\documentclass[11pt]{amsart}

\usepackage{amssymb,amsmath,amscd,amsfonts,a4,psfrag,graphicx,
latexsym,epsfig,color}
\usepackage[active]{srcltx}
\usepackage[all]{xy}

\everymath{\displaystyle}

\newtheorem{teoIntro}{Theorem}
\newtheorem{coroIntro}[teoIntro]{Corollary}

\newtheorem{teo}{Theorem}[section]
\newtheorem{lem}[teo]{Lemma}
\newtheorem{cor}[teo]{Corollary}
\newtheorem{exa}[teo]{Example}
\newtheorem{prop}[teo]{Proposition}
\newtheorem{defi}[teo]{Definition}

\newtheorem{remark}[teo]{Remark}

\newcommand{\flecheIso}[4]{                     
            \begin{array}{ccll} #1 & \overset{\sim}{\longrightarrow} & #2 \\   %
                         #3 &\longmapsto & #4          %
            \end{array}}

\newcommand{\fonc}[5]{                     
            \begin{array}{rcll}#1 :& #2 & \longrightarrow & #3 \\   %
                         &#4 &\longmapsto & #5          %
            \end{array}}

\allowdisplaybreaks 

\newcommand{\mc}{\mathbb{C}}
\newcommand{\mz}{\mathbb{Z}}

\newcommand{\mn}{\mathbb{N}}

\newcommand{\Ll}{{\mathcal L}}

\newcommand{\Oo}{{\mathcal O}}

\newcommand{\lra}{\longrightarrow}
\newcommand{\ra}{\rightarrow}

\newcommand{\Endproof}{\hfill$\Box$}

\newcommand{\smallblacktriangle}{%
  \vcenter{\hbox{\,\scalebox{0.75}{$\blacktriangleright$}\,}}%
}

\title[quantum moduli and $\mathfrak{g}$-skein algebras]{Noetherian and Affine properties of \\ quantum moduli and $\mathfrak{g}$-skein algebras}

\author[St\'ephane Baseilhac, Matthieu Faitg, Philippe Roche]{St\'ephane Baseilhac\textsuperscript{1}, Matthieu Faitg\textsuperscript{2}, Philippe Roche\textsuperscript{3}}

\begin{document}


\maketitle

\begin{center}
\textsuperscript{1,3}IMAG, Univ Montpellier, CNRS, Montpellier, France.

\smallskip

\textsuperscript{2}Institut de Mathématiques de Toulouse, 118 route de Narbonne, F-31062 Toulouse, France.

\medskip

\textsuperscript{1}stephane.baseilhac@umontpellier.fr, \textsuperscript{2}matthieu.faitg@univ-tlse3.fr, \textsuperscript{3}philippe.roche@umontpellier.fr
\end{center}

\begin{abstract} We prove that the quantum moduli algebra associated to a possibly punctured compact oriented surface and a complex semisimple Lie algebra $\mathfrak{g}$ is a Noetherian and finitely generated ring. If the surface has punctures, we prove also that it has no non-trivial zero divisors (i.e., it is a domain). Moreover, we show that the quantum moduli algebra is isomorphic to the skein algebra of the surface, defined by means of the Reshetikhin-Turaev functor for the quantum group $U_q(\mathfrak{g})$, and which coincides with the Kauffman bracket skein algebra when $\mathfrak{g}=\mathfrak{sl}_2$. We obtain these results by a similar study of quantum graph algebras, which we show to be isomorphic to stated skein algebras.
\end{abstract} 
\medskip



\tableofcontents

\section{Introduction}\label{INTRO}
\subsection{Context} In this section we present qualitative features of our results, and relations with connected results in the previous literature. In Section \ref{summary} we give precise statements and discuss our methods in details. 

The quantum moduli algebras have been introduced in the mid $'90$s by Alekseev-Grosse-Schomerus \cite{AGS1,AGS2,AS} and Buffenoir-Roche \cite{BuR1,BuR2}. They are associative, non commutative algebras defined over the ground field $\mc(q^{1/D})$, quantizing the algebras of functions on the moduli spaces of flat $\mathfrak{g}$-connections on surfaces, where $q$ is a formal variable, $\mathfrak{g}$ is a complex finite dimensional simple Lie algebra, and $D$ is the smallest positive integer such that $DP\subset Q$, where $P$, $Q$ are the weight and root lattices of $\mathfrak{g}$ (see \S \ref{prelimLieAlgebras} for the generalization to semisimple Lie algebras).

The main purpose of this paper is to prove that, for every $\mathfrak{g}$ and every finite type surface, the quantum moduli algebra is a Noetherian and finitely generated ring, with no non-trivial zero divisors when the surface has punctures (i.e. it is a domain). Moreover, it is isomorphic to a $\mathfrak{g}$-skein algebra. The resulting properties of skein algebras were previously known only in the case $\mathfrak{g}=\mathfrak{sl}_2$ (\cite{Bu}, \cite{PS}), and \cite{FrohmanSikora} proved finite generation for $\mathfrak{g}=\mathfrak{sl}_3$. For stated skein algebras, again when $\mathfrak{g}=\mathfrak{sl}_2$, this was proved in \cite{LY}. All these papers use presentations by generators and relations, together with geometric techniques based on curves or graphs on surfaces, and hardly seem to to be generalizable to Lie algebras $\mathfrak{g}$ of higher rank. On the contrary,  our approach based on quantum moduli algebras allows one to prove the results for all cases of $\mathfrak{g}$, and does it in a uniform, intrinsic way. It provides also a direct, fruitful connection with quantum group theory.

Denote by $\Sigma_{g,n}$ the oriented surface with genus $g$ and $n$ punctures ($n\in \mn$), and by $\Sigma_{g,n}^\circ$ the surface with one boundary component, obtained from $\Sigma_{g,n}$ by removing an open $2$-disk (these surfaces are pictured in \S \ref{sectionStatedSkein}). To $\mathfrak{g}$ and $\Sigma_{g,n}^\circ$ one can associate an algebra $\mathcal{L}_{g,n}(\mathfrak{g})$ (see Definition \ref{def:Lgn}), called {\it graph algebra}, which is endowed with an action of the Drinfeld-Jimbo (simply-connected) quantum group $U_q = U_q(\mathfrak{g})$ making $\mathcal{L}_{g,n}(\mathfrak{g})$ a module-algebra. The subalgebra of invariant elements $\mathcal{L}_{g,n}^{U_q}(\mathfrak{g})$ is the {\it quantum moduli algebra} of $\mathfrak{g}$ and $\Sigma_{g,n}^\circ$. The case $g=0$ was studied in \cite{BR1,BR2}. 

There are at least two areas of good motivations to study the quantum moduli algebras. On one hand, the definition of $\mathcal{L}_{g,n}(\mathfrak{g})$ relies on Hopf algebra and quantum group theory. It is built by a twisting procedure from $2g+n$ copies of the quantum coordinate algebra $\mathcal{O}_q(G)$ associated to a complex semisimple algebraic group $G$ with Lie algebra $\mathfrak{g}$, and it is therefore a natural problem to study the algebraic structure and representation theory of $\mathcal{L}_{g,n}(\mathfrak{g})$ and $\mathcal{L}_{g,n}^{U_q}(\mathfrak{g})$. 

Our first main result makes progresses in this direction. We show that $\mathcal{L}_{g,n}(\mathfrak{g})$ and $\mathcal{L}_{g,n}^{U_q}(\mathfrak{g})$ are Noetherian, finitely generated rings, and are domains. Such properties are well-known for $\mathcal{O}_q(G)$ and the subalgebra of invariant elements under the coadjoint action of $U_q(\mathfrak{g})$ (see e.g. \cite[Prop. 3.117]{VY}). In our context we use also tools, like filtrations, which are standard for quantum groups, and a version of the Hilbert--Nagata theorem in invariant theory. The similar result when $g=0$ was obtained in \cite{BR2}. The present genus $g>0$ case is substantially more complicated.

On the other hand, connections with quantum topology are well-established since the $'90$s. For instance, \cite{ASduke} showed that the Witten-Reshetikin-Turaev representations of the mapping class groups of surfaces can be recovered from certain representations of $\mathcal{L}_{g,n}(\mathfrak{g})$. Also, \cite{BFK, BFK2} proved (with a slightly different formalism) that $\mathcal{L}_{g,n}^{U_q}(\mathfrak{sl}_2)$ is isomorphic to the Kauffman bracket skein algebra of $\Sigma_{g,n}^\circ$, and recently these results have been extended in \cite{Faitg3,FaitgMCG} and \cite {FaitgHol, korinman} respectively. It is a natural problem to extend the isomorphism of $\mathcal{L}_{g,n}^{U_q}(\mathfrak{g})$ with a skein algebra beyond $\mathfrak{g} = \mathfrak{sl}_2$. By using such an isomorphism one can naturally expect that $\mathcal{L}_{g,n}(\mathfrak{g})$, which has an algebraic flavour by definition, provides good tools to study the skein algebras.

Our second main result achieves this goal: we construct an (explicit) isomorphism of $\mathcal{L}_{g,n}(\mathfrak{g})$ with an algebra of ``stated ribbon graphs'' in the thickened surface $\Sigma_{g,n}^\circ \times [0,1]$. It restricts to an isomorphism between $\mathcal{L}_{g,n}^{U_q}(\mathfrak{g})$ and the skein algebra of the surface $\Sigma_{g,n}^{\circ}$ (whose skein relations are given by the Reshetikin--Turaev functor for $U_q(\mathfrak{g})$). 

\indent Combining our two main results we see that the skein algebra of $\Sigma_{g,n}^\circ$ associated to an arbitrary semisimple Lie algebra $\mathfrak{g}$ is a Noetherian and finitely generated domain. 

Finally, the quantum moduli algebra of $\Sigma_{g,n}$, including the case $n=0$, can be defined by using the notion of quantum reduction developed by several authors and applied to $\mathcal{L}_{g,n}(\mathfrak{g})$. The topological counterpart of this notion is the operation of gluing a $2$-disk along the boundary component of $\Sigma_{g,n}^\circ$. We describe quantum reduction in detail, especially for $\mathcal{L}_{g,n}(\mathfrak{g})$, and show that the resulting algebra is a Noetherian and finitely generated ring, indeed isomorphic to the skein algebra of $\Sigma_{g,n}$. For closed surfaces $(n=0)$, whether or not it is a domain is still an open question at this stage for general $\mathfrak{g}$; it is the case for $\mathfrak{g} =\mathfrak{sl}_2$ (see \cite{PS}).

\smallskip

We can formulate most of our constructions for general quasitriangular Hopf algebras $H$, thus obtaining an $H$-module algebra $\mathcal{L}_{g,n}(H)$ and a subalgebra of $H$-invariant elements $\mathcal{L}_{g,n}^H(H)$. We do so in the text, and then make the required adaptation to handle the case of $U_q(\mathfrak{g})$, which is quasitriangular only in a categorical completion. 

\smallskip

Another approach to $\mathcal{L}_{g,n}(H)$ is \textit{via} factorization homology. The seminal paper is \cite{BBJ}, where it is in particular shown that $\textstyle \int_{\Sigma_{g,n}^{\circ}}\!\!H\text{-}\mathrm{mod} \cong A_{g,n}\text{-}\mathrm{mod}_H$ for a $H$-module-algebra $A_{g,n}$ which they proved isomorphic to $\mathcal{L}_{g,n}(H)$. Factorization homology can be realized by skein categories \cite{cooke}, where $A_{g,n}$ takes the name of ``internal skein algebra''. For $H = U^{\mathrm{ad}}_q(\mathfrak{sl}_2)$, it is known that internal skein algebras are isomorphic to stated skein algebras \cite{haioun}, and this is thought to be true for all $U_q^{\mathrm{ad}}(\mathfrak{g})$. Also, for $H = U_q^{\mathrm{ad}}(\mathfrak{g})$, \cite[Cor.\,4]{cooke} asserts that $A_{g,n}^{H\text{-}\mathrm{inv}}$ is isomorphic to a skein algebra whose definition is equivalent to the one used here. A survey on the relations between the different approaches is in \cite[\S 4.6]{JordanSurvey}.

\smallskip

We note that the algebras $\mathcal{L}_{g,n}(\mathfrak{g})$ and $\mathcal{L}_{g,n}^{U_q}(\mathfrak{g})$ have integral forms, which are subalgebras $\mathcal{L}_{g,n}^A(\mathfrak{g})$ and $(\mathcal{L}_{g,n}^A)^{U_A^{\mathrm{res}}}(\mathfrak{g})$ defined over the ground ring $A=\mc[q^{1/D},q^{-1/D}]$, and such that $\mathcal{L}_{g,n}^A(\mathfrak{g}) \otimes_{A} \mc(q^{1/D}) = \mathcal{L}_{g,n}(\mathfrak{g})$ and $(\mathcal{L}_{g,n}^A)^{U_A^{\mathrm{res}}}(\mathfrak{g}) \otimes_{A} \mc(q^{1/D}) = \mathcal{L}_{g,n}^{U_q}(\mathfrak{g})$, where $U_A^{\mathrm{res}}$ is Lusztig's restricted quantum group associated to the adjoint quantum group $U_q^\mathrm{ad}(\mathfrak{g})$. It is because of integrality properties of the $R$-matrix of $U_A^{\mathrm{res}}$ that the twists involved in the definition of $\mathcal{L}_{g,n}(\mathfrak{g})$ yield a well-defined algebra structure on $\mathcal{L}_{g,n}^A(\mathfrak{g})$. 

Trivially our first main result implies that $\mathcal{L}_{g,n}^A(\mathfrak{g})$ and $(\mathcal{L}_{g,n}^A)^{U_A^{\mathrm{res}}}(\mathfrak{g})$ are domains. By using the Kashiwara-Lusztig theory of canonical basis, we have shown in \cite{BR2} that $\mathcal{L}_{0,n}^A(\mathfrak{g})$ is a Noetherian and finitely generated algebra. We expect that these results still hold true in genus $g > 0$. 
\smallskip

In \cite{BFR} we study the algebraic properties and representations of these algebras when the parameter $q$ is specialized to a root of unity. As a further direction of research, we note that the results of this paper should allow one to study stratifications of the prime and primitive spectra of the skein algebras, similarly as those of quantized coordinate rings described e.g. in \cite{BG}, Part II.

\subsection{Summary of results}\label{summary}
\indent We first give an overview of the definition of $\mathcal{L}_{g,n}(H)$. Let $H$ be a quasitriangular Hopf algebra with $R$-matrix $R \in H^{\otimes 2}$, and let $H^{\circ}$ be its restricted dual (see \S \ref{sectionPreliminaires}). For $g,n \in \mathbb{N}$, the algebra $\mathcal{L}_{g,n}(H)$ is the vector space $(H^{\circ})^{\otimes (2g+n)}$ with a product ``twisted by $R$''. There is a coadjoint action $\mathrm{coad}^r$ on $\mathcal{L}_{g,n}(H)$, which gives it the structure of a right $H$-module-algebra (\S \ref{sectionDefLgnH}). In particular we have the subalgebra $\mathcal{L}_{g,n}^H(H)$ of $H$-invariant elements for this action.

\smallskip

\indent The definitions and some results will be given for general $H$ but we are mainly interested in the case where $H$ is a quantum group. So let $\mathfrak{g}$ be a complex simple Lie algebra and $G$ be the simply-connected algebraic Lie group with Lie algebra $\mathfrak{g}$; for the extension to semisimple $\mathfrak{g}$, see \S \ref{prelimLieAlgebras}. Fix a formal variable $q$ and denote by $U_q = U_q(\mathfrak{g})$ the simply connected Drinfeld-Jimbo quantum group defined over $\mathbb{C}(q)$, by $U_q^{\mathrm{ad}}\subset U_q$ the adjoint quantum group, and by $\mathcal{O}_q = \mathcal{O}_q(G)$ the associated quantized coordinate algebra (see \S \ref{sectionPrelimUq}, \S \ref{sectionOq}). In this situation we denote the resulting algebra by $\mathcal{L}_{g,n}(\mathfrak{g})$ or simply $\mathcal{L}_{g,n}$, which is $\mathcal{O}_q(G)^{\otimes (2g+n)}$ as a $\mathbb{C}(q^{1/D})$-vector space (see \S \ref{prelimLieAlgebras} for the definition of the integer $D$), and by $\mathcal{L}_{g,n}^{U_q}(\mathfrak{g})$ or simply $\mathcal{L}_{g,n}^{U_q}$ the subalgebra of $U_q$-invariant elements. 

\smallskip

\indent The definition of $\mathcal{L}_{g,n}(H)$ relies on the special cases $\mathcal{L}_{0,1}(H)$ and $\mathcal{L}_{1,0}(H)$:
\begin{equation}\label{defLgnIntroBraidedTensor}
\mathcal{L}_{g,n}(H) = \mathcal{L}_{1,0}(H)^{\widetilde{\otimes} g} \: \widetilde{\otimes} \: \mathcal{L}_{0,1}(H)^{\widetilde{\otimes} n}
\end{equation}
where $\widetilde{\otimes}$ is the braided tensor product in the braided tensor category of right $H$-modules, as defined in \cite[Lem 9.2.12]{majidFoundations} and recalled before Proposition  \ref{propBraidedTensProduct}. Hence the algebras $\mathcal{L}_{0,1}(H)$ and $\mathcal{L}_{1,0}(H)$ play a special role and have to be examined first. The papers \cite{BR1, BR2} were focused on $g=0$ and in particular $\mathcal{L}_{0,1}(H)$. Here we are interested in $\mathcal{L}_{g,n}(H)$ for arbitrary $g$. So we start with $\mathcal{L}_{1,0}(H)$ in \S \ref{sectionHandleAlgebra}. This algebra is very different from $\mathcal{L}_{0,1}(H)$; for instance $\mathcal{L}_{0,1}(H)$ is strongly related to $H$ while $\mathcal{L}_{1,0}(H)$ is strongly related to the Heisenberg double of $H^\circ$ (\S \ref{sectionHeisenberg}, \S \ref{SectionL10ForUq}).

\smallskip

\indent Here is our first main result for the algebra $\mathcal{L}_{g,n} := \mathcal{L}_{g,n}(\mathfrak{g})$:
\begin{teoIntro}[Theorems \ref{ThmLgnNoetherian}, \ref{thmLgnUqFinGen}, \ref{TheoremePhignInjectif}] \label{thmNoethIntro}
1. The algebra $\mathcal{L}_{g,n}$ is a Noetherian domain.
\\2. The algebra $\mathcal{L}_{g,n}^{U_q}$ is Noetherian and finitely generated.
\end{teoIntro}
\noindent We note that it is not difficult to prove that $\mathcal{L}_{g,n}$ is finitely generated (Prop. \ref{propLgnFinGen}). Moreover, it follows from item 1 that $\mathcal{L}_{g,n}^{U_q}$ is a domain. 
\smallskip

Let us discuss a bit how we prove Theorem \ref{thmNoethIntro}. The main ideas of the proof are similar to those for $\mathcal{L}_{0,n}$ in \cite{BR1,BR2}; however when $g>0$ the presence of the algebra $\mathcal{L}_{1,0}$ requires many non-trivial generalizations and new computations.

\smallskip

\indent To prove that $\mathcal{L}_{g,n}$ is Noetherian, we use filtrations. A filtration for $\mathcal{L}_{0,1}$ was introduced in \cite[\S 3.14.4]{VY}. In \S \ref{sectionL10Noetherien} we modify it in order to define a filtration of $\mathcal{L}_{1,0}$ and show that the associated graded algebra is Noetherian, which implies that $\mathcal{L}_{1,0}$ is Noetherian. Then in \S \ref{sectionNoetherianityLgn} we first define a filtration of $\mathcal{L}_{g,n}$ whose associated graded algebra has a kind of ``quasi-commutative" product, in the sense of Lemma \ref{critereNoetherien}. This allows us in a second step to use tensor products of the filtrations of $\mathcal{L}_{0,1}$ and $\mathcal{L}_{1,0}$ to get an associated graded algebra which is Noetherian.

\smallskip

\indent The proof of the item 2 of Theorem \ref{thmNoethIntro} is based on a generalization of the Hilbert--Nagata theorem. Let $G$ be a group acting on a graded algebra $A$ in such a way that the action is compatible with the multiplication and the grading. The Hilbert--Nagata theorem gives sufficient conditions for the subalgebra of $G$-invariant elements of $A$ to be Noetherian and finitely generated. In \S \ref{sectionNoetherianityLgnInv} we generalize this theorem to the case where $A$ is a graded module-algebra over a Hopf algebra $H$. We then apply this general result to the case where $H = U_q^{\mathrm{ad}}$ and $A$ is a ``graded truncation'' of $\mathcal{L}_{g,n}$.

\smallskip

\indent We prove that $\mathcal{L}_{g,n}$ does not have non-trivial zero divisors by using a morphism $\Phi_{g,n}$ called the Alekseev morphism, which has many applications and that we discuss now. In \S \ref{sectionAlekseevMorphism}, relying on the formulas given in \cite{A}, we define for any quasitriangular Hopf algebra $H$ a morphism of algebras
\[ \Phi_{g,n} : \mathcal{L}_{g,n}(H) \to \mathcal{HH}(H^{\circ})^{\otimes g} \otimes H^{\otimes n}. \] 
Here $\mathcal{HH}(H^{\circ})$ is the ``two-sided Heisenberg double'', an algebra which we introduce in \S \ref{sectionTwoSidedHeisenberg} and which extends the usual Heisenberg double $\mathcal{H}(H^{\circ})$ recalled in \S \ref{sectionHeisenberg}. For $g>1$ it is necessary to use $\mathcal{HH}(H^{\circ})$ instead of $\mathcal{H}(H^{\circ})$ to make sense of the formulas in \cite{A} if $H$ is not finite-dimensional (this point is explained at the end of \S \ref{sectionDefAlekseev}). It is important to define and analyze $\Phi_{0,1}$ and $\Phi_{1,0}$ first since their properties are used in the proofs for $\Phi_{g,n}$. The morphism $\Phi_{0,1}$ is well-known in quantum group theory (see e.g. \cite{Bau1}), while $\Phi_{1,0} : \mathcal{L}_{1,0}(H) \to \mathcal{H}(H^{\circ})$ is defined in \S \ref{sectionHeisenberg} following \cite{A}.

\smallskip

\indent Let us return to the case $H = U_q^{\mathrm{ad}}(\mathfrak{g})$. In that situation $\Phi_{g,n}$ takes values in $\mathcal{HH}(\mathcal{O}_q)^{\otimes g} \otimes U_q^{\otimes n}$. We prove that $\mathcal{HH}(\mathcal{O}_q)^{\otimes g} \otimes U_q^{\otimes n}$ does not contain non-trivial zero divisors (Prop. \ref{propHHqUqIntegre}) and that $\Phi_{g,n}$ is injective (Th. \ref{TheoremePhignInjectif}), which implies that $\mathcal{L}_{g,n}$ does not have non-trivial zero divisors. In \cite{Bau1} it was already proved that $\Phi_{0,1} : \mathcal{L}_{0,1} \to U_q$ is injective; moreover, that paper showed that the image of $\Phi_{0,1}$ is $U_q^{\mathrm{lf}}$, namely the subspace of locally finite elements of $U_q$ for the adjoint action. In Theorem \ref{ThmPhi10} we show a similar result for $\Phi_{1,0}$, that is, $\Phi_{1,0} : \mathcal{L}_{1,0} \to \mathcal{H}(\mathcal{O}_q)$ is injective and its image is the subspace of locally finite elements for an action of $U_q$ on $\mathcal{H}(\mathcal{O}_q)$ that we introduce in \S \ref{sectionHeisenberg}.

\medskip

\indent In \S \ref{sectionTopologicalInterpretation} we relate the algebra $\mathcal{L}_{g,n}(H)$ and its subalgebra of $H$-invariant elements $\mathcal{L}_{g,n}^H(H)$ to skein theory. Let $H$ be a ribbon Hopf algebra over a field $k$, and let $F_{\mathrm{RT}} : \mathrm{Rib}_H \to H\text{-mod}$ be the Reshetikhin--Turaev functor, which to an $H$-colored oriented ribbon graph associates some $H$-linear morphism \cite{RT}. Let $\Sigma$ be an oriented surface, possibly with boundary. The skein algebra of $\Sigma$ associated to $H$, denoted by $\mathcal{S}_H(\Sigma)$, is the $k$-vector space generated by the isotopy classes of $H$-colored oriented ribbon links (with coupons) modulo the skein relations:
\begin{center}
\begingroup%
  \makeatletter%
  \providecommand\color[2][]{%
    \errmessage{(Inkscape) Color is used for the text in Inkscape, but the package 'color.sty' is not loaded}%
    \renewcommand\color[2][]{}%
  }%
  \providecommand\transparent[1]{%
    \errmessage{(Inkscape) Transparency is used (non-zero) for the text in Inkscape, but the package 'transparent.sty' is not loaded}%
    \renewcommand\transparent[1]{}%
  }%
  \providecommand\rotatebox[2]{#2}%
  \newcommand*\fsize{\dimexpr\f@size pt\relax}%
  \newcommand*\lineheight[1]{\fontsize{\fsize}{#1\fsize}\selectfont}%
  \ifx\svgwidth\undefined%
    \setlength{\unitlength}{303.53577343bp}%
    \ifx\svgscale\undefined%
      \relax%
    \else%
      \setlength{\unitlength}{\unitlength * \real{\svgscale}}%
    \fi%
  \else%
    \setlength{\unitlength}{\svgwidth}%
  \fi%
  \global\let\svgwidth\undefined%
  \global\let\svgscale\undefined%
  \makeatother%
  \begin{picture}(1,0.12354393)%
    \lineheight{1}%
    \setlength\tabcolsep{0pt}%
    \put(-0.00054172,0.05317266){\color[rgb]{0,0,0}\makebox(0,0)[lt]{\lineheight{1.25}\smash{\begin{tabular}[t]{l}$\displaystyle\sum_i \lambda_iF_{\mathrm{RT}}(T_i) =  0 \quad \implies \quad \sum_i \lambda_i$\end{tabular}}}}%
    \put(0.70188301,0.0530706){\color[rgb]{0,0,0}\makebox(0,0)[lt]{\lineheight{1.25}\smash{\begin{tabular}[t]{l}$= 0$ in $\mathcal{S}_H(\Sigma)$.\end{tabular}}}}%
    \put(0,0){\includegraphics[width=\unitlength,page=1]{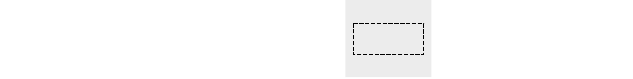}}%
    \put(0.59610599,0.05429781){\color[rgb]{0,0,0}\makebox(0,0)[lt]{\lineheight{1.25}\smash{\begin{tabular}[t]{l}$T_i$\end{tabular}}}}%
    \put(0.59275272,0.01031215){\color[rgb]{0,0,0}\makebox(0,0)[lt]{\lineheight{1.25}\smash{\begin{tabular}[t]{l}$\ldots$\end{tabular}}}}%
    \put(0.59306857,0.11005255){\color[rgb]{0,0,0}\makebox(0,0)[lt]{\lineheight{1.25}\smash{\begin{tabular}[t]{l}$\ldots$\end{tabular}}}}%
    \put(0,0){\includegraphics[width=\unitlength,page=2]{skeinRelationsIntro.pdf}}%
  \end{picture}%
\endgroup%

\end{center}
The $T_i$ are any ribbon graphs and the $\lambda_i \in k$ are any scalars such that the linear equation on the left holds. The right hand-side represents a linear combination of links which are equal outside of the cube in $\Sigma \times [0,1]$ which is depicted in grey. The product of two links $L_1, L_2$ in $\mathcal{S}_H(\Sigma)$ is obtained by putting $L_1$ below $L_2$ in $\Sigma \times [0,1]$. 

\smallskip

\indent The stated skein algebra $\mathcal{S}_H^{\mathrm{st}}(\Sigma)$ is a generalization of $\mathcal{S}_H(\Sigma)$ where one uses ribbon graphs in $\Sigma \times [0,1]$ instead of links. The endpoints of these ribbon graphs are required to be in $\partial\Sigma \times [0,1]$ and are labelled by ``states'' \textit{i.e.} vectors in $H$-modules. Stated skein algebras have been introduced and studied for $H = U_q^{\mathrm{ad}}(\mathfrak{sl}_2)$ in \cite{Le, CL, korinman} and for $H = U_q^{\mathrm{ad}}(\mathfrak{sl}_{m+1})$ in \cite{LS}. In the present paper we will deal with stated skeins only for the surface $\Sigma = \Sigma_{g,n}^{\circ, \bullet}$ obtained by removing one point $\bullet$ on the circle $\partial(\Sigma_{g,n}^{\circ})$, and $\mathcal{S}_H^{\mathrm{st}}(\Sigma_{g,n}^{\circ, \bullet})$ is defined in \S \ref{sectionStatedSkein} for any ribbon Hopf algebra $H$. Our definition agrees with the one which will be given in \cite{CKL} for more general surfaces than $\Sigma_{g,n}^{\circ, \bullet}$.
\smallskip

\indent In \cite[\S 4.1]{FaitgHol} a ``holonomy map'' hol was defined, which to a $H$-colored oriented ribbon graph $\mathbf{T}$ in $(\Sigma_{g,n}^{\circ,\bullet}) \times [0,1]$ associates a tensor $\mathrm{hol}(\mathbf{T})$ with coefficients in $\mathcal{L}_{g,n}(H)$. The type of the tensor $\mathrm{hol}(\mathbf{T})$ depends of the orientation and the number of endpoints of $\mathbf{T}$. The map hol generalizes the Reshetikhin--Turaev functor to the surfaces $\Sigma_{g,n}^{\circ,\bullet}$. In \S \ref{sectionIsoSstLgn} we refine it to a ``stated holonomy map'' $\mathrm{hol}^{\mathrm{st}} : \mathcal{S}_H^{\mathrm{st}}(\Sigma_{g,n} ^{\circ, \bullet}) \to \mathcal{L}_{g,n}(H)$, which is a morphism of algebras due to the properties of hol. Then in \S \ref{sectionIsoSLgninv} we note that there is a natural algebra morphism $I : \mathcal{S}_H(\Sigma_{g,n}^{\circ}) \to \mathcal{S}_H^{\mathrm{st}}(\Sigma_{g,n}^{\circ, \bullet})$ simply obtained by seeing a link as a ribbon graph without boundary points. When restricted to links, $\mathrm{hol}$ and $\mathrm{hol}^{\mathrm{st}}$ are equal and give a morphism $W = \mathrm{hol}^{\mathrm{st}} \circ I : \mathcal{S}_H(\Sigma_{g,n}^{\circ}) \to \mathcal{L}_{g,n}^H(H)$  (\S \ref{sectionIsoSLgninv}). In this way we recover the ``Wilson loop map'' $W$ already defined and studied in \cite{BuR2, BFK} (the latter paper also defined holonomy for certain tangles called $q$-nets).

Consider the following conditions on $H$:
\begin{align*}
\begin{split}
\bullet &\text{ the base field of } H \text{ is algebraically closed and } H\text{-}\mathrm{mod} \text{ is semisimple,}\\
\bullet & \text{ or } H = U_q^{\mathrm{ad}}(\mathfrak{g}) \text{ over } \mathbb{C}(q),\\
\bullet & \text{ or } H \text{ is finite-dimensional.}
\end{split}
\end{align*}
Our second main result is:

\begin{teoIntro}[Theorems \ref{ThStatedHolIso} and \ref{thWilsonIso}]\label{thmSkeinIntro}
When $H$ satisfies any one of the above conditions, we have:
\\1. $\mathrm{hol}^{\mathrm{st}} : \mathcal{S}_H^{\mathrm{st}}(\Sigma_{g,n}^{\circ,\bullet}) \to \mathcal{L}_{g,n}(H)$ is an isomorphism of algebras.
\\2. If the category $H\text{-}{\rm mod}$ of finite dimensional $H$-modules is semisimple, then $W: \mathcal{S}_H(\Sigma_{g,n}^{\circ}) \to \mathcal{L}_{g,n}^H(H)$ is an isomorphism of algebras.
\end{teoIntro}
\noindent This result applies in particular for $H = U_q^{\mathrm{ad}}(\mathfrak{g})$. It is a generalization of \cite[\S 5]{FaitgHol} for the first item and of \cite{BFK2} (see \cite[\S 8.2]{BR1} in our context) for the second item, which proved it for $H = U_q^{\mathrm{ad}}(\mathfrak{sl}_2)$. If $H\text{-}\mathrm{mod}$ is not semisimple the second item fails because in general $\mathcal{L}_{g,n}^H(H)$ is bigger than $\mathrm{im}(W)$, as explained in \S \ref{sectionRemarksSemisimplicity}. To sum up:

\medskip

\begin{center}
\begingroup%
  \makeatletter%
  \providecommand\color[2][]{%
    \errmessage{(Inkscape) Color is used for the text in Inkscape, but the package 'color.sty' is not loaded}%
    \renewcommand\color[2][]{}%
  }%
  \providecommand\transparent[1]{%
    \errmessage{(Inkscape) Transparency is used (non-zero) for the text in Inkscape, but the package 'transparent.sty' is not loaded}%
    \renewcommand\transparent[1]{}%
  }%
  \providecommand\rotatebox[2]{#2}%
  \newcommand*\fsize{\dimexpr\f@size pt\relax}%
  \newcommand*\lineheight[1]{\fontsize{\fsize}{#1\fsize}\selectfont}%
  \ifx\svgwidth\undefined%
    \setlength{\unitlength}{361.24627995bp}%
    \ifx\svgscale\undefined%
      \relax%
    \else%
      \setlength{\unitlength}{\unitlength * \real{\svgscale}}%
    \fi%
  \else%
    \setlength{\unitlength}{\svgwidth}%
  \fi%
  \global\let\svgwidth\undefined%
  \global\let\svgscale\undefined%
  \makeatother%
  \begin{picture}(1,0.19867145)%
    \lineheight{1}%
    \setlength\tabcolsep{0pt}%
    \put(0.78095327,0.17035836){\color[rgb]{0,0,0}\makebox(0,0)[lt]{\lineheight{1.25}\smash{\begin{tabular}[t]{l}$\mathcal{L}_{g,n}(H)$\end{tabular}}}}%
    \put(0.78134638,0.02594336){\color[rgb]{0,0,0}\makebox(0,0)[lt]{\lineheight{1.25}\smash{\begin{tabular}[t]{l}$\mathcal{L}_{g,n}^H(H)$\end{tabular}}}}%
    \put(0.83819244,0.10038505){\color[rgb]{0,0,0}\makebox(0,0)[lt]{\lineheight{1.25}\smash{\begin{tabular}[t]{l}(subalgebra)\end{tabular}}}}%
    \put(0,0){\includegraphics[width=\unitlength,page=1]{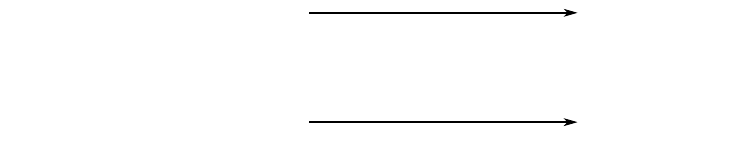}}%
    \put(0.57264788,0.18971962){\color[rgb]{0,0,0}\makebox(0,0)[lt]{\lineheight{1.25}\smash{\begin{tabular}[t]{l}$\sim$\end{tabular}}}}%
    \put(0.44025275,0.15182417){\color[rgb]{0,0,0}\makebox(0,0)[lt]{\lineheight{1.25}\smash{\begin{tabular}[t]{l}Stated holonomy $\mathrm{hol}^{\mathrm{st}}$\end{tabular}}}}%
    \put(0.4601502,0.04749758){\color[rgb]{0,0,0}\makebox(0,0)[lt]{\lineheight{1.25}\smash{\begin{tabular}[t]{l}$\sim$ {\footnotesize (if $H$-mod semisimple)}\end{tabular}}}}%
    \put(0.45451384,0.00245082){\color[rgb]{0,0,0}\makebox(0,0)[lt]{\lineheight{1.25}\smash{\begin{tabular}[t]{l}Wilson loop map $W$\end{tabular}}}}%
    \put(0.57524315,0.09669986){\color[rgb]{0,0,0}\makebox(0,0)[lt]{\lineheight{1.25}\smash{\begin{tabular}[t]{l}$\circlearrowleft$\end{tabular}}}}%
    \put(-0.00031581,0.17038305){\color[rgb]{0,0,0}\makebox(0,0)[lt]{\lineheight{1.25}\smash{\begin{tabular}[t]{l}Stated skein algebra $\mathcal{S}_H^{\mathrm{st}}(\Sigma_{g,n}^{\circ,\bullet})$\end{tabular}}}}%
    \put(0.09324317,0.02790869){\color[rgb]{0,0,0}\makebox(0,0)[lt]{\lineheight{1.25}\smash{\begin{tabular}[t]{l}Skein algebra $\mathcal{S}_H(\Sigma_{g,n}^{\circ})$\end{tabular}}}}%
    \put(0,0){\includegraphics[width=\unitlength,page=2]{diagramme_hol_iso.pdf}}%
    \put(0.29676866,0.088588){\color[rgb]{0,0,0}\makebox(0,0)[lt]{\lineheight{1.25}\smash{\begin{tabular}[t]{l}$I$\end{tabular}}}}%
  \end{picture}%
\endgroup%

\end{center}

\smallskip

\noindent When $H\text{-}\mathrm{mod}$ is semisimple we get in particular that the morphism $I$ is injective (Corollary \ref{coroISkeinInjective}). 
\medskip

\indent Take $H = U_q^{\mathrm{ad}}(\mathfrak{g})$. The category $\mathcal{C}$ of type $1$ finite dimensional $U_q^{\mathrm{ad}}(\mathfrak{g})$-modules is semisimple. Write $\mathcal{S}^{\mathrm{st}}_{\mathfrak{g}}(\Sigma_{g,n}^{\circ,\bullet})$, $\mathcal{S}_{\mathfrak{g}}(\Sigma_{g,n}^{\circ})$ and $\mathcal{S}_{\mathfrak{g}}(\Sigma_{g,n})$ for the stated skein and skein algebras of the respective surfaces $\Sigma_{g,n}^{\circ,\bullet}$, $\Sigma_{g,n}^{\circ}$ and $\Sigma_{g,n}$, where the ribbon graphs are colored by objects and morphisms in $\mathcal{C}$. They are $\mathbb{C}(q^{1/D})$-algebras. In particular, for $\mathfrak{g} = \mathfrak{sl}_2(\mathbb{C})$ these algebras are respectively the Kauffman bracket skein algebra and the stated skein algebra of \cite{Le, CL}. Also, in the particular case of $\mathfrak{g} = \mathfrak{sl}_{m+1}(\mathbb{C})$, they can be described in terms of special ribbon graphs called webs \cite{sikora, LS}.

\smallskip

\indent The conjunction of Theorems \ref{thmNoethIntro} and \ref{thmSkeinIntro} finally gives:
\begin{coroIntro}\label{coroSkeinIntro}
1. The stated skein algebra $\mathcal{S}^{\mathrm{st}}_{\mathfrak{g}}(\Sigma_{g,n}^{\circ,\bullet})$ is a finitely generated Noetherian domain.
\\2. The skein algebra $\mathcal{S}_{\mathfrak{g}}(\Sigma_{g,n}^{\circ})$ is a finitely generated Noetherian domain.
\end{coroIntro}

In \S \ref{qredsection} we extend our results to the surface $\Sigma_{g,n}$. It is a fact that the skein algebra $\mathcal{S}_H(\Sigma_{g,n})$ is a quotient of $\mathcal{S}_H(\Sigma_{g,n}^\circ)$ (see \S \ref{topqrsection}). Then the item 2 of Corollary \ref{coroSkeinIntro} implies:
\begin{coroIntro} The skein algebra $\mathcal{S}_{\mathfrak{g}}(\Sigma_{g,n})$ is finitely generated and Noetherian.
\end{coroIntro}

We consider the quantum reduction $\mathcal{L}_{g,n}^{\mathrm{qr}}(H)$ of $\mathcal{L}_{g,n}(H)$ associated to the counit $\varepsilon\colon H\ra k$ (see below), and we show that, under suitable hypotheses satisfied e.g. in the case of $H= U_q^{\mathrm{ad}}(\mathfrak{g})$, we have:

\begin{teoIntro}[\S \ref{Lgnqrsection} and \S \ref{topqrsection}] \label{teoclosedcase}
1. There is a surjective algebra morphism $\pi\colon \mathcal{L}_{g,n}^H(H) \ra \mathcal{L}_{g,n}^{\mathrm{qr}}(H)$.\\
2. The isomorphism $W : \mathcal{S}_H(\Sigma_{g,n}^\circ) \ra \mathcal{L}_{g,n}^H(H)$ descends to an algebra isomorphism $W^{\rm qr} : \mathcal{S}_H(\Sigma_{g,n}) \ra \mathcal{L}_{g,n}^{\mathrm{qr}}(H)$.\end{teoIntro}

The notion of quantum reduction of module algebras over quantum groups was introduced in \cite{Lu} as an analog of Hamiltonian reduction in symplectic geometry \cite{MarWe74} (see \cite{AM} for the broader notion of quasi-Hamiltonian reduction, well-suited to character varieties of surface groups). 

In that setup, one considers a symplectic space $X$ endowed with a Hamiltonian action of a Lie group $G$ generated by a moment map $\mu^{cl}\colon X\ra \mathfrak{g}^*$, and the reduction procedure describes symplectic leaves in $X/G$ as quotients $(\mu^{cl})^{-1}(\mathcal{C})/G$, where $\mathcal{C}$ is a coadjoint orbit in $\mathfrak{g}^*$. 

In the quantum setup one works dually, and considers a module algebra $A$ over a Hopf algebra $H$. The quantum moment map is a morphism of algebras $\mu\colon H'\ra A$, where $H'\subset H$ is a coideal subalgebra, and $\mu$ satisfies an equation which guarantees that it generates the action of $H'$ on $A$. The reduction procedure describes algebras of invariant elements in certain quotients of $A$, defined by means of $\mu$ and characters of $H'$. 

An ansatz of quantum reduction was implemented in \cite{A} and \cite{BNR}. Its proper definition and fundamental properties were settled in \cite{VV}, which applied it to double affine Hecke algebras (which are closely related to $\Ll_{1,0}(\mathfrak{g})$). Quantum reduction was also applied to quivers in \cite{JordanQuiver} and \cite{GJS}, where it was axiomatized in a categorical setting. 

In the case $H=U_q^{\rm ad}(\mathfrak{g})$, the quantum reduction in Theorem \ref{teoclosedcase} is obtained from a quantum moment map $\mu\colon U_q^{lf}(\mathfrak{g})\ra \mathcal{L}_{g,n}(\mathfrak{g})$, where $U_q^{lf}(\mathfrak{g})\subset U_q(\mathfrak{g})$ is the subalgebra of locally finite elements (which is known to be isomorphic to $\mathcal{O}_q(G)$), and one uses the counit $\varepsilon$ as a character of $U_q^{\rm lf}$. In the case $g=1$, $\mu$ is essentially the quantum moment map denoted by $\partial_{\lhd}$ in Proposition 1.8.2 of \cite{VV}, and for $\mathfrak{g}=\mathfrak{gl}_{m+1}$ and arbitrary genus $g$, it coincides with the quantum moment map of \cite{JordanQuiver}. 

For the reader's convenience, before we prove Theorem \ref{teoclosedcase} we recall the general setup of quantum reduction (\S \ref{prelqrsection}), and we describe the quantum moment map $\mu$ leading to $\mathcal{L}_{g,n}^{\rm qr}(H)$, providing all arguments we did not find in the literature (\S \ref{QMMLgn}). In particular there is a subtlety in the normalization of $\mu$ (see Remark \ref{normliftholonomy}).  


\medskip

\noindent \textbf{Acknowledgements.} We thank F. Costantino for discussing the definition of stated skein algebras to appear in \cite{CKL}, and T.T.Q. L\^e for a question which led us to Corollary \ref{coroISkeinInjective}. We are grateful to them and to J. Korinman for discussions during the workshop ``Skein algebras in Toulouse'' in March 2023, funded by the LabEx CIMI. We also thank the anonymous referees for their careful reading and their relevant comments which improved the exposition.
\\M.F. is supported by the CIMI Labex ANR 11-LABX-0040. Part of this work was done when M.F. was a postdoc in the University of Hamburg, supported by the DFG under Germany’s Excellence Strategy - EXC 2121 “Quantum Universe” - 390833306.

\section{Preliminaries}\label{sectionPreliminaires}
The notations and conventions used in this paper agree with those from \cite{BR2}.

\subsection{Hopf algebras}\label{sectionHopfAlgebras}
\indent Let $H$ be a Hopf algebra over a field $k$ (see e.g. \cite[Chap. III]{kassel} or \cite[\S 4.1]{CP}). We denote by
\[ \Delta : H \to H \otimes H, \qquad \varepsilon : H \to k, \qquad S : H \to H \]
the coproduct, counit and antipode of $H$. The unit of $H$ is denoted by $1$. We occasionally write $\Delta_H$, $\varepsilon_H$, $S_H$, $1_H$ if there is an ambiguity. We assume that the antipode $S$ is invertible.

\smallskip

\indent For an element $X \in H^{\otimes n}$ we use the notation $\textstyle X = \sum_{(X)} X_{(1)} \otimes \ldots \otimes X_{(n)}$ as a substitute for $\textstyle X = \sum_i X_{(1),i} \otimes \ldots \otimes X_{(n),i}$.

\smallskip
For the coproduct we write $\textstyle \Delta(h) = \sum_{(h)} h_{(1)} \otimes h_{(2)}$ (Sweedler's notation) instead of $\textstyle \sum_{(\Delta(h))} \Delta(h)_{(1)} \otimes \Delta(h)_{(2)}$. The opposite coproduct, denoted by $\Delta^{\mathrm{op}}$, is defined by $\textstyle \Delta^{\mathrm{op}}(h) = \sum_{(h)} h_{(2)} \otimes h_{(1)}$. The co-opposite Hopf algebra $H^{\mathrm{cop}}$ is the algebra $H$ endowed with the coproduct $\Delta^{\mathrm{op}}$, the counit $\varepsilon$ and the antipode $S^{-1}$.

\smallskip

\indent In \S \ref{sectionHandleAlgebra}, \S \ref{sectionGraphAlgebra} and \S \ref{sectionAlekseevMorphism}, $H$ is a quasitriangular Hopf algebra as defined e.g. in \cite[\S4.2]{CP} or \cite[\S VIII.2]{kassel}. We write the $R$-matrix of $H$ as $\textstyle R = \sum_{(R)} R_{(1)} \otimes R_{(2)} \in H^{\otimes 2}$. In \S \ref{sectionTopologicalInterpretation} we moreover assume that $H$ is a ribbon Hopf algebra, which means that it contains an element $v$ with the properties listed e.g. in \cite[\S XIV.6]{kassel} or \cite[\S4.2.C]{CP}; in particular
\begin{equation}\label{vproperties}
v \text{ is central and } \Delta(v) = (v\otimes v)(R'R)^{-1}\ ,
\end{equation}
where $R' = \textstyle \sum_{(R)} R_{(2)} \otimes R_{(1)}$. Let $\textstyle u = \sum_{(R)} S(R_{(2)})R_{(1)}$ be the Drinfeld element, then $g = uv^{-1}$ is the pivotal element of $H$. It satisfies $\Delta(g) = g \otimes g$ and $S^2(h) = ghg^{-1}$ for all $h \in H$.

\smallskip

\indent Let $H^* = \mathrm{Hom}_k(H,k)$ be the dual vector space of $H$. For a finite-dimensional $H$-module $V$ and $v \in V$, $f \in V^*$, we define $_V\phi^f_v \in H^*$ by $_V\phi^f_v(h) = f(h \cdot v)$, where $\cdot$ is the action of $H$ on $V$. The linear form $_V\phi^f_v$ is called a {\em matrix coefficient of $V$}. When a basis $(v_j)$ of $V$ is given we allow ourselves to write $_V\phi^i_j$ instead of $_V\phi^{v^i}_{v_j}$ for a better readability, where $(v^i)$ is the dual basis. The {\em restricted dual of $H$}, denoted by $H^{\circ}$, is the subspace of $H^*$ spanned by the matrix coefficients of finite-dimensional $H$-modules. Let us set
\begin{equation}\label{hopfMatrixCoeffs}
\begin{array}{c}
_V\phi^f_v \, {_W}\phi^l_w = {_{V \otimes W}\phi^{f \otimes l}_{v \otimes w}}, \quad \Delta(_V\phi^i_j) = \sum_{k=1}^{\dim(V)} {_V\phi^i_k} \otimes {_V\phi^k_j},\\[1.3em]
1_ {H^{\circ}} = {_k\phi^1_1}, \quad \varepsilon(_V\phi^i_j) = \delta_{i,j}, \quad S(_V\phi^i_j) = {_{V^*}\phi^j_i}
\end{array}
\end{equation}
or in other words
\[ \varphi\psi = (\varphi \otimes \psi) \circ \Delta, \quad \Delta(\varphi)(x \otimes y) = \varphi(xy), \quad 1_{H^{\circ}} = \varepsilon, \quad \varepsilon(\varphi) = \varphi(1), \quad S(\varphi) = \varphi \circ S \]
for all $\varphi \in H^{\circ}$. Then $H^{\circ}$ is a Hopf algebra. The above formula for $\Delta(\varphi)$ is not well-defined for a general $\varphi \in H^*$, which explains the relevance of $H^{\circ}$. If $H$ is finite-dimensional then $H^{\circ} = H^*$ because $\varphi = {_H\phi^{\varphi}_1}$ for all $\varphi \in H^*$.

\smallskip

\indent Let $\mathcal{C}$ be a full subcategory of $H$-mod (the category of finite-dimensional $H$-modules) such that the trivial module $k$ is in $\mathcal{C}$ and if $V,W \in \mathcal{C}$ then $V \otimes W \in \mathcal{C}$ and $V^* \in \mathcal{C}$. Let $H^{\circ}_{\mathcal{C}}$ be the subspace of $H^{\circ}$ spanned by the matrix coefficients of the objects in $\mathcal{C}$. It is clear from \eqref{hopfMatrixCoeffs} that $H^{\circ}_{\mathcal{C}}$ is a subalgebra of $H^{\circ}$.

\subsection{Filtrations of algebras}\label{sectionFiltrations}
Let $(S, \leq)$ be an ordered abelian monoid and let $A$ be an associative $k$-algebra ($k$ is a field). Recall that an algebra filtration of $A$ indexed by $S$ is a family $F = (F^s)_{s \in S}$ of linear subspaces $F^s\subset A$ such that $1\in F^0$, for all $s,t\in S$ we have $F^s\subset F^t$ if $s\leq t$ and $F^sF^t\subset F^{s+t}$, and
$$A = \cup_{s\in S} F^s.$$
We denote by $\textstyle \mathrm{gr}_F(A) = \bigoplus_{s \in S} \mathrm{gr}_F(A)_s$ the associated graded algebra which is defined by $\mathrm{gr}_F(A)_s = F^s/F^{<s}$, where $\textstyle F^{< s} = \sum_{r<s}F^r$, and endowed with the product
\[ (a + F^{<s})(b + F^{<t}) = ab + F^{<(s+t)} \]
for $a \in F^s$ and $b \in F^t$.
\smallskip

\indent Assume now that $A$ has a decomposition $\textstyle A = \bigoplus_{s \in S} X_s$ as a $k$-vector space, where the family $X = (X_s)_{s \in S}$ is such that $\textstyle X_r X_s \subset \bigoplus_{t \leq r+s} X_t$ for all $r,s$. Define $\textstyle \Sigma^s(X) = \bigoplus_{r \leq s} X_r$. If the partial order $\leq$ on $S$ satisfies
\begin{equation}\label{confluentOrder}
\forall \, r,s \in S, \quad \exists\, m \in S, \quad r \leq m \text{ and } s \leq m
\end{equation}
(or in other words, any finite subset of $S$ has an upper bound) then the condition $\textstyle A = \bigcup_{s \in S} \Sigma^s(X)$ is satisfied. Then $\Sigma(X) = \bigl( \Sigma^s(X) \bigr)_{s \in S}$ is a filtration of the algebra $A$. Moreover, $\textstyle \Sigma^{<s}(X) = \bigoplus_{r<s}X_r$ so that $\mathrm{gr}_{\Sigma(X)}(A)$ can be identified with $\textstyle \bigoplus_{s \in S} X_s$ as graded vector spaces, and under this identification the product $\circ$ in $\mathrm{gr}_{\Sigma(X)}(A)$ is
\[ x \circ y = \pi_{r+s}(xy) \]
where $x \in X_r$, $y \in X_s$, $xy$ is the product in $A$ and $\textstyle \pi_{r+s} : \Sigma^{r+s}(X) \to X_{r+s}$ is the canonical projection. All the filtrations used in the subsequent sections are of this form.

\smallskip

\indent Recall that an order relation $\leq$ is called {\em well-founded} if any decreasing chain $s_1 \geq s_2 \geq \ldots$ is eventually constant. The relevance of filtered algebras in the present paper comes from the following criterion :
\begin{lem}\label{lemmaFiltrationNoetherian}
Assume that the order relation $\leq$ on $S$ is well-founded and let $F = (F^s)_{s \in S}$ be a filtration of $A$. If $\mathrm{gr}_F(A)$ is Noetherian, then $A$ is Noetherian.
\end{lem}
\noindent See e.g. \cite[Lem. 3.130]{VY} for the proof. We note that the converse statement is false in general (see \cite[\S 1.6.9]{MC-R} for a simple counter-example).

\smallskip

\indent A strategy to prove the Noetherianity of some algebra $A$ is then to find a filtration $F$ of $A$ such that $\mathrm{gr}_F(A)$ is simpler to analyze. For the algebras $A$ considered in this paper, we will define $F$ in such a way that the following criterion from \cite[Prop. I.8.17]{BG} applies to $\mathrm{gr}_F(A)$:

\begin{lem}\label{critereNoetherien}
If an associative $k$-algebra is  generated by elements $u_1, \ldots, u_m$ such that
\[ \forall\, 1 \leq j < i \leq m, \quad u_i u_j = q_{ij} u_j u_i + \sum_{s=1}^{j-1} \sum_{t=1}^m \left( \alpha^{ij}_{st} u_s u_t + \beta^{ij}_{st} u_t u_s \right) \]
for certain scalars $q_{ij} \in k^{\times}$, $\alpha^{ij}_{st}, \beta^{ij}_{st} \in k$, then it is Noetherian.
\end{lem}

\indent Finally, we record two classical transfer results:
\begin{lem}\label{lemmaFinGenGr}
Assume that the order relation $\leq$ on $S$ is well-founded and let $F = (F^s)_{s \in S}$ be a filtration of $A$. If $\mathrm{gr}_F(A)$ is finitely generated, then $A$ is finitely generated as well.
\end{lem}
\begin{proof}
The easy proof by well-founded induction is left to the reader.
\end{proof}
We say that $F = (F^s)_{s \in S}$ is {\it locally bounded below} if for each nonzero element $a \in A$ the set of all $s\in S$ such that $a\in F^s$ has a minimal element. We have (see e.g. \cite[\S 1.6.6]{MC-R} or \cite[Lem. 3.130]{VY}):
\begin{lem}\label{lemmadomainGr}
Assume that $(S,\leq)$ is totally ordered and $F = (F^s)_{s \in S}$ is a filtration of $A$ which is locally bounded below. If $\mathrm{gr}_F(A)$ has no non-trivial zero divisors, then $A$ has no non-trivial zero divisors as well.
\end{lem}

\subsection{Semisimple Lie algebras}\label{prelimLieAlgebras}
Let $\mathfrak{g}$ be a finite-dimensional complex simple Lie algebra; one can more generally adapt the notations to handle the semisimple case, by replacing $D$ below with the lowest common multiple of the corresponding integers for the simple components of $\mathfrak{g}$. Here we fix the notations regarding $\mathfrak{g}$.

\smallskip

\indent We denote by $m$ the rank of $\mathfrak{g}$ and its Cartan matrix by $(a_{ij})$. We fix a Cartan subalgebra $\mathfrak{h} \subset \mathfrak{g}$ and a basis of simple roots $\alpha_i \in \mathfrak{h}^*_{\mathbb{R}}$. Let $\mathfrak{b}_{\pm}$ be the associated Borel subalgebras. We denote by $N$ the number of positive roots of $\mathfrak{g}$.

\smallskip

\indent Let $d_1, \ldots, d_m$ be the unique coprime integers such that the matrix $(d_ia_{ij})$ is symmetric. There is an inner product $(-,-)$ on $\mathfrak{h}^* _{\mathbb{R}}$ defined by $(\alpha_i,\alpha_j) = d_i a_{ij}$ on the simple roots. The simple coroots are $\alpha_i^{\vee} = d_i^{-1} \alpha_i$ for every $1 \leq i \leq m$, so that $(\alpha_i,\alpha_i^{\vee}) = 2$.

\smallskip

\indent The root lattice is $\textstyle Q = \bigoplus_{i=1}^m \mathbb{Z} \alpha_i \subset \mathfrak{h}^*_{\mathbb{R}}$. The weight lattice $P$ is the $\mathbb{Z}$-lattice formed by all the $\lambda \in \mathfrak{h}^*_{\mathbb{R}}$ satisfying $(\lambda, \alpha_i^{\vee}) \in \mathbb{Z}$ for every $i$. The fundamental weights $\varpi_i$, $1 \leq i \leq m$, are defined by $(\varpi_i, \alpha_j^{\vee}) = \delta_{i,j}$ for all $j$. Then $\textstyle P = \bigoplus_{i=1}^m \mathbb{Z} \varpi_i$. We have $Q \subset P$ and we denote by $D$ the smallest positive integer such that $DP \subset Q$. 
\smallskip

\indent The cone of dominant integral weights is $\textstyle P_+ = \bigoplus_{i=1}^m \mathbb{N} \varpi_i$. We also put $\textstyle Q_+ = \bigoplus_{i=1}^m \mathbb{N} \alpha_i$. Note that $Q_+ \not\subset P_+$. But $P_+ \subset D^{-1}Q_+$, which is due to the classical fact that the inverse of the Cartan matrix has coefficients in $D^{-1}\mathbb{N}$.

\smallskip

\indent The standard partial order $\leq$ on $P$ is defined by $\lambda \leq \mu$ if and only if $\mu - \lambda \in Q_+$. We will need another partial order $\preceq$ on $P$, defined by $\lambda \preceq \mu$ if and only if $\mu - \lambda \in D^{-1}Q_+$. Note that $\lambda \leq \mu$ implies $\lambda \preceq \mu$. The advantage of $\preceq$ is that it satisfies the condition \eqref{confluentOrder}, contrarily to $\leq$.

We put $\lambda \prec \mu$ when $\lambda \preceq \mu$ and $\lambda \ne \mu$.
\subsection{Quantum group $U_q(\mathfrak{g})$}\label{sectionPrelimUq}
\indent Let $\mathfrak{g}$ be a semisimple finite-dimensional Lie algebra of rank $m$ and let $\mathbb{C}(q)$ be the field of fractions of $\mathbb{C}[q]$, where $q$ is an indeterminate. The {\em simply connected quantum group} $U_q = U_q(\mathfrak{g})$ is the $\mathbb{C}(q)$-algebra generated by $E_i, F_i, L_i^{\pm 1}$ for $1 \leq i \leq m$ modulo the relations
\begin{align*}
&L_i^{\pm 1} L_i^{\mp 1} = 1, \qquad L_i L_j = L_j L_i, \qquad L_iE_j = q_i^{\delta_{i,j}}E_jL_i, \qquad L_iF_j = q_i^{-\delta_{i,j}}F_jL_i,\\
& E_iF_j - F_jE_i = \delta_{i,j} \frac{K_i - K_i^{-1}}{q_i - q_i^{-1}},\\
&q\text{-Serre relations (see e.g. \cite[Def. 3.13]{VY}})
\end{align*}
where $q_i = q^{d_i}$ and for $\textstyle \mu = \sum_{i=1}^m n_i \varpi_i \in P$ we set $\textstyle K_{\mu} = \prod_{i=1}^m L_i^{n_i}$ and $\textstyle K_i = K_{\alpha_i} = \prod_{j=1}^m L_j^{a_{ji}}$. The Hopf algebra structure on $U_q$ is given by
\[ \begin{array}{lll}
\Delta(E_i) = E_i \otimes K_i + 1 \otimes E_i, & \Delta(F_i) = F_i \otimes 1 + K_i^{-1} \otimes F_i, & \Delta(L_i) = L_i \otimes L_i,\\
S(E_i) = -E_i K_i^{-1}, & S(F_i) = -K_i F_i, & S(L_i) = L_i^{-1},\\
\varepsilon(E_i) = 0, & \varepsilon(F_i) = 0, & \varepsilon(L_i) = 1.
\end{array} \]
We denote by $U_q(\mathfrak{h})$, $U_q(\mathfrak{n}_+)$ and $U_q(\mathfrak{n}_-)$ the subalgebras of $U_q$ generated respectively by $(K_{\mu})_{\mu \in P}$, $(E_i)_{1 \leq i \leq m}$ and $(F_i)_{1 \leq i \leq m}$.

\smallskip

\indent The {\em adjoint quantum group} $U_q^{\mathrm{ad}} = U_q^{\mathrm{ad}}(\mathfrak{g})$ is the Hopf subalgebra of $U_q$ generated by the elements $E_i, F_i, K_i$. In particular $K_{\alpha} \in U_q^{\mathrm{ad}}$ for all $\alpha \in Q$.

\smallskip

\indent Fix a reduced expression $s_{i_1} \ldots s_{i_N}$ of the longest element $w_0$ of the Weyl group of $\mathfrak{g}$, where as usual $s_i : \mathfrak{h}^*_{\mathbb{R}} \to \mathfrak{h}^*_{\mathbb{R}}$, $\alpha_j \mapsto \alpha_j - a_{ij}\alpha_i$. This expression induces a total ordering of the positive roots:
\[ \beta_1 = \alpha_{i_1}, \quad \beta_k = s_{i_1} \ldots s_{i_{k-1}}(\alpha_{i_k}) \]
for all $2 \leq k \leq N$. The {\em root vectors} of $U_q$ associated to this ordering are 
\[ E_{\beta_1} = E_{i_1}, \: E_{\beta_k} = T_{i_1} \ldots T_{i_{k-1}}(E_{i_k}), \qquad F_{\beta_1} = F_{i_1}, \: F_{\beta_k} = T_{i_1} \ldots T_{i_{k-1}}(F_{i_k}) \]
for all $2 \leq k \leq N$, where the $T_i$ are Lusztig's algebra automorphisms whose defining formulas can be found e.g. in \cite[Th. 3.58]{VY}. We record that for any $\mu \in P$,
\begin{equation}\label{weightOfRootVectors}
K_{\mu} E_{\beta_j} = q^{(\mu,\beta_j)} E_{\beta_j} K_{\mu}, \qquad K_{\mu} F_{\beta_j} = q^{-(\mu,\beta_j)} F_{\beta_j} K_{\mu}.
\end{equation}

\smallskip

\indent For $\mathbf{t} = (t_1, \ldots, t_N) \in \mathbb{N}^N$, $\mu \in P$ and $\mathbf{s} = (s_1, \ldots, s_N) \in \mathbb{N}^N$ consider the monomial
\begin{equation}\label{PBWmonomial}
X(\mathbf{t}, \mu, \mathbf{s}) = F_{\beta_N}^{t_N} \ldots F_{\beta_1}^{t_1} K_{\mu} E_{\beta_N}^{s_N} \ldots E_{\beta_1}^{s_1}.
\end{equation}
It is a classical theorem that these monomials form a basis of $U_q$, called the {\em Poincar\'e-Birkhoff-Witt basis} (PBW basis) (see e.g. \cite{CP}). In particular the elements $X(\mathbf{0}, 0, \mathbf{s})$ form a basis of $U_q(\mathfrak{n}_+)$ (where $\mathbf{0} = (0, \ldots, 0)$), the elements $X(\mathbf{t}, 0, \mathbf{0})$ form a basis of $U_q(\mathfrak{n}_-)$ and the elements $K_{\mu} = X(\mathbf{0}, \mu, \mathbf{0})$ form a basis of $U_q(\mathfrak{h})$.

\smallskip

\indent Recall that the {\em height} of an element in $Q$ is $\textstyle \mathrm{ht}\left( \sum_i k_i \alpha_i \right) = \sum_i k_i$, and let
\[ \mathrm{ht}\bigl( X(\mathbf{t}, \mu, \mathbf{s}) \bigr) = \sum_{i=1}^N (t_i + s_i) \mathrm{ht}(\beta_i). \]
Then the {\em degree} of $X(\mathbf{t}, \mu, \mathbf{s})$ is defined as
\[ d\bigl( X(\mathbf{t}, \mu, \mathbf{s}) \bigr) = \bigl( s_N, \ldots, s_1, t_N, \ldots, t_1, \mathrm{ht}\bigl( X(\mathbf{t}, \mu, \mathbf{s}) \bigr) \bigr) \in \mathbb{N}^{2N+1}. \]
The set $\mathbb{N}^{2N+1}$ is a totally ordered additive monoid for the lexicographic order given by
\begin{equation}\label{RevLexN}
(1, 0, \ldots, 0) < (0, 1, 0, \ldots, 0) < \ldots < (0, \ldots, 0, 1, 0)  < (0, \ldots, 0, 1).
\end{equation}
For $\mathbf{m} \in \mathbb{N}^{2N+1}$, denote by $\mathcal{F}_{\mathrm{DCK}}^{\mathbf{m}}$ the subspace of $U_q$ spanned by the monomials $X(\mathbf{t}, \mu, \mathbf{s})$ of degree $\leq \mathbf{m}$. It was shown in \cite[Prop 1.7]{DC-K} that $(\mathcal{F}_{\mathrm{DCK}}^{\mathbf{m}})_{\mathbf{m} \in \mathbb{N}^{2N+1}}$ is an algebra filtration of $U_q$ and that the graded algebra $\mathrm{gr}_{\mathcal{F}_{\mathrm{DCK}}}(U_q)$ is quasi-polynomial in the sense of \cite[Section 1.8]{DC-K}. It is generated by the cosets 
\[ \overline{E_{\beta_i}} = E_{\beta_i} + \mathcal{F}_{\mathrm{DCK}}^{<d(E_{\beta_i})}, \quad \overline{F_{\beta_i}} = F_{\beta_i} + \mathcal{F}_{\mathrm{DCK}}^{<d(F_{\beta_i})}, \quad \overline{K_{\mu}} = K_{\mu} + \mathcal{F}_{\mathrm{DCK}}^{<d(K_\nu)} \]
for $1 \leq i \leq N$ and $\mu \in P$ (note that $\mathcal{F}_{\mathrm{DCK}}^{<d(K_\mu)}=\{0\}$ since $d(K_\mu)=\mathbf{0}$), modulo the relations
\begin{equation}\label{commutationGrDCK}
\begin{array}{lll}
\overline{E_{\beta_i}} \, \overline{E_{\beta_j}} = q^{(\beta_i,\beta_j)} \overline{E_{\beta_j}} \, \overline{E_{\beta_i}}, & \overline{F_{\beta_i}} \, \overline{F_{\beta_j}} = q^{(\beta_i,\beta_j)} \overline{F_{\beta_j}} \, \overline{F_{\beta_i}}, & \overline{E_{\beta_i}} \, \overline{F_{\beta_j}} = \overline{F_{\beta_j}} \, \overline{E_{\beta_i}}, \\[.5em]
\overline{K_{\mu}} \, \overline{E_{\beta_j}} = q^{(\mu,\beta_j)} \overline{E_{\beta_j}} \, \overline{K_{\mu}}, & \overline{K_{\mu}} \, \overline{F_{\beta_j}} = q^{-(\mu,\beta_j)} \overline{F_{\beta_j}} \, \overline{K_{\mu}}, & \overline{K_{\mu}} \, \overline{K_{\nu}} = \overline{K_{\mu + \nu}}.
\end{array}
\end{equation}

\smallskip

The following facts describe the effect of the coproduct $\Delta$ on the filtration $\mathcal{F}_{\mathrm{DCK}}$; this will be useful later.
\begin{prop}\label{propHtCoproduit}
We have
\[ \Delta(E_{\beta_j}) = E_{\beta_j} \otimes K_{\beta_j} + \sum_k c_k X(\mathbf{0}, 0, \mathbf{s}'_k) \otimes X(\mathbf{0}, \mu_k, \mathbf{s}''_k) + 1 \otimes E_{\beta_j} \]
where $\mathbf{s}'_k$, $\mathbf{s}''_k\in \mathbb{N}^N$ and $\mu_k\in P$ satisfy $\mathrm{ht}\bigl( X(\mathbf{0}, 0, \mathbf{s}'_k) \bigr) < \mathrm{ht}(\beta_j)$ and $\mathrm{ht}\bigl( X(\mathbf{0}, \mu_k, \mathbf{s}_k'') \bigr) < \mathrm{ht}(\beta_j)$ for all $k$, and $c_k \in \mathbb{C}(q)$. Similarly
\[ \Delta(F_{\beta_j}) = F_{\beta_j} \otimes 1 + \sum_k d_k X(\mathbf{t}'_k, \nu_k, \mathbf{0}) \otimes X(\mathbf{t}_k'', 0, \mathbf{0}) + K_{-\beta_j} \otimes F_{\beta_j} \]
with $\mathrm{ht}\bigl( X(\mathbf{t}'_k, \nu_k, \mathbf{0}) \bigr) < \mathrm{ht}(\beta_j)$ and $\mathrm{ht}\bigl( X(\mathbf{t}_k'', 0, \mathbf{0}) \bigr) < \mathrm{ht}(\beta_j)$ for all $k$, and $d_k \in \mathbb{C}(q)$.
\end{prop}
\begin{proof}
Let
\[ U_q(\mathfrak{n}_+)_{\lambda} = \bigl\{ x \in U_q(\mathfrak{n}_+) \, \big| \, \forall \, \nu \in P, \: K_{\nu}xK_{\nu}^{-1} = q^{(\nu,\lambda)}x \bigr\} \]
which is the space of vectors in $U_q(\mathfrak{n}_+)$ with weight $\lambda \in P$ for the adjoint action of $U_q$. Due to \eqref{weightOfRootVectors} we have $E_{\beta_j} \in U_q(\mathfrak{n}_+)_{\beta_j}$. By \cite[Lem 4.12]{jantzen} (see also \cite[Prop. 1]{krahmer} for the exact statement used here) there exist elements $x'_i \in U_q(\mathfrak{n}_+)_{\gamma_i}$ and $x''_i \in U_q(\mathfrak{n}_+)_{\beta_j - \gamma_i}$ with $\gamma_i \in Q$ satisfying $0 < \gamma_i < \beta_j$ such that
\[ \Delta(E_{\beta_j}) = E_{\beta_j} \otimes K_{\beta_j} + \sum_i x'_i \otimes x''_iK_{\gamma_i} + 1 \otimes E_{\beta_j}. \]
Since the monomials $X(\mathbf{0},0,\mathbf{s})$ form a basis of $U_q(\mathfrak{n}_+)$ we can rewrite this as
\[ \Delta(E_{\beta_j}) = E_{\beta_j} \otimes K_{\beta_j} + \sum_k c_k X(\mathbf{0}, 0, \mathbf{s}'_k) \otimes X(\mathbf{0}, \mu_k, \mathbf{s}''_k) + 1 \otimes E_{\beta_j} \]
where for all $k$: $c_k \in \mathbb{C}(q), $ $X(\mathbf{0}, 0, \mathbf{s}'_k) \in U_q(\mathfrak{n}_+)_{\mu_k}$ and $X(\mathbf{0}, \mu_k, \mathbf{s}''_k) \in U_q(\mathfrak{n}_+)_{\beta_j - \mu_k}$ for some $\mu_k \in Q$ such that $0 < \mu_k < \beta_j$. Write $\mathbf{s}'_k = (s'_{k,1}, \ldots, s'_{k,N})$ and note by \eqref{weightOfRootVectors} that
\[ K_{\nu} X(\mathbf{0}, 0, \mathbf{s}'_k) K_{\nu}^{-1} = q^{(\nu, s'_{k,1}\beta_1 + \ldots + s'_{k,N}\beta_N)} X(\mathbf{0}, 0, \mathbf{s}'_k) \]
for all $\nu \in P$. Hence $s'_{k,1}\beta_1 + \ldots + s'_{k,N}\beta_N = \mu_k$ and we get
\[ \mathrm{ht}\bigl( X(\mathbf{0}, 0, \mathbf{s}'_k) \bigr) = \mathrm{ht}(s'_{k,1}\beta_1 + \ldots + s'_{k,N}\beta_N) = \mathrm{ht}(\mu_k) < \mathrm{ht}(\mu_k) + \mathrm{ht}(\beta_j - \mu_k) = \mathrm{ht}(\beta_j). \]
We obtain similarly that $\mathrm{ht}\bigl( X(\mathbf{0}, \mu_k, \mathbf{s}''_k) \bigr) < \mathrm{ht}(\beta_j)$. The proof for $F_{\beta_j}$ is completely analogous.
\end{proof}

\begin{cor}\label{coroCoproduitSurFiltrationDCK}
For all $\mathbf{m} \in \mathbb{N}^{2N+1}$ we have $\Delta(\mathcal{F}_{\mathrm{DCK}}^{\mathbf{m}}) \subset \mathcal{F}_{\mathrm{DCK}}^{\mathbf{m}} \otimes \mathcal{F}_{\mathrm{DCK}}^{\mathbf{m}}$.
\end{cor}
\begin{proof}
By Proposition \ref{propHtCoproduit} and by definition of the degree $d$ and of the order $<$ on $\mathbb{N}^{2N+1}$ we have $\Delta(E_{\beta_j}) \in \left(\mathcal{F}_{\mathrm{DCK}}^{d(E_{\beta_j})}\right)^{\otimes 2}$ and $\Delta(F_{\beta_j}) \in \left(\mathcal{F}_{\mathrm{DCK}}^{d(F_{\beta_j})}\right)^{\otimes 2}$. Hence
\begin{align*}
\Delta\bigl( X(\mathbf{t},\mu,\mathbf{s}) \bigr) &= \Delta(F_{\beta_N})^{t_N} \ldots \Delta(F_{\beta_1})^{t_1} \Delta(K_{\mu}) \Delta(E_{\beta_N})^{s_N} \ldots \Delta(E_{\beta_1})^{s_1}\\
&\in  \left(\mathcal{F}_{\mathrm{DCK}}^{\sum_{j=1}^N t_j d(F_{\beta_j}) + s_j d(E_{\beta_j})}\right)^{\otimes 2} = \left( \mathcal{F}_{\mathrm{DCK}}^{d( X(\mathbf{t},\mu,\mathbf{s}))}\right)^{\otimes 2}
\end{align*}
and the result follows.
\end{proof}

\subsection{Categorical completion $\mathbb{U}_q$ and $R$-matrix}\label{sectionCategoricalCompletion}
\indent The quantum group $U_q^{\mathrm{ad}}(\mathfrak{g})$ is not quasitriangular in the usual sense. We quickly recall from \cite{BR1} how to overcome this issue.

\smallskip

\indent Let $\mathcal{C} = \mathcal{C}_q(\mathfrak{g})$ be the full subcategory of type $1$ finite-dimensional modules in $U_q^{\mathrm{ad}}$-mod, i.e., finite dimensional $\mc(q)$-vector spaces which are weight modules where the generator $K_i$ has eigenvalues in $q_i^\mz$ (see e.g. \cite[\S 10.1.A]{CP} or \cite[\S I.6.12]{BG}). Note that $\mathcal{C}$ is equivalent to the full subcategory of type $1$ finite-dimensional $U_q$-modules.

The {\em categorical completion of $U^{\mathrm{ad}}_q$} is the subalgebra $\mathbb{U}^{\mathrm{ad}}_q$ of $\textstyle \prod_{V \in \mathcal{C}} \mathrm{End}_{\mathbb{C}(q)}(V)$ consisting of the collections $(a_V)_{V \in \mathcal{C}}$ which satisfy $f\circ a_V = a_{V'} \circ f$ for any $f \in \mathrm{Hom}_{\mathcal{C}}(V,V')$, \textit{i.e.} natural transformations $a : \mathcal{U} \Rightarrow \mathcal{U}$ where $\mathcal{U} : \mathcal{C} \to \mathrm{Vect}_{\mathbb{C}(q)}$ is the forgetful functor. It admits a ``generalized'' Hopf algebra structure such that the map
\begin{equation}\label{embeddingUqIntoItsCompletion}
\iota : U_q^{\mathrm{ad}} \to \mathbb{U}^{\mathrm{ad}}_q, \quad h \mapsto (h_V)_{V \in \mathcal{C}}
\end{equation}
is a morphism of Hopf algebras, where $h_V$ is the representation of $h$ on $V$. The morphism $\iota$ is known to be injective and thus $U_q^{\mathrm{ad}}$ can be seen as a Hopf subalgebra of $\mathbb{U}^{\mathrm{ad}}_q$. If moreover we let $\mathbb{U}_q = \mathbb{U}_q^{\mathrm{ad}} \otimes_{\mathbb{C}(q)} \mathbb{C}(q^{1/D})$ be the extension of scalars to $\mathbb{C}(q^{1/D})$, then $\iota$ can be extended to an embedding of the simply connected quantum group:
\[ \iota : U_q \to \mathbb{U}_q. \]

\smallskip

\indent There is an analogous notion of categorical completion of $U_q \otimes U_q$, denoted by $\mathbb{U}_q \otimes \mathbb{U}_q$, see \cite[\S2, \S3.2]{BR1}. There is an $R$-matrix for $\mathbb{U}_q$, which lives in $\mathbb{U}_q \otimes \mathbb{U}_q$; this means that it is a family of $\mathbb{C}(q)$-linear maps
\begin{equation}\label{RmatrixInCompletion}
R = \bigl(R_{V,W} : V \otimes W \to V \otimes W\bigr)_{V, W \in \mathcal{C}}
\end{equation}
which is natural (\textit{i.e.} it commutes with the $U_q$-morphisms of the form $f \otimes g$). The definition of $R$  is derived from the $R$-matrix of $U_h(\mathfrak{g})$ \cite[Th. 8.3.9]{CP}: we have $R = \Theta \hat R$, where
\begin{itemize}
\item for all $V,W \in \mathcal{C}$, if $v \in V$ and $w \in W$ are weight vectors of weights $\mu$ and $\nu$ we put
\begin{equation}\label{ThetaOnWeightVectors}
\Theta(v \otimes w) = q^{(\mu,\nu)} v \otimes w.
\end{equation}
Since any module in $\mathcal{C}$ has a basis of weight vectors, this defines a linear map $\Theta_{V,W} : V \otimes W \to V \otimes W$ and $\Theta = (\Theta_{V,W})_{V,W \in \mathcal{C}}$ is an element of $\mathbb{U}_q^{\otimes 2}$.
\item $\hat R$ is written formally as $\textstyle \prod_{i=1}^N \sum_{n \in \mathbb{N}} z_{i,n}E_{\beta_i}^n \otimes F_{\beta_i}^n$, where the $z_{i,n}$ are coefficients in $\mathbb{C}(q)$ such that $z_{i,0} = 1$ for all $i$. Since the actions of $E_{\beta_i}$ and $F_{\beta_i}$ are nilpotent on any finite-dimensional module, the action of these infinite sums on $V \otimes W$ gives a well-defined linear map $\hat R_{V,W} : V \otimes W \to V \otimes W$ for all $V,W \in \mathcal{C}$ and $\hat R = (\hat R_{V,W})_{V,W \in \mathcal{C}}$ is an element of $\mathbb{U}_q^{\otimes 2}$.
\end{itemize}

\smallskip

\indent Since $\Theta_{V,W} \in \mathrm{End}_{\mathbb{C}(q)}(V \otimes W) = \mathrm{End}_{\mathbb{C}(q)}(V) \otimes \mathrm{End}_{\mathbb{C}(q)}(W)$, we can write $\textstyle \Theta_{V,W} = \sum_i \Theta^V_{(1),i} \otimes \Theta^W_{(2),i}$. From this observation we allow ourselves to use the notation $\Theta = \Theta_{(1)} \otimes \Theta_{(2)}$, which will be very convenient in later computations. 
\smallskip

In e.g. Theorem \ref{TheoremePhignInjectif} we will need the following observation. Let $v$ be a vector of weight $\mu \in P$ in some $V \in \mathcal{C}$; then we have
\begin{equation}\label{ThetaPoids}
\Theta_{(1)}v \otimes \Theta_{(2)} = v \otimes K_{\mu}, \qquad \Theta_{(1)} \otimes \Theta_{(2)}v = K_{\mu} \otimes v.
\end{equation}
Indeed, $\Theta_{V,-} = \bigl(\Theta_{V,W} : V \otimes W \to V \otimes W \bigr)_{W \in \mathcal{C}}$ is a family of linear maps indexed by $W$, so by definition $\Theta_{V,-} \in \mathrm{End}(V) \otimes \mathbb{U}_q$, and thus $\Theta_{(1)}v \otimes \Theta_{(2)} \in V \otimes \mathbb{U}_q$. If $w \in W$ is a vector of weight $\nu$ we have $\Theta_{V,W}(v \otimes w) = q^{(\mu,\nu)} v \otimes w = v \otimes K_{\mu}w$, which means that $\Theta_{(1)}v \otimes \Theta_{(2)} = v \otimes K_{\mu}$ as desired. We also note that
\begin{equation}\label{ThetaPoidsInverse}
(S \otimes \mathrm{id})(\Theta) =  (\mathrm{id} \otimes S)(\Theta) = \Theta^{-1}
\end{equation}
and we allow ourselves to write these elements as $S(\Theta_{(1)}) \otimes \Theta_{(2)}$ and $\Theta_{(1)} \otimes S(\Theta_{(2)})$. Finally
\begin{equation}\label{commutationKTheta}
K_{\nu}\Theta_{(1)} \otimes \Theta_{(2)} = \Theta_{(1)}K_{\nu} \otimes \Theta_{(2)}, \qquad \Theta_{(1)} \otimes K_{\nu}\Theta_{(2)} = \Theta_{(1)} \otimes \Theta_{(2)}K_{\nu}
\end{equation}
for all $\nu \in P$.

\smallskip

\indent Using the PBW basis \eqref{PBWmonomial} to rewrite $\hat R$, we get
\begin{equation}\label{expressionCanoniqueR}
R = \Theta_{(1)} \otimes \Theta_{(2)} + \sum_{k \geq 1} z_k \, \Theta_{(1)}X(\mathbf{0},0,\mathbf{s}_k) \otimes \Theta_{(2)}X(\mathbf{s}_k,0,\mathbf{0})
\end{equation}
where $z_k \in \mathbb{C}(q)$ and for each $k\geq 1$ we have $\mathbf{s}_k \neq \mathbf{0}$. This expression of $R$ will be abridged as a formal sum
\[ R = \sum_{(R)} R_{(1)} \otimes R_{(2)}. \]
Let $R_{(1)} \otimes R_{(2)}$ be one of the summands in \eqref{expressionCanoniqueR}; we see from \eqref{weightOfRootVectors} and \eqref{commutationKTheta} that
\begin{equation}\label{commutationRK}
K_{\nu}R_{(1)} \otimes R_{(2)} = q^{(\nu, \gamma)} R_{(1)}K_{\nu} \otimes R_{(2)} \qquad R_{(1)} \otimes K_{\nu} R_{(2)} = q^{-(\nu,\gamma)} R_{(1)} \otimes R_{(2)}K_{\nu}
\end{equation}
for some $\gamma \in Q_+$. In particular $\gamma = 0$ if and only if $R_{(1)} \otimes R_{(2)}$ is the first summand of $R$, namely $\Theta_{(1)} \otimes \Theta_{(2)}$.

\subsection{Quantized coordinate algebra $\mathcal{O}_q(G)$}\label{sectionOq}
Recall that $\mathcal{C} = \mathcal{C}_q(\mathfrak{g})$ is the full subcategory of $U_q^{\mathrm{ad}}$-mod whose objects are the (finite-dimensional) modules of type $1$. The subcategory $\mathcal{C}$ contains the trivial module and is stable under tensor product and taking duals.

\smallskip

\indent Let $G$ be the connected, complex, semisimple algebraic group $G$ with Lie algebra $\mathfrak{g}$.  The {\em quantized coordinate algebra} $\mathcal{O}_q = \mathcal{O}_q(G)$ is the subspace of $(U_q^{\mathrm{ad}})^{\circ}$ spanned over $\mathbb{C}(q)$ by the family of all matrix coefficients of objects in $\mathcal{C}$ \cite[Chap. I.7]{BG}. Note that $\mathcal{O}_q = (U_q^{\mathrm{ad}})^{\circ}_{\mathcal{C}}$ in the notations of \S\ref{sectionHopfAlgebras} and hence $\mathcal{O}_q$ is a Hopf algebra. We denote by $\star$ the product in $\mathcal{O}_q$. Since $\mathcal{C}$ is semisimple, $\mathcal{O}_q$ is spanned over $\mathbb{C}(q)$ by the matrix coefficients of the irreducible finite-dimensional $U_q^{\mathrm{ad}}$-modules of type $1$.

\smallskip

\indent Let $\mathcal{O}_q(q^{1/D}) = \mathcal{O}_q \otimes_{\mathbb{C}(q)} \mathbb{C}(q^{1/D})$ be the extension of scalars to $\mathbb{C}(q^{1/D})$, where $D \in \mathbb{N}$ is defined in \S \ref{prelimLieAlgebras}. There is a non-degenerate pairing
\begin{equation}\label{pairingOqCategoricalCompletion}
\langle \cdot, \cdot \rangle : \mathcal{O}_q(q^{1/D}) \otimes \mathbb{U}_q \to \mathbb{C}(q^{1/D}), \quad \bigl\langle {_V\phi^w_v}, (a_X)_{X \in \mathcal{C}} \bigr\rangle = w(a_Vv)
\end{equation}
which extends the evaluation pairing $\mathcal{O}_q \otimes U_q^{\mathrm{ad}} \to \mathbb{C}(q)$.

\smallskip

\indent For each $\mu \in P_+$ let $V_{\mu}$ be the irreducible $U_q^{\mathrm{ad}}$-module of type $1$ with highest weight $\mu$.
We denote by $C(\mu)$ the subspace of $\mathcal{O}_q$ spanned by the matrix coefficients of $V_{\mu}$. We have $\textstyle \mathcal{O}_q = \bigoplus_{\mu \in P_+} C(\mu)$. We have the decomposition
\[ V_{\mu} \otimes V_{\nu} = V_{\mu + \nu} \oplus \bigoplus_{\lambda < \mu + \nu} N_{\mu,\nu}^{\lambda} V_{\lambda} \]
where the $N_{\mu,\nu}^{\lambda} \in \mathbb{N}$ are multiplicities and $\lambda \leq \kappa$ means that $\kappa - \lambda \in Q_+$. We deduce that
\begin{equation}\label{filtrationOq}
C(\mu) \star C(\nu) = C(\mu + \nu) \oplus \bigoplus_{\lambda < \mu + \nu} \delta_{\mu,\nu}^{\lambda} \, C(\lambda)
\end{equation}
where $\delta_{\mu,\nu}^{\lambda}$ is $0$ or $1$, depending if $N_{\mu,\nu}^{\lambda} = 0$ or $N_{\mu,\nu}^{\lambda} > 0$.

\section{The loop algebra $\mathcal{L}_{0,1}$ and the handle algebra $\mathcal{L}_{1,0}$}\label{sectionHandleAlgebra}
\indent Let $H$ be a quasitriangular Hopf algebra with an invertible antipode. In the first papers on combinatorial quantization like \cite{A, AGS1, BuR1}, the {\em loop algebra} $\mathcal{L}_{0,1}(H)$ and the {\em handle algebra} $\mathcal{L}_{1,0}(H)$ were defined by matrix relations describing the commutation relations in the algebra. Here we give a  more intrinsic definition of $\mathcal{L}_{1,0}(H)$, as a twist of $\mathcal{L}_{0,1}(H)^{\otimes 2}$. The seminal definition of $\mathcal{L}_{1,0}(H)$ based on matrix relations will be recovered in Proposition \ref{presentationL10}.

\subsection{Definition of $\mathcal{L}_{0,1}(H)$}\label{subsectionL10H}
Recall that an invertible element $J \in H^{\otimes 2}$ satisfying
\begin{equation*}
\begin{array}{l}
(J \otimes 1) \, (\Delta \otimes \mathrm{id})(J) = (1 \otimes J) \, (\mathrm{id} \otimes \Delta)(J),\\[.2em]
(\varepsilon \otimes \mathrm{id})(J) = (\mathrm{id} \otimes \varepsilon)(J) = 1.
\end{array}
\end{equation*}
is called a {\em twist} for $H$. We can define a new Hopf algebra $H_J$, called the {\em twist of $H$ by $J$}, which is $H$ as an algebra but whose coproduct, counit and antipode are (for $h\in H$)
\[ \Delta_{H_J}(h) = J \Delta_H(h) J^{-1}, \quad \varepsilon_{H_J}(h) = \varepsilon_H(h), \quad S_{H_J}(h) = u S_H(h) u^{-1} \]
with $\textstyle u = \sum_{(J)} J_{(1)}S(J_{(2)})$ if we write $\textstyle J = \sum_{(J)} J_{(1)} \otimes J_{(2)}$. If $(A,m_A,\cdot)$ is a right $H$-module-algebra we can define a right $H_J$-module-algebra $A_J$, called the {\em twist of $A$ by $J$}, which is $A$ as a vector space and whose product is
\begin{equation}\label{twistedProduct}
m_{A_J}(x \otimes y) = \sum_{(J)} m_A\bigl( x \cdot J_{(1)} \otimes y \cdot J_{(2)}\bigr) 
\end{equation}
(where $m_A$ is the product in $A$). The right action $\cdot : A_J \otimes H_J \to A_J$ is equal to the original action $\cdot : A \otimes H \to A$.

\smallskip

\indent Following the conventions of \cite{BR1,BR2} recalled in \S \ref{sectionHopfAlgebras}, we write the $R$-matrix of $H$ as
\[ R = \sum_{(R)} R_{(1)} \otimes R_{(2)} \in H^{\otimes 2}. \]
Let us recall from \cite[\S 4]{BR1} the construction of the {\em loop algebra} $\mathcal{L}_{0,1}(H)$ as a twist of the restricted dual $H^{\circ}$, i.e. the subspace of $H^*$ spanned by the matrix coefficients of finite-dimensional $H$-modules (see \S\ref{sectionHopfAlgebras}). Denote by $\rhd$ and $\lhd$ the left and right coregular actions of $H$ on $H^{\circ}$ respectively:
\begin{equation}\label{coregularActions}
x\rhd \varphi = \sum_{(\varphi)}\varphi_{(1)} \langle \varphi_{(2)}, x\rangle, \quad \varphi \lhd x = \sum_{(\varphi)} \langle \varphi_{(1)}, x\rangle \varphi_{(2)}
\end{equation}
where $x \in H$, $\varphi \in H^{\circ}$ and we use Sweedler's notation $\textstyle \Delta_{H^{\circ}}(\varphi) = \sum_{(\varphi)} \varphi_{(1)} \otimes \varphi_{(2)}$. Let $H \otimes H^{\mathrm{cop}}$ be endowed with the usual Hopf algebra structure on tensor products. Introduce
\begin{equation}\label{twistL01}
F = R_{32}R_{42} = \sum_{(R^1),(R^2)} (1 \otimes R_{(2)}^1R_{(2)}^2) \otimes (R_{(1)}^1 \otimes R_{(1)}^2) \in (H \otimes H^{\mathrm{cop}})^{\otimes 2}.
\end{equation}
Here we write $\textstyle R_{32} = \sum_{(R)} 1 \otimes R_{(2)} \otimes R_{(1)} \otimes 1$ and $\textstyle R_{42} = \sum_{(R)} 1 \otimes R_{(2)} \otimes 1 \otimes R_{(1)}$, which are embeddings of $R \in H^{\otimes 2}$ into $(H \otimes H^{\mathrm{cop}})^{\otimes 2}$, while $R^1$ and $R^2$ denote two copies of $R$. It is straightforward to check that the element $F$ is a twist for $H \otimes H^{\mathrm{cop}}$, so we can consider the twisted Hopf algebra $(H \otimes H^{\mathrm{cop}})_F$, which we denote by $A_{0,1}(H)$.

\smallskip

\indent Note that $H^{\circ}$ is a right $(H \otimes H^{\mathrm{cop}})$-module-algebra for the action
\begin{equation}\label{actionD}
 \varphi \cdot (x \otimes y) = S(y) \rhd \varphi \lhd x \qquad (\text{with } \varphi \in H^{\circ},\, x \in H,\, y \in H^{\mathrm{cop}}).
\end{equation}
\begin{defi} $\mathcal{L}_{0,1}(H)$ is the twist $(H^{\circ})_F$, with $F$ in \eqref{twistL01}.
\end{defi}
Explicitly, it is the vector space $H^{\circ}$ with the product
\begin{equation}\label{produitL01}
\varphi\psi = \sum_{(R^1),(R^2)} \bigl( R_{(2)}^2 S(R_{(2)}^1) \rhd \varphi \bigr) \star \bigl( R_{(1)}^2 \rhd \psi \lhd R_{(1)}^1 \bigr)
\end{equation}
where $\star$ is the usual product on $H^{\circ}$, given by $\eta \star \gamma = (\eta \otimes \gamma) \circ \Delta_H$. The formula \eqref{produitL01} is derived from \eqref{twistL01} using the general definition of the twisted product in \eqref{twistedProduct} together with the fact that $S$ is an antimorphism of algebras and that $(S \otimes S)(R) =R$.

\smallskip

\indent By construction, $\mathcal{L}_{0,1}(H)$ is a right $A_{0,1}(H)$-module-algebra for the action \eqref{actionD}. The coproduct $\Delta_H$, viewed as a map $H \to A_{0,1}(H)$, is a morphism of Hopf algebras. It follows that $\mathcal{L}_{0,1}(H)$ is a right $H$-module-algebra for the right coadjoint action, given by
\begin{equation}\label{coadL01}
\mathrm{coad}^r(h)(\varphi) = \sum_{(h)} S(h_{(2)}) \rhd \varphi \lhd h_{(1)}.
\end{equation}
 
\subsection{Definition of $\mathcal{L}_{1,0}(H)$} Let us denote
\begin{align*}\gamma & := R_{32} R_{31} R_{42} R_{14}^{-1} \\
& = \sum_{(R^1), \ldots (R^4)} \!\!\!\bigl(R_{(2)}^2 S(R_{(1)}^4) \otimes R_{(2)}^1R_{(2)}^3 \bigr) \otimes \bigl(R_{(1)}^1R_{(1)}^2 \otimes R_{(1)}^3 R_{(2)}^4\bigr) \in A_{0,1}(H)^{\otimes 2}
\end{align*}
where $R^1, \ldots, R^4$ are four copies of $R$ and we used that $R^{-1} = (S \otimes \mathrm{id})(R)$.
\begin{lem}\label{lemmeTwistL10}
The element $\gamma$ is a bicharacter of $A_{0,1}(H)$, which means that
\[ (\Delta_{A_{0,1}(H)} \otimes \mathrm{id})(\gamma) = \gamma_{23}\gamma_{13}, \quad (\mathrm{id} \otimes \Delta_{A_{0,1}(H)})(\gamma) = \gamma_{12} \gamma_{13} \]
where the subscripts denote embeddings of $\gamma \in A_{0,1}(H)^{\otimes 2}$ into $A_{0,1}(H)^{\otimes 3}$. It follows that
\[ \Gamma = 1 \otimes 1 \otimes \gamma \otimes 1 \otimes 1 \in A_{0,1}(H)^{\otimes 2} \otimes A_{0,1}(H)^{\otimes 2} \]
is a twist for $A_{0,1}(H)^{\otimes 2}$.
\end{lem}
\begin{proof}
The two equalities are obtained by straightforward computations using the definition of the coproduct in $A_{0,1}(H)$ and the defining properties of the $R$-matrix given e.g. in \cite[\S 4.2]{CP}; this is left to the reader. The last claim is a general fact.\end{proof}

\noindent It follows that we have the twisted Hopf algebra $\bigl(A_{0,1}(H)^{\otimes 2}\bigr)_{\Gamma}$, which we denote by $A_{1,0}(H)$.
\begin{defi} $\mathcal{L}_{1,0}(H)$ is the twist by $\Gamma$ of the right $A_{0,1}(H)^{\otimes 2}$-module-algebra $\mathcal{L}_{0,1}(H)^{\otimes 2}$:
\[ \mathcal{L}_{1,0}(H) = (\mathcal{L}_{0,1}(H) \otimes \mathcal{L}_{0,1}(H))_{\Gamma}. \]
\end{defi}

\indent Explicitly, $\mathcal{L}_{1,0}(H)$ is the vector space $\mathcal{L}_{0,1}(H) \otimes \mathcal{L}_{0,1}(H)$ with the product
\begin{equation}\label{productL10}
(\beta \otimes \alpha)(\beta' \otimes \alpha') = \sum_{(R^1),\ldots,(R^4)} \beta \bigl( R_{(2)}^4 R_{(1)}^3 \rhd \beta' \lhd R_{(1)}^1 R_{(1)}^2 \bigr) \otimes \bigl( R_{(2)}^3 S(R_{(2)}^1) \rhd \alpha \lhd R_{(2)}^2 R_{(1)}^4  \bigr) \alpha'
\end{equation}
where we used that $(S \otimes S)(R) =R$. Note that the maps
\begin{equation}\label{embeddingL01inL10}
\fonc{\mathfrak{i}_A}{\mathcal{L}_{0,1}(H)}{\mathcal{L}_{1,0}(H)}{\alpha}{1 \otimes \alpha}, \qquad \fonc{\mathfrak{i}_B}{\mathcal{L}_{0,1}(H)}{\mathcal{L}_{1,0}(H)}{\beta}{\beta \otimes 1}
\end{equation}
are embeddings of algebras, and we have
\begin{equation}\label{factorisationElementsL10}
\beta \otimes \alpha = \mathfrak{i}_B(\beta) \, \mathfrak{i}_A(\alpha).
\end{equation}
Therefore the handle algebra $\mathcal{L}_{1,0}(H)$ contains two distinct copies of the loop algebra $\mathcal{L}_{0,1}(H)$. They will be used to give a presentation of $\mathcal{L}_{1,0}(H)$ in Proposition \ref{presentationL10}. The choice of the subscripts $A$ and $B$ to denote the embeddings will be explained in \eqref{matABmars25}. 

\smallskip

\begin{lem}\label{lemmaIteratedDelta10}
The iterated coproduct $\textstyle \Delta^{(3)}_H : h \mapsto \sum_{(h)} h_{(1)} \otimes h_{(2)} \otimes h_{(3)} \otimes h_{(4)}$, viewed as a map $H \to A_{1,0}(H)$, is a morphism of Hopf algebras.
\end{lem}
\begin{proof}
Straightforward computations left to the reader.
\end{proof}
\noindent Since by definition $\mathcal{L}_{1,0}(H)$ is a right $A_{1,0}(H)$-module-algebra, the lemma implies that $\mathcal{L}_{1,0}(H)$ is a right $H$-module-algebra with action
\begin{equation}\label{coadL10}
\mathrm{coad}^r(h)(\beta \otimes \alpha) = \sum_{(h)} \mathrm{coad}^r(h_{(1)})(\beta) \otimes \mathrm{coad}^r(h_{(2)})(\alpha)
\end{equation}
where we use the action $\mathrm{coad}^r$ from \eqref{coadL01}. In particular the maps $\mathfrak{i}_A, \mathfrak{i}_B$ in \eqref{embeddingL01inL10} are embeddings of $H$-module-algebras.

\smallskip

\indent We now recover the definition of $\mathcal{L}_{1,0}(H)$ by matrix relations, which is very helpful in certain proofs. To introduce the notations, let us first recall the case of $\mathcal{L}_{0,1}(H)$, which was treated in \cite[Prop. 4.6]{BR1}. Let $V$ be a finite-dimensional $H$-module, $(v_i)$ be a basis of $V$ and $(v^i)$ be its dual basis. Denote by $_V\phi^i_j$ the matrix coefficients in this basis: $_V\phi^i_j(h) = v^i(h \cdot v_j)$ for $h \in H$. Let $E_{ij}$ be the basis elements of $\mathrm{End}(V)$ given by $E_{ij}(v_k) = \delta_{jk}v_i$. We define
\begin{equation}\label{defMatrixL01}
\overset{V}{M} = \sum_{i,j} {_V\phi^i_j} \otimes E_{ij} \in \mathcal{L}_{0,1}(H) \otimes \mathrm{End}(V)
\end{equation}
which can be seen as a matrix of size $\dim(V)$ with coefficients in $\mathcal{L}_{0,1}(H)$. Note that $\overset{V}{M}$ does not depend on the choice of a basis in $V$. By its very definition, the collection of matrices $M$ satisfies {\em naturality}:
\begin{equation}\label{naturalite}
(\mathrm{id} \otimes f)\overset{V}{M} = \overset{W}{M}(\mathrm{id} \otimes f)
\end{equation}
for any $H$-linear morphism $f : V \to W$.
Consider the embeddings
\[ \fonc{j_1}{\mathrm{End}(V)}{\mathrm{End}(V) \otimes \mathrm{End}(W)}{X}{X \otimes \mathrm{id}_W}, \quad \fonc{j_2}{\mathrm{End}(W)}{\mathrm{End}(V) \otimes \mathrm{End}(W)}{Y}{\mathrm{id}_V \otimes Y} \]
where $V$ and $W$ are any finite-dimensional $H$-modules. We define
\[ \overset{V}{M}_1 = (\mathrm{id} \otimes j_1)\bigl(\overset{V}{M}\bigr), \:\: \overset{W}{M}_2 = (\mathrm{id} \otimes j_2)\bigl(\overset{W}{M}\bigr) \in \mathcal{L}_{0,1}(H) \otimes \mathrm{End}(V) \otimes \mathrm{End}(W). \]
Note that $\overset{V \otimes W}{M}$ can be viewed as an element of $\mathcal{L}_{0,1}(H) \otimes \mathrm{End}(V) \otimes \mathrm{End}(W)$ since $\mathrm{End}(V \otimes W) = \mathrm{End}(V) \otimes \mathrm{End}(W)$. Finally $x_{V,W}$ denotes the representation of some $x \in H^{\otimes 2}$ on $V \otimes W$ and we implicitly identify $x_{V,W}$ with $1 \otimes x_{V,W} \in \mathcal{L}_{0,1}(H) \otimes \mathrm{End}(V) \otimes \mathrm{End}(W)$. Then $\mathcal{L}_{0,1}(H)$ is spanned by the coefficients of the matrices $\overset{V}{M}$ for all $V$ in $H\text{-}\mathrm{mod}$ and \eqref{produitL01} is equivalent to the set of all {\it fusion relations} :
\begin{equation}\label{fusionrelation}
\overset{V \otimes W}{M} = \overset{V}{M}_1 \, (R')_{V,W} \, \overset{W}{M}_2 \, (R')_{V,W}^{-1}
\end{equation}
where $\textstyle R' = \sum_{(R)} R_{(2)} \otimes R_{(1)}$. Thanks to the braiding in $H\text{-}\mathrm{mod}$ and naturality \eqref{naturalite}, it is equivalent to
\begin{equation}\label{fusionrelationbis}
\overset{V \otimes W}{M} = R_{V,W}^{-1} \, \overset{W}{M}_2 \, R_{V,W}\, \overset{V}{M}_1.
\end{equation}
The equivalence of \eqref{fusionrelation} and \eqref{fusionrelationbis} gives the so-called reflection equation:
\begin{equation}\label{reflectionEquation}
R_{V,W} \overset{V}{M}_1 \, (R')_{V,W} \, \overset{W}{M}_2 = \overset{W}{M}_2 \, R_{V,W}\, \overset{V}{M}_1 (R')_{V,W}.
\end{equation}
From the fusion relation one can check that $\overset{V}{M}$ is invertible, see e.g. \cite[Prop. 3.3]{FaitgHol}. 

\indent For $\mathcal{L}_{1,0}(H)$, consider the matrices
\begin{equation}\label{matABmars25} \overset{V}{A} = (\mathfrak{i}_A \otimes \mathrm{id})\bigl(\overset{V}{M}\bigr),  \:\: \overset{V}{B} = (\mathfrak{i}_B \otimes \mathrm{id})\bigl(\overset{V}{M}\bigr) \in \mathcal{L}_{1,0}(H) \otimes \mathrm{End}(V).
\end{equation}
Since the matrix $\overset{V}{M}$ is invertible, so are $\overset{V}{A}$ and $\overset{V}{B}$. In mathematical physics, $\overset{V}{A}$ and $\overset{V}{B}$ are seen as quantum holonomy matrices associated to the usual generators $a,b$ of the fundamental group of the torus with an open disk removed. They can be given a stated skein theoretic interpretation (\cite[\S 4]{FaitgHol}), which will be developed in \S \ref{sectionTopologicalInterpretation} and \S \ref{QMMLgn}.

\begin{prop}\label{presentationL10}
The following relations hold true for any finite-dimensional $H$-modules $V$, $W$:
\begin{align*}
&\overset{V \otimes W}{A} = \overset{V}{A}_1 \, (R')_{V,W} \, \overset{W}{A}_2 \, (R')_{V,W}^{-1} \quad \text{(fusion relation)}\\
&\overset{V \otimes W}{B} = \overset{V}{B}_1 \, (R')_{V,W} \, \overset{W}{B}_2 \, (R')_{V,W}^{-1} \quad \text{(fusion relation)}\\
&R_{V,W} \, \overset{V}{B}_1 \, (R')_{V,W} \, \overset{W}{A}_2 = \overset{W}{A}_2 \, R_{V,W} \, \overset{V}{B}_1 \, R_{V,W}^{-1} \quad \text{(exchange relation)}
\end{align*}
where $\textstyle R' = \sum_{(R)} R_{(2)} \otimes R_{(1)}$. These matrix equalities entirely describe the product in $\mathcal{L}_{1,0}(H)$.
\end{prop}
\begin{proof}
We see from \eqref{productL10} and \eqref{factorisationElementsL10} that the product in $\mathcal{L}_{1,0}(H)$ is entirely described by the following formulas:
\begin{equation}\label{productL10WithEmbeddings}
\begin{array}{rl}
\mathfrak{i}_A(\varphi)\mathfrak{i}_A(\psi) &= \mathfrak{i}_A(\varphi \psi),\\
\mathfrak{i}_B(\varphi)\mathfrak{i}_B(\psi) &= \mathfrak{i}_B(\varphi \psi),\\
\mathfrak{i}_A(\varphi)\mathfrak{i}_B(\psi) &= \sum_{(R^1),\ldots,(R^4)} \!\!\!\!\!\!\mathfrak{i}_B\bigl( R_{(2)}^4 R_{(1)}^3 \rhd \psi \lhd R_{(1)}^1 R_{(1)}^2 \bigr) \mathfrak{i}_A\bigl( R_{(2)}^3 S(R_{(2)}^1) \rhd \varphi \lhd R_{(2)}^2 R_{(1)}^4  \bigr).
\end{array}
\end{equation}
Assume that $\varphi$, $\psi$ are matrix coefficients ${_V\phi^i_j}$, ${_W\phi^k_l}$ in some bases of $V$ and $W$. We have already recalled above that the product in $\mathcal{L}_{0,1}(H)$ for matrix coefficients is equivalent to \eqref{fusionrelation}. Hence the two first equalities in \eqref{productL10WithEmbeddings} for matrix coefficients are equivalent to the fusion relations for the matrices $\overset{V}{A}$ and $\overset{V}{B}$. The exchange relation is equivalent to the third equality in \eqref{productL10WithEmbeddings} for matrix coefficients, as we now show. Note that by the very definition of a representation we have
\[ \sum_{i,j} (h \rhd {_V\phi}^i_j) \otimes E_{ij} = \sum_{i,j} {_V\phi}^i_j \otimes (E_{ij} \, h_V), \qquad \sum_{i,j} ({_V\phi}^i_j \lhd h) \otimes E_{ij} = \sum_{i,j} {_V\phi}^i_j \otimes (h_V E_{ij}) \]
where $h_V \in \mathrm{End}(V)$ is the representation of $h \in H$ on $V$. 
Hence
\begin{align*}
&\overset{W}{A}_2 \, \overset{V}{B}_1 = \sum_{i,j,k,l} \mathfrak{i}_A({_W\phi}^k_l)\, \mathfrak{i}_B({_V\phi}^i_j) \otimes E_{ij} \otimes E_{kl}\\
=& \sum_{\substack{i,j,k,l\\(R^1),\ldots,(R^4)}} \mathfrak{i}_B\bigl( R_{(2)}^4 R_{(1)}^3 \rhd {_V\phi}^i_j \lhd R_{(1)}^1 R_{(1)}^2 \bigr)\,\mathfrak{i}_A\bigl( R_{(2)}^3 S(R_{(2)}^1) \rhd {_W\phi}^k_l \lhd R_{(2)}^2 R_{(1)}^4 \bigr) \otimes E_{ij} \otimes E_{kl}\\
=& \sum_{\substack{i,j,k,l\\(R^1), \ldots, (R^4)}} \mathfrak{i}_B({_V\phi}^i_j)\,\mathfrak{i}_A({_W\phi}^k_l) \otimes \bigl(R_{(1)}^1\bigr)_V \, \bigl(R_{(1)}^2\bigr)_V \, E_{ij} \, \bigl(R_{(2)}^4\bigr)_V \, \bigl(R_{(1)}^3\bigr)_V\\[-1.2em]
& \qquad\qquad\qquad\qquad\qquad\qquad\qquad\qquad\qquad\qquad \otimes \bigl(R_{(2)}^2\bigr)_W \, \bigl(R_{(1)}^4\bigr)_W \, E_{kl} \, \bigl(R_{(2)}^3\bigr)_W \, S\bigl(R_{(2)}^1\bigr)_W\\
=& \sum_{(R^1), \ldots, (R^4)} \bigl(R_{(1)}^1\bigr)_{V1} \, \bigl(R_{(1)}^2\bigr)_{V1} \, \overset{V}{B}_1 \, \bigl(R_{(2)}^4\bigr)_{V1} \, \bigl(R_{(1)}^3\bigr)_{V1} \, \bigl(R_{(2)}^2\bigr)_{W2} \, \bigl(R_{(1)}^4\bigr)_{W2} \, \overset{W}{A}_2 \, \bigl(R_{(2)}^3\bigr)_{W2} \, S\bigl(R_{(2)}^1\bigr)_{W2}\\
=& \sum_{(R)} \bigl(R_{(1)}\bigr)_{V1} \, R_{V,W} \, \overset{V}{B}_1 \, (R')_{V,W} \, \overset{W}{A}_2 \, R_{V,W} \, S\bigl(R_{(2)}\bigr)_{W2}.
\end{align*}
Using that $\textstyle \sum_{(R^1),(R^2)} R^1_{(1)}R^2_{(1)} \otimes S(R^2_{(2)})R^1_{(2)} = 1 \otimes 1$, we get
\[ R_{V,W} \, \overset{V}{B}_1 \, (R')_{V,W} \, \overset{W}{A}_2 = \sum_{(R)}  \bigl( R_{(1)} \bigr)_{V1} \, \overset{W}{A}_2 \, \overset{V}{B}_1 \, \bigl(R_{(2)}\bigr)_{W2} \, R_{V,W}^{-1} = \overset{W}{A}_2 \, R_{V,W} \, \overset{V}{B}_1 \, R_{V,W}^{-1} \]
as claimed.
\\As a result the equalities in \eqref{productL10WithEmbeddings}, which determine the product in $\mathcal{L}_{1,0}(H)$, are equivalent to the matrix equalities when $\varphi$ and $\psi$ are matrix coefficients. But the matrix coefficients span $\mathcal{L}_{0,1}(H)$ as a vector space, so if we know the formulas \eqref{productL10WithEmbeddings} for matrix coefficients we can actually deduce the product in $\mathcal{L}_{1,0}(H)$, which proves the last claim.
\end{proof}
\medskip

We will need the following result when discussing the quantum moment maps (see Section \ref{QMMLgn}). Assume now that $H$ has a ribbon element $v$. For every finite-dimensional $H$-module $V$ define matrices in $\mathcal{L}_{1,0}(H) \otimes {\rm End}(V)$ by:
\begin{equation}\label{matpartielles}
\overset{V}{X} =v_V\overset{V}{B}\overset{V}{A}{}^{-1}\ ,\quad \overset{V}{Y} =v_V\overset{V}{B}{}^{-1}\overset{V}{A}
\end{equation}
where $v_V \in \mathrm{End}(V)$ is the representation of $v$ on $V$, and is identified with the element $1_{\mathcal{L}_{1,0}(H)} \otimes v_V\in \mathcal{L}_{1,0}(H) \otimes {\rm End}(V)$.
\begin{lem}\label{fusionforXY} The matrices $\overset{V}{X}$, $\overset{V}{Y}$, and $\overset{V}{X}\overset{V}{Y} =v_V^2\overset{V}{B}\overset{V}{A}{}^{-1}\overset{V}{B}{}^{-1}\overset{V}{A}$ satisfy the fusion relation \eqref{fusionrelation}.
\end{lem}
\begin{proof}Recall first that $v$ is central. It follows that $v_V$ is a $H$-linear endomorphism so by naturality  \eqref{naturalite} we have $v_V \overset{V}{M} = \overset{V}{M} v_V$. Moreover $(v_V)_1 \, x_{V,W} = x_{V,W} \, (v_V)_1$ and $(v_W)_2 \, x_{V,W} = x_{V,W} \, (v_W)_2$ for all $x \in H^{\otimes 2}$. Let us show the claim for $\overset{V}{X}$; explanations are below the computation: 
\begin{align*}
\overset{V\otimes W}{X} & = (v_V)_1 \, (v_W)_2 \, (R'R)_{V,W}^{-1} \, \overset{V}{B}_1 \, (R')_{V,W} \, \overset{W}{B}_2 \, (R')_{V,W}^{-1}\, (R')_{V,W}\, \overset{W}{A}{}^{-1}_2 (R')_{V,W}^{-1}\, \overset{V}{A}{}^{-1}_1 \\
& = (v_V)_1 \, (v_W)_2 \, R_{V,W}^{-1} \, (R')_{V,W}^{-1} \, \overset{V}{B}_1 \, (R')_{V,W} \, \overset{W}{B}_2 \,  \overset{W}{A}{}^{-1}_2 (R')_{V,W}^{-1}\, \overset{V}{A}{}^{-1}_1\\
& = (v_V)_1 \, (v_W)_2 \, R_{V,W}^{-1}\, \overset{W}{B}_2 \, R_{V,W} \, \overset{V}{B}_1 \,  R_{V,W}^{-1}\, \overset{W}{A}{}^{-1}_2 (R')_{V,W}^{-1}\, \overset{V}{A}{}^{-1}_1\\
&= (v_V)_1 \, (v_W)_2 \, R_{V,W}^{-1}\, \overset{W}{B}_2 \, \overset{W}{A}{}^{-1}_2 \, R_{V,W} \, \overset{V}{B}_1 \,  \overset{V}{A}{}^{-1}_1\\
& = R_{V,W}^{-1} \, \overset{W}{X}_2 \, R_{V,W}\, \overset{V}{X}_1 = \overset{V}{X}_1 \, R'_{V,W} \, \overset{W}{X}_2 \, (R')_{V,W}^{-1}.
\end{align*}
For the first equality we used \eqref{vproperties} and the fusion relations for the matrices $B$ and $A$, for the third equality we used a variant of \eqref{reflectionEquation}, for the fourth equality we used the exchange relation in Proposition \ref{presentationL10}, for the fifth equality we used the remarks made at the begining of the proof and for the sixth equality we used the equivalence of \eqref{fusionrelationbis} and \eqref{fusionrelation}.
\\\noindent The same arguments apply to $\overset{V}{Y}$. In order to prove the claim for $\overset{V}{X}\overset{V}{Y}$ we first show that $\overset{V}{X}$ and $\overset{V}{Y}$ satisfy the exchange relation
$$R_{V,W}\, \overset{V}{X}_1 \, R_{V,W}^{-1} \,  \overset{W}{Y}_2 = \overset{W}{Y}_2 R_{V,W}\,  \overset{V}{X}_1 \, R_{V,W}^{-1}.$$
Since $v$ is central, the observations at the begining of this proof imply that the above exchange relation is equivalent to
$$R_{V,W}\, \overset{V}{B}_1\overset{V}{A}{}^{-1} _1\, R_{V,W}^{-1} \,  \overset{W}{B}{}^{-1}_2\overset{W}{A}_2 = \overset{W}{B}{}^{-1}_2\overset{W}{A}_2 \, R_{V,W}\,  \overset{V}{B}_1\overset{V}{A}{}^{-1}_1 \, R_{V,W}^{-1}.$$
Now, by using again \eqref{reflectionEquation} and the exchange relation in Proposition \ref{presentationL10} we get
\begin{align*}
R_{V,W}\, \overset{V}{B}_1\overset{V}{A}{}^{-1} _1\, R_{V,W}^{-1} \,  \overset{W}{B}{}^{-1}_2\overset{W}{A}_2  & = R_{V,W}\, \overset{V}{B}_1\, R'_{V,W}\, \overset{W}{B}{}^{-1}_2\, (R'_{V,W})^{-1} \,
\overset{V}{A}{}^{-1} _1 \, R'_{V,W}\, \overset{W}{A}_2\\
 & = \overset{W}{B}{}^{-1}_2 \, R_{V,W}\, \overset{V}{B}_1\, R'_{V,W}\, (R'_{V,W})^{-1} \,
\overset{V}{A}{}^{-1} _1 \, R'_{V,W}\, \overset{W}{A}_2\\
& = \overset{W}{B}{}^{-1}_2 \, R_{V,W}\, \overset{V}{B}_1\, R'_{V,W}\, \overset{W}{A}_2 \, R_{V,W}\, \overset{V}{A}{}^{-1} _1 \,  R_{V,W}^{-1}\\
&= \overset{W}{B}{}^{-1}_2 \,  \overset{W}{A}_2 \, R_{V,W}\, \overset{V}{B}_1\, R_{V,W}^{-1}\, R_{V,W}\, \overset{V}{A}{}^{-1} _1 \,  R_{V,W}^{-1}\\
& = \overset{W}{B}{}^{-1}_2\overset{W}{A}_2 \, R_{V,W}\,  \overset{V}{B}_1\overset{V}{A}{}^{-1}_1 \, R_{V,W}^{-1}.
\end{align*}
We can now deduce that $\overset{V}{X}\overset{V}{Y}$ satisfies the fusion relation:
\begin{align*}
\overset{V\otimes W}{X}\ \overset{V\otimes W}{Y} & = \overset{V}{X}_1 \, R'_{V,W} \, \overset{W}{X}_2 \, (R')_{V,W}^{-1}\, \overset{V}{Y}_1 \, R'_{V,W} \, \overset{W}{Y}_2 \, (R')_{V,W}^{-1}\\
& = \overset{V}{X}_1 \, R'_{V,W} \, (R')^{-1}_{V,W}\, \overset{V}{Y}_1 \, R'_{V,W} \, \overset{W}{X}_2 \, (R')^{-1}_{V,W}\, R'_{V,W} \, \overset{W}{Y}_2 \, (R')_{V,W}^{-1}\\
& = \overset{V}{X}_1 \, \overset{V}{Y}_1 \, R'_{V,W} \, \overset{W}{X}_2 \,  \overset{W}{Y}_2 \, (R')_{V,W}^{-1}.\qedhere
\end{align*}
\end{proof}

\subsection{Morphism $\mathcal{L}_{1,0}(H) \to \mathcal{H}(H^{\circ})$}\label{sectionHeisenberg} Since the left coregular action  \eqref{coregularActions} endows $H^{\circ}$ with a structure of $H$-module-algebra, we can consider the smash product $H^{\circ} \# H$, which is denoted by $\mathcal{H}(H^{\circ})$ and is called the Heisenberg double of $H^{\circ}$ \cite[\S 4.1.10]{Mon}. Explicitly, the algebra $\mathcal{H}(H^{\circ})$ is the vector space $H^{\circ} \otimes H$ with the product
\begin{equation}\label{productHeisenberg}
(\varphi \,\#\, x)(\psi \,\#\, y) = \sum_{(x)}\varphi \star (x_{(1)} \rhd \psi) \,\#\, x_{(2)}y
\end{equation}
where we write $\varphi \,\#\, x$ for the element $\varphi \otimes x \in H^{\circ} \otimes H$ and $\star$ is the usual product in $H^{\circ}$.

\smallskip

\indent Consider the following {\em right} action of $H$ on $\mathcal{H}(H^{\circ})$, which will be used in Proposition \ref{propPhi10} below:
\begin{equation}\label{actionOnHeisenberg}
(\varphi \,\#\, x) \cdot h = \sum_{(h)} S(h_{(2)}) \rhd \varphi \lhd h_{(3)} \,\#\, S(h_{(1)})xh_{(4)}.
\end{equation}

\begin{lem}
The action \eqref{actionOnHeisenberg} endows $\mathcal{H}(H^{\circ})$ with a structure of right $H$-module-algebra.
\end{lem}
\begin{proof}
We compute:
\begin{align*}
&\sum_{(h)} \bigl( (\varphi \,\#\, x) \cdot h_{(1)} \bigr) \bigl( (\psi \,\#\, y) \cdot h_{(2)} \bigr)\\
=\:& \sum_{(h)} \bigl( S(h_{(2)}) \rhd \varphi \lhd h_{(3)} \,\#\, S(h_{(1)})xh_{(4)} \bigr) \bigl( S(h_{(6)}) \rhd \psi \lhd h_{(7)} \,\#\, S(h_{(5)})yh_{(8)} \bigr)\\
=\:& \sum_{\substack{(x),(h)\\(S(h_{(1)})),(h_{(4)})}} \bigl( S(h_{(2)}) \rhd \varphi \lhd h_{(3)} \bigr) \star \bigl( S(h_{(1)})_{(1)}x_{(1)}h_{(4)(1)}S(h_{(6)}) \rhd \psi \lhd h_{(7)} \bigr)\\[-1.7em]
&\qquad\qquad\qquad\qquad\qquad\qquad\qquad\qquad\qquad\qquad \,\#\, S(h_{(1)})_{(2)}x_{(2)}h_{(4)(2)}S(h_{(5)})yh_{(8)}\\[.8em]
=\:& \sum_{(x),(h)} \bigl( S(h_{(3)}) \rhd \varphi \lhd h_{(4)} \bigr) \star \bigl( S(h_{(2)})x_{(1)}h_{(5)}S(h_{(8)}) \rhd \psi \lhd h_{(9)} \bigr)\\[-1em]
&\qquad\qquad\qquad\qquad\qquad\qquad\qquad\qquad\qquad\qquad \,\#\, S(h_{(1)})x_{(2)}h_{(6)}S(h_{(7)})yh_{(10)}\\[.8em]
=\:& \sum_{(x),(h)} \bigl( S(h_{(3)}) \rhd \varphi \lhd h_{(4)} \bigr) \star \bigl( S(h_{(2)})x_{(1)} \rhd \psi \lhd h_{(5)} \bigr) \,\#\,S(h_{(1)})x_{(2)}yh_{(6)}\\
=\:& \sum_{(x),(h)} S(h_{(2)}) \rhd \bigl(\varphi \star (x_{(1)} \rhd \psi)\bigr) \lhd h_{(3)} \,\#\, S(h_{(1)})x_{(2)}yh_{(4)} = \bigl( (\varphi \,\#\, x)(\psi \,\#\, y) \bigr) \cdot h
\end{align*}
We used the properties of the antipode $S$ and for the last step we used that both $\rhd$ and $\lhd$ are $H$-module-algebra structures on $H^{\circ}$.
\end{proof}

\indent Recall the map (see \cite[Th. 4.3]{BR1} and the references therein)
\begin{equation}\label{RSDmap}
\fonc{\Phi_{0,1}}{\mathcal{L}_{0,1}(H)}{H}{\varphi}{(\varphi \otimes \mathrm{id})(RR') = \sum_{(R^1), (R^2)} \bigl\langle \varphi, R^1_{(1)}R^2_{(2)} \bigr\rangle R^1_{(2)} R^2_{(1)}}
\end{equation}
where $R^1$, $R^2$ denote two copies of $R \in H^{\otimes 2}$. The map $\Phi_{0,1}$ is a morphism of $H$-module-algebras when $\mathcal{L}_{0,1}(H)$ is endowed with the action $\mathrm{coad}^r$ in \eqref{coadL01} and $H$ is endowed with the right adjoint action $\textstyle \mathrm{ad}^r(h)(x) = \sum_{(h)}S(h_{(1)})xh_{(2)}$. Moreover it take values in $H^{\mathrm{lf}}$, the subspace of locally finite elements for $\mathrm{ad}^r$. 

\smallskip

\indent The morphism \eqref{RSDmap} relates $\mathcal{L}_{0,1}(H)$ and $H$. We now discuss a morphism which relates $\mathcal{L}_{1,0}(H)$ and $\mathcal{H}(H^{\circ})$. Let
\begin{equation}\label{defLFHeisenberg}
\mathcal{H}(H^{\circ})^{\mathrm{lf}} = \bigl\{ v \in \mathcal{H}(H^{\circ}) \, \bigl| \, \dim(v \cdot H) < \infty \bigr. \bigr\}
\end{equation}
be the subspace of locally finite elements for the action \eqref{actionOnHeisenberg}. If $H$ is finite-dimensional, we obviously have $\mathcal{H}(H^{\circ})^{\mathrm{lf}} = \mathcal{H}(H^{\circ})$; but we will see later in the case of $H = U_q^{\mathrm{ad}}$ that $\mathcal{H}(H^{\circ})^{\mathrm{lf}}$ is in general a strict subspace. The map $\Phi_{1,0}$ in the next proposition was introduced (in a different form) in \cite[\S 5]{A}.

\begin{prop}\label{propPhi10}
We have a morphism of right $H$-module-algebras
\[ \begin{array}{lrcl}
\Phi_{1,0} : & \mathcal{L}_{1,0}(H) & \longrightarrow & \mathcal{H}(H^{\circ})\\
 & \beta \otimes \alpha & \longmapsto & \sum_{(R^1), (R^2),(R^3)} \left(R^1_{(2)} R^2_{(2)} \rhd \beta \lhd R^3_{(1)} R^1_{(1)} \right) \,\#\, \left(R^3_{(2)} R^2_{(1)} \Phi_{0,1}(\alpha)\right)
\end{array} \]
where $R^1, R^2, R^3$ are three copies of $R \in H^{\otimes 2}$ and $\mathcal{H}(H^{\circ})$ is endowed with the action \eqref{actionOnHeisenberg}. The morphism $\Phi_{1,0}$ actually takes values in $\mathcal{H}(H^{\circ})^{\mathrm{lf}}$.
\end{prop}
\begin{proof}
Let us first show that $\Phi_{1,0}$ is a morphism of right $H$-modules. Due to \eqref{factorisationElementsL10} it is enough to show that $\Phi_{1,0} \circ \mathfrak{i}_A$ and $\Phi_{1,0} \circ \mathfrak{i}_B$ are $H$-linear morphisms $\mathcal{L}_{0,1}(H) \to \mathcal{H}(H^{\circ})$. Since $\Phi_{1,0}\bigl( \mathfrak{i}_A(\alpha) \bigr) = 1_{H^{\circ}} \,\#\, \Phi_{0,1}(\alpha)$, we use that
$\Phi_{0,1}$ intertwines $\mathrm{coad}^r$ and $\mathrm{ad}^r$ to get:
\begin{align*}
\Phi_{1,0}\bigl( \mathrm{coad}^r(h)(\mathfrak{i}_A(\alpha)) \bigr) &= \Phi_{1,0}\circ \mathfrak{i}_A\bigl( \mathrm{coad}^r(h)(\alpha) \bigr) = 1_{H^{\circ}} \,\#\, \Phi_{0,1}\bigl( \mathrm{coad}^r(h)(\alpha) \bigr)\\
& = \sum_{(h)} 1_{H^{\circ}} \,\#\, S(h_{(1)}) \Phi_{0,1}(\alpha) h_{(2)} = \bigl(1_{H^{\circ}} \,\#\, \Phi_{0,1}(\alpha)\bigr)\cdot h.
\end{align*}
For $\mathfrak{i}_B(\beta)$ we compute:
\begin{align*}
&\Phi_{1,0}\bigl( \mathrm{coad}^r(h)(\mathfrak{i}_B(\beta)) \bigr) = \Phi_{1,0}\circ \mathfrak{i}_B\bigl( \mathrm{coad}^r(h)(\beta) \bigr)\\
=\:& \sum_{(R^1), (R^2),(R^3),(h)} \bigl(R^1_{(2)} R^2_{(2)} S(h_{(2)})\rhd \beta \lhd h_{(1)}R^3_{(1)} R^1_{(1)} \bigr) \,\#\, R^3_{(2)} R^2_{(1)}\\
=\:& \sum_{(R^1), (R^2),(R^3),(h)} \bigl(R^1_{(2)} R^2_{(2)} S(h_{(4)})\rhd \beta \lhd h_{(3)}R^3_{(1)} R^1_{(1)} \bigr) \,\#\, S(h_{(1)})h_{(2)}R^3_{(2)} R^2_{(1)}S(h_{(5)})h_{(6)}\\
=\:& \sum_{(R^1), (R^2),(R^3),(h)} \bigl(R^1_{(2)} S(h_{(5)}) R^2_{(2)} \rhd \beta \lhd R^3_{(1)} h_{(2)} R^1_{(1)} \bigr) \,\#\, S(h_{(1)})R^3_{(2)} h_{(3)} S(h_{(4)}) R^2_{(1)}h_{(6)}\\
=\:& \sum_{(R^1), (R^2),(R^3),(h)} \bigl(S\bigl(h_{(3)}S^{-1}(R^1_{(2)})\bigr) R^2_{(2)} \rhd \beta \lhd R^3_{(1)} h_{(2)} R^1_{(1)} \bigr) \,\#\, S(h_{(1)})R^3_{(2)} R^2_{(1)}h_{(4)}\\
=\:& \sum_{(R^1), (R^2),(R^3),(h)} \bigl(S(h_{(2)})R^1_{(2)} R^2_{(2)} \rhd \beta \lhd R^3_{(1)} R^1_{(1)} h_{(3)} \bigr) \,\#\, S(h_{(1)})R^3_{(2)} R^2_{(1)}h_{(4)}\\
=\:& \Phi_{1,0}\bigl( \mathfrak{i}_B(\beta) \bigr) \cdot h.
\end{align*}
We simply used that $R \Delta_H = \Delta^{\mathrm{op}}_H R$ and that $(\mathrm{id} \otimes S^{-1})(R) = R^{-1}$.
\\The simplest way to show that $\Phi_{1,0}$ is a morphism of algebras uses the description of the product in $\mathcal{L}_{1,0}(H)$ based on the matrices $\overset{V}{A}$, $\overset{V}{B}$ (Proposition \ref{presentationL10}); see \cite{A} (with different conventions), and \cite[Prop 4.6]{Faitg3} where this is done in detail in the present setup.
\\The last claim is due to the fact that all the elements of $\mathcal{L}_{1,0}(H)$ are locally finite. Indeed, as $H$-modules $\mathcal{L}_{1,0}(H)$ is $\mathcal{L}_{0,1}(H) \otimes \mathcal{L}_{0,1}(H)$ with the diagonal $coad^r$ action (see \eqref{coadL10}), and $\mathcal{L}_{0,1}(H) = H^\circ$ under $coad^r$, which is locally finite.
\end{proof}

\subsection{The case $H = U^{\mathrm{ad}}_q(\mathfrak{g})$}\label{SectionL10ForUq}
Recall from \S \ref{sectionPrelimUq} that $U_q = U_q(\mathfrak{g})$ and $U^{\mathrm{ad}}_q = U^{\mathrm{ad}}_q(\mathfrak{g})$ denote respectively the simply connected and adjoint quantum groups, which are Hopf algebra over $\mathbb{C}(q)$ and where $q$ is an indeterminate. The previous definitions of $\mathcal{L}_{0,1}(H)$ and $\mathcal{L}_{1,0}(H)$ must be adapted when $H = U^{\mathrm{ad}}_q$ because this Hopf algebra is not quasitriangular in the usual sense. Indeed, as explained in \S \ref{sectionCategoricalCompletion}, the $R$-matrix lives in the categorical completion $\mathbb{U}_q^{\otimes 2}$. We can overcome this issue thanks to the pairing \eqref{pairingOqCategoricalCompletion} and the fact that for any $\varphi \in \mathcal{O}_q$
\begin{equation}\label{eqWellDefRMatrixOnOq}
\sum_{(R)} \langle \varphi, R_{(1)} \rangle R_{(2)} \in U_q  \quad \text{and} \quad \sum_{(R)} \langle \varphi, R_{(2)} \rangle R_{(1)} \in U_q
\end{equation}
(note that {\it a priori} these elements are in $\mathbb{U}_q$).

\smallskip

\indent As in \S\ref{sectionOq}, we denote by $\mathcal{O}_q(q^{1/D})$ the extension of scalars to $\mathbb{C}(q^{1/D})$ of $\mathcal{O}_q = \mathcal{O}_q(G)$. Using \eqref{pairingOqCategoricalCompletion} we extend the left and right coregular actions \eqref{coregularActions} of $U_q^{\mathrm{ad}}$ on $\mathcal{O}_q$ to left and right coregular actions of $\mathbb{U}_q$ on $\mathcal{O}_q(q^{1/D})$. Hence $\mathcal{O}_q(q^{1/D})$ is a right $(\mathbb{U}_q \otimes \mathbb{U}^{\mathrm{cop}}_q)$-module-algebra for the action \eqref{actionD}. Recall from \S\ref{sectionCategoricalCompletion} that $R \in \mathbb{U}_q^{\otimes 2}$ has coefficients in $\mathbb{C}(q^{1/D})$, thus the same is true for the twist $F$ defined in \eqref{twistL01}. We have the twisted Hopf algebra $\mathbb{A}_{0,1} = (\mathbb{U}_q \otimes \mathbb{U}^{\mathrm{cop}}_q)_F$ and the twisted module-algebra $\mathcal{O}_q(q^{1/D})_F$. The latter is the $\mathbb{C}(q^{1/D})$-vector space $\mathcal{O}_q(q^{1/D})$ with the product \eqref{produitL01}, which makes sense due to \eqref{eqWellDefRMatrixOnOq}. In \cite[Prop. 4.1]{BR1} it is shown that restricting the product of $\mathcal{O}_q(q^{1/D})_F$ on the $\mathbb{C}(q)$-subspace $\mathcal{O}_q$ gives a $\mathbb{C}(q)$-subalgebra. So we can define

\begin{defi}\label{defiL01}
$\mathcal{L}_{0,1}(\mathfrak{g})$ is the $\mathbb{C}(q)$-vector space $\mathcal{O}_q$ endowed with the product \eqref{produitL01}.
\end{defi}
\noindent We simply write $\mathcal{L}_{0,1}$ when $\mathfrak{g}$ is fixed.
\smallskip

\indent Let $\Gamma \in \mathbb{A}_{0,1}^{\otimes 2}$ be the twist which was introduced in Lemma \ref{lemmeTwistL10}. We have the twisted Hopf algebra $\mathbb{A}_{1,0} = (\mathbb{A}_{0,1}^{\otimes 2})_{\Gamma}$ and since $\mathcal{O}_q(q^{1/D})_F \otimes \mathcal{O}_q(q^{1/D})_F$ is a right $\mathbb{A}_{0,1}^{\otimes 2}$-module-algebra, we can define
\begin{defi}
$\mathcal{L}_{1,0}(\mathfrak{g})$ is the twist $\bigl(\mathcal{O}_q(q^{1/D})_F \otimes \mathcal{O}_q(q^{1/D})_F\bigr)_{\Gamma}$.
\end{defi}
\noindent We simply write $\mathcal{L}_{1,0}$ when $\mathfrak{g}$ has been fixed. Explicitly, $\mathcal{L}_{1,0}(\mathfrak{g})$ is the $\mathbb{C}(q^{1/D})$-vector space $\mathcal{O}_q(q^{1/D}) \otimes \mathcal{O}_q(q^{1/D})$ endowed with the product \eqref{productL10}. 

\begin{remark}{\rm Contrarily to the case of $\mathcal{L}_{0,1}$ discussed before Definition \ref{defiL01}, the space $\bigl((\mathcal{O}_q)_F \otimes (\mathcal{O}_q)_F\bigr)_{\Gamma}$ is not a $\mathbb{C}(q)$-subalgebra of $\mathcal{L}_{1,0}$. This is because the terms coming from the $\Theta$ factors of the $R$-matrices (see \S \ref{sectionCategoricalCompletion}), which have coefficients in $\mathbb{C}(q^{1/D})$, do not compensate for each other in the exchange relation of Proposition \ref{presentationL10}. Actually, this same fact implies that $\bigl(\mathcal{O}_q(q^{2/D}) \otimes \mathcal{O}_q(q^{2/D})\bigr)_{\Gamma}$ is a $\mathbb{C}(q^{2/D})$-subalgebra of $\mathcal{L}_{1,0}$. Therefore, if $D=2$ (e.g. for $\mathfrak{g} = \mathfrak{sl}_2$) we can work over $\mathbb{C}(q)$, and for general $D$, we could have defined $\mathcal{L}_{1,0}$ over the ground field $\mathbb{C}(q^{2/D})$.}
\end{remark}

\indent By construction $\mathcal{L}_{0,1}$ is a right $\mathbb{A}_{0,1}$-module-algebra and $\mathcal{L}_{1,0}$ is a right $\mathbb{A}_{1,0}$-module-algebra. Specializing the results of \S \ref{subsectionL10H}, we have morphisms of Hopf algebras
\[ \mathbb{U}_q \xrightarrow{\Delta_{\mathbb{U}_q}} \mathbb{A}_{0,1} \quad \text{and} \quad \mathbb{U}_q \xrightarrow{\Delta^{(3)}_{\mathbb{U}_q}} \mathbb{A}_{1,0} \]
where $\Delta^{(3)}_{\mathbb{U}_q}$ is the iterated coproduct (see Lemma \ref{lemmaIteratedDelta10}). Combining this with the morphism of Hopf algebras $\iota : U_q \to \mathbb{U}_q$ introduced below \eqref{embeddingUqIntoItsCompletion}, we obtain that $\mathcal{L}_{0,1}$ and $\mathcal{L}_{1,0}$ are right $U_q$-module-algebras for the actions $\mathrm{coad}^r$ in \eqref{coadL01} and \eqref{coadL10} respectively.

\medskip

\indent Using the pairing from \eqref{pairingOqCategoricalCompletion} and the embedding $\iota : U_q \to \mathbb{U}_q$, we get a pairing
\[ \langle \cdot, \cdot \rangle : \mathcal{O}_q(q^{1/D}) \otimes U_q(q^{1/D}) \to \mathbb{C}(q^{1/D}) \]
where $U_q(q^{1/D}) = U_q \otimes _{\mathbb{C}(q)} \mathbb{C}(q^{1/D})$ is the extension of scalars to $\mathbb{C}(q^{1/D})$.
It follows that the left (and right) coregular action of $U_q(q^{1/D})$ on $\mathcal{O}_q(q^{1/D})$ is well-defined and allows us to define the Heisenberg double $\mathcal{H}_q = \mathcal{H}_q(\mathfrak{g})$ as the smash product
\[ \mathcal{H}_q = \mathcal{O}_q(q^{1/D}) \,\#\, U_q(q^{1/D}). \]
Explicitly, it is the $\mathbb{C}(q^{1/D})$-vector space $\mathcal{O}_q(q^{1/D}) \otimes U_q(q^{1/D})$ endowed with the product \eqref{productHeisenberg}. For simplicity of notation we often do not write the extension of scalars to $\mathbb{C}(q^{1/D})$.

\begin{prop}\label{HqSansDivDeZero}
The algebra $\mathcal{H}_q$ is a domain.
\end{prop}
\begin{proof}
Recall from \S\ref{sectionPreliminaires} the filtration $(\mathcal{F}_{\mathrm{DCK}}^{\mathbf{m}})_{\mathbf{m} \in \mathbb{N}^{2N+1}}$ of $U_q$ which is based on the PBW basis. By Corollary \ref{coroCoproduitSurFiltrationDCK} the family of subspaces $\bigl( \mathcal{O}_q \# \mathcal{F}_{\mathrm{DCK}}^{\mathbf{m}} \bigr)_{\mathbf{m} \in \mathbb{N}^{2N+1}}$ is a filtration of the algebra $\mathcal{H}_q$. Indeed, if $\varphi\,\#\,x \in \mathcal{O}_q \, \# \, \mathcal{F}_{\mathrm{DCK}}^{\mathbf{m}}$ and $\psi\,\#\,y \in \mathcal{O}_q \, \# \, \mathcal{F}_{\mathrm{DCK}}^{\mathbf{n}}$ we have
\[ (\varphi\,\#\,x)(\psi\,\#\,y) = \sum_{(x)}\varphi \star (x_{(1)} \rhd \psi) \,\#\, x_{(2)}y \in \mathcal{O}_q \, \# \, \mathcal{F}_{\mathrm{DCK}}^{\mathbf{m} + \mathbf{n}}. \]
The associated graded algebra is
\[ \mathrm{gr}_{\mathcal{O}_q \# \mathcal{F}_{\mathrm{DCK}}}(\mathcal{H}_q) = \mathcal{O}_q \, \# \, \mathrm{gr}_{\mathcal{F}_{\mathrm{DCK}}}(U_q). \]
Let us explain this equality more precisely. Recall from \S\ref{sectionHopfAlgebras} that $_V\phi^i_j$ denote the matrix coefficients of finite-dimensional $U_q^{\mathrm{ad}}$-modules, and assume that for each $V$ we use a basis of weight vectors $(v_j)$ with weights $(\epsilon_j)$. $\mathcal{O}_q$ is generated by such matrix coefficients. For simplicity, write $_V\phi^i_j$, $\overline{E_{\beta_i}}$, $\overline{F_{\beta_i}}$, $\overline{K_{\mu}}$ instead of the cosets $(_V\phi^i_j \,\#\, 1) + \mathcal{O}_q \,\#\, \mathcal{F}_{\mathrm{DCK}}^{<\mathbf{0}}$, $(\varepsilon \,\#\, E_{\beta_i}) + \mathcal{O}_q \,\#\, \mathcal{F}_{\mathrm{DCK}}^{< d(E_{\beta_i})}$, $(\varepsilon \,\#\, F_{\beta_i}) + \mathcal{O}_q \,\#\, \mathcal{F}_{\mathrm{DCK}}^{< d(F_{\beta_i})}$, $(\varepsilon \,\#\, K_{\mu}) + \mathcal{O}_q \,\#\, \mathcal{F}_{\mathrm{DCK}}^{<d(K_\mu)}$ in $\mathrm{gr}_{\mathcal{O}_q \# \mathcal{F}_{\mathrm{DCK}}}(\mathcal{H}_q)$. Thanks to Proposition \ref{propHtCoproduit} we get
\[\overline{E_{\beta_k}} \, {_V\phi^i_j} = {_V\phi^i_j} \, \overline{E_{\beta_k}}, \quad \overline{F_{\beta_k}} \, {_V\phi^i_j} = q^{-(\beta_k, \epsilon_j)} {_V\phi^i_j} \, \overline{F_{\beta_k}}, \quad \overline{K_{\nu}} \, {_V\phi^i_j} = q^{(\nu, \epsilon_j)} {_V\phi^i_j} \, \overline{K_{\nu}}. \]
It follows from these relations and \eqref{commutationGrDCK} that $\mathrm{gr}_{\mathcal{O}_q \# \mathcal{F}_{\mathrm{DCK}}}(\mathcal{H}_q)$ is a quasi-polynomial ring over $\mathcal{O}_q$, generated over $\mathcal{O}_q$ by $\overline{E_{\beta_i}}$, $\overline{F_{\beta_i}}$, $\overline{K_{\nu}}$ (with $1 \leq i \leq N$, $\mu \in P$). Since $\mathcal{O}_q$ is a domain \cite[Th. I.8.9]{BG}, it follows from a general result (see e.g. \cite[\S 1.2.9]{MC-R}) that $\mathrm{gr}_{\mathcal{O}_q \# \mathcal{F}_{\mathrm{DCK}}}(\mathcal{H}_q)$ is a domain. By Lemma \ref{lemmadomainGr}, $\mathcal{H}_q$ is a domain as well.
\end{proof}

\smallskip

\indent It follows from \eqref{eqWellDefRMatrixOnOq} that the formulas of the maps $\Phi_{0,1}$ and $\Phi_{1,0}$ in \S \ref{sectionHeisenberg} are well-defined and give morphisms of right $U_q$-module-algebras
\[ \Phi_{0,1} : \mathcal{L}_{0,1} \to U_q, \qquad \Phi_{1,0} : \mathcal{L}_{1,0} \to \mathcal{H}_q \]
where the action on $\mathcal{L}_{0,1}$ and $\mathcal{L}_{1,0}$ is $\mathrm{coad}^r$ in \eqref{coadL01} and \eqref{coadL10}, the action on $U_q$ is the right adjoint action $\textstyle \mathrm{ad}^r(h)(x) = \sum_{(h)} S(h_{(1)})xh_{(2)}$ and the action on $\mathcal{H}_q$ is \eqref{actionOnHeisenberg}. It is known that $\Phi_{0,1}$ affords an isomorphism of $U_q$-module-algebras $\mathcal{L}_{0,1} \overset{\sim}{\rightarrow} U_q^{\mathrm{lf}}$, where $U_q^{\mathrm{lf}}$ is the subspace of locally finite elements for $\mathrm{ad}^r$ \cite[Th. 3]{Bau1} (also see \cite[Th 4.3]{BR1} for this statement in the present framework). We now prove an analogous statement for $\Phi_{1,0}$. Denote by $\mathcal{H}_q^{\mathrm{lf}}$ the subspace of locally finite elements (see \eqref{defLFHeisenberg}).

\begin{teo}\label{ThmPhi10}
1. The morphism $\Phi_{1,0} : \mathcal{L}_{1,0} \to \mathcal{H}_q$ is injective.
\\2. The algebra $\mathcal{L}_{1,0}$ is a domain.
\\3. The image of $\Phi_{1,0}$ is $\mathcal{H}_q^{\mathrm{lf}}$.
\end{teo}
\begin{proof}
For simplicity we write $U_q$ and $\mathcal{O}_q$ instead of $U_q(q^{1/D})$ and $\mathcal{O}_q(q^{1/D})$.
\\1. Recall first that
$\Phi_{1,0}\bigl( \beta \otimes \alpha \bigr) = \Phi_{1,0}\bigl( \mathfrak{i}_B(\beta) \, \mathfrak{i}_A(\alpha) \bigr) = \Phi_{1,0}\bigl(\mathfrak{i}_B(\beta)\bigr) \, \bigl( 1_{\mathcal{O}_q} \,\#\, \Phi_{0,1}(\alpha) \bigr).$
For $\lambda, \sigma \in P$, consider
\[ _{\lambda}(\mathcal{O}_q)_{\sigma} = \bigl\{ \varphi \in \mathcal{O}_q \, | \, \forall \, \nu \in P, \: K_{\nu} \rhd \varphi = q^{(\nu, \lambda)} \varphi \text{ and } \varphi \lhd K_{\nu} = q^{(\sigma, \nu)} \varphi \bigr\} \]
and note that $\mathcal{O}_q$ is the direct sum of all these subspaces. It follows that
\[ \mathcal{L}_{1,0} = \bigoplus_{\lambda, \sigma \in P} \mathfrak{i}_B\bigl( {_{\lambda}(\mathcal{O}_q)_{\sigma}} \bigr) \, \mathfrak{i}_A\bigl( \mathcal{L}_{0,1} \bigr), \quad \mathcal{H}_q = \bigoplus_{\lambda,\sigma \in P} {_{\lambda}(\mathcal{O}_q)_{\sigma}} \,\#\, U_q. \]
By \eqref{commutationRK}, if $\beta \in {_{\lambda}(\mathcal{O}_q)_{\sigma}}$ we have
\begin{align*}
&K_{\nu} \rhd \! \bigl(R^1_{(2)} R^2_{(2)} \rhd \beta \lhd R^3_{(1)} R^1_{(1)} \bigr) \,\#\, R^3_{(2)} R^2_{(1)} = q^{(\nu,\lambda - \gamma_1 - \gamma_2)} \bigl(R^1_{(2)} R^2_{(2)} \rhd \beta \lhd R^3_{(1)} R^1_{(1)} \bigr) \,\#\, R^3_{(2)} R^2_{(1)}\\
&\bigl(R^1_{(2)} R^2_{(2)} \rhd \beta \lhd R^3_{(1)} R^1_{(1)} \bigr) \! \lhd K_{(\nu)} \,\#\, R^3_{(2)} R^2_{(1)} = q^{(\nu, \sigma - \gamma_1 - \gamma_3)} \bigl(R^1_{(2)} R^2_{(2)} \rhd \beta \lhd R^3_{(1)} R^1_{(1)} \bigr) \,\#\, R^3_{(2)} R^2_{(1)}
\end{align*}
for some $\gamma_1, \gamma_2, \gamma_3 \in Q_+$. As a result $\textstyle \Phi_{1,0}(\mathfrak{i}_B(\beta)) \in \bigoplus_{(\lambda', \sigma') \leq (\lambda, \sigma)} {_{\lambda'}(\mathcal{O}_q)_{\sigma'}} \,\#\, U_q$, where $\leq$ is the partial order on $P^2$ defined by  $(\lambda',\sigma') \leq (\lambda,\sigma)$ if and only if $\lambda - \lambda' \in Q_+$ and $\sigma - \sigma' \in Q_+$. Hence thanks to the expression of $R$ in \eqref{expressionCanoniqueR} we see that
\begin{equation}\label{propPhi10SurSommeDirecte}
\begin{array}{rl}
\textstyle \Phi_{1,0}(\mathfrak{i}_B(\beta)) \!\!\!&\!\!\in \bigl(\Theta^1_{(2)} \Theta^2_{(2)} \rhd \beta \lhd \Theta^3_{(1)} \Theta^1_{(1)} \bigr) \,\#\, \Theta^3_{(2)} \Theta^2_{(1)} + \bigoplus_{(\lambda',\sigma') < (\lambda,\sigma)} {_{\lambda'}(\mathcal{O}_q)_{\sigma'}} \,\#\, U_q\\[1.5em]
&= q^{(\lambda,\sigma)} \beta \,\#\,  K_{\lambda + \sigma} + \bigoplus_{(\lambda',\sigma') < (\lambda,\sigma)} {_{\lambda'}(\mathcal{O}_q)_{\sigma'}} \,\#\, U_q
\end{array}
\end{equation}
where the second equality is due to \eqref{ThetaPoids} (and $\Theta^1$, $\Theta^2$, $\Theta^3$ denote three copies of $\Theta$). We are now in position to show our result. Let $x \in \mathcal{L}_{1,0}$ be a non-zero element and write it as $\textstyle x = \sum_{\lambda, \sigma \in P} \sum_{i \in I_{\lambda,\sigma}} \mathfrak{i}_B(\beta_{\lambda,\sigma,i}) \, \mathfrak{i}_A(\alpha_{\lambda,\sigma,i})$ with $\beta_{\lambda,\sigma,i} \in {_{\lambda}(\mathcal{O}_q)_{\sigma}}$ for all $i \in I_{\lambda, \sigma}$ and such that for each $\lambda, \sigma$ the elements $\bigl(\beta_{\lambda,\sigma,i}\bigr)_{i \in I_{\lambda, \sigma}}$ are linearly independent. Take a $(\lambda,\sigma)$ maximal for the order $\leq$ such that there exists at least one $i \in I_{\lambda,\sigma}$ with $\beta_{\lambda,\sigma,i} \neq 0$ and $\alpha_{\lambda,\sigma,i} \neq 0$. Then using \eqref{propPhi10SurSommeDirecte} one obtains
\[ \Phi_{1,0}(x) \in q^{(\lambda,\sigma)} \sum_{i \in I_{\lambda,\sigma}} \beta_{\lambda, \sigma,i} \,\#\,  K_{\lambda + \sigma}\Phi_{0,1}(\alpha_{\lambda,\sigma,i}) + \bigoplus_{(\lambda',\sigma') \neq (\lambda,\sigma)} {_{\lambda'}(\mathcal{O}_q)_{\sigma'}} \,\#\, U_q. \]
The morphism $\Phi_{0,1}$ is injective (see the comments before the theorem); moreover $K_{\lambda+\sigma}$ is invertible, so there exists at least one $i \in I_{\lambda,\sigma}$ with $K_{\lambda + \sigma}\Phi_{0,1}(\alpha_{\lambda,\sigma,i}) \neq 0$ and since the elements $\beta_{\lambda,\sigma,i}$ are linearly independent we conclude that $\Phi_{1,0}(x) \neq 0$.

\smallskip

\indent 2. Follows from Proposition \ref{HqSansDivDeZero} and item 1.

\smallskip

\indent 3. We already know from Proposition \ref{propPhi10} that $\mathrm{im}(\Phi_{1,0}) \subset \mathcal{H}_q^{\mathrm{lf}}$. Let us prove the converse inclusion. Let $\mathcal{O}_q \otimes U_q$ be the tensor product of the right modules $\bigl( \mathcal{O}_q, \mathrm{coad}^r \bigr)$ and $\bigl( U_q, \mathrm{ad}^r \bigr)$, \textit{i.e.} the right action is
\begin{equation}\label{actionOqUqProofHqlf}
(\varphi \otimes x) \cdot h = \sum_{(h)} S(h_{(2)}) \rhd \varphi \lhd h_{(1)} \otimes S(h_{(3)})xh_{(4)}.
\end{equation}
As recalled above the theorem, $\Phi_{0,1}$ provides an isomorphism of $U_q$-modules $\bigl(\mathcal{O}_q, \mathrm{coad}^r) \overset{\sim}{\to} \bigl(U_q^{\mathrm{lf}}, \mathrm{ad}^r\bigr)$. Hence we have an isomorphism of $U_q$-modules
\[ \Phi_{0,1} \otimes \mathrm{id} : \: \mathcal{O}_q \otimes U_q \, \overset{\sim}{\longrightarrow} \, U_q^{\mathrm{lf}} \otimes U_q \]
and we see that $\bigl( \mathcal{O}_q \otimes U_q \bigr)^{\mathrm{lf}} \cong \bigl(U_q^{\mathrm{lf}} \otimes U_q\bigr)^{\mathrm{lf}} \subset (U_q \otimes U_q)^{\mathrm{lf}} = U_q^{\mathrm{lf}} \otimes U_q^{\mathrm{lf}}$ where the last equality is due to \cite[Th. 2]{KLNY}. The converse inclusion being obvious, we conclude that $\bigl( \mathcal{O}_q \otimes U_q \bigr)^{\mathrm{lf}} = \mathcal{O}_q \otimes U_q^{\mathrm{lf}}$.
Now consider the map
\[ \fonc{I}{\mathcal{H}_q}{\mathcal{O}_q \otimes U_q}{\varphi \,\#\, x}{\sum_{(R^1),(R^2),(R^3)}} S\bigl( R^1_{(2)} R^2_{(1)} \bigr) \rhd \varphi \lhd R^1_{(1)}R^3_{(1)} \otimes S\bigl( R^3_{(2)}R^2_{(2)} \bigr)x. \]
It is an isomorphism of $U_q$-modules, thanks to a straightforward computation which uses $R \Delta = \Delta^{\mathrm{op}} R$ to pass from the action \eqref{actionOnHeisenberg} to the action \eqref{actionOqUqProofHqlf}. Another computation which uses $(\mathrm{id} \otimes S^{-1})(R) = R^{-1}$ reveals that
\begin{equation}\label{IPhi10}
I \circ \Phi_{1,0}(\beta \otimes \alpha) = \sum_{(\beta)} \beta_{(1)} \otimes \Phi_{0,1}\bigl( \beta_{(2)}\alpha \bigr)
\end{equation}
where $\textstyle \sum_{(\beta)} \beta_{(1)} \otimes \beta_{(2)}$ is the coproduct of $\beta$ in $\mathcal{O}_q$ and $\beta_{(2)}\alpha$ is the product of $\beta_{(2)}$ and $\alpha$ in $\mathcal{L}_{0,1}$. We need one more fact: there is a map $S_{\mathcal{L}_{0,1}} : \mathcal{L}_{0,1} \to \mathcal{L}_{0,1}$ such that $\textstyle \sum_{(\varphi)} \varphi_{(1)} S_{\mathcal{L}_{0,1}}(\varphi_{(2)}) = \varepsilon(\varphi)1_{\mathcal{L}_{0,1}}$ for all $\varphi \in \mathcal{L}_{0,1}$;
it is given by
\[ S_{\mathcal{L}_{0,1}}(\varphi) = \sum_{(R)} S_{\mathcal{O}_q}\bigl( S(R_{(1)}) \rhd \varphi \lhd R_{(2)}u^{-1} \bigr) \]
where $S_{\mathcal{O}_q}$ is the antipode of $\mathcal{O}_q$ and $\textstyle u = \sum_{(R)}S(R_{(2)})R_{(1)}$ is the Drinfeld element, which satisfies $S^2 (x) = uxu^{-1}$ for all $x \in U_q$. In fact, $\mathcal{L}_{0,1}$ is isomorphic to a coend \cite[Prop. 6.3]{FaitgMCG} which has a natural structure of a Hopf algebra, and $S_{\mathcal{L}_{0,1}}$ is the antipode for this structure. We are finally ready to conclude. Take $\psi \otimes y \in (\mathcal{O}_q \otimes U_q)^{\mathrm{lf}}$. By the above discussion we know that $y \in U_q^{\mathrm{lf}}$, so there exists $\gamma \in \mathcal{L}_{0,1}$ such that $y = \Phi_{0,1}(\gamma)$. By \eqref{IPhi10} and the property of $S_{\mathcal{L}_{0,1}}$ we find
\[ I \circ \Phi_{1,0}\biggl( \sum_{(\psi)} \psi_{(1)} \otimes S_{\mathcal{L}_{0,1}}(\psi_{(2)})\gamma \biggr) = \sum_{(\psi)} \psi_{(1)} \otimes \Phi_{0,1}\bigl(\psi_{(2)} S_{\mathcal{L}_{0,1}}(\psi_{(3)})\gamma\bigr) = \psi \otimes y \]
which proves that $\mathrm{im}(I \circ \Phi_{1,0}) = (\mathcal{O}_q \otimes U_q)^{\mathrm{lf}}$. Since $I$ is an isomorphism of $U_q$-modules the result follows.
\end{proof}

Note that $\mathcal{H}_q^{\mathrm{lf}}$ is a strict subspace of $\mathcal{H}_q$. Observe for instance that the action \eqref{actionOnHeisenberg} is such that $(1_{\mathcal{O}_q} \,\#\, x) \cdot h = 1_{\mathcal{O}_q} \,\#\, \mathrm{ad}^r(h)(x)$. Thus if $x \in U_q$ is not locally finite then $1_{\mathcal{O}_q} \,\#\, x \not\in \mathcal{H}_q^{\mathrm{lf}}$. For instance one can take $x = K_i$.

\subsection{Noetherianity of $\mathcal{L}_{1,0}$}\label{sectionL10Noetherien} The strategy is to construct a filtration on $\mathcal{L}_{1,0}$ and to show that the associated graded algebra satisfies the relations in Lemma \ref{critereNoetherien}. First we need to recall and adapt a filtration of $\mathcal{L}_{0,1}$ taken from \cite{VY}.

\smallskip

\indent For $\mu \in P_+$ and $\lambda \in P$ we define
\begin{align*}
 {_{\lambda}C(\mu)} &= \bigl\{ \varphi \in C(\mu) \, \big| \, \forall \, \nu \in P, \: K_{\nu} \rhd \varphi = q^{(\lambda,\nu)}\varphi \bigr\},\\
C(\mu)_{\lambda} &= \bigl\{ \varphi \in C(\mu) \, \big| \, \forall \, \nu \in P, \: \varphi \lhd K_{\nu} = q^{(\lambda,\nu)}\varphi \bigr\}
\end{align*}
where $C(\mu)$ is the subspace of $\mathcal{O}_q$ spanned by the matrix coefficients of $V_{\mu}$, the irreducible $U_q^{\mathrm{ad}}$-module with highest weight $\mu$ (it was first introduced in Section \ref{sectionOq}). Consider the ordered abelian monoid
\begin{equation}\label{defIndexSetLambda}
\Lambda = \bigl\{ (\mu, \lambda) \in P_+ \times P \, \big| \, \lambda \text{ is a weight of } V_{\mu} \bigr\}
\end{equation}
for the order $\preceq$ defined by $(\mu',\lambda') \preceq (\mu,\lambda)$ if and only if $\mu' \preceq \mu$ and $\lambda'  \preceq \lambda$, where the order $\preceq$ on $P$ is defined in \S\ref{prelimLieAlgebras}. Note that $\preceq$ is well-founded on $\Lambda$ (although it is not on $P_+ \times P$) and satisfies condition \eqref{confluentOrder}.

\smallskip

\indent It is not difficult to show that
\[ _{\lambda}C(\mu) \, {_{\lambda'}C(\mu')} = \bigoplus_{(\nu, \kappa) \preceq (\mu+\mu', \lambda+\lambda')} {_{\kappa}C(\nu)} \]
in $\mathcal{L}_{0,1}$ (see \cite[\S 3.14.4]{VY}, but with a slightly different product), and similarly for the subspaces $C(\mu)_{\lambda}$. By the general discussion of \S\ref{sectionFiltrations} it follows that we have filtrations $\mathcal{F}_{\ell}$, $\mathcal{F}_r$ of the algebra $\mathcal{L}_{0,1}$ given by
\[
\mathcal{F}_{\ell}^{\mu, \lambda} = \bigoplus_{(\mu', \lambda') \preceq (\mu, \lambda)}  {_{\lambda'}C(\mu')} \quad \text{and} \quad \mathcal{F}_r^{\mu, \lambda} = \bigoplus_{(\mu', \lambda') \preceq (\mu, \lambda)}  C(\mu')_{\lambda'}
\]
for all $(\mu, \lambda) \in \Lambda$. The filtration $\mathcal{F}_{\ell}$ was used in \cite[\S 3]{BR2} where it was denoted by $\mathcal{F}_2$.

\smallskip

\indent For $x \in \{\ell, r\}$, let $\mathrm{gr}_{\mathcal{F}_x}(\mathcal{L}_{0,1})$ be the graded algebra associated to $\mathcal{F}_x$ (see \S \ref{sectionFiltrations}). The vector space $\mathrm{gr}_{\mathcal{F}_x}(\mathcal{L}_{0,1})$ can be identified with $\mathcal{O}_q$: we identify $\varphi \in {_{\lambda}C(\mu)}$ (resp. $\varphi \in C(\mu)_{\lambda}$) with $\varphi + \mathcal{F}_{\ell}^{\prec(\mu,\lambda)}$ (resp. $\varphi + \mathcal{F}_r^{\prec(\mu,\lambda)}$). We denote by $\circ_x$ the product on $\mathcal{O}_q$ obtained through this identification.

\smallskip

\indent We introduce some notations in order to describe the product $\circ_x$. For $\mu \in P_+$, denote by $e_1^{\mu}, \ldots, e_m^{\mu}$ a basis of weight vectors of $V_{\mu}$. Let $\epsilon^{\mu}_i \in P$ be the weight of $e^{\mu}_i$ and assume that the vectors are numbered in such a way that $\epsilon^{\mu}_i \succ \epsilon^{\mu}_j$ implies $i < j$. Let $e^1_{\mu}, \ldots, e^m_{\mu}$ be the dual basis and let ${_{\mu}\phi^i_j} : x \mapsto e^i_{\mu}(x \cdot e_j^{\mu})$ be the associated matrix coefficients. Note that
\begin{equation}\label{weightMatrixCoeffs}
K_{\nu} \rhd {_{\mu}\phi^i_j} = q^{(\nu, \epsilon^{\mu}_j)} {_{\mu}\phi^i_j}, \quad \text{and} \quad {_{\mu}\phi^i_j} \lhd K_{\nu} = q^{(\nu, \epsilon_i^{\mu})} {_{\mu}\phi^i_j}
\end{equation}
or in other words ${_{\mu}\phi^i_j} \in {_{\epsilon^{\mu}_j}C(\mu)}$ and ${_{\mu}\phi^i_j} \in C(\mu)_{\epsilon^{\mu}_i}$. We use similar notations for a basis $e^{\eta}_1, \ldots, e^{\eta}_n$ of $V_{\eta}$. Finally for $\varphi \in C(\mu)$ and $\psi \in C(\nu)$, we put
\[ \varphi \, \overline{\star} \, \psi = \pi_{\mu+\nu}(\varphi \star \psi) \]
where $\star$ is the usual product in $\mathcal{O}_q$ and $\pi_{\mu+\nu} : C(\mu)\star C(\nu) \to C(\mu + \nu)$ is the projection defined by \eqref{filtrationOq}.

\smallskip

\indent The following facts were proved in \cite[pp. 177-178]{VY} for the filtration $\mathcal{F}_{\ell}$ (but note that in \cite{VY} the product in $\mathcal{L}_{0,1}$ is slightly different from ours). We give a proof of these facts for the filtration $\mathcal{F}_r$.

\begin{lem}\label{lemmaProductsInLeftAndRightGr}
For $x \in \{\ell, r\}$ and with the notations just introduced, we have
\\1. ${_{\mu}\phi^i_j} \circ_x  {_{\eta}\phi^k_l} = q^{(\epsilon^{\mu}_j, \epsilon^{\eta}_l - \epsilon^{\eta}_k)} {_{\mu}\phi^i_j} \: \overline{\star} \: {_{\eta}\phi^k_l} + \sum_{s=j+1}^{m} \sum_{t=1}^{l-1} \alpha^{ijkl}_{st} {_{\mu}\phi^i_s} \: \overline{\star} \: {_{\eta}\phi^k_t}$, with $\alpha^{ijkl}_{st} \in \mathbb{C}(q^{1/D})$. 
\\Moreover $\alpha^{ijkl}_{st} = 0$ unless $\epsilon^{\mu}_s \prec \epsilon^{\mu}_j$ and $\epsilon^{\eta}_t \succ \epsilon^{\eta}_l$. It follows in particular that ${_{\mu}\phi^i_j} \circ_{\ell}  {_{\eta}\phi^k_l} = {_{\mu}\phi^i_j} \circ_r  {_{\eta}\phi^k_l}$.
\\2. ${_{\mu}\phi^k_l} \circ_x {_{\eta}\phi^i_j} = q^{(\epsilon^{\mu}_i, \epsilon^{\eta}_k - \epsilon^{\eta}_l)} {_{\eta}\phi^i_j} \: \overline{\star} \: {_{\mu}\phi^k_l} + \sum_{s=i+1}^{n} \sum_{t=1}^{k-1} \beta^{ijkl}_{st} {_{\eta}\phi^s_j} \: \overline{\star} \: {_{\mu}\phi^t_l}$, with $\beta^{ijkl}_{st} \in \mathbb{C}(q^{1/D})$. 
\\Moreover $\beta^{ijkl}_{st} = 0$ unless $\epsilon^{\mu}_s \prec \epsilon^{\mu}_i$ and $\epsilon^{\eta}_t \succ \epsilon^{\eta}_k$.
\\3. It follows that
\begin{equation}\label{formuleEchangeGrL01}
\begin{array}{rl}
{_{\eta}\phi^k_l} \circ_x {_{\mu}\phi^i_j} \: = \: q^{(\epsilon^{\mu}_i + \epsilon^{\mu}_j, \epsilon^{\eta}_k - \epsilon^{\eta}_l)} {_{\mu}\phi^i_j} \circ_x {_{\eta}\phi^k_l} &\!\!\!+\: \sum_{s=j}^{m} \sum_{t=1}^{l} \sum_{u=i+1}^{m} \sum_{v=1}^{k-1} \delta^{ijkl}_{stuv} \, {_{\eta}\phi^u_s} \circ_x {_{\mu}\phi^v_t}\\
&+\: \sum_{s=j+1}^{m} \sum_{t=1}^{l-1} \gamma^{ijkl}_{st} {_{\mu}\phi^i_s} \circ_x {_{\eta}\phi^k_t}
\end{array}
\end{equation}
with $\delta^{ijkl}_{stuv}, \gamma^{ijkl}_{st} \in \mathbb{C}(q^{1/D})$. Moreover $\delta^{ijkl}_{stuv} = 0$ unless $\epsilon^{\mu}_s \preceq \epsilon^{\mu}_j, \epsilon^{\eta}_t \succeq \epsilon^{\eta}_l, \epsilon^{\mu}_u \prec \epsilon^{\mu}_i, \epsilon^{\eta}_v \succ \epsilon^{\eta}_k$ and $\gamma^{ijkl}_{st} = 0$ unless $\epsilon^{\mu}_s \prec \epsilon^{\mu}_j, \epsilon^{\eta}_t \succ \epsilon^{\eta}_l$.
\end{lem}
\begin{proof}
Recall the formula \eqref{expressionCanoniqueR} for the $R$-matrix; rewrite it more simply as $\textstyle R = \Theta_{(1)} \otimes \Theta_{(2)} + \sum_r X^+_r \otimes X^-_r$ with $X^+_r \otimes X^-_r \in \Theta_{(1)}U_q(\mathfrak{n}_+) \otimes \Theta_{(2)}U_q(\mathfrak{n}_-)$ for each $r$. We have a key-property regarding the weights for the adjoint action: 
\[ \forall \, r, \:\: \exists \, \gamma_r \in Q^+\!\setminus\!\{0\}, \:\: \forall \, \nu \in P, \quad K_{\nu}X^{\pm}_rK_{\nu}^{-1} = q^{\pm (\nu,\gamma_r)} X^{\pm}_r \]
which simply follows from \eqref{weightOfRootVectors}. Hence by \eqref{weightMatrixCoeffs} we have
\begin{align}
\begin{split}\label{coregOnNumberedMatCoeffs}
&\bigl[ X^-_r \rhd {_{\mu}\phi^i_j} \bigr] \in \mathrm{span}\bigl\{ {_{\mu}\phi^i_s} \, \big| \, s \text{ such that } \epsilon^{\mu}_s = \epsilon^{\mu}_j - \gamma_r \bigr\} \subset \mathrm{span}\bigl\{ {_{\mu}\phi^i_s} \, \big| \, j+1 \leq s \leq m \bigr\},\\
&\bigl[ X^+_r \rhd {_{\eta}\phi^k_l} \bigr] \subset \mathrm{span}\bigl\{ {_{\eta}\phi^k_t} \, \big| \, t \text{ such that } \epsilon^{\eta}_t = \epsilon^{\eta}_l + \gamma_r \bigr\} \subset \mathrm{span}\bigl\{ {_{\eta}\phi^k_t} \, \big| \, 1 \leq t \leq l-1 \bigr\},\\
& \bigl[ {_{\eta}\phi^k_l} \lhd X^+_r \bigr] \subset \mathrm{span}\bigl\{ {_{\eta}\phi^u_l} \, \big| \, u \text{ such that } \epsilon^{\eta}_u = \epsilon^{\eta}_k - \gamma_r \bigr\} \subset \mathrm{span}\bigl\{ {_{\eta}\phi^u_l} \, \big| \, k+1 \leq u \leq n \bigr\}
\end{split}
\end{align}
where the inclusions are due to the chosen numbering for the basis of weight vectors in $V_{\mu}$.
The formula \eqref{produitL01} for the product in $\mathcal{L}_{0,1}$ thus gives (with implicit summation on $r,r'$)
\begin{align*}
{_{\mu}\phi^i_j} \, {_{\eta}\phi^k_l} =&\: \bigl( \Theta_{(2)}^2 S(\Theta_{(2)}^1) \rhd {_{\mu}\phi^i_j}  \bigr) \star \bigl( \Theta_{(1)}^2 \rhd {_{\eta}\phi^k_l} \lhd \Theta_{(1)}^1 \bigr) + \bigl( X^-_r S(\Theta_{(2)}) \rhd {_{\mu}\phi^i_j}  \bigr) \star \bigl( X^+_r \rhd {_{\eta}\phi^k_l} \lhd \Theta_{(1)} \bigr)\\
&+ \bigl( \Theta_{(2)} S(X^-_r) \rhd {_{\mu}\phi^i_j}  \bigr) \star \bigl( \Theta_{(1)} \rhd {_{\eta}\phi^k_l} \lhd X^+_r \bigr) + \bigl( X^-_r S(X^-_{r'}) \rhd {_{\mu}\phi^i_j}  \bigr) \star \bigl( X^+_r \rhd {_{\eta}\phi^k_l} \lhd X^+_{r'} \bigr)\\
=&\: q^{(\epsilon^{\mu}_j, \epsilon^{\eta}_l - \epsilon^{\eta}_k)} {_{\mu}\phi^i_j} \star {_{\eta}\phi^k_l} + \sum_{s=j+1}^m \sum_{t=1}^{l-1} \alpha^{ijkl}_{st} {_{\mu}\phi^i_s} \star {_{\eta}\phi^k_t} + \sum_{s=j+1}^m \sum_{u = k+1}^n B^{ijkl}_{su} {_{\mu}\phi^i_s} \star {_{\eta}\phi^u_l}\\
&\qquad \qquad \qquad\qquad\qquad\qquad\:\:+ \sum_{s=j+1}^m \sum_{t=1}^{l-1} \sum_{u=k+1}^n D^{ijkl}_{stu} {_{\mu}\phi^i_s} \star {_{\eta}\phi^u_t}
\end{align*}
for certain scalars $\alpha^{ijkl}_{st}, B^{ijkl}_{su}, D^{ijkl}_{stu} \in \mathbb{C}(q^{1/D})$, where for the second equality we used \eqref{ThetaPoids}, \eqref{ThetaPoidsInverse}, and \eqref{weightMatrixCoeffs} to evaluate the action of the $\Theta$'s (first term in the result), and we used \eqref{coregOnNumberedMatCoeffs} to get the other terms. More precisely, we actually have $\alpha^{ijkl}_{st} = 0$ unless there exists $0 \neq \gamma \in Q_+$ such that $\epsilon^{\mu}_s = \epsilon^{\mu}_j - \gamma$ and $\epsilon^{\eta}_t = \epsilon^{\eta}_l + \gamma$, $B^{ijkl}_{su} = 0$ unless $\epsilon^{\mu}_s \prec \epsilon^{\mu}_j$ and $\epsilon^{\eta}_u \prec \epsilon^{\eta}_k$, $D^{ijkl}_{stu} = 0$ unless $\epsilon^{\mu}_s \prec \epsilon^{\mu}_j$ and $\epsilon^{\eta}_l \prec \epsilon^{\eta}_t$ and $\epsilon^{\eta}_u \prec \epsilon^{\eta}_k$. 
Of course the condition for $\alpha^{ijkl}_{st}$ implies $\epsilon^{\mu}_s \prec \epsilon^{\mu}_j$ and $\epsilon^{\eta}_t \succ \epsilon^{\eta}_l$ as in the statement of the Lemma but for the proof we need this more precise property, showing that $\epsilon^{\mu}_s + \epsilon^{\eta}_t = \epsilon^{\mu}_j + \epsilon^{\eta}_l$ for all $s,t$. Write $_{\lambda} C(\mu)_{\lambda'} = {_\lambda}C(\mu) \cap C(\mu)_{\lambda'}$. Then by  \eqref{weightMatrixCoeffs} we have
\[ \alpha^{ijkl}_{st} {_{\mu}\phi^i_s} \star {_{\eta}\phi^k_t} \in \sum_{0 \neq \gamma \in Q^+}{_{\epsilon^{\mu}_j - \gamma}}C(\mu)_{\epsilon^{\mu}_i} \star {_{\epsilon^{\eta}_l + \gamma}}C(\eta)_{\epsilon^{\eta}_k} \subset \sum_{\nu \preceq \mu + \eta} {_{\epsilon^{\mu}_j + \epsilon^{\eta}_l}}C(\nu)_{\epsilon^{\mu}_i + \epsilon^{\eta}_k}. \]
Thus, among the terms produced by $\alpha^{ijkl}_{st} {_{\mu}\phi^i_s} \star {_{\eta}\phi^k_t}$, only those (and all of them) which are in $C(\mu + \eta)$ are kept in $\mathrm{gr}_{\mathcal{F}_\ell}(\mathcal{L}_{0,1})$ and $\mathrm{gr}_{\mathcal{F}_r}(\mathcal{L}_{0,1})$, which is precisely the effect of the product $\overline{\star}$. Same remark for $q^{(\epsilon^{\mu}_j, \epsilon^{\eta}_l - \epsilon^{\eta}_k)} {_{\mu}\phi^i_j} \star {_{\eta}\phi^k_l}$. Next,
\[ B^{ijkl}_{su} {_{\mu}\phi^i_s} \star {_{\eta}\phi^u_l} \in \sum_{\substack{\lambda \,\prec\, \epsilon^{\mu}_j\\ \kappa \,\prec\, \epsilon^{\eta}_k}} {_\lambda}C(\mu)_{\epsilon^{\mu}_i} \star {_{\epsilon^\eta_l}}C(\eta)_{\kappa} \subset \sum_{\substack{\nu \,\preceq\, \mu + \eta\\ \lambda' \prec\, \epsilon^{\mu}_j + \epsilon^{\eta}_l \\ \kappa' \prec\, \epsilon^{\mu}_i + \epsilon^{\eta}_k}} {_{\lambda'}}C(\nu)_{\kappa'} \subset \mathcal{F}_\ell^{\prec (\mu + \eta, \, \epsilon^{\mu}_j + \epsilon^{\eta}_l)} \cap \mathcal{F}_r^{\prec (\mu + \eta, \, \epsilon^{\mu}_i + \epsilon^{\eta}_k)} \]
and thus all these terms disappear in $\mathrm{gr}_{\mathcal{F}_\ell}(\mathcal{L}_{0,1})$ and $\mathrm{gr}_{\mathcal{F}_r}(\mathcal{L}_{0,1})$. Same remark for $D^{ijkl}_{stu} {_{\mu}\phi^i_s} \star {_{\eta}\phi^u_t}$. We have thus obtained the announced formula.
\\2. Using the relation $R\Delta = \Delta^{\mathrm{op}}R$, we see that the formula \eqref{produitL01} for the product in $\mathcal{L}_{0,1}$ can be rewritten as follows:
\[ \alpha \beta = \sum_{(R^1), (R^2)} \bigl( \beta \lhd R^1_{(1)} R^2_{(1)} \bigr) \star \bigl( S(R^1_{(2)}) \rhd \alpha \lhd R^2_{(2)} \bigr). \]
To obtain the desired equalities in the claim, it suffices to repeat a proof like in item 1 but with this alternative formula for the product in $\mathcal{L}_{0,1}$.
\\3. Each of the formulas in items 1 and 2 allows one to express the product $\overline{\star}$ in terms of the product $\circ_x$ (using induction). We get the result by comparing the two resulting expressions for $\overline{\star}$; for more details see \cite[pp. 177-178]{VY}.
\end{proof}

\indent Recall that $\varpi_1, \ldots, \varpi_m$ denote the fundamental weights of the semisimple Lie algebra $\mathfrak{g}$. Any irreducible $U_q^{\mathrm{ad}}$-module $V_{\mu}$ is a direct summand of some tensor product of the modules $V_{\varpi_1}, \ldots, V_{\varpi_m}$ (see e.g. \cite[\S I.6.12]{BG}). Hence the matrix coefficients ${_{\varpi_r}\phi^i_j}$ generate $\mathcal{O}_q$ as an algebra, and there is of course a finite number of such elements. From item 1 of Lemma \ref{lemmaProductsInLeftAndRightGr}, one can deduce that these matrix coefficients generate $\mathrm{gr}_{\mathcal{F}_x}(\mathcal{L}_{0,1})$ as well (see \cite[pp. 178]{VY}). We form an ordered list $(u_1, \ldots, u_p)$ containing all the elements ${_{\varpi_r}\phi^i_j}$ such that for any $u_b = {_{\varpi_r}\phi^i_j}$ and $u_a = {_{\varpi_s}\phi^k_l}$ the following condition is satisfied:
\begin{equation}\label{conditionOrdreCoeffMatricielsFondamentauxL01}
\biggl[ \epsilon^{\varpi_r}_i \prec \epsilon^{\varpi_s}_k \:\text{ or }\: \bigl( \epsilon^{\varpi_r}_i = \epsilon^{\varpi_s}_k \:\text{ and }\: \epsilon^{\varpi_r}_j \prec \epsilon^{\varpi_s}_l\bigr) \biggr] \implies b < a.
\end{equation}
With this choice of numbering, the formula \eqref{formuleEchangeGrL01} can be rewritten as
\begin{equation}\label{relationsNoetherianiteL01}
\forall \, 1 \leq b < a \leq p, \quad u_a \circ_x u_b = q_{ab} u_b \circ_x u_a + \sum_{s=1}^{b-1} \sum_{t=1}^p \alpha^{ab}_{st} u_s \circ_x u_t
\end{equation}
for certain scalars $q_{ab} \in \mathbb{C}(q)^{\times}$ and $\alpha^{ab}_{st} \in \mathbb{C}(q)$. It follows that $\mathrm{gr}_{\mathcal{F}_x}(\mathcal{L}_{0,1})$ is Noetherian by Lemma \ref{critereNoetherien} and hence $\mathcal{L}_{0,1}$ is Noetherian by Lemma \ref{lemmaFiltrationNoetherian}.

\medskip

\indent We are now ready to study the Noetherianity for $\mathcal{L}_{1,0}$. Consider
\begin{equation}\label{filtrationL10}
\mathcal{F}^{\mu, \lambda, \eta, \kappa} = \mathcal{F}_r^{\mu, \lambda} \otimes \mathcal{F}_{\ell}^{\eta, \kappa} \subset \mathcal{L}_{1,0}  \quad \text{for } (\mu, \lambda), (\eta, \kappa) \in \Lambda.
\end{equation}
Here we use that $\mathcal{L}_{1,0}$ is $\mathcal{O}_q(q^{1/D})^{\otimes 2}$ as a vector space and we implicitly extend the scalars of $\mathcal{F}_x^{\mu,\lambda}$ to $\mathbb{C}(q^{1/D})$. On the indexing set $\Lambda \times \Lambda$ of the family $\mathcal{F}$ of spaces $\mathcal{F}^{\mu, \lambda, \eta, \kappa}$, we put the product order, again denoted by $\preceq$, namely $(\mu', \lambda', \eta', \kappa') \preceq (\mu, \lambda, \eta, \kappa)$ if and only if $(\mu', \lambda') \preceq (\mu, \lambda)$ and $(\eta', \kappa') \preceq (\eta, \kappa)$.
\begin{lem}
$\mathcal{F}$ is a filtration of the algebra $\mathcal{L}_{1,0}$.
\end{lem}
\begin{proof}
Take $\beta \otimes \alpha \in C(\mu)_{\lambda} \otimes {_{\kappa}C(\eta)}$ and $\beta' \otimes \alpha' \in C(\mu')_{\lambda'} \otimes {_{\kappa'}C(\eta')}$ and recall that the product in $\mathcal{L}_{1,0}$ is described in \eqref{productL10}. By \eqref{commutationRK}, we have
\begin{equation}\label{weightsInProductL10}
\begin{array}{rcl}
\bigl( R_{(2)}^4 R_{(1)}^3 \rhd \beta' \lhd R_{(1)}^1 R_{(1)}^2 \bigr) \lhd K_{\nu}\!\!\!&=&\!\!\!q^{(\nu, \lambda' - \gamma_1 - \gamma_2)} \bigl( R_{(2)}^4 R_{(1)}^3 \rhd \beta' \lhd R_{(1)}^1 R_{(1)}^2 \bigr)\\[.5em]
K_{\nu} \rhd \bigl( R_{(2)}^3 S(R_{(2)}^1) \rhd \alpha \lhd R_{(2)}^2 R_{(1)}^4  \bigr)\!\!\!\!&=&\!\!\!q^{(\nu, \kappa - \gamma_1 - \gamma_3)} \bigl( R_{(2)}^3 S(R_{(2)}^1) \rhd \alpha \lhd R_{(2)}^2 R_{(1)}^4  \bigr)
\end{array}
\end{equation}
for some $\gamma_1, \gamma_2, \gamma_3 \in Q_+$. In other words:
\[ R_{(2)}^4 R_{(1)}^3 \rhd \beta' \lhd R_{(1)}^1 R_{(1)}^2 \in \mathcal{F}_r^{\mu',\lambda'}, \quad  R_{(2)}^3 S(R_{(2)}^1) \rhd \alpha \lhd R_{(2)}^2 R_{(1)}^4\in \mathcal{F}_{\ell}^{\eta,\kappa} \]
Since $\mathcal{F}_l$ and $\mathcal{F}_r$ are filtrations of the algebra $\mathcal{L}_{0,1}$ we thus have
\[ \beta \bigl( R_{(2)}^4 R_{(1)}^3 \rhd \beta' \lhd R_{(1)}^1 R_{(1)}^2 \bigr) \in \mathcal{F}_r^{\mu + \mu', \lambda + \lambda'}, \quad  \bigl( R_{(2)}^3 S(R_{(2)}^1) \rhd \alpha \lhd R_{(2)}^2 R_{(1)}^4 \bigr) \alpha'\in \mathcal{F}_{\ell}^{\eta + \eta',\kappa + \kappa'} \]
which gives the claim by the definition of $\mathcal{F}$.
\end{proof}

\indent We identify $\mathrm{gr}_{\mathcal{F}}(\mathcal{L}_{1,0})$ with $\mathcal{O}_q(q^{1/D})^{\otimes 2}$ as a vector space: $\beta \otimes \alpha \in C(\mu)_{\lambda} \otimes {_{\kappa}C(\eta)}$ is identified with $\beta \otimes \alpha + \mathcal{F}^{\prec (\mu, \lambda, \eta, \kappa)}$. We denote by $\diamond$ the resulting product on $\mathcal{O}_q(q^{1/D})^{\otimes 2}$. Let $\beta \in C(\mu)_{\lambda}$ and $\alpha \in {_{\kappa}C(\eta)}$; then \eqref{expressionCanoniqueR} and \eqref{weightsInProductL10} yield
\begin{equation}\label{formulaDiamondProduct}
\mathfrak{i}_A(\alpha) \diamond \mathfrak{i}_B(\beta) = \sum_{(R)} \mathfrak{i}_B\bigl( R_{(2)} \Theta^3_{(1)} \rhd \beta \lhd \Theta^1_{(1)} \Theta^2_{(1)} \bigr) \diamond \mathfrak{i}_A\bigl( \Theta^3_{(2)} S(\Theta^1_{(2)}) \rhd \alpha \lhd \Theta^2_{(2)}R_{(1)} \bigr)
\end{equation}
where $\Theta^1$, $\Theta^2$, $\Theta^3$ are three copies of $\Theta$.

\begin{teo}\label{TheoL10Noetherian}
The algebra $\mathcal{L}_{1,0}$ is Noetherian.
\end{teo}
\begin{proof}
We use the notations and assumptions introduced just before Lemma \ref{lemmaProductsInLeftAndRightGr} for basis elements and matrix coefficients. We are going to show that $\mathrm{gr}_{\mathcal{F}}(\mathcal{L}_{1,0})$ is Noetherian. The first task is to simplify the formula \eqref{formulaDiamondProduct} when $\alpha = {_{\eta}\phi^k_l}$ and $\beta = {_{\mu}\phi^i_j}$ are matrix coefficients of some irreducible modules $V_{\eta}$ and $V_{\mu}$. By \eqref{commutationRK} and \eqref{weightMatrixCoeffs}, we have
\begin{align*}
\bigl( K_{\nu}\rhd (R_{(2)} \rhd {_{\mu}\phi^i_j}) \bigr) \otimes ({_{\eta}\phi^k_l} \lhd R_{(1)}) &= q^{(\nu, \epsilon^{\mu}_j - \gamma)}(R_{(2)} \rhd {_{\mu}\phi^i_j}) \otimes ({_{\eta}\phi^k_l} \lhd R_{(1)})\\
(R_{(2)} \rhd {_{\mu}\phi^i_j}) \otimes \bigl( ({_{\eta}\phi^k_l} \lhd R_{(1)}) \lhd K_{\nu}) \bigr) &= q^{(\nu, \epsilon^{\eta}_k - \gamma)} (R_{(2)} \rhd {_{\mu}\phi^i_j}) \otimes ({_{\eta}\phi^k_l} \lhd R_{(1)})
\end{align*}
for some $\gamma \in Q_+$. Hence
\[ (R_{(2)} \rhd {_{\mu}\phi^i_j}) \otimes ({_{\eta}\phi^k_l} \lhd R_{(1)}) \in \mathrm{span}\bigl\{ {_{\mu}\phi^i_s} \otimes {_{\eta}\phi^t_l} \bigr\}_{\epsilon^{\mu}_s \preceq \epsilon^{\mu}_j, \: \epsilon^{\eta}_t \preceq \epsilon^{\eta}_k} \subset \mathrm{span}\bigl\{ {_{\mu}\phi^i_s} \otimes {_{\eta}\phi^t_l} \bigr\}_{s \geq j, t \geq k} \]
where the inclusion is due to the assumption on the numbering of the basis vectors. Using the expression \eqref{expressionCanoniqueR} we obtain more precisely
\begin{align*}
\sum_{(R)} (R_{(2)} \rhd {_{\mu}\phi^i_j}) \otimes ({_{\eta}\phi^k_l} \lhd R_{(1)}) &\in (\Theta_{(2)} \rhd {_{\mu}\phi^i_j}) \otimes ({_{\eta}\phi^k_l} \lhd \Theta_{(1)}) + \mathrm{span}\bigl\{ {_{\mu}\phi^i_s} \otimes {_{\eta}\phi^t_l} \bigr\}_{\epsilon^{\mu}_s \prec \epsilon^{\mu}_j, \: \epsilon^{\eta}_t \prec \epsilon^{\eta}_k}\\
&= q^{(\epsilon^{\mu}_j, \epsilon^{\eta}_k)} {_{\mu}\phi^i_j} \otimes {_{\eta}\phi^k_l} + \mathrm{span}\bigl\{ {_{\mu}\phi^i_s} \otimes {_{\eta}\phi^t_l} \bigr\}_{\epsilon^{\mu}_s \prec \epsilon^{\mu}_j, \: \epsilon^{\eta}_t \prec \epsilon^{\eta}_k}.
\end{align*}
We can now simplify \eqref{formulaDiamondProduct}:
\begin{equation}\label{echangeCoeffsMatricielsL10}
\begin{array}{rl}
&\mathfrak{i}_A({_{\eta}\phi^k_l}) \diamond \mathfrak{i}_B({_{\mu}\phi^i_j})\\[.7em]
&= \sum_{(R)} \mathfrak{i}_B\bigl( R_{(2)} \Theta^3_{(1)} \rhd {_{\mu}\phi^i_j} \lhd \Theta^1_{(1)} \Theta^2_{(1)} \bigr) \diamond \mathfrak{i}_A\bigl( \Theta^3_{(2)} S(\Theta^1_{(2)}) \rhd {_{\eta}\phi^k_l} \lhd \Theta^2_{(2)}R_{(1)} \bigr)\\[2.5em]
&= q^{(\epsilon^{\mu}_i, \epsilon^{\eta}_k - \epsilon^{\eta}_l) + (\epsilon^{\mu}_j, \epsilon^{\eta}_l)} \sum_{(R)} \mathfrak{i}_B\bigl( R_{(2)} \rhd {_{\mu}\phi^i_j} \bigr) \diamond \mathfrak{i}_A\bigl( {_{\eta}\phi^k_l} \lhd R_{(1)} \bigr)\\
&= q^{(\epsilon^{\mu}_i, \epsilon^{\eta}_k - \epsilon^{\eta}_l) + (\epsilon^{\mu}_j, \epsilon^{\eta}_l + \epsilon^{\eta}_k)} \, \mathfrak{i}_B({_{\mu}\phi^i_j}) \diamond \mathfrak{i}_A({_{\eta}\phi^k_l}) + \sum_{s = j+1}^m \sum_{t = k+1}^n c_{st}^{ijkl} \, \mathfrak{i}_B({_{\mu}\phi^i_s}) \diamond \mathfrak{i}_A({_{\eta}\phi^t_l})
\end{array}
\end{equation}
for some $c_s^{ijkl} \in \mathbb{C}(q^{1/D})$ such that $c_s^{ijkl} = 0$ unless $\epsilon^{\mu}_s \prec \epsilon^{\mu}_j$ and $\epsilon^{\eta}_t \prec \epsilon^{\eta}_k$.

\smallskip

\indent Recall from above \eqref{conditionOrdreCoeffMatricielsFondamentauxL01} the finite sequence of elements $(u_k)$ which generate $\mathrm{gr}_{\mathcal{F}_x}(\mathcal{L}_{0,1})$. Since $\mathfrak{i}_B\bigl( \mathcal{F}_r^{\mu,\lambda} \bigr) = \mathcal{F}^{\mu,\lambda,0,0}$ and $\mathfrak{i}_A\bigl( \mathcal{F}_l^{\eta,\kappa} \bigr) = \mathcal{F}^{0,0,\eta,\kappa}$, the linear maps $\mathfrak{i}_B$ and $\mathfrak{i}_A$ are morphisms of algebras $\mathrm{gr}_{\mathcal{F}_r}(\mathcal{L}_{0,1}) \to \mathrm{gr}_{\mathcal{F}}(\mathcal{L}_{1,0})$ and $\mathrm{gr}_{\mathcal{F}_l}(\mathcal{L}_{0,1}) \to \mathrm{gr}_{\mathcal{F}}(\mathcal{L}_{1,0})$. Moreover, we have $\mathfrak{i}_B(\beta) \diamond \mathfrak{i}_A(\alpha) = \beta \otimes \alpha$. It follows that the algebra $\mathrm{gr}_{\mathcal{F}}(\mathcal{L}_{1,0})$ is generated by the elements $\mathfrak{i}_B(u_k)$ and $\mathfrak{i}_A(u_k)$. We organize them in a sequence $(x_1, \ldots, x_{2p})$ by
\[ x_1 = \mathfrak{i}_B(u_1), \ldots, x_p = \mathfrak{i}_B(u_p), \: x_{p+1} = \mathfrak{i}_A(u_1), \ldots, x_{2p} = \mathfrak{i}_A(u_p). \]
Thanks to relations \eqref{relationsNoetherianiteL01} and \eqref{echangeCoeffsMatricielsL10} we have:
\begin{align}
&\forall \, 1 \leq b < a \leq p, \quad x_a \diamond x_b = q_{ab} x_b \diamond x_a + \sum_{s=1}^{b-1} \sum_{t=1}^p \alpha^{ab}_{st} x_s \diamond x_t \nonumber\\
&\forall \, p+1 \leq b < a \leq 2p, \quad x_a \diamond x_b = q_{ab} x_b \diamond x_a + \sum_{s=1}^{b-1} \sum_{t=1}^p \alpha^{ab}_{st} x_s \diamond x_t \nonumber\\
&\forall \, 1 \leq b \leq p < a \leq 2p, \quad x_a \diamond x_b = q'_{ab} x_b \diamond x_a + \sum_{s=1}^{b-1} \sum_{t=1}^p \lambda^{ab}_{st} x_s \diamond x_t \label{echangeL10Generateurs}
\end{align}
for certain scalars $q_{ab}, q'_{ab} \in \mathbb{C}(q^{1/D})^{\times}$ and $\alpha^{ab}_{st}, \lambda^{ab}_{st} \in \mathbb{C}(q^{1/D})$. These three relations cover all the possible cases for the indices $a$ and $b$, and they all have the form required by Lemma \ref{critereNoetherien}. It follows that $\mathrm{gr}_{\mathcal{F}}(\mathcal{L}_{1,0})$ is Noetherian, and that $\mathcal{L}_{1,0}$ is Noetherian as well thanks to Lemma \ref{lemmaFiltrationNoetherian}.
\end{proof}

\section{The graph algebra $\mathcal{L}_{g,n}$}\label{sectionGraphAlgebra}
Following the definition of $\mathcal{L}_{0,n}(H)$ in \cite{BR1}, we will define $\mathcal{L}_{g,n}(H)$ as a twist of the tensor product of $g$ copies of $\mathcal{L}_{1,0}(H)$ and $n$ copies of $\mathcal{L}_{0,1}(H)$ and then explain the equivalence with the other constructions based on matrix relations \cite{A, AGS1, BuR1} or braided tensor product \cite{AS}. See also \cite{MW} for a similar construction based on twisting; in their work $\mathcal{L}_{g,n}(H)$ corresponds to $\mathcal{A}^*_{\Gamma_{g,n}}$, where $\Gamma_{g,n}$ is the ``daisy'' graph.

\smallskip

\indent The definition of $\mathcal{L}_{g,n}(H)$ in \S \ref{sectionDefLgnH} works for any quasitriangular Hopf algebra $H$. The general definition is adapted to $H = U^{\mathrm{ad}}_q(\mathfrak{g})$ in \S \ref{sectionLgnUq}. Our main results are in \S \ref{sectionNoetherianityLgn} and \S \ref{sectionNoetherianityLgnInv}.

\subsection{Definition of $\mathcal{L}_{g,n}(H)$}\label{sectionDefLgnH}
Let $H$ be a quasitriangular Hopf algebra with an invertible antipode. As in \cite[\S 6.1]{BR1} we will use twisted tensor products, so let us recall a few facts about this operation. For two Hopf algebras $A$ and $B$, a {\em bicharacter} is an invertible element $\chi \in B \otimes A$ such that
\[ (\Delta_B \otimes \mathrm{id}_A)(\chi) = \chi_{23} \, \chi_{13}, \qquad (\mathrm{id}_B \otimes \Delta_A)(\chi) = \chi_{12} \, \chi_{13}, \]
where the subscripts describe embeddings in $B \otimes B \otimes A$ and $B \otimes A \otimes A$ respectively; for instance $\textstyle \chi_{23} = \sum_{(\chi)} 1_B \otimes \chi_{(1)} \otimes \chi_{(2)}$. For such a $\chi$, the element $X= 1_A \otimes \chi \otimes 1_B$ is a twist for $A \otimes B$ (recall the definition of a twist in \S \ref{subsectionL10H}). The twisted Hopf algebra $(A \otimes B)_X$ is denoted by $A \otimes^{\chi} B$ and is called the {\em twisted tensor product of $A$ and $B$ by $\chi$}. If $M$ is a $A$-module-algebra and $N$ is a $B$-module-algebra, then $M \otimes N$ is a $(A \otimes B)$-module-algebra; the twist $(M \otimes N)_X$ is denoted by $M \otimes^{\chi} N$ and is called the {\em twisted tensor product of $M$ and $N$ by $\chi$}.

\begin{exa}
1. The element $R' = \textstyle \sum_{(R)} R_{(2)} \otimes R_{(1)}$ is a bicharacter (with $A = B =H$).
\\2. An easy computation reveals that the Hopf algebra $A_{0,1}(H)$, which was defined as a twist of  $H\otimes H^{\rm cop}$ by $F=R_{32}R_{42}$ in section \ref{subsectionL10H}, coincides with $H \otimes^{R'} H$.
\\3. The Hopf algebra $A_{1,0}(H)$ from section \ref{subsectionL10H} was defined as $A_{0,1}(H) \otimes^{\gamma} A_{0,1}(H)$ (see Lemma \ref{lemmeTwistL10}) and $\mathcal{L}_{1,0}(H)$ is $\mathcal{L}_{0,1}(H) \otimes^{\gamma} \mathcal{L}_{0,1}(H)$.
\end{exa}

\indent Let $f_A : H \to A$ and $f_B : H \to B$ be morphisms of Hopf algebras and denote $f_A \odot f_B = (f_A \otimes f_B)\Delta$. Then
\begin{enumerate}
\item $(f_B \otimes f_A)(R')$ is a bicharacter,
\item $f_A \odot f_B$ is a morphism of Hopf algebras $H \to A \otimes^{(f_B \otimes f_A)(R')} B$.
\end{enumerate}

\smallskip

\indent The twisted tensor product allows us to iterate the twist operation, as follows. Recall from \S \ref{subsectionL10H} that $\mathcal{L}_{0,1}(H)$ and $\mathcal{L}_{1,0}(H)$ are module-algebras over the Hopf algebras $A_{0,1}(H)$ and $A_{1,0}(H)$ respectively.  Let us put by convention $A_{0,0}(H) = \mathcal{L}_{0,0}(H) = k$ (the base field of $H$). For $n \geq 0$ we define
\begin{align*}
A_{0,n+1}(H) &= A_{0,n}(H) \otimes^{J_n} A_{0,1}(H) \quad \text{with } J_n = \bigl( \Delta \otimes \Delta^{\odot n} \bigr)(R')\\
\mathcal{L}_{0,n+1}(H) &= \mathcal{L}_{0,n}(H) \otimes^{J_n} \mathcal{L}_{0,1}(H)
\end{align*}
where $\Delta^{\odot 0} = \varepsilon$ and $\Delta^{\odot n} = \Delta^{\odot (n-1)} \odot \Delta$. For $g \geq 0$ we define
\begin{align*}
A_{g+1,0}(H) &= A_{g,0}(H) \otimes^{K_g} A_{1,0}(H) \quad \text{with } K_g = \bigl( \Delta^{(3)} \otimes (\Delta^{(3)})^{\odot g} \bigr)(R')\\
\mathcal{L}_{g+1,0}(H) &= \mathcal{L}_{g,0}(H) \otimes^{K_g} \mathcal{L}_{1,0}(H)
\end{align*}
where $\Delta^{(3)} = (\Delta \otimes \Delta)\Delta$ is the third iterated coproduct, $(\Delta^{(3)})^{\odot 0} = \varepsilon$ and $(\Delta^{(3)})^{\odot g} = (\Delta^{(3)})^{\odot (g-1)} \odot \Delta^{(3)}$. To see that $J_n$ (resp. $K_g$) is a bicharacter, recall that $\Delta : H \to A_{0,1}(H)$ (resp. $\Delta^{(3)} : H \to A_{1,0}(H)$) is a morphism of Hopf algebras and make an induction from items 1 and 2 above.

\smallskip

\indent In general we put
\[
A_{g,n}(H) = A_{g,0}(H) \otimes^{F_{g,n}} A_{0,n}(H) \quad \text{with } F_{g,n} = \bigl(\Delta^{\odot n} \otimes (\Delta^{(3)})^{\odot g}\bigr)(R') \]
and
\begin{defi}\label{def:Lgn}
The graph algebra is $\mathcal{L}_{g,n}(H) = \mathcal{L}_{g,0}(H) \otimes^{F_{g,n}} \mathcal{L}_{0,n}(H)$.
\end{defi}
By construction $\mathcal{L}_{g,n}(H)$ is a (right) $A_{g,n}(H)$-module-algebra. As vectors spaces
\[ A_{g,n}(H) = H^{\otimes (4g + 2n)} \quad \text{and} \quad \mathcal{L}_{g,n}(H) = (H^{\circ})^{\otimes (2g+n)} \]
and the right action is given by
\[ (\varphi_1 \otimes \ldots \otimes \varphi_{2g+n})\cdot (x_1 \otimes \ldots \otimes x_{4g+2n}) = \bigl(S(x_2) \rhd \varphi_1 \lhd x_1 \bigr) \otimes \ldots \otimes \bigl( S(x_{4g+2n}) \rhd \varphi_{2g+n} \lhd x_{4g+2n-1} \bigr). \]
Thanks to the morphism of Hopf algebras $(\Delta^{(3)})^{\odot g} \odot \Delta^{\odot n} = \Delta^{(4g + 2n - 1)} : H \to A_{g,n}(H)$, we obtain a right $H$-module-algebra structure on $\mathcal{L}_{g,n}(H)$, denoted by $\mathrm{coad}^r$:
\begin{equation}\label{coadLgn}
\begin{array}{rl}
\mathrm{coad}^r(h)(\varphi_1 \otimes \ldots \otimes \varphi_{2g+n}) \!\!\!\!&= (\varphi_1 \otimes \ldots \otimes \varphi_{2g+n}) \cdot \Delta^{(4g + 2n - 1)}(h)\\[.5em]
&= \sum_{(h)} \mathrm{coad}^r(h_{(1)})(\varphi_1) \otimes \ldots \otimes \mathrm{coad}^r(h_{(2g+n)})(\varphi_{2g+n})
\end{array}
\end{equation}
where the right action $\mathrm{coad}^r$ in the second line is defined in \eqref{coadL01}. In other words, $\mathcal{L}_{g,n}(H)$ is $\mathcal{L}_{0,1}(H)^{\otimes (2g+n)}$ as a $H$-module.

\smallskip

\indent We now reformulate the definition of $\mathcal{L}_{g,n}(H)$ as a braided tensor product, which will be more convenient in order to describe the product. Let $\mathrm{Mod}\text{-}H$ be the category of right $H$-modules; it is braided, with braiding:
\[ \fonc{c_{U,V}}{U \otimes V}{V \otimes U}{u \otimes v}{\sum_{(R)} \, (v \cdot R_{(1)}) \otimes (u \cdot R_{(2)})} .\]
Let $(M, m_M, 1_M)$ and $(N, m_N, 1_N)$ be $H$-module-algebras (\textit{i.e.} algebras in $\mathrm{Mod}\text{-}H$) and define
\begin{align*}
&m_{M \, \widetilde{\otimes} \, N} : (M \otimes N) \otimes (M \otimes N) \xrightarrow{\mathrm{id}_M \,\otimes \,c_{N,M} \,\otimes\, \mathrm{id}_N} M \otimes M \otimes N \otimes N \xrightarrow{m_M \otimes m_N} M \otimes N, \\
&1_{M \, \widetilde{\otimes} \, N} = 1_M \otimes 1_N.
\end{align*}
This endows $M \otimes N$ with a structure of (right) $H$-module-algebra, denoted by $M \, \widetilde{\otimes} \, N$ and called the braided tensor product of $M$ and $N$ \cite[Lem. 9.2.12]{majidFoundations}. It is easy to see that $\widetilde{\otimes}$ is associative. We often write a pure tensor $m \otimes n \in M \, \widetilde{\otimes} \, N$ as $m \, \widetilde{\otimes} \, n$ to stress that we use the braided product in $M \otimes N$. Explicitly:
\begin{equation}\label{multiplicationInBraidedTensorProduct}
(m \,\widetilde{\otimes}\, n)(m' \,\widetilde{\otimes}\, n') = \sum_{(R)}m (m' \cdot R_{(1)}) \,\widetilde{\otimes}\, (n \cdot R_{(2)})n'.
\end{equation}
The embeddings of $k$-vector spaces
\begin{equation}\label{embeddingsBraided}
\fonc{j_M}{M}{M \, \widetilde{\otimes} \, N}{m}{m \,\widetilde{\otimes}\, 1} \quad \fonc{j_N}{N}{M \, \widetilde{\otimes} \, N}{n}{1 \,\widetilde{\otimes}\, n}
\end{equation}
are morphisms of $H$-module-algebras and we have
\begin{align}
j_M(m) j_N(n) &= m \,\widetilde{\otimes}\, n,\label{factorisationBraidedTensors}\\
j_N(n)j_M(m) &= \sum_{(R)} j_M(m \cdot R_{(1)}) j_N(n \cdot R_{(2)}).\label{productBraided}
\end{align}

\begin{prop}\label{propBraidedTensProduct}
$\mathcal{L}_{g,n}(H) = \mathcal{L}_{1,0}(H)^{\widetilde{\otimes} g} \: \widetilde{\otimes} \: \mathcal{L}_{0,1}(H)^{\widetilde{\otimes} n}$ as right $H$-module-algebras.
\end{prop}
\begin{proof}
Note that $K_g = \bigl( \Delta^{(3)} \otimes \Delta^{(4g-1)} \bigr)(R')$; then by definition of the multiplication $m_{\mathcal{L}_{g+1,0}(H)}$ in $\mathcal{L}_{g+1,0}(H)$ (recall \eqref{twistedProduct} and the definition of the twisted tensor product) we have for $u,v \in \mathcal{L}_{g,0}(H)$ and $x,y \in \mathcal{L}_{1,0}(H)$:
\begin{align*}
&m_{\mathcal{L}_{g+1,0}(H)}\bigl( (u \otimes x) \otimes (v \otimes y) \bigr)\\
=\:&\sum_{(K_g)} m_{\mathcal{L}_{g,0}(H)}\bigl( u \otimes \bigl(v\cdot (K_g)_{(2)}\bigr) \bigr) \otimes m_{\mathcal{L}_{1,0}(H)}\bigl( \bigl(x \cdot (K_g)_{(1)}\bigr) \otimes y \bigr)\\
=\:&\sum_{(R)} m_{\mathcal{L}_{g,0}(H)}\bigl( u \otimes \bigl(v\cdot \Delta^{(4g-1)}(R_{(1)})\bigr) \bigr) \otimes m_{\mathcal{L}_{1,0}(H)}\bigl( \bigl(x \cdot \Delta^{(3)}(R_{(2)})\bigr) \otimes y \bigr)\\
=\:&\sum_{(R)} m_{\mathcal{L}_{g,0}(H)}\bigl( u \otimes \mathrm{coad}^r(R_{(1)})(v) \bigr) \otimes m_{\mathcal{L}_{1,0}(H)}\bigl( \mathrm{coad}^r(R_{(2)})(x) \otimes y \bigr) =(u \, \widetilde{\otimes} \, x)(v \, \widetilde{\otimes} \, y).
\end{align*}
For the third equality we used \eqref{coadLgn}. Hence $\mathcal{L}_{g+1,0}(H) = \mathcal{L}_{g,0}(H) \, \widetilde{\otimes} \, \mathcal{L}_{1,0}(H)$ and it follows that $\mathcal{L}_{g,0}(H) = \mathcal{L}_{1,0}(H)^{\widetilde{\otimes} g}$. One shows similarly that $\mathcal{L}_{0,n}(H) = \mathcal{L}_{0,1}(H)^{\widetilde{\otimes} n}$ and that $\mathcal{L}_{g,0}(H) \otimes^{F_{g,n}} \mathcal{L}_{0,n}(H) = \mathcal{L}_{g,0}(H) \, \widetilde{\otimes} \, \mathcal{L}_{0,n}(H)$.
\end{proof}

Thanks to Proposition \ref{propBraidedTensProduct} and \eqref{embeddingsBraided} we see that there are embeddings of $H$-module-algebras $j_i : \mathcal{L}_{1,0}(H) \to \mathcal{L}_{g,n}(H)$ for $1 \leq i \leq g$ and $j_i : \mathcal{L}_{0,1}(H) \to \mathcal{L}_{g,n}(H)$ for $g+1 \leq i \leq g+n$, given by
\begin{equation}\label{embeddingsL01L10InLgn}
j_i(x) = 1^{\widetilde{\otimes} (i-1)} \, \widetilde{\otimes} \, x \, \widetilde{\otimes} \, 1^{\widetilde{\otimes}(g+n-i)}
\end{equation}
where $1$ is the unit of $\mathcal{L}_{1,0}(H)$ (\textit{i.e.} $\varepsilon \otimes \varepsilon$) or of $\mathcal{L}_{0,1}(H)$ (\textit{i.e.} $\varepsilon$) depending on the position.
Recall the embeddings $\mathfrak{i}_A, \mathfrak{i}_B$ from \eqref{embeddingL01inL10}; we denote
\begin{equation}\label{embeddingsFrakI}
\begin{array}{l}
\mathfrak{i}_{A(i)} = j_i \circ \mathfrak{i}_A, \:\: \mathfrak{i}_{B(i)} = j_i \circ \mathfrak{i}_B \quad \text{ for } 1 \leq i \leq g,\\[.2em]
\mathfrak{i}_{M(i)} = j_i \quad \text{ for } g+1 \leq i \leq g+n
\end{array}
\end{equation}
which are all embeddings $\mathcal{L}_{0,1}(H) \to \mathcal{L}_{g,n}(H)$. Then from \eqref{factorisationBraidedTensors} and \eqref{factorisationElementsL10} we find
\begin{equation}\label{monomialsLgn}
\begin{array}{rl}
&\mathfrak{i}_{B(1)}(\varphi_1) \, \mathfrak{i}_{A(1)}(\varphi_2) \ldots \mathfrak{i}_{B(g)}(\varphi_{2g-1}) \, \mathfrak{i}_{A(g)}(\varphi_{2g}) \, \mathfrak{i}_{M(g+1)}(\varphi_{2g+1}) \ldots \mathfrak{i}_{M(g+n)}(\varphi_{2g+n})\\[.5em]
= \!\!\!\!& j_1(\varphi_1 \otimes \varphi_2) \ldots j_g(\varphi_{2g-1} \otimes \varphi_{2g}) \, j_{g+1}(\varphi_{2g+1}) \ldots j_{g+n}(\varphi_{2g+n})\\[.5em]
=\!\!\!\!& (\varphi_1 \otimes \varphi_2) \, \widetilde{\otimes} \, \ldots \, \widetilde{\otimes} \, (\varphi_{2g-1} \otimes \varphi_{2g}) \, \widetilde{\otimes} \, \varphi_{2g+1} \, \widetilde{\otimes} \, \ldots \, \widetilde{\otimes} \, \varphi_{2g+n}
\end{array}
\end{equation}
and any element in $\mathcal{L}_{g,n}(H)$ is a linear combination of such elements.

\begin{prop}\label{productLgn}
The product in $\mathcal{L}_{g,n}(H)$ is entirely determined by the following equalities:
\begin{align*}
\mathfrak{i}_{X(i)}(\varphi) \, \mathfrak{i}_{X(i)}(\psi) &= \mathfrak{i}_{X(i)}(\varphi \psi)\\
\mathfrak{i}_{A(i)}(\varphi) \, \mathfrak{i}_{B(i)}(\psi) &= \sum_{(R^1),\ldots,(R^4)} \!\!\!\mathfrak{i}_{B(i)}\bigl( R_{(2)}^4 R_{(1)}^3 \rhd \psi \lhd R_{(1)}^1 R_{(1)}^2 \bigr) \, \mathfrak{i}_{A(i)}\bigl( R_{(2)}^3 S(R_{(2)}^1) \rhd \varphi \lhd R_{(2)}^2 R_{(1)}^4  \bigr)\\
\mathfrak{i}_{Y(j)}(\varphi) \, \mathfrak{i}_{X(i)}(\psi) &= \sum_{(R^1),\ldots,(R^4)} \!\!\! \mathfrak{i}_{X(i)}\bigl( S(R^3_{(1)}R^4_{(1)}) \rhd \psi \lhd R^1_{(1)}R^2_{(1)} \bigr) \, \mathfrak{i}_{Y(j)}\bigl( S(R^1_{(2)}R^3_{(2)}) \rhd \varphi \lhd R^2_{(2)}R^4_{(2)} \bigr)
\end{align*}
for all $\varphi, \psi \in \mathcal{L}_{0,1}(H)$, where $1 \leq i < j \leq g+n$ and $R^1, \ldots, R^4$ are four copies of $R \in H^{\otimes 2}$. In the last equality $X(i)$ is $B(i)$ or $A(i)$ if $1 \leq i \leq g$ and is $M(i)$ if $g + 1 \leq i \leq g + n$; same convention for $Y(j)$.
\end{prop}
\begin{proof}
It is clear that these formulas determine the product since they allow us to compute the product of any two elements from \eqref{monomialsLgn}. The first formula simply expresses the fact that $\mathfrak{i}_{X(i)}$ is a morphism of algebras. The second formula is obtained from \eqref{productL10} and \eqref{factorisationElementsL10} thanks to the algebra embedding $j_i : \mathcal{L}_{1,0}(H) \to \mathcal{L}_{g,n}(H)$. Let us prove the third formula. Assume for instance that $i < j \leq g$. Using Proposition \ref{propBraidedTensProduct} and the associativity of $\widetilde{\otimes}$ we can write $\mathcal{L}_{g,n}(H) = \mathcal{L}_{i,0}(H) \, \widetilde{\otimes} \, \mathcal{L}_{g-i,n}(H)$; let $j_{\mathcal{L}_{i,0}(H)}$ and $j_{\mathcal{L}_{g-i,n}(H)}$ be the corresponding $H$-module-algebra embeddings (see \eqref{embeddingsBraided}). It is easy to see that
\[ \mathfrak{i}^{(g,n)}_{X(i)} = j_{\mathcal{L}_{i,0}(H)} \circ \mathfrak{i}^{(i,0)}_{X(i)}, \qquad \mathfrak{i}^{(g,n)}_{Y(j)} = j_{\mathcal{L}_{g-i,n}(H)} \circ \mathfrak{i}^{(g-i,n)}_{Y(j-i)} \]
where $\mathfrak{i}^{(g,n)}_{X(l)}$, $\mathfrak{i}^{(i,0)}_{X(l)}$ and $\mathfrak{i}^{(g-i,n)}_{X(l)}$ are the $H$-module-algebra embeddings from \eqref{embeddingsFrakI} for $\mathcal{L}_{g,n}(H)$, $\mathcal{L}_{i,0}(H)$ and $\mathcal{L}_{g-i,n}(H)$ respectively. Hence we have
\begin{align*}
&\mathfrak{i}^{(g,n)}_{Y(j)}(\varphi) \, \mathfrak{i}^{(g,n)}_{X(i)}(\psi) =  j_{\mathcal{L}_{g-i,n}(H)}\bigl(\mathfrak{i}^{(g-i,n)}_{Y(j-i)}(\varphi)\bigr) j_{\mathcal{L}_{i,0}(H)}\bigl(\mathfrak{i}^{(i,0)}_{X(i)}(\psi)\bigr)\\
=\:& \sum_{(R)} j_{\mathcal{L}_{i,0}(H)}\!\left(\mathrm{coad}^r(R_{(1)})\bigl(\mathfrak{i}^{(i,0)}_{X(i)}(\psi)\bigr)\right) j_{\mathcal{L}_{g-i,n}(H)}\!\left(\mathrm{coad}^r(R_{(2)})\bigr(\mathfrak{i}^{(g-i,n)}_{Y(j-i)}(\varphi)\bigl)\right)\\
=\:& \sum_{(R), (R_{(1)}), (R_{(2)})} j_{\mathcal{L}_{i,0}(H)}\!\left(\mathfrak{i}^{(i,0)}_{X(i)}\!\left( S((R_{(1)})_{(2)}) \rhd \psi \lhd (R_{(1)})_{(1)}\right)\right)\\[-1.3em]
&\hspace{13em} \times j_{\mathcal{L}_{g-i,n}(H)}\!\left(\mathfrak{i}^{(g-i,n)}_{Y(j-i)}\!\left( S((R_{(2)})_{(2)}) \rhd \varphi \lhd (R_{(2)})_{(1)}\right)\right)\\
=\:&\sum_{(R^1),\ldots,(R^4)} \!\!\! \mathfrak{i}_{X(i)}^{(g,n)}\bigl( S(R^3_{(1)}R^4_{(1)}) \rhd \psi \lhd R^1_{(1)}R^2_{(1)} \bigr) \, \mathfrak{i}^{(g,n)}_{Y(j)}\bigl( S(R^1_{(2)}R^3_{(2)}) \rhd \varphi \lhd R^2_{(2)}R^4_{(2)} \bigr).
\end{align*}
For the second equality we used \eqref{productBraided}, for the third equality we used the $H$-linearity of $\mathfrak{i}^{(i,0)}_{X(i)}$, $\mathfrak{i}^{(g-i,n)}_{Y(j-i)}$ and the definition \eqref{coadL01} of the $H$-action on $\mathcal{L}_{0,1}(H)$, and for the last equality we used that $(\Delta \otimes \Delta)(R) = R_{14}R_{13}R_{24}R_{23}$, which follows from the axioms of a $R$-matrix. The other cases for $i$ and $j$ are treated similarly.
\end{proof}

Recall the notations introduced before Proposition \ref{presentationL10}, in particular the matrix $\overset{V}{M} \in \mathcal{L}_{0,1}(H) \otimes \mathrm{End}(V)$ for any finite-dimensional $H$-module $V$. Consider the following elements of $\mathcal{L}_{g,n}(H) \otimes \mathrm{End}(V)$, or in other words, matrices of size $\dim(V)$ with coefficients in $\mathcal{L}_{g,n}(H)$:
\begin{equation}\label{matricesABMgn}
\begin{array}{l}
\overset{V}{A}(i) = (\mathfrak{i}_{A(i)} \otimes \mathrm{id})\bigl( \overset{V}{M} \bigr), \:\: \overset{V}{B}(i) = (\mathfrak{i}_{B(i)} \otimes \mathrm{id})\bigl( \overset{V}{M} \bigr) \quad \text{for } 1 \leq i \leq g\\[.4em]
\overset{V}{M}(i) = (\mathfrak{i}_{M(i)} \otimes \mathrm{id})\bigl( \overset{V}{M} \bigr) \quad \text{for } g+1 \leq i \leq g+n
\end{array}
\end{equation}
In mathematical physics $\overset{V}{A}(i), \overset{V}{B}(i), \overset{V}{M}(i)$ are seen as quantum holonomy matrices associated to the usual generators $a_i, b_i, m_i$ of the fundamental group of the surface $\Sigma_{g,n}$ with an open disk removed (see \cite{A}). They have a stated skein theoretic interpretation described in \cite{FaitgHol}, and can be used e.g. in the proof of Theorem \ref{thWilsonIso} and Lemma \ref{lemmaIsoInvariantsHomH}.

\begin{prop}\label{presentationLgn}
The following relations hold true for any finite-dimensional $H$-modules $V,W$:
\begin{align*}
& \overset{V \otimes W}{X}\!\!\!\!(i) = \overset{V}{X}\hspace{-1pt}(i)_1\,(R')_{VW}\,\overset{W}{X}\hspace{-1pt}(i)_2\, (R')_{VW}^{-1} \quad \text{ for } 1 \leq i \leq g+n,\\
& R_{VW}\,\overset{V}{B}(i)_1\, (R')_{VW}\,\overset{W}{A}(i)_2 = \overset{W}{A}(i)_2\, R_{VW}\,\overset{V}{B}(i)_1\, R_{VW}^{-1} \quad \text{ for } 1 \leq i \leq g,\\
& R_{VW}\,\overset{V}{X}\hspace{-1pt}(i)_1\, R_{VW}^{-1} \,\overset{W}{Y}\hspace{-1pt}(j)_2 = \overset{W}{Y}\hspace{-1pt}(j)_2\, R_{VW}\,\overset{V}{X}\hspace{-1pt}(i)_1\, R_{VW}^{-1} \quad  \text{ for } 1 \leq i < j \leq g+n,
\end{align*}
with the same notations as in Prop \ref{productLgn}. These matrix equalities entirely describe the product in $\mathcal{L}_{g,n}(H)$.
\end{prop}
\begin{proof}
This is completely parallel to the proof of Proposition \ref{presentationL10}. By the same arguments we already get that the first and second matrix relations are equivalent to the first and second formulas in Proposition \ref{productLgn}. It remains to show that the third matrix relation is equivalent to the third formula in Proposition \ref{productLgn}; this is similar to the computation in the proof of Proposition \ref{presentationL10} and is left as an exercise for the reader.
\end{proof}
\begin{remark}
In many papers the matrix relations in Proposition \ref{presentationLgn} are taken as a definition for the product in $\mathcal{L}_{g,n}(H)$.
\end{remark}

\subsection{The case $H = U_q^{\mathrm{ad}}(\mathfrak{g})$}\label{sectionLgnUq}
Since $U_q^{\mathrm{ad}}(\mathfrak{g})$ is not quasitriangular in the usual sense (\S \ref{sectionCategoricalCompletion}), the above definition of $\mathcal{L}_{g,n}(H)$ must be slightly adapted. We use the notations from sections \ref{SectionL10ForUq} and \ref{sectionDefLgnH}. Let us put by convention $\mathbb{A}_{0,0} = \mathcal{L}_{0,0}(q^{1/D}) = \mathbb{C}(q^{1/D})$. Recall from \S \ref{SectionL10ForUq} that $\mathcal{L}_{0,1}$ and $\mathcal{L}_{1,0}$ are module-algebras over the Hopf algebras $\mathbb{A}_{0,1}$ and $\mathbb{A}_{1,0}$ respectively. For $n \geq 0$ we define
\begin{align*}
&\mathbb{A}_{0,n+1} = \mathbb{A}_{0,n} \otimes^{J_n} \mathbb{A}_{0,1} \quad \text{with } J_n = \bigl( \Delta_{\mathbb{U}_q} \otimes \Delta_{\mathbb{U}_q}^{\odot n} \bigr)(R')\\
&\mathcal{L}_{0,n+1}(q^{1/D}) = \mathcal{L}_{0,n}(q^{1/D}) \otimes^{J_n} \mathcal{O}_q(q^{1/D})_F
\end{align*}
where $\textstyle R' = \sum_{(R)} R_{(2)} \otimes R_{(1)}$ is the flip of the $R$-matrix $R \in \mathbb{U}_q^{\otimes 2}$. For $g \geq 0$ let
\begin{align*}
&\mathbb{A}_{g+1,0} = \mathbb{A}_{g,0} \otimes^{K_g} \mathbb{A}_{1,0} \quad \text{with } K_g = \bigl( \Delta_{\mathbb{U}_q}^{(3)} \otimes (\Delta_{\mathbb{U}_q}^{(3)})^{\odot g} \bigr)(R')\\
&\mathcal{L}_{g+1,0} = \mathcal{L}_{g,0} \otimes^{K_g} \mathcal{L}_{1,0}.
\end{align*}
In \cite[Prop. 6.2]{BR1} it is shown that restricting the product of $\mathcal{L}_{0,n}(q^{1/D})$ on the $\mathbb{C}(q)$-subspace $\mathcal{L}_{0,1}^{\otimes n}$ gives a $\mathbb{C}(q)$-subalgebra, which is denoted by $\mathcal{L}_{0,n}$. However, as we have seen in section \ref{SectionL10ForUq}, the algebra $\mathcal{L}_{1,0}$ (and hence $\mathcal{L}_{g,0}$) can only be defined over $\mathbb{C}(q^{1/D})$. So we can define $\mathcal{L}_{g,n}$ over $\mathbb{C}(q^{1/D})$ as follows:
\begin{align*}
\mathbb{A}_{g,n} &= \mathbb{A}_{g,0} \otimes^{F_{g,n}} \mathbb{A}_{0,n} \quad \text{with } F_{g,n} = \bigl(\Delta_{\mathbb{U}_q}^{\odot n} \otimes (\Delta_{\mathbb{U}_q}^{(3)})^{\odot g}\bigr)(R')\\
\mathcal{L}_{g,n} &= \mathcal{L}_{g,0} \otimes^{F_{g,n}} \mathcal{L}_{0,n}(q^{1/D}).
\end{align*}
Explicitly, $\mathcal{L}_{g,n}$ is the $\mathbb{C}(q^{1/D})$-vector space $\mathcal{O}_q(q^{1/D})^{\otimes (2g+n)}$ endowed with the product described in Proposition \ref{productLgn}. 

\smallskip

\indent By construction $\mathcal{L}_{g,n}$ is a right $\mathbb{A}_{g,n}$-module-algebra. As in \S \ref{sectionDefLgnH}, we have a morphism of Hopf algebras $\Delta_{\mathbb{U}_q}^{(4g+2n-1)} : \mathbb{U}_q \to \mathbb{A}_{g,n}$. Combining this with the morphism of Hopf algebras $\iota : U_q \to \mathbb{U}_q$ introduced below \eqref{embeddingUqIntoItsCompletion}, we obtain that $\mathcal{L}_{g,n}$ is a right $U_q$-module-algebra for the action $\mathrm{coad}^r$ in \eqref{coadLgn}.

\smallskip

\indent Note that the factorization of $\mathcal{L}_{g,n}$ as the braided tensor product $\mathcal{L}_{1,0}^{\widetilde{\otimes} g} \:\widetilde{\otimes} \: \mathcal{L}_{0,1}^{\widetilde{\otimes} n}$ (Prop. \ref{propBraidedTensProduct}) holds true for $H = U_q^{\mathrm{ad}}$ as well, despite the fact that the whole category of $U_q^{\mathrm{ad}}$-modules (not necessarily finite-dimensional) is not braided. This is because all the elements of $\mathcal{L}_{0,1}$ and $\mathcal{L}_{1,0}$ are locally-finite, and so the pairing \eqref{pairingOqCategoricalCompletion} and \eqref{eqWellDefRMatrixOnOq} ensure that the formula \eqref{multiplicationInBraidedTensorProduct} for the product in $M \,\widetilde{\otimes} \, N$ is well-defined when $M = \mathcal{L}_{g,n}$, $N = \mathcal{L}_{g',n'}$ and the right action $\cdot$ is $\mathrm{coad}^r$.

\begin{prop}\label{propLgnFinGen}
The algebra $\mathcal{L}_{g,n}$ is finitely generated.
\end{prop}
\begin{proof}
Recall the algebra $\mathrm{gr}_{\mathcal{F}_l}(\mathcal{L}_{0,1})$ from \S \ref{sectionL10Noetherien}. It is proved in \cite[pp. 178]{VY} that $\mathrm{gr}_{\mathcal{F}_l}(\mathcal{L}_{0,1})$ is finitely generated, thanks to the first item in Lemma \ref{lemmaProductsInLeftAndRightGr}. As a result $\mathcal{L}_{0,1}$ is finitely generated, see Lemma \ref{lemmaFinGenGr}. Let $x_1, \ldots, x_m$ be the generators of $\mathcal{L}_{0,1}$. We deduce from \eqref{monomialsLgn} that the elements $\mathfrak{i}_{B(k)}(x_s)$, $\mathfrak{i}_{A(k)}(x_s)$, $\mathfrak{i}_{M(l)}(x_s)$ with $1 \leq k \leq g$, $g+1 \leq l \leq g+n$, $1 \leq s \leq m$ generate $\mathcal{L}_{g,n}$.
\end{proof}

\subsection{Noetherianity of $\mathcal{L}_{g,n}$}\label{sectionNoetherianityLgn}
Here are the steps of the proof:
\begin{enumerate}
\item we introduce a filtration $\mathcal{G}_1$ of $\mathcal{L}_{g,n}$ whose graded algebra $\mathrm{gr}_{\mathcal{G}_1}(\mathcal{L}_{g,n})$ simplifies the third relation in Proposition \ref{productLgn},
\item we construct a filtration $\mathcal{G}_2$ of $\mathrm{gr}_{\mathcal{G}_1}(\mathcal{L}_{g,n})$ thanks to the filtrations on $\mathcal{L}_{0,1}$ and on $\mathcal{L}_{1,0}$ from \S \ref{sectionL10Noetherien},
\item using Lemma \ref{critereNoetherien}, we show that $\mathrm{gr}_{\mathcal{G}_2}\bigl( \mathrm{gr}_{\mathcal{G}_1}(\mathcal{L}_{g,n}) \bigr)$ is Noetherian,
\item it follows from Lemma \ref{lemmaFiltrationNoetherian} that $\mathrm{gr}_{\mathcal{G}_1}(\mathcal{L}_{g,n})$ is Noetherian and that $\mathcal{L}_{g,n}$ is Noetherian.
\end{enumerate}

We write $\mathcal{L}_{0,1}$ instead of $\mathcal{L}_{0,1}(q^{1/D})$ for simplicity of notation. For $\mu, \omega \in P_+$ and $\lambda \in P$ consider the subspaces
\begin{align*}
&(\mathcal{L}_{0,1})_{\mu, \lambda} = \bigl\{ \varphi \in C(\mu) \, \big| \, \forall \, \nu \in P, \: \mathrm{coad}^r(K_{\nu}^{-1})(\varphi) = q^{(\lambda,\nu)}\varphi \bigr\},\\
&(\mathcal{L}_{1,0})_{\mu, \omega, \lambda} = \bigl\{ x \in C(\mu) \otimes C(\omega) \, \big| \, \forall \, \nu \in P, \: \mathrm{coad}^r(K_{\nu}^{-1})(x) = q^{(\lambda,\nu)}x \bigr\}
\end{align*}
where $\mathrm{coad}^r$ is defined in \eqref{coadL01} for $\mathcal{L}_{0,1}$ and in \eqref{coadL10} for $\mathcal{L}_{1,0}$, and $C(\mu)$ is the subspace of matrix coefficients of the irreducible $U_q^{\mathrm{ad}}$-module with highest weight $\mu$. For a finite sequence $\bigl( [\mu], [\lambda] \bigr)$, with $[\mu] = (\mu_1, \ldots, \mu_{2g+n}) \in (P_+)^{2g+n}$ and $[\lambda] = (\lambda_1, \ldots, \lambda_{g+n})\in P^{g+n}$, we put
\[ (\mathcal{L}_{g,n})_{[\mu], [\lambda]} = (\mathcal{L}_{1,0})_{\mu_1,\mu_2, \lambda_1} \otimes \ldots \otimes (\mathcal{L}_{1,0})_{\mu_{2g-1}, \mu_{2g}, \lambda_g} \otimes (\mathcal{L}_{0,1})_{\mu_{2g+1}, \lambda_{g+1}} \otimes \ldots \otimes (\mathcal{L}_{0,1})_{\mu_{2g+n}, \lambda_{g+n}}. \]
Recall the partial order $\preceq$ on $P$ defined in \S \ref{prelimLieAlgebras}. For $\mu, \mu', \omega, \omega' \in P_+$ and $\lambda, \lambda' \in P$ we write $(\mu',\lambda') \preceq (\mu,\lambda)$ (resp. $(\mu',\omega',\lambda') \preceq (\mu,\omega,\lambda)$) to mean that $\mu' \preceq \mu$ and $\lambda' \preceq \lambda$ (resp. $\mu' \preceq \mu$ and $\omega' \preceq \omega$ and $\lambda' \preceq \lambda$). We say that $\bigl( [\mu'], [\lambda'] \bigr) \prec_1 \bigl( [\mu], [\lambda] \bigr)$ if and only if
\begin{align*}
&\:(\mu'_{2g+n}, \lambda'_{g+n}) \prec (\mu_{2g+n}, \lambda_{g+n})\\
\text{ or } &\: \Big( (\mu'_{2g+n}, \lambda'_{g+n}) = (\mu_{2g+n}, \lambda_{g+n}) \text{ and } (\mu'_{2g+n-1}, \lambda'_{g+n-1}) \prec (\mu_{2g+n-1}, \lambda_{g+n-1}) \Big) \text{ or } \ldots
\end{align*}
This defines a partial order $\preceq_1$ on $(P_+)^{2g+n} \times P^{g+n}$, which is a block lexicographic order starting from the right. For instance if $g=1$ and $n=1$ we have $\bigl( [\mu'_1,\mu'_2,\mu'_3], [\lambda'_1,\lambda'_2] \bigr) \prec_1 \bigl( [\mu_1,\mu_2,\mu_3], [\lambda_1,\lambda_2] \bigr)$ if and only if
\[ (\mu'_3, \lambda'_2) \prec (\mu_3, \lambda_2) \text{ or } \Big( (\mu'_3, \lambda'_2) = (\mu_3, \lambda_2) \text{ and } (\mu'_1, \mu'_2, \lambda'_1) \prec (\mu_1, \mu_2, \lambda_1) \Big). \]
Now let
\[ \Xi = \bigl\{ \bigl( [\mu], [\lambda] \bigr) \in (P_+)^{2g+n} \times P^{g+n} \, \big| \, (\mathcal{L}_{g,n})_{[\mu], [\lambda]} \neq 0 \bigr\}. \]
Then $(\Xi, \preceq_1)$ is an ordered abelian monoid and the order $\preceq_1$ is well-founded on $\Xi$. Moreover $\preceq_1$ satisfies the condition \eqref{confluentOrder} and it follows that
\[ \mathcal{G}_1^{[\mu], [\lambda]} = \bigoplus_{( [\mu'], [\lambda']) \preceq_1 ([\mu], [\lambda])} (\mathcal{L}_{g,n})_{[\mu'], [\lambda']} \qquad (\text{with } \bigl( [\mu'], [\lambda']\bigr) \in \Xi) \]
defines a filtration $\mathcal{G}_1$ of the vector space $\mathcal{L}_{g,n}$ indexed by $\Xi$. For any $1 \leq k \leq g$ and $g+1 \leq l \leq g+n$ it is convenient to use the embeddings
\begin{equation}\label{specialSequenceForLgn}
S_k : P_+^2 \times P \to P_+^{2g+n} \times P^{g+n}, \qquad S_l : P_+ \times P \to P_+^{2g+n} \times P^{g+n}
\end{equation}
such that $S_k(\mu,\omega,\lambda)$ has the $2k$-th entry equal to $\mu$, the $2k+1$-th entry equal to $\omega$, the $2g+n+k$-th entry equal to $\lambda$, and all remaining entries equal to $0$; similarly $S_l(\mu,\lambda)$ has the $g+l$-th entry equal to $\mu$, the $2g+n+l$-th entry equal to $\lambda$, and all remaining entries equal to $0$. These embeddings are such that  
\begin{equation}\label{propertySpecialSequences}
(\mathcal{L}_{g,n})_{S_k(\mu,\omega,\lambda)} = j_k\bigl( (\mathcal{L}_{1,0})_{\mu,\omega,\lambda} \bigr), \qquad (\mathcal{L}_{g,n})_{S_l(\mu, \lambda)} = j_{l}\bigl( (\mathcal{L}_{0,1})_{\mu, \lambda} \bigr)
\end{equation}
for $\mu,\omega \in P_+$, $\lambda \in P$, and where $j_k$, $j_l$ are the embeddings from \eqref{embeddingsL01L10InLgn}.

\begin{lem}
$\mathcal{G}_1$ is a filtration of the algebra $\mathcal{L}_{g,n}$.
\end{lem}
\begin{proof}
Let
\[ y_1 \otimes \ldots \otimes y_{g+n} = j_1(y_1) \ldots j_{g+n}(y_{g+n}) \in (\mathcal{L}_{g,n})_{[\mu'], [\lambda']} \]
where by definition we have $y_k \in (\mathcal{L}_{1,0})_{\mu'_{2k-1}, \mu'_{2k}, \lambda'_k}$ for $1 \leq k \leq g$ and $y_l \in (\mathcal{L}_{0,1})_{\mu'_{g+l}, \lambda'_l}$ for $g+1 \leq l \leq g+n$. Take an element $x \in (\mathcal{L}_{0,1})_{\mu, \lambda}$; so $j_m(x) \in (\mathcal{L}_{g,n})_{S_m(\mu, \lambda)}$ for any $m \in \{g+1, \ldots, g+n\}$ (see \eqref{specialSequenceForLgn} and \eqref{propertySpecialSequences}). By Proposition \ref{propBraidedTensProduct} and \eqref{productBraided} we have
\begin{align*}
&j_m(x) \, j_1(y_1) \ldots j_{g+n}(y_{g+n})\\
=\:\,&\sum_{(R^1), \ldots, (R^{m-1})} j_1\!\left( \mathrm{coad}^r(R^1_{(1)})(y_1) \right) \ldots j_{m-1}\!\left( \mathrm{coad}^r(R^{m-1}_{(1)})(y_{m-1}) \right)\\[-1em]
&\hphantom{\sum_{(R^1), \ldots, (R^{m-1})}} \:\:\times j_m\!\left( \mathrm{coad}^r\bigl( R^1_{(2)} \ldots R^{m-1}_{(2)} \bigr)(x)y_m \right) j_{m+1}(y_{m+1}) \ldots j_{g+n}(y_{g+n}).
\end{align*}
For any term in this sum and any $\nu \in P$, the fact that $\mathcal{L}_{0,1}$ is a $U_q$-module-algebras for $\mathrm{coad}^r$ together with \eqref{commutationRK} gives
\begin{multline*}
\mathrm{coad}^r(K_{\nu}^{-1})\!\left( \mathrm{coad}^r\bigl(R^1_{(2)} \ldots R^{m-1}_{(2)}\bigr)(x) \, y_m \right)\\
= q^{(\nu,\, \lambda + \lambda'_m - \gamma_1 - \ldots - \gamma_{m-1})} \mathrm{coad}^r\bigl(R^1_{(2)} \ldots R^{m-1}_{(2)}\bigr)(x) \, y_m
\end{multline*}
for some $\gamma_1, \ldots, \gamma_{m-1} \in Q_+$. Moreover due to \eqref{filtrationOq} and \eqref{produitL01} we have
\[ \textstyle \mathrm{coad}^r\bigl(R^1_{(2)} \ldots R^{m-1}_{(2)}\bigr)(x) \, y_m \in \bigoplus_{\kappa \preceq \mu + \mu_{g+m}'} C(\kappa). \]
It follows from these two facts that $j_m(x) \, j_1(y_1) \ldots j_{g+n}(y_{g+n}) \in \mathcal{G}_1^{S_m(\mu, \lambda) + ([\mu'],[\lambda'])}$ by definition of $\mathcal{G}_1$. If now $x$ is an element in $(\mathcal{L}_{1,0})_{\mu, \omega, \lambda}$, $j_m(x) \in (\mathcal{L}_{g,n})_{S_m(\mu, \omega, \lambda)}$ for any $m \in \{1, \ldots, g\}$, and we have an analogous proof. These two particular cases imply the result.
\end{proof}

\indent Let us identify $\mathrm{gr}_{\mathcal{G}_1}(\mathcal{L}_{g,n})$ and $\mathcal{L}_{g,n}$ as vector spaces; namely $x \in (\mathcal{L}_{g,n})_{[\mu],[\lambda]}$ is identified with the coset $x + \mathcal{G}_1^{\prec_1 ([\mu],[\lambda])} \in \mathrm{gr}_{\mathcal{G}_1}(\mathcal{L}_{g,n})$. Denote by $\bullet$ the product on $\mathrm{gr}_{\mathcal{G}_1}(\mathcal{L}_{g,n})$. By the identification above, we can regard it as a product on the vector space $\mathcal{L}_{g,n}$. Consider the canonical projections
\[ \pi_{\eta} : \bigoplus_{\kappa \preceq \eta} C(\kappa) \to C(\eta), \quad \pi_{(\eta,\eta')} : \bigoplus_{\substack{\kappa \preceq \eta\\\kappa' \preceq \eta'}} C(\kappa) \otimes C(\kappa') \to C(\eta) \otimes C(\eta') \]
which allow us to define the truncated product $\overline{\,\cdot\,}$ on $\mathcal{L}_{0,1}$ and $\mathcal{L}_{1,0}$:
\[ x \, \overline{\,\cdot\,} \, y = \pi_{\mu+\mu'}(xy) \qquad (\text{resp. } x \, \overline{\,\cdot\,} \, y = \pi_{(\mu+\mu', \omega + \omega')}(xy)) \]
for $x \in C(\mu)$ and $y \in C(\mu')$ (resp. $x \in C(\mu) \otimes C(\omega)$ and $y \in C(\mu') \otimes C(\omega')$). We denote by $(\mathcal{L}_{0,1}, \overline{\,\cdot\,})$ and $(\mathcal{L}_{1,0}, \overline{\,\cdot\,})$ the vector spaces $\mathcal{L}_{0,1}$ and $\mathcal{L}_{1,0}$ endowed with the products $\overline{\,\cdot\,}$ above. The next lemma describes $\bullet$ thanks to the embeddings of vector spaces $j_k$ from \eqref{embeddingsL01L10InLgn}.

\begin{lem}\label{lemmaProductGrG1}
Let $k,l \in \{1, \ldots, g+n\}$, $\mu,\omega,\mu',\omega' \in P_+$ and $\lambda, \lambda' \in P$. Take $x$ in $(\mathcal{L}_{1,0})_{\mu,\omega,\lambda}$ or $(\mathcal{L}_{0,1})_{\mu, \lambda}$, depending if $k$ is $\leq g$ or $>g$. Similarly, take $y$ in $(\mathcal{L}_{1,0})_{\mu', \omega', \lambda'}$ or $(\mathcal{L}_{0,1})_{\mu',\lambda'}$, depending if $l$ is $\leq g$ or $>g$. Then
\begin{align*}
&j_k(x) \bullet j_k(y) = j_k(x \,\overline{\,\cdot\,}\, y) \quad \text{ for all } k,\\
&j_k(x) \bullet j_l(y) = j_k(x) j_l(y) \quad \text{if } k<l,\\
&j_k(x) \bullet j_l(y) = q^{(\lambda,\lambda')} j_l(y)j_k(x) \quad \text{if } k>l.
\end{align*}
\end{lem}
\begin{proof}
Let $\simeq$ be the identification $\mathrm{gr}_{\mathcal{G}_1}(\mathcal{L}_{g,n}) \simeq \mathcal{L}_{g,n}$. Let us prove the first formula. Take for instance $k > g$. By assumption $j_k(x) \in (\mathcal{L}_{g,n})_{S_k(\mu, \lambda)}$ and $j_k(y) \in (\mathcal{L}_{g,n})_{S_k(\mu',\lambda')}$. Then $\textstyle xy \in \bigoplus_{\kappa \preceq \mu + \mu'} (\mathcal{L}_{0,1})_{\kappa,\lambda+\lambda'}$ and we have
\[ j_k(x) \bullet j_k(y) \simeq \left( j_k(x) + \mathcal{G}_1^{\prec_1 S_k(\mu,\lambda)} \right) \left( j_k(y) + \mathcal{G}_1^{\prec_1 S_k(\mu',\lambda')} \right) = j_k(xy) + \mathcal{G}_1^{\prec_1 S_k(\mu+\mu',\lambda+\lambda')} \simeq j_k(x\, \overline{\,\cdot\,} \,y). \]
The second formula is obtained similarly. For the third formula, recall the expression of $R$ in \eqref{expressionCanoniqueR} and its commutations rules with $K_{\nu}$ in \eqref{commutationRK}. Then from \eqref{productBraided} and the definition of $\mathcal{G}_1$ we get
\begin{align*}
&j_k(x) \bullet j_l(y) \simeq \left( j_k(x) + \mathcal{G}_1^{\prec_1 S_k(\mu,\lambda)} \right) \left( j_l(y) + \mathcal{G}_1^{\prec_1 S_l(\mu',\lambda')} \right) = j_k(x)j_l(y) + \mathcal{G}_1^{\prec_1 S_k(\mu,\lambda) + S_l(\mu',\lambda')}\\
=\:\,& \sum_{(R)} j_l\bigl( \mathrm{coad}^r(R_{(1)})(y) \bigr) j_k\bigl( \mathrm{coad}^r(R_{(2)})(x) \bigr) + \mathcal{G}_1^{\prec_1 S_k(\mu,\lambda) + S_l(\mu',\lambda')}\\
=\:\,& j_l\bigl( \mathrm{coad}^r(\Theta_{(1)})(y) \bigr) j_k\bigl( \mathrm{coad}^r(\Theta_{(2)})(x) \bigr) + \mathcal{G}_1^{\prec_1 S_k(\mu,\lambda) + S_l(\mu',\lambda')}\\
=\:\,& q^{(\lambda,\lambda')}j_l(y)j_k(x) + \mathcal{G}_1^{\prec_1 S_k(\mu,\lambda) + S_l(\mu',\lambda')} \simeq q^{(\lambda,\lambda')}j_l(y)j_k(x). \qedhere
\end{align*}
\end{proof}

\indent Recall from \S \ref{sectionL10Noetherien} the filtration $\mathcal{F}_\ell$ on $\mathcal{L}_{0,1}$ indexed by $\Lambda$ (where $\Lambda$ is defined in \eqref{defIndexSetLambda}) and the filtration $\mathcal{F}$ on $\mathcal{L}_{1,0}$ indexed by $\Lambda \times \Lambda$. Since we identified $\mathrm{gr}_{\mathcal{G}_1}(\mathcal{L}_{g,n})$ with $\mathcal{L}_{1,0}^{\otimes g} \otimes \mathcal{L}_{0,1}^{\otimes n}$ as a vector space, we can consider the family of subspaces $\mathcal{G}_2$ indexed by $\Lambda^{2g+n}$ and defined by
\begin{multline*}
\mathcal{G}_2^{\mu_1,\lambda_1,\eta_1,\kappa_1, \ldots, \mu_g,\lambda_g,\eta_g,\kappa_g, \mu_{g+1},\lambda_{g+1}, \ldots, \mu_{g+n},\lambda_{g+n}}\\
= \mathcal{F}^{\mu_1,\lambda_1,\eta_1,\kappa_1} \otimes \ldots \otimes \mathcal{F}^{\mu_g,\lambda_g,\eta_g,\kappa_g} \otimes \mathcal{F}_\ell^{\mu_{g+1},\lambda_{g+1}} \otimes \ldots \otimes \mathcal{F}_\ell^{\mu_{g+n},\lambda_{g+n}}.
\end{multline*}
We endow $P^{4g+2n}$ with the partial order $\preceq_2$ given by
\[ (p_1, \ldots, p_{4g+2n}) \preceq_2 (p'_1, \ldots, p'_{4g+2n}) \quad \iff \quad \forall \, i, \: p_i \preceq p'_i  \]
which gives by restriction a partial order on $\Lambda^{2g+n}$ which satisfies \eqref{confluentOrder} and is well-founded. Then $\mathcal{G}_2$ becomes a filtration of the vector space $\mathrm{gr}_{\mathcal{G}_1}(\mathcal{L}_{g,n})$.
\begin{lem}
$\mathcal{G}_2$ is a filtration of the algebra $\mathrm{gr}_{\mathcal{G}_1}(\mathcal{L}_{g,n})$.
\end{lem}
\begin{proof}
Take $j_1(x_1) \ldots j_{g+n}(x_{g+n}) \in \mathcal{G}_2^S$ and $j_1(y_1) \ldots j_{g+n}(y_{g+n}) \in \mathcal{G}_2^{S'}$ for some $S,S' \in \Lambda^{2g+n}$, where $j_k$ is the embedding of vector spaces \eqref{embeddingsL01L10InLgn}. We can assume that these elements are homogeneous in $\mathrm{gr}_{\mathcal{G}_1}(\mathcal{L}_{g,n})$. Then according to Lemma \ref{lemmaProductGrG1}:
\[ \bigl( j_1(x_1) \ldots j_{g+n}(x_{g+n}) \bigr) \bullet \bigl( j_1(y_1) \ldots j_{g+n}(y_{g+n}) \bigr) = q^Nj_1(x_1\,\overline{\,\cdot\,}\,y_1) \ldots j_{g+n}(x_{g+n}\,\overline{\,\cdot\,}\,y_{g+n}) \]
for some $N \in (1/D)\mathbb{Z}$. It is easy to see that $\mathcal{F}$ and $\mathcal{F}_\ell$ are also filtrations of $(\mathcal{L}_{1,0}, \overline{\,\cdot\,})$ and $(\mathcal{L}_{0,1}, \overline{\,\cdot\,})$ respectively, which gives the result.
\end{proof}

For $1 \leq k \leq g$ and $g+1 \leq l \leq g+n$ we will use the embeddings
\[ \mathbf{S}_k : \Lambda^2 \to (\Lambda^2)^g \times \Lambda^n, \qquad \mathbf{S}_l : \Lambda \to (\Lambda^2)^g \times \Lambda^n \]
such that, for $\xi$, $\zeta\in \Lambda$ we have $\mathbf{S}_k(\xi,\zeta)$ has the $2k$-th entry equal to $\xi$, the $2k+1$-th entry equal to $\zeta$, and all remaining entries equal to $0$; similarly $\mathbf{S}_l(\xi)$ has the $g+l$-th entry equal to $\xi$, and all remaining entries equal to $0$.
\begin{teo}\label{ThmLgnNoetherian}
The algebra $\mathcal{L}_{g,n}$ is Noetherian.
\end{teo}
\begin{proof}
We are going to show that $\mathrm{gr}_{\mathcal{G}_2}\bigl( \mathrm{gr}_{\mathcal{G}_1}(\mathcal{L}_{g,n}) \bigr)$ is Noetherian. Recall that $\mathrm{gr}_{\mathcal{G}_1}(\mathcal{L}_{g,n})$ is identified with $\mathcal{L}_{g,n}$ endowed with the product $\bullet$. To avoid confusion we do not identify again $\mathrm{gr}_{\mathcal{G}_2}\bigl( \mathrm{gr}_{\mathcal{G}_1}(\mathcal{L}_{g,n}) \bigr)$ with $\mathcal{L}_{g,n}$ as a vector space and instead we use cosets $x + \mathcal{G}_2^{\prec_2 S}$.

\noindent Note first that thanks to the first formula in Lemma \ref{lemmaProductGrG1} the linear maps $\mathfrak{i}_{B(k)}, \mathfrak{i}_{A(k)}, \mathfrak{i}_{M(s)} : \mathcal{L}_{0,1} \to \mathcal{L}_{g,n}$ from \eqref{embeddingsFrakI} are morphisms of algebras
\[ \overline{\mathfrak{i}}_{B(k)}, \overline{\mathfrak{i}}_{A(k)}, \overline{\mathfrak{i}}_{M(s)} : \bigl( \mathcal{L}_{0,1}, \overline{\,\cdot\,} \, \bigr) \to \mathrm{gr}_{\mathcal{G}_1}(\mathcal{L}_{g,n}). \]
It is easy to see that the filtration $\mathcal{F}_x$ ($x \in \{\ell,r\}$) of $\mathcal{L}_{0,1}$ is also a filtration of $\bigl( \mathcal{L}_{0,1}, \overline{\,\cdot\,} \, \bigr)$ and that $\mathrm{gr}_{\mathcal{F}_x}\bigl( \mathcal{L}_{0,1}, \overline{\,\cdot\,} \, \bigr) = \mathrm{gr}_{\mathcal{F}_x}(\mathcal{L}_{0,1})$. If $\varphi, \psi \in \mathcal{L}_{0,1}$ are such that $\mathfrak{i}_B(\varphi) \in \mathcal{F}^{\mu,\lambda,0,0}$ and $\mathfrak{i}_A(\psi) \in \mathcal{F}^{0,0,\eta,\kappa}$ we let
\[ \widehat{\mathfrak{i}}_{B(k)}(\varphi) = \overline{\mathfrak{i}}_{B(k)}(\varphi) + \mathcal{G}_2^{\prec_2 \mathbf{S}_k(\mu,\lambda,0,0)}, \qquad
\widehat{\mathfrak{i}}_{A(k)}(\psi) = \overline{\mathfrak{i}}_{A(k)}(\psi) + \mathcal{G}_2^{\prec_2 \mathbf{S}_k(0,0,\eta,\kappa)} \]
for $1 \leq k \leq g$. Similarly, for $\varphi \in \mathcal{F}^{\mu,\lambda}_\ell \subset \mathcal{L}_{0,1}$ and $g+1 \leq s \leq g+n$ we let
\[ \widehat{\mathfrak{i}}_{M(s)}(\varphi) = \overline{\mathfrak{i}}_{M(s)}(\varphi) + \mathcal{G}_2^{\prec_2 \mathbf{S}_s(\mu,\lambda)}. \]
Since $\overline{\mathfrak{i}}_{B(k)}, \overline{\mathfrak{i}}_{A(k)}, \overline{\mathfrak{i}}_{M(s)}$ are morphisms of algebras and
\[ \overline{\mathfrak{i}}_{B(k)}\bigl( \mathcal{F}_r^{\prec \mu,\lambda} \bigr) = \mathcal{G}_2^{\prec_2 \mathbf{S}_k(\mu,\lambda,0,0)}, \quad \overline{\mathfrak{i}}_{A(k)}\bigl( \mathcal{F}_\ell^{\prec \eta,\kappa} \bigr) = \mathcal{G}_2^{\prec_2 \mathbf{S}_k(0,0,\eta,\kappa)}, \quad \overline{\mathfrak{i}}_{M(s)}\bigl( \mathcal{F}_\ell^{\prec \mu,\lambda} \bigr) = \mathcal{G}_2^{\prec_2 \mathbf{S}_s(\mu,\lambda)} \]
we obtain morphisms of algebras
\begin{align*}
\widehat{\mathfrak{i}}_{B(k)} :&\:\, \mathrm{gr}_{\mathcal{F}_r}\bigl( \mathcal{L}_{0,1}, \overline{\,\cdot\,} \, \bigr) = \mathrm{gr}_{\mathcal{F}_r}(\mathcal{L}_{0,1}) \to \mathrm{gr}_{\mathcal{G}_2}\bigl( \mathrm{gr}_{\mathcal{G}_1}(\mathcal{L}_{g,n}) \bigr),\\
\widehat{\mathfrak{i}}_{A(k)}, \: \widehat{\mathfrak{i}}_{M(s)} :&\:\, \mathrm{gr}_{\mathcal{F}_\ell}\bigl( \mathcal{L}_{0,1}, \overline{\,\cdot\,} \, \bigr) = \mathrm{gr}_{\mathcal{F}_\ell}(\mathcal{L}_{0,1}) \to \mathrm{gr}_{\mathcal{G}_2}\bigl( \mathrm{gr}_{\mathcal{G}_1}(\mathcal{L}_{g,n}) \bigr).
\end{align*}
Recall from \S \ref{sectionL10Noetherien} that $\mathrm{gr}_{\mathcal{F}_x}(\mathcal{L}_{0,1})$ with $x \in \{\ell,r\}$ is generated by a finite number of matrix coefficients $u_1, \ldots, u_p$. Thanks to the second formula in Lemma \ref{lemmaProductGrG1}, we deduce that $\mathrm{gr}_{\mathcal{G}_2}\bigl( \mathrm{gr}_{\mathcal{G}_1}(\mathcal{L}_{g,n}) \bigr)$ is generated by the finite set of elements $\widehat{\mathfrak{i}}_{B(k)}(u_r)$, $\widehat{\mathfrak{i}}_{A(k)}(u_r)$, $\widehat{\mathfrak{i}}_{M(s)}(u_r)$ for $1 \leq k \leq g$, $g+1 \leq s \leq g+n$ and $1 \leq r \leq p$. By applying the morphism $\widehat{\mathfrak{i}}_{X(r)}$ to \eqref{relationsNoetherianiteL01}, where $X(r)$ is $A(r)$, $B(r)$ or $M(r)$, we obtain
\[ \forall \, 1 \leq b < a \leq p, \quad \widehat{\mathfrak{i}}_{X(r)}(u_a)\,\widehat{\mathfrak{i}}_{X(r)}(u_b) = q_{ab} \, \widehat{\mathfrak{i}}_{X(r)}(u_b)\, \widehat{\mathfrak{i}}_{X(r)}(u_a) + \sum_{s=1}^{b-1} \sum_{t=1}^p \alpha^{ab}_{st} \, \widehat{\mathfrak{i}}_{X(r)}(u_s) \, \widehat{\mathfrak{i}}_{X(r)}(u_t) \]
with certain scalars $q_{ab} \in \mathbb{C}(q^{1/D})^{\times}$ and $\alpha^{ab}_{st} \in \mathbb{C}(q^{1/D})$. Similarly, it follows from \eqref{echangeL10Generateurs} that
\[ \forall \, 1 \leq b < a \leq p, \quad \widehat{\mathfrak{i}}_{A(r)}(u_a) \, \widehat{\mathfrak{i}}_{B(r)}(u_b) = q'_{ab} \, \widehat{\mathfrak{i}}_{B(r)}(u_b) \, \widehat{\mathfrak{i}}_{A(r)}(u_a) + \sum_{s=1}^{b-1} \sum_{t=1}^p \lambda^{ab}_{st} \, \widehat{\mathfrak{i}}_{B(r)}(u_s) \, \widehat{\mathfrak{i}}_{A(r)}(u_t) \]
for any $1 \leq r \leq g$ and with certain scalars $q'_{ab} \in \mathbb{C}(q^{1/D})^{\times}$ and $\lambda^{ab}_{st} \in \mathbb{C}(q^{1/D})$.
Finally, we see from the third formula in Lemma \ref{lemmaProductGrG1} that
\[ \forall \, 1 \leq a, b \leq p, \quad \widehat{\mathfrak{i}}_{Y(s)}(u_a) \, \widehat{\mathfrak{i}}_{X(r)}(u_b) = q''_{ab} \, \widehat{\mathfrak{i}}_{X(r)}(u_b) \, \widehat{\mathfrak{i}}_{Y(s)}(u_a) \]
where $r<s$ and $X(r)$ (resp. $Y(s)$) is $A(r)$, $B(r)$ or $M(r)$ (resp. $A(s)$, $B(s)$ or $M(s)$) and $q''_{ab} \in \mathbb{C}(q^{1/D})^{\times}$.

\noindent We define a sequence $(\widehat{x}_1, \ldots, \widehat{x}_{(2g+n)p})$ by
\begin{align*}
&\widehat{x}_1 = \widehat{\mathfrak{i}}_{B(1)}(u_1), \ldots, \widehat{x}_p = \widehat{\mathfrak{i}}_{B(1)}(u_p), \: \widehat{x}_{p+1} = \widehat{\mathfrak{i}}_{A(1)}(u_1), \ldots, \widehat{x}_{2p} = \widehat{\mathfrak{i}}_{A(1)}(u_p),\\
&\hspace{1.5em}\vdots\hspace{11.5em}\vdots\hspace{11.5em}\vdots\\
&\widehat{x}_{2(g-1)p+1} = \widehat{\mathfrak{i}}_{B(g)}(u_1), \ldots, \widehat{x}_{2(g-1)p+p} = \widehat{\mathfrak{i}}_{B(g)}(u_p), \: \widehat{x}_{2(g-1)p+p+1} = \widehat{\mathfrak{i}}_{A(g)}(u_1), \ldots, \widehat{x}_{2gp} = \widehat{\mathfrak{i}}_{A(g)}(u_p),\\
&\widehat{x}_{2gp+1} = \widehat{\mathfrak{i}}_{M(g+1)}(u_1), \ldots, \widehat{x}_{(2g+1)p} = \widehat{\mathfrak{i}}_{M(g+1)}(u_p),\\
&\hspace{1.5em}\vdots\hspace{13.5em}\vdots\\
&\widehat{x}_{(2g+n-1)p+1} = \widehat{\mathfrak{i}}_{M(g+n)}(u_1), \ldots, \widehat{x}_{(2g+n)p} = \widehat{\mathfrak{i}}_{M(g+n)}(u_p).
\end{align*}
The relations above show that the generators $\widehat{x}_s$ satisfy the relations required by Lemma \ref{critereNoetherien}. It follows that $\mathrm{gr}_{\mathcal{G}_2}\bigl( \mathrm{gr}_{\mathcal{G}_1}(\mathcal{L}_{g,n}) \bigr)$ is Noetherian. Using Lemma \ref{lemmaFiltrationNoetherian} two times, we obtain that $\mathrm{gr}_{\mathcal{G}_1}(\mathcal{L}_{g,n})$ is Noetherian and that $\mathcal{L}_{g,n}$ is Noetherian.
\end{proof}

\begin{remark}
One can wonder why we did not use the simpler filtration by the subspaces $\textstyle \widetilde{\mathcal{G}}_1^{[\lambda]} = \cup_{[\mu] \in P^{2g+n}} \mathcal{G}_1^{[\mu],[\lambda]}$ indexed by $P^{g+n}$ with lexicographic order from the right, for which all the above proofs seem simpler since the product $\overline{\,\cdot\,}$ does not appear in $\mathrm{gr}_{\widetilde{\mathcal{G}}_1}(\mathcal{L}_{g,n})$. The problem is that the lexicographic order on $P^{g+n}$ is not well-founded, hence Lemma \ref{lemmaFiltrationNoetherian} does not apply to $\mathrm{gr}_{\widetilde{\mathcal{G}}_1}(\mathcal{L}_{g,n})$ and the last argument in the proof of Theorem \ref{ThmLgnNoetherian} becomes wrong.
\end{remark}

\subsection{Noetherianity and finiteness of $\mathcal{L}_{g,n}^{U_q}$}\label{sectionNoetherianityLgnInv}
The subalgebra of $U_q$-invariant elements is
\[ \mathcal{L}_{g,n}^{U_q} = \bigl\{ x \in \mathcal{L}_{g,n} \, \big| \, \forall \, h \in U_q, \: \mathrm{coad}^r(h)(x) = \varepsilon(h)x \bigr\} \]
where the right action $\mathrm{coad}^r$ was defined in \eqref{coadLgn}. More generally if $V$ is a module over a Hopf algebra $H$, the subspace of invariant elements is
\[ V^H =  \bigl\{ v \in V \: | \: \forall \, h \in H, \: h\cdot v = \varepsilon(h)v \bigr\} \]
where $\varepsilon$ is the counit of $H$.

\smallskip

\indent We start with a general result about the structure of invariant elements in module-algebras; it is just an abstract formulation of the discussion below Theorem 3.2 in \cite{BR2}, which is itself inspired by \cite[Chap. 3, \S 1]{DC}. Denote by $\mathcal{I}$ an additive abelian monoid endowed with a partial order $\leq$ such that
\begin{equation}\label{conditionsOrdreHilbertNagata}
\begin{array}{l}
\forall \, i \in \mathcal{I}, \:\: 0 \leq i \\
\forall \, i,j,k \in \mathcal{I}, \:\: i \leq j \implies i+k \leq j+k.
\end{array}
\end{equation}
We assume moreover that $\leq$ is well-founded. For instance, $\mathcal{I} = P_+$ endowed with the order $\preceq$ (see \S \ref{prelimLieAlgebras}) satisfies these assumptions.
\begin{teo}[Hilbert--Nagata theorem for module-algebras]\label{HilbertNagata}
Let $H$ be a Hopf algebra and let $\textstyle A = \bigoplus_{i \in \mathcal{I}} A_i$ be a graded $k$-algebra which is a $H$-module-algebra. Assume that
\begin{itemize}
\item $A_0 = k\, 1$,
\item $A$ is Noetherian and finitely generated,
\item for each $i \in \mathcal{I}$, $A_i$ is stable under the action of $H$ and is a semisimple $H$-module.
\end{itemize}
Then the subalgebra $A^H$ is Noetherian and finitely generated.
\end{teo}
\noindent The proof relies on three basic lemmas. Note that since $A_i$ is semisimple we can write 
\[ \textstyle A_i = (A_i)^H \oplus \bigoplus_{n} V_{i,n}\]
where the $V_{i,n}$ are some simple $H$-modules not isomorphic to the trivial module $k$. Projecting along this decomposition gives a {\it canonical} $H$-morphism $\mathfrak{R}_i : A_i \to (A_i)^H$. Since each $A_i$ is $H$-stable we have $\textstyle A^H = \bigoplus_{i \in \mathcal{I}} (A_i)^H$ and we can define
\begin{equation}\label{defReynolds}
\mathfrak{R} = \bigoplus_{i \in \mathcal{I}} \mathfrak{R}_i : A \to A^H. 
\end{equation}
The morphism $\mathfrak{R}$ is called the {\em Reynolds operator}.
\begin{lem}\label{lemmaReynolds}
For $x \in A^H$ and $y \in A$, it holds $\mathfrak{R}(xy) = x\mathfrak{R}(y)$ and $\mathfrak{R}(yx) = \mathfrak{R}(y)x$.
\end{lem}
\begin{proof}
We can assume without loss of generality that $x \in (A_i)^H$ and $y \in A_j$. Consider the map
\[ \fonc{m_x}{A_j}{A_{i+j}}{y}{xy} \]
It is $H$-linear:
\[ h \cdot m_x(y) = h\cdot(xy) = \sum_{(h)} (h_{(1)} \cdot x)(h_{(2)} \cdot y) = \sum_{(h)} (\varepsilon(h_{(1)})x)(h_{(2)} \cdot y) = x(h \cdot y) = m_x(h \cdot y). \]
Let $A_j = (A_j)^H \oplus V_j$ and $A_{i+j} = (A_{i+j})^H \oplus V_{i+j}$, where the semisimple modules $V_j$ and $V_{i+j}$ do not have any direct summand isomorphic to the trivial module $k$. Then by Schur's lemma
\[ \mathrm{Hom}_H(A_j, A_{i+j}) = \mathrm{Hom}_k\bigl( (A_j)^H, (A_{i+j})^H \bigr) \oplus \mathrm{Hom}_H(V_j, V_{i+j}). \]
It follows that if we write $y = \mathfrak{R}(y) \oplus v \in (A_j)^H \oplus V_j$, we have
\[ xy = m_x(y) = m_x(\mathfrak{R}(y) \oplus v) = m_x\bigl(\mathfrak{R}(y)\bigr) \oplus m_x(v) = x\mathfrak{R}(y) \oplus xv \in (A_{i+j})^H \oplus V_{i+j}. \]
But by definition of $\mathfrak{R}$ we also have $xy = \mathfrak{R}(xy) \oplus w$, which gives the desired equality. The second one is shown similarly.
\end{proof}

\begin{lem}\label{lemmaFiniteGeneratorsOfIdeals}
Let $A$ be a Noetherian algebra and $X \subset A$. There exists a finite subset $X_{\mathrm{fin}} \subset X$ such that $AX = AX_{\mathrm{fin}}$, where $AY$ denotes the left ideal generated by $Y \subset A$. The same is true for right ideals.
\end{lem}
\begin{proof}
Since $A$ is Noetherian, $AX$ is finitely generated. So there exists elements
\[ g_1 = \sum_{i \in I_1} a_i^{(1)} x_i^{(1)}, \: \ldots, \: g_n = \sum_{i \in I_n} a_i^{(n)} x_i^{(n)} \]
(where $a_i^{(j)} \in A$, $x_i^{(j)} \in X$ and $I_j$ is finite for each $j$) such that $AX = Ag_1 + \ldots + Ag_n$. Take $X_{\mathrm{fin}} = \{ x_i^{(1)} \}_{i \in I_1} \cup \ldots \cup \{ x_i^{(n)} \}_{i \in I_n}$. Since $X_{\mathrm{fin}}$ is a (finite) subset of $X$, we have $AX_{\mathrm{fin}} \subset AX$. On the other hand, since $g_1, \ldots, g_n \in AX_{\mathrm{fin}}$, we have $AX \subset A(AX_{\mathrm{fin}}) = AX_{\mathrm{fin}}$.
\end{proof}

The last lemma is a well-known fact:
\begin{lem}\label{gradedNoetherianFinGen}
Let $\textstyle A = \bigoplus_{i \in \mathcal{I}} A_i$ be a graded $k$-algebra such that $A_0 = k \, 1$. If $A$ is Noetherian, then it is finitely generated.
\end{lem}
\begin{proof}
Let $\textstyle A_+ = \bigoplus_{i \in \mathcal{I}\setminus \{0\}} A_i$. Since $A_+$ is an ideal, it is finitely generated: there exists $g_1, \ldots, g_n$ such that $A_+ = Ag_1 + \ldots + Ag_n$. We can assume that the elements $g_s$ are homogeneous (\text{i.e.} $g_s \in A_{j_s}$ with $j_s > 0$). We will show that $A_i \subset k\langle g_1, \ldots, g_n \rangle$ by well-founded induction on $i$. Firstly, $A_0 = k\,1 \subset k\langle g_1, \ldots, g_n \rangle$. Now assume that $\textstyle \bigoplus_{j < i} A_j \subset k\langle g_1, \ldots, g_n \rangle$ for some $i \in \mathcal{I}\setminus \{0\}$ and take $x \in A_i$. Since $A_i \subset A_+$ we have $\textstyle x = \sum_{s=1}^n a_s g_s$. The elements $x$ and $g_s$ being homogeneous, the $a_s$ can be assumed to be homogeneous as well: $a_s \in A_{l_s}$. Then $j_s + l_s = i$ for all $s$, which implies $l_s < i$ due to the assumptions \eqref{conditionsOrdreHilbertNagata} on $\leq$. Hence $a_s \in \textstyle \bigoplus_{j < i} A_j$ so that $a_s \in k\langle g_1, \ldots, g_n \rangle$ by the induction hypothesis. It follows that $x \in k\langle g_1, \ldots, g_n \rangle$, as desired.
\end{proof}

\begin{proof}[Proof of Theorem \ref{HilbertNagata}]
Let $I$ be a left ideal in $A^H$; we want to show that it is finitely generated. Lemma \ref{lemmaFiniteGeneratorsOfIdeals} gives us elements $x_1, \ldots, x_n \in I$ such that $AI = Ax_1 + \ldots + Ax_n$. Let us show that $\mathfrak{R}(AI) = I$. The inclusion $I \subset \mathfrak{R}(AI)$ is obvious since $\mathfrak{R}(I) = I$. Conversely, take $\textstyle x = \sum_{s=1}^n a_s x_s \in AI$. Applying Lemma \ref{lemmaReynolds}, we get $\textstyle \mathfrak{R}(x) = \sum_{s=1}^n \mathfrak{R}(a_s)x_s \in I$, so that $\mathfrak{R}(AI) \subset I$. Now, using Lemma \ref{lemmaReynolds} again, we find
\[ I = \mathfrak{R}(AI) = \mathfrak{R}(Ax_1 + \ldots + Ax_n) = \mathfrak{R}(A)x_1 + \ldots + \mathfrak{R}(A)x_n = A^H x_1 + \ldots + A^H x_n \]
which means that $I$ is generated by $x_1, \ldots, x_n$. The same arguments show that right ideals of $A^H$ are finitely generated as well. Thus $A^H$ is Noetherian. To show that the algebra $A^H$ is finitely generated, note that it is graded by $\mathcal{I}$ if we define $(A^H)_i = (A_i)^H$; hence Lemma \ref{gradedNoetherianFinGen} applies.
\end{proof}

\indent We will now use Theorem \ref{HilbertNagata} to prove that the subalgebra $\mathcal{L}^{U_q}_{g,n}$ is Noetherian and finitely generated, which is our main result in this section. Recall that $\mathcal{L}_{g,n}$ is $\mathcal{O}_q(q^{1/D})^{\otimes (2g+n)}$ as a vector space. For $[\mu] = (\mu_1, \ldots, \mu_{2g+n}) \in P_+^{2g+n}$ let
\[ C([\mu]) = C(\mu_1) \otimes \ldots \otimes C(\mu_{2g+n}) \subset \mathcal{L}_{g,n} \]
where $C(\mu_i)$ is the subspace of matrix coefficients of the irreducible $U_q^{\mathrm{ad}}$-module with highest weight $\mu_i$. We say that $[\mu'] \preceq [\mu]$ if and only if $\mu'_i \preceq \mu_i$ for all $1 \leq i \leq 2g+n$, where we recall that $\mu'_i \preceq \mu_i$ means that $\mu_i - \mu'_i \in D^{-1}Q_+$. This gives an order on $P_+^{2g+n}$ which satisfies the conditions \eqref{confluentOrder} and \eqref{conditionsOrdreHilbertNagata} and is well-founded. We see from the formulas in Proposition \ref{productLgn} and \eqref{produitL01} that the subspaces
\[ \mathcal{W}^{[\mu]} = \bigoplus_{[\mu'] \preceq [\mu]} C([\mu']) \]
form a filtration $\mathcal{W}$ of the algebra $\mathcal{L}_{g,n}$ indexed by $P_+^{2g+n}$. Let
\begin{equation}\label{canoProjCmu}
\textstyle \pi_{[\mu]} : \bigoplus_{[\mu'] \preceq [\mu]} C([\mu']) \to C([\mu])
\end{equation}
be the canonical projection and put $x \,\overline{\,\cdot\,}\,y = \pi_{[\mu]+[\eta]}(xy)$ for $x \in C([\mu])$, $y \in C([\eta])$. Then $\mathrm{gr}_{\mathcal{W}}(\mathcal{L}_{g,n})$ is identified with the vector space $\mathcal{L}_{g,n}$ endowed with the product $\overline{\,\cdot\,}$.

\begin{teo}\label{thmLgnUqFinGen}
The algebra $\mathcal{L}_{g,n}^{U_q}$ is Noetherian and finitely generated.
\end{teo}
\begin{proof}
Recall the filtration $\mathcal{G}_1$ of $\mathcal{L}_{g,n}$ that we used in \S \ref{sectionNoetherianityLgn}. We see that $\mathcal{G}_1$ is a filtration of $\mathrm{gr}_{\mathcal{W}}(\mathcal{L}_{g,n})$ as well and that
\[ \mathrm{gr}_{\mathcal{G}_1}\bigl(\mathrm{gr}_{\mathcal{W}}(\mathcal{L}_{g,n})\bigr) = \mathrm{gr}_{\mathcal{G}_1}(\mathcal{L}_{g,n}). \]
In the proof of Theorem \ref{ThmLgnNoetherian} we showed that $\mathrm{gr}_{\mathcal{G}_2}\bigl( \mathrm{gr}_{\mathcal{G}_1}(\mathcal{L}_{g,n}) \bigr)$ is Noetherian and finitely generated. It follows from Lemmas \ref{lemmaFiltrationNoetherian} and \ref{lemmaFinGenGr} that $\mathrm{gr}_{\mathcal{G}_1}(\mathcal{L}_{g,n})$ also has these properties. Using again these lemmas, we obtain that $\mathrm{gr}_{\mathcal{W}}(\mathcal{L}_{g,n})$ is Noetherian and finitely generated.

\noindent Since $\mathcal{L}_{g,n}$ is a $U_q$-module-algebra and the canonical projections $\pi_{[\mu]}$ in \eqref{canoProjCmu} are $U_q$-linear for $\mathrm{coad}^r$, we see that $\mathrm{gr}_{\mathcal{W}}(\mathcal{L}_{g,n})$ is a $U_q$-module-algebra for $\mathrm{coad}^r$. It satisfies all the assumptions of Theorem \ref{HilbertNagata} (recall that $U_q\text{-}\mathrm{mod}$ is semisimple). Hence $\mathrm{gr}_{\mathcal{W}}(\mathcal{L}_{g,n})^{U_q}$ is Noetherian and finitely generated. To conclude, let
\[ \mathcal{W}^{[\mu]}_{\mathrm{inv}} = \bigl( \mathcal{W}^{[\mu]} \bigr)^{U_q} = \mathcal{W}^{[\mu]} \cap \mathcal{L}_{g,n}^{U_q} \subset \mathcal{L}_{g,n}^{U_q}. \]
Then $\mathcal{W}_{\mathrm{inv}}$ is a filtration of $\mathcal{L}_{g,n}^{U_q}$ and we have
\[ \mathrm{gr}_{\mathcal{W}}(\mathcal{L}_{g,n})^{U_q} = \mathrm{gr}_{\mathcal{W}_{\mathrm{inv}}}\bigl(\mathcal{L}_{g,n}^{U_q}\bigr). \]
The result then follows from Lemma \ref{lemmaFiltrationNoetherian} for Noetherianity and from Lemma \ref{lemmaFinGenGr} for finiteness.
\end{proof}

\section{The Alekseev morphism}\label{sectionAlekseevMorphism}
\indent Let $H$ be a quasitriangular Hopf algebra with an invertible antipode $S$. Usually the Alekseev morphism \cite{A} is expected to be a morphism of algebras
\[ \mathcal{L}_{g,n}(H) \to \mathcal{H}(H^{\circ})^{\otimes g} \otimes H^{\otimes n} \]
where $\mathcal{H}(H^{\circ})$ is the Heisenberg double (\S \ref{sectionHeisenberg}). However, without further assumptions on $H$ (e.g. finite-dimensionality, as in \cite{FaitgMCG}) we have to introduce a bigger algebra than $\mathcal{H}(H^{\circ})$ in order to make sense of the formulas given in \cite{A}; we call this algebra the two-sided Heisenberg double. We construct it in \S \ref{sectionTwoSidedHeisenberg}, and define the Alekseev morphism in \S \ref{sectionDefAlekseev}. The case of finite-dimensional $H$ is discussed in \S \ref{AlekseevFinDim}. Then in \S \ref{sectionAlekseevUq} we focus on the quantum group $H = U_q^{\mathrm{ad}}(\mathfrak{g})$ and we use the Alekseev morphism to prove that $\mathcal{L}_{g,n}\bigl( U_q^{\mathrm{ad}}(\mathfrak{g}) \bigr)$ is a domain.

\subsection{The two-sided Heisenberg double}\label{sectionTwoSidedHeisenberg} This section does not use the $R$-matrix of $H$. Recall that $H^{\mathrm{cop}}$ is $H$ with the opposite coproduct. The restricted dual $H^{\circ}$ endowed with its usual product $\star$ is a left $(H \otimes H^{\mathrm{cop}})$-module-algebra for the action
\[ (x \otimes y) \vdash \varphi = x\rhd \varphi \lhd S^{-1}(y). \]
As a result we can make the following definition:
\begin{defi}\label{defTwoSidedHeisenberg}
The two-sided Heisenberg double is the smash product $H^{\circ} \# (H \otimes H^{\mathrm{cop}})$. It is an associative algebra which we denote by $\mathcal{HH}(H^{\circ})$.
\end{defi}

By the general definition of a smash product (see e.g. \cite[Def. 4.1.3]{Mon}), $\mathcal{HH}(H^{\circ})$ is the vector space $H^{\circ} \otimes (H \otimes H^{\mathrm{cop}})$ endowed with the multiplication
\begin{equation}\label{produitSurHH}
\begin{array}{rl}
\bigl( \varphi \,\#\, (x \otimes y) \bigr) \bigl( \psi \,\#\, (z \otimes t) \bigr) \!\!&= \sum_{(x \,\otimes\, y)} \varphi \star \bigl( (x \otimes y)_{(1)} \vdash \psi \bigr) \,\#\, (x \otimes y)_{(2)}(z \otimes t)\\[1.5em]
&= \sum_{(x),(y)} \varphi \star \bigl(x_{(1)} \rhd \psi \lhd S^{-1}(y_{(2)})\bigr) \,\#\, (x_{(2)}z \otimes y_{(1)}t).
\end{array}
\end{equation}
where we write $\varphi \,\#\, (x \otimes y)$ for the element $\varphi \otimes (x \otimes y) \in H^{\circ} \otimes (H \otimes H)$. From this formula we see that $H^{\circ} \#\, (1 \otimes 1)$, $\varepsilon \,\#\, (H \otimes 1)$ and $\varepsilon \,\#\, (1 \otimes H)$ are subalgebras of $\mathcal{HH}(H^{\circ})$ which are canonically isomorphic to $H^{\circ}$, $H$ and $H$ respectively. Moreover $H^{\circ} \# (H \otimes 1)$ is a subalgebra of $\mathcal{HH}(H^{\circ})$ which is canonically isomorphic to the Heisenberg double $\mathcal{H}(H^{\circ})$ from \S\ref{sectionHeisenberg}.

\smallskip

\indent We note that $\mathcal{HH}(H^{\circ})$ has a natural representation:

\begin{prop}\label{propRepresentationTwoSidedHeisenberg}
There is a representation $\blacktriangleright$ of $\mathcal{HH}(H^{\circ})$ on $H^{\circ}$ given by
\[ \varphi \,\#\, (x \otimes y) \blacktriangleright \psi = \varphi \star \bigl( x \rhd \psi \lhd S^{-1}(y) \bigr). \]
\end{prop}
\begin{proof}
This is actually a general fact: for any left $H$-module-algebra $A$ there is a representation $\blacktriangleright$ of $A \,\#\, H$ on $A$ defined by $(a \,\#\, h) \blacktriangleright x = a(h \cdot x)$. The proof of this claim is straightforward.
\end{proof}

\indent The representation of $\mathcal{H}(H^{\circ})$ on $H^{\circ}$ (obtained by restricting $\blacktriangleright$ to this subalgebra) is known to be faithful \cite[Lem. 9.4.2]{Mon}. In general this is not true for the representation of $\mathcal{HH}(H^{\circ})$ on $H^{\circ}$. For instance if $H$ is finite-dimensional we have $H^{\circ} = H^*$ and $\dim\bigl( \mathcal{HH}(H^*) \bigr) = \dim(H)^3$ while $\dim\bigl( \mathrm{End}_k(H^*) \bigr) = \dim(H)^2$, which forces the representation morphism $\rho : \mathcal{HH}(H^*) \to \mathrm{End}_k(H^*)$ to have a non-zero kernel.


\smallskip

\indent Finally, let us mention another construction of $\mathcal{HH}(H^{\circ})$. We know from \S \ref{sectionHeisenberg} that $\mathcal{H}(H^{\circ})$ is a right $H$-module-algebra; it follows that $\mathcal{H}(H^{\circ})$ is a left $H^{\mathrm{cop}}$-module-algebra for the action
\[ h \cdot (\varphi \,\#\, x) = (\varphi\,\#\, x) \cdot S^{-1}(h) \]
where the $\cdot$ on the right-hand side is defined in \eqref{actionOnHeisenberg}. As a result we can consider the smash product $\mathcal{H}(H^{\circ}) \# H^{\mathrm{cop}}$ and this algebra is isomorphic to $\mathcal{HH}(H^{\circ})$:
\[ \flecheIso{\mathcal{H}(H^{\circ}) \# H^{\mathrm{cop}}}{\mathcal{HH}(H^{\circ})}{(\varphi \,\#\, x) \,\#\, y}{\sum_{(y)} \varphi \,\#\, \bigl(xy_{(2)} \otimes S(y_{(1)})\bigr)} \]

\subsection{Definition of the Alekseev morphism}\label{sectionDefAlekseev}
Recall that $H$ is a quasitriangular Hopf algebra with an invertible antipode $S$. We use the following notations in $\mathcal{HH}(H^{\circ})$:
\begin{itemize}
\item we write $\varphi$ instead of $\varphi \,\#\, (1 \otimes 1)$,
\item we write $h$ instead of $\varepsilon \,\#\, (h \otimes 1)$,
\item we write $\widetilde{h}$ instead of $\varepsilon \,\#\, (1 \otimes h)$,
\end{itemize}
where $\varphi \in H^{\circ}$ and $h \in H$. The multiplication in $\mathcal{HH}(H^{\circ})$ is then described by
\begin{equation}\label{commutationsHH}
h \varphi = \sum_{(h)} (h_{(1)} \rhd \varphi) \, h_{(2)}, \qquad \widetilde{h} \varphi = \sum_{(h)} \bigl(\varphi \lhd S^{-1}(h_{(2)})\bigr) \, \widetilde{h_{(1)}}, \qquad \widetilde{h}g = g \widetilde{h}
\end{equation}
for all $\varphi \in H^{\circ}$ and $h,g \in H$. In particular the Heisenberg double $\mathcal{H}(H^{\circ})$ is identified with the subalgebra of elements of the form $\textstyle \sum_i \varphi_i \, h_i$. We note also that by the very definition $\widetilde{gh} = \widetilde{g}\,\widetilde{h}$. Finally, we endow $\mathcal{HH}(H^{\circ})^{\otimes g} \otimes H^{\otimes n}$ with the usual multiplication on a tensor product of algebras.

\medskip

We will use the embeddings $j_k : \mathcal{L}_{1,0}(H) \to \mathcal{L}_{g,n}(H)$ for $1 \leq k \leq g$ and $j_{g+l} : \mathcal{L}_{0,1}(H) \to \mathcal{L}_{g,n}(H)$ for $1 \leq l \leq n$ from \eqref{embeddingsL01L10InLgn} and the morphisms $\Phi_{0,1} : \mathcal{L}_{0,1}(H) \to H$ and $\Phi_{1,0} : \mathcal{L}_{1,0}(H) \to \mathcal{H}(H^{\circ}) \subset \mathcal{HH}(H^{\circ})$ from \eqref{RSDmap} and Proposition \ref{propPhi10}. Let
\[ \fonc{\mathsf{D}_{g,n}}{H}{\mathcal{HH}(H^{\circ})^{\otimes g} \otimes H^{\otimes n}}{h}{\sum_ {(h)} \widetilde{h_{(1)}}\,h_{(2)} \otimes \ldots \otimes \widetilde{h_{(2g-1)}}\,h_{(2g)} \otimes h_ {(2g+1)} \otimes \ldots \otimes h_{(2g+n)}} \]
which is a morphism of algebras. We use the convention $\mathsf{D}_{0,0}(h) = \varepsilon(h)$. Note for further use that
\begin{equation}\label{inductionGamma}
\mathsf{D}_{0,n+1}(h) = \sum_{(h)} h_{(1)} \otimes \mathsf{D}_{0,n}(h_{(2)}), \qquad \mathsf{D}_{g+1,n}(h) = \sum_{(h)} \widetilde{h_{(1)}} \, h_{(2)} \otimes \mathsf{D}_{g,n}(h_{(3)}).
\end{equation}
\begin{teo}\label{thmMorphismeAlekseev}
There is a morphism of algebras, called \emph{Alekseev morphism},
\[ \Phi_{g,n} : \mathcal{L}_{g,n}(H) \to \mathcal{HH}(H^{\circ})^{\otimes g} \otimes H^{\otimes n} \]
defined by
\begin{align*}
&\Phi_{g,n}\bigl(j_{g+l}(\varphi)\bigr) = \sum_{(R)} 1^{\otimes g} \, \otimes 1^{\otimes (l-1)} \otimes \Phi_{0,1}\!\left( \mathrm{coad}^r\bigl(R_{(1)}\bigr)(\varphi)\right) \otimes \mathsf{D}_{0,n-l}\bigl(R_{(2)}\bigr)\\
&\Phi_{g,n}\bigl( j_k(x) \bigr) = \sum_{(R)} 1^{\otimes (k-1)} \otimes \Phi_{1,0} \!\left( \mathrm{coad}^r\bigl( R_{(1)} \bigr)(x) \right) \otimes \mathsf{D}_{g-k,n}\bigl(R_{(2)}\bigr)
\end{align*}
for any $\varphi \in \mathcal{L}_{0,1}(H)$, $1 \leq l \leq n$, $x \in \mathcal{L}_{1,0}(H)$, $1 \leq k \leq g$, and where $\mathrm{coad}^r$ is defined in \eqref{coadL01} and \eqref{coadL10}.
\end{teo}
\noindent By \eqref{monomialsLgn} the value of $\Phi_{g,n}$ on general elements is
\begin{align*}
&\Phi_{g,n}\bigl( x_1 \, \widetilde{\otimes} \, \ldots \, \widetilde{\otimes} \, x_g \, \widetilde{\otimes} \, \varphi_1 \, \widetilde{\otimes} \, \ldots \, \widetilde{\otimes} \, \varphi_n \bigr)\\
=\:& \Phi_{g,n}\bigl( j_1(x_1) \bigr) \ldots \Phi_{g,n}\bigl( j_g(x_g) \bigr) \Phi_{g,n}\bigl( j_{g+1}(\varphi_1) \bigr) \ldots \Phi_{g,n}\bigl( j_{g+n}(\varphi_n) \bigr).
\end{align*}
Note that for $g>0$ the morphism $\Phi_{g,n}$ actually takes values in $\mathcal{H}(H^{\circ}) \otimes \mathcal{HH}(H^{\circ})^{\otimes (g-1)} \otimes H^{\otimes n}$ where $\mathcal{H}(H^{\circ})$ is identified with the subalgebra $H^{\circ} \,\#\, (H \otimes 1)$ in $\mathcal{HH}(H^{\circ})$. This remark will be important in the proof of Theorem \ref{TheoremePhignInjectif} in \S\ref{sectionAlekseevUq} below.

\smallskip

\indent The following lemma is the key point for the proof of Theorem \ref{thmMorphismeAlekseev}:
\begin{lem}\label{lemmaCommutationGammaPhi}
For all $h \in H$ and $x \in \mathcal{L}_{g,n}(H)$ we have
\[ \mathsf{D}_{g,n}(h) \, \Phi_{g,n}(x) = \sum_ {(h)}\Phi_{g,n}\!\left( \mathrm{coad}^r\bigl( S^{-1}(h_ {(2)}) \bigr)(x) \right) \mathsf{D}_{g,n}(h_ {(1)}) .\]
\end{lem}
\begin{proof}
The proof is by induction. First we prove the result for $\Phi_{0,n}$ by induction on $n$. Then we fix $n$ and we prove the result for $\Phi_{g,n}$ by induction on $g$. So let us consider the case $g=0$. Thanks to the relation $R\Delta = \Delta^{\mathrm{op}}R$ we get for all $h \in H$ and $\varphi \in \mathcal{L}_{0,1}(H)$
\begin{equation}\label{commutationHPhi01}
\begin{array}{rl}
h\,\Phi_{0,1}(\varphi) \!\!\!&= \sum_{(R^1),(R^2)} \varphi\bigl( R^1_{(1)}R^2_{(2)} \bigr) \, hR^1_{(2)} R^2_{(1)}\\[1.4em]
& = \sum_{(R^1),(R^2),(h)} \varphi\!\left( S^{-1}(h_{(3)})h_{(2)}R^1_{(1)}R^2_{(2)} \right) \, h_{(1)}R^1_{(2)} R^2_{(1)}\\[1.4em]
& = \sum_{(R^1),(R^2),(h)} \varphi\!\left( S^{-1}(h_{(3)})R^1_{(1)}R^2_{(2)} h_{(2)}\right) \, R^1_{(2)} R^2_{(1)} h_{(1)}\\[1.4em]
& = \sum_{(h)}\Phi_{0,1}\!\left( \mathrm{coad}^r\bigl( S^{-1}(h_{(2)})\bigr)(\varphi) \right) h_{(1)}
\end{array}
\end{equation}
which is the desired formula for $(g,n) = (0,1)$. Now assume that the formula is true for some $n \geq 1$ (still with $g=0$) and let $\varphi \,\widetilde{\otimes}\, x \in \mathcal{L}_{0,n+1}(H) = \mathcal{L}_{0,1}(H) \,\widetilde{\otimes}\, \mathcal{L}_{0,n}(H)$. We have in $H^{\otimes (n+1)}$
\begin{align*}
&\mathsf{D}_{0,n+1}(h) \Phi_{0,n+1}(\varphi \,\widetilde{\otimes}\, x) = \sum_{(R),(h)} h_{(1)}\Phi_{0,1}\!\left( \mathrm{coad}^r\bigl(R_{(1)}\bigr)(\varphi)\right) \otimes \mathsf{D}_{0,n}\bigl(h_{(2)}R_{(2)}\bigr)\Phi_{0,n}(x)\\
=\:& \sum_{(R),(h)} \Phi_{0,1}\!\left( \mathrm{coad}^r\bigl(R_{(1)}S^{-1}(h_ {(2)})\bigr)(\varphi)\right)h_ {(1)} \otimes \mathsf{D}_{0,n}\bigl(h_{(3)}R_{(2)}\bigr)\Phi_{0,n}(x)\\
=\:& \sum_{(R),(h)} \Phi_{0,1}\!\left( \mathrm{coad}^r\bigl(S^{-1}(h_ {(3)})R_{(1)} \bigr)(\varphi)\right)h_ {(1)} \otimes \mathsf{D}_{0,n}\bigl(R_{(2)}h_{(2)} \bigr)\Phi_{0,n}(x)\\
=\:& \sum_{(R),(h)} \Phi_{0,1}\!\left( \mathrm{coad}^r\bigl(S^{-1}(h_ {(4)})R_{(1)} \bigr)(\varphi)\right)h_ {(1)} \otimes \mathsf{D}_{0,n}\bigl(R_{(2)}\bigr)\Phi_{0,n}\bigl(\mathrm{coad}^r\bigl( S^{-1}(h_ {(3)}) \bigr)(x) \bigr)\mathsf{D}_{0,n}\bigl(h_{(2)} \bigr)\\
=\:& \sum_{(h)} \Phi_{0,n+1}\bigl( \mathrm{coad}^r\bigl( S^{-1}(h_{(2)}) \bigr)(\varphi \,\widetilde{\otimes}\, x) \bigr) \mathsf{D}_{0,n+1}(h_{(1)}).
\end{align*}
For the first equality we used \eqref{inductionGamma}, for the second we used \eqref{commutationHPhi01}, for the third we used that $(S \otimes \mathrm{id})(R) = R^{-1}$ together with $R^{-1}\Delta^{\mathrm{op}} = \Delta R^{-1}$, for the fourth we used the induction hypothesis and for the fifth we used \eqref{inductionGamma}, the definition of $\Phi_{0,n+1}$ and the definition of $\mathrm{coad}^r$ in \eqref{coadLgn}. Now we treat the case $g > 0$. Observe first that for all $h \in H$ and $\beta \otimes \alpha \in \mathcal{L}_{1,0}(H)$, we have in $\mathcal{HH}(H^{\circ})$:
\begin{align*}
&\mathsf{D}_{1,0}(h) \Phi_{1,0}(\beta \otimes \alpha) = \sum_{(R^1),(R^2),(R^3),(h)} \widetilde{h_{(1)}}h_{(2)} \left( R^1_{(2)}R^2_{(2)} \rhd \beta \lhd R^3_{(1)}R^1_{(1)} \right)R^3_{(2)}R^2_{(1)}\Phi_{0,1}(\alpha)\\
=\:&\sum_{(R^1),(R^2),(R^3),(h)} \left( h_ {(3)}R^1_{(2)}R^2_{(2)} \rhd \beta \lhd R^3_{(1)}R^1_{(1)}S^{-1}(h_{(2)}) \right) h_{(4)}R^3_{(2)}R^2_{(1)} \Phi_{0,1}(\alpha) \widetilde{h_{(1)}}\\
=\:&\sum_{(R^1),(R^2),(R^3),(h)} \left( R^1_{(2)}R^2_{(2)}h_{(3)} \rhd \beta \lhd S^{-1}(h_{(4)})R^3_{(1)}R^1_{(1)} \right) R^3_{(2)}R^2_{(1)} h_{(2)} \Phi_{0,1}(\alpha)\widetilde{h_{(1)}}\\
=\:&\sum_{(h)} \Phi_{1,0}\bigl( \mathrm{coad}^r\bigl( S^{-1}(h_ {(3)}) \bigr)(\beta) \otimes \varepsilon \bigr) \Phi_{0,1}\bigl( \mathrm{coad}^r\bigl( S^{-1}(h_ {(2)})\bigr)(\alpha) \bigr) \mathsf{D}_{1,0}(h_{(1)})\\
=\:&\sum_{(h)} \Phi_{1,0}\left( \mathrm{coad}^r\bigl(S^{-1}(h_{(2)})\bigr)(\beta \otimes \alpha)\right) \mathsf{D}_{1,0}(h_{(1)})
\end{align*}
which is the formula for the case $(g,n)=(1,0)$. The first equality is the definition of $\Phi_{1,0}$ (Proposition \ref{propPhi10}), the second equality uses \eqref{commutationsHH}, the third equality uses $R\Delta = \Delta^{\mathrm{op}}R$ three times, the fourth equality uses \eqref{commutationHPhi01}, the definition of $\Phi_{1,0}$ and the definition of $\mathsf{D}_{1,0}$ and the fifth equality uses the definition of $\Phi_{1,0}$ and \eqref{coadL10}. For the general case, fix $n$ and make an induction on $g$; this is completely similar to the computation for $\Phi_{0,n+1}$ above, using \eqref{inductionGamma} and the result for $\Phi_{1,0}$ just proved.
\end{proof}

\begin{proof}[Proof of Theorem \ref{thmMorphismeAlekseev}]
We will obtain the result by the same type of induction as in the proof of Lemma \ref{lemmaCommutationGammaPhi}, \textit{i.e.} first one proves the result for $\Phi_{0,n}$ by induction on $n$ and then one fixes $n$ and proves the result for $\Phi_{g,n}$ by induction on $g$. Let us do for instance the second part, the first part being completely similar and left to the reader. So assume that $\Phi_{g,n}$ is a morphism of algebras and let $x \,\widetilde{\otimes}\, u, y \,\widetilde{\otimes}\, v \in \mathcal{L}_{g+1,n}(H) = \mathcal{L}_{1,0}(H) \,\widetilde{\otimes}\, \mathcal{L}_{g,n}(H)$. We have
\begin{align*}
&\Phi_{g+1,n}\bigl( (x \,\widetilde{\otimes}\, u)(y \,\widetilde{\otimes}\, v) \bigr) = \sum_{(R)} \Phi_{g+1,n}\bigl( x\,\mathrm{coad}^r(R_ {(1)})(y) \,\widetilde{\otimes}\, \mathrm{coad}^r(R_ {(2)})(u)\,v \bigr)\\
=\:& \sum_{(R^1), (R^2)} \Phi_{1,0}\bigl( \mathrm{coad}^r(R^2_{(1)})\bigl( x\,\mathrm{coad}^r(R^1_{(1)})(y)\bigr) \bigr) \otimes \mathsf{D}_{g,n}(R^2_{(2)})\,\Phi_{g,n}\bigl( \mathrm{coad}^r(R^1_{(2)})(u)\,v \bigr)\\
=\:& \sum_{(R^1), (R^2), (R^3)} \Phi_{1,0}\bigl( \mathrm{coad}^r(R^2_{(1)})(x) \bigr) \,\Phi_{1,0}\bigl( \mathrm{coad}^r(R^1_{(1)}R^3_{(1)})(y) \bigr)\\[-1.2em]
&\hspace{17.7em} \otimes \mathsf{D}_{g,n}(R^2_{(2)}R^3_{(2)})\,\Phi_{g,n}\bigl( \mathrm{coad}^r(R^1_{(2)})(u)\bigr)\,\Phi_{g,n}(v)\\[.5em]
=\:& \sum_{\substack{(R^1), (R^2),\\(R^3), (R^4)}} \Phi_{1,0}\bigl( \mathrm{coad}^r(R^2_{(1)})(x) \bigr)\, \Phi_{1,0}\bigl( \mathrm{coad}^r(R^1_{(1)}R^3_{(1)}R^4_{(1)})(y) \bigr)\\[-1.9em]
&\hspace{11em} \otimes \mathsf{D}_{g,n}(R^2_{(2)}) \,\Phi_{g,n}\bigl( \mathrm{coad}^r\bigl(R^1_{(2)}S^{-1}(R^3_{(2)})\bigr)(u) \bigr) \, \mathsf{D}_{g,n}(R^4_{(2)})\,\Phi_{g,n}(v)\\[.5em]
=\:& \sum_{(R^2), (R^4)} \Phi_{1,0}\bigl( \mathrm{coad}^r(R^2_{(1)})(x) \bigr)\, \Phi_{1,0}\bigl( \mathrm{coad}^r(R^4_{(1)})(y) \bigr) \otimes \mathsf{D}_{g,n}(R^2_{(2)}) \,\Phi_{g,n}(u) \, \mathsf{D}_{g,n}(R^4_{(2)})\,\Phi_{g,n}(v)\\
=\:& \Phi_{g+1,n}(x \,\widetilde{\otimes}\, u)\,\Phi_{g+1,n}(y \,\widetilde{\otimes}\, v).
\end{align*}
The first equality uses \eqref{multiplicationInBraidedTensorProduct}, the second uses the definition of $\Phi_{g+1,n}$, the third uses that $\mathcal{L}_{1,0}(H)$ is a module-algebra for $\mathrm{coad}^r$ together with the formula $(\Delta \otimes \mathrm{id})(R) = R_{13}R_ {23}$ and then uses that $\Phi_{1,0}$ and $\Phi_{g,n}$ are morphisms of algebras (by Proposition \ref{propPhi10} and the induction hypothesis), the fourth uses Lemma \ref{lemmaCommutationGammaPhi} and the formula $(\mathrm{id} \otimes \Delta)(R)= R_{13}R_{12}$, the fifth uses $(\mathrm{id} \otimes S^{-1})(R) = R^{-1}$ and the sixth uses the definition of $\Phi_{g+1,n}$.
\end{proof}

\indent For completeness let us provide an alternative definition of $\Phi_{g,n}$ based on the matrices $\overset{V}{B}(i)$, $\overset{V}{A}(i)$, $\overset{V}{M}(j)$ from \eqref{matricesABMgn}, which was first introduced in \cite{A}. Let $R^{(+)} = R$ and $\textstyle R^{(-)} = (R')^{-1} = \sum_{(R)} R_{(2)} \otimes S(R_{(1)})$ and write as usual $\textstyle R^{(\pm)} = \sum_{(R^{(\pm)})} R^{(\pm)}_{(1)} \otimes R^{(\pm)}_{(2)}$. Note that $R^{(-)}$ is also an $R$-matrix. For a finite-dimensional $H$-module $V$ we define
\begin{align*}
&\overset{V}{L}{^{(\pm)}} = \sum_{(R^{(\pm)})} R^{(\pm)}_{(2)} \otimes \rho_V\!\left(R^{(\pm)}_{(1)}\right) \in H \otimes \mathrm{End}_k(V)\\
&\overset{V}{\widetilde{L}}{^{(\pm)}} = \sum_{(R^{(\pm)})} \widetilde{R^{(\pm)}_{(2)}} \otimes \rho_V\!\left(R^{(\pm)}_{(1)}\right) \in \mathcal{HH}(H^{\circ}) \otimes \mathrm{End}_k(V)\\
&\overset{V}{T} = \sum_{i,j} {_V\phi^{e^i}_{e_j}} \otimes E_{ij} \in H^{\circ} \otimes \mathrm{End}(V)
\end{align*}
where $\rho_V : H \to \mathrm{End}_k(V)$ is the representation morphism, ${_V\phi^{e^i}_{e_j}} : h \mapsto e^i(h \cdot e_j)$ are the matrix coefficients of $V$ in some basis $(e_i)$ and $E_{ij}(e_k) = \delta_{jk}e_i$. These elements can be seen as matrices of size $\dim(V)$ with coefficients in $\mathcal{HH}(H^{\circ})$. They satisfy various exchange relations; for instance it is an exercise to show that \eqref{commutationsHH} implies
\[ \overset{V}{L}{^{(\pm)}_1} \overset{W}{T}_2 = \overset{W}{T}_2 \, \overset{V}{L}{^{(\pm)}_1} \, R^{(\pm)}_{VW}, \quad R^{(\pm)}_{VW} \, \overset{V}{\widetilde{L}}{^{(\pm)}_1} \, \overset{W}{T}_2 = \overset{W}{T}_2 \, \overset{V}{\widetilde{L}}{^{(\pm)}_1}, \quad \overset{V}{L}{^{(\pm)}_1} \, \overset{W}{\widetilde{L}}{^{(\pm)}_2} = \overset{W}{\widetilde{L}}{^{(\pm)}_2}\, \overset{V}{L}{^{(\pm)}_1} \]
where we use the notations introduced before Proposition \ref{presentationL10}. Let $\overset{V}{M}$ (resp. $\overset{V}{A}$, $\overset{V}{B}$) be the matrix with coefficients in $\mathcal{L}_{0,1}(H)$ (resp. in $\mathcal{L}_{1,0}(H)$) introduced before Proposition \ref{presentationL10}. We have
\begin{equation}\label{matrixFormPhi01Phi10}
\Phi_{0,1}\bigl( \overset{V}{M} \bigr) = \overset{V}{L}{^{(+)}} \overset{V}{L}{^{(-)-1}}, \quad \Phi_{1,0}\bigl( \overset{V}{A} \bigr) = \overset{V}{L}{^{(+)}} \overset{V}{L}{^{(-)-1}}, \quad \Phi_{1,0}\bigl( \overset{V}{B} \bigr) = \overset{V}{L}{^{(+)}} \, \overset{V}{T} \, \overset{V}{L}{^{(-)-1}}
\end{equation}
where $\Phi_{0,1}\bigl( \overset{V}{M} \bigr)$ is viewed as a matrix with coefficients in $H$ while $\Phi_{1,0}\bigl( \overset{V}{A} \bigr)$ and $\Phi_{1,0}\bigl( \overset{V}{B} \bigr)$ are viewed as matrices with coefficients in $\mathcal{H}(H^{\circ})$. Finally for $1 \leq k \leq g$ let $J_k :\mathcal{HH}(H^{\circ}) \to \mathcal{HH}(H^{\circ})^{\otimes g} \otimes H^{\otimes n}$ and for $g+1 \leq l \leq g+n$ let $J_{g+l} : H \to \mathcal{HH}(H^{\circ})^{\otimes g} \otimes H^{\otimes n}$ be the obvious embeddings. We define
\begin{align*}
&\overset{V}{L}{^{(\pm)}}(i) = (J_i \otimes \mathrm{id})\bigl( \overset{V}{L}{^{(\pm)}}  \bigr) \quad \text{for } 1 \leq i \leq g+n,\\
&\overset{V}{\widetilde{L}}{^{(\pm)}}(k) = (J_k \otimes \mathrm{id})\bigl( \overset{V}{\widetilde{L}}{^{(\pm)}}  \bigr), \quad \overset{V}{T}(k) = (J_k \otimes \mathrm{id})\bigl( \overset{V}{T}  \bigr) \quad \text{for } 1 \leq k \leq g.
\end{align*}
A straightforward computation based on the axiom $(\mathrm{id} \otimes \Delta)(R) = R_{13}R_{12}$ reveals that
\begin{equation}\label{matrixFormD}
\sum_{(R)} \left(1^{\otimes g} \, \otimes 1^{\otimes l} \otimes \mathsf{D}_{0,n-l}\bigl(R_{(2)}\bigr) \right) \otimes \rho_V\bigl(R_ {(1)}\bigr) = \overset{V}{L}{^{(+)}}(g+n) \ldots \overset{V}{L}{^{(+)}}(g+l+1)
\end{equation}
and similarly
\begin{equation}\label{matrixFormDBis}
\begin{array}{l}
\sum_{(R)} \left(1^{\otimes k} \otimes \mathsf{D}_{g-k,n}\bigl(R_{(2)}\bigr)\right) \otimes \rho_V\bigl(R_ {(1)}\bigr)\\
= \overset{V}{L}{^{(+)}}(g+n) \ldots \overset{V}{L}{^{(+)}}(g+1) \, \overset{V}{L}{^{(+)}}(g) \overset{V}{\widetilde{L}}{^{(+)}}(g) \ldots \overset{V}{L}{^{(+)}}(k+1) \overset{V}{\widetilde{L}}{^{(+)}}(k+1).
\end{array}
\end{equation}
Let us denote the matrices in \eqref{matrixFormD} and \eqref{matrixFormDBis} by $\Lambda_{0,n-l}$ and $\Lambda_{g-k,n}$ respectively, with the convention $\Lambda_{0,0} = 1$. A computation which combines \eqref{matrixFormPhi01Phi10} with  \eqref{matrixFormD} and \eqref{matrixFormDBis} yields
\begin{equation}\label{matrixFormAlekseev}
\begin{array}{l}
\Phi_{g,n}\bigl( \overset{V}{M}(g+l) \bigr) = \Lambda_{0, n-l} \, \overset{V}{L}{^{(+)}}(g+l) \, \overset{V}{L}{^{(-)}}(g+l)^{-1} \Lambda_{0, n-l}^{-1},\\
\Phi_{g,n}\bigl( \overset{V}{A}(k) \bigr) = \Lambda_{g-k,n} \, \overset{V}{L}{^{(+)}}(k) \overset{V}{L}{^{(-)}}(k)^{-1} \, \Lambda_{g-k, n}^{-1},\\
\Phi_{g,n}\bigl( \overset{V}{B}(k) \bigr) = \Lambda_{g-k,n} \, \overset{V}{L}{^{(+)}}(k) \, \overset{V}{T}(k)  \,\overset{V}{L}{^{(-)}}(k)^{-1} \, \Lambda_{g-k, n}^{-1}
\end{array}
\end{equation}
for $1 \leq l \leq n$ and $1 \leq k \leq g$. This is the matrix definition of $\Phi_{g,n}$.

\begin{remark}
Another possible definition of $\Phi_{g,n}$ uses $\overset{V}{L}{^{(-)}}$, $\overset{V}{\widetilde{L}}{^{(-)}}$ instead of $\overset{V}{L}{^{(+)}}$, $\overset{V}{\widetilde{L}}{^{(+)}}$ to define different conjugation matrices $\Lambda$, where the conjugation starts from the first tensorand (here it starts from the last tensorand). See e.g. \cite[\S3.3]{FaitgMCG}. For $g=0$ our definition here agrees with  the one in \cite[\S6.2]{BR1}, as can be seen from the first equality in \eqref{matrixFormAlekseev}.
\end{remark}

\subsection{Finite-dimensional case}\label{AlekseevFinDim}
\indent Here we quickly discuss the case where $H$ is finite-dimensional, which is treated in detail in \cite{FaitgMCG}. In this case one can use $\mathcal{H}(H^{\circ})$ instead of $\mathcal{HH}(H^{\circ})$.

\smallskip

Recall from \S\ref{sectionTwoSidedHeisenberg} the representation of $\mathcal{H}(H^*)$ on $H^*$, which is faithful. Since $\dim\bigl( \mathcal{H}(H^*) \bigr) = \dim\bigl( \mathrm{End}_k(H^*) \bigr)$ the representation morphism $\rho : \mathcal{H}(H^*) \to \mathrm{End}_k(H^*)$ is an isomorphism. For $h \in H$ let $Q_h \in \mathcal{H}(H^*)$ be defined by $\rho(Q_h)(\psi) = \psi \lhd S^{-1}(h)$ for all $\psi \in H^*$ (be aware that $Q_h$ was denoted by $\widetilde{h}$ in \cite{FaitgMCG}). We have $Q_{hg} = Q_h\,Q_g$ so that $\{ Q_h \,|\, h \in H\}$ is a subalgebra of $\mathcal{H}(H^*)$ isomorphic to $H$. These elements obey the following commutation rules in $\mathcal{H}(H^*)$:
\begin{equation}\label{commutationQHeisenberg}
Q_h \, \varphi = \sum_{(h)} \bigl(\varphi \lhd S^{-1}(h_{(2)})\bigr) \, Q_{h_{(1)}}, \qquad Q_h \, g = g \, Q_h
\end{equation}
for all $\varphi \in H^*$ and $h,g \in H$.
Let
\[ \fonc{\mathsf{D}^{\mathrm{fin}}_{g,n}}{H}{\mathcal{H}(H^*)^{\otimes g} \otimes H^{\otimes n}}{h}{\sum_ {(h)} Q_ {h_{(1)}}\,h_{(2)} \otimes \ldots \otimes Q_{h_{(2g-1)}}\,h_{(2g)} \otimes h_ {(2g+1)} \otimes \ldots \otimes h_{(2g+n)}}. \]
which is a morphism of algebras and define
\[ \Phi_{g,n}^{\mathrm{fin}} : \mathcal{L}_{g,n}(H) \to \mathcal{H}(H^*)^{\otimes g} \otimes H^{\otimes n} \]
by the formulas of Theorem \ref{thmMorphismeAlekseev} but with $\mathsf{D}^{\mathrm{fin}}$ instead of $\mathsf{D}$. Thanks to \eqref{commutationQHeisenberg}, the proof of Theorem \ref{thmMorphismeAlekseev} remains valid with the elements $Q_h$ instead of $\widetilde{h}$ and $\Phi_{g,n}^{\mathrm{fin}}$ is a morphism of algebras. 

\smallskip

\indent Hence we have two possible different morphisms in the finite-dimensional case: $\Phi_{g,n}$ and $\Phi_{g,n}^{\mathrm{fin}}$. They are related as follows. Let $r_{V_i} : H \to \mathrm{End}_k(V_i)$ be any representations of $H$ for $1 \leq i \leq n$ and let $\rho : \mathcal{HH}(H^*) \to \mathrm{End}_k(H^*)$ be the representation from Proposition \ref{propRepresentationTwoSidedHeisenberg}, which by restriction gives the isomorphism $\mathcal{H}(H^*) \cong \mathrm{End}_k(H^*)$. Then the two representations of $\mathcal{L}_{g,n}(H)$
\begin{align*}
&\bigl( \rho^{\otimes g} \otimes r_{V_1} \otimes \ldots \otimes r_{V_n} \bigr) \circ \Phi_{g,n}^{\mathrm{fin}} : \mathcal{L}_{g,n}(H) \to \mathrm{End}_k\bigl( (H^*)^{\otimes g} \otimes V_1 \otimes \ldots \otimes V_n \bigr),\\
&\bigl( \rho^{\otimes g} \otimes r_{V_1} \otimes \ldots \otimes r_{V_n} \bigr) \circ \Phi_{g,n} : \mathcal{L}_{g,n}(H) \to \mathrm{End}_k\bigl( (H^*)^{\otimes g} \otimes V_1 \otimes \ldots \otimes V_n \bigr)
\end{align*}
are {\em equal}, simply because $\rho(Q_h) = \rho(\widetilde{h})$. This property is satisfying, as one of the main applications of $\Phi_{g,n}$ is to construct representations of $\mathcal{L}_{g,n}(H)$.

\smallskip

\indent When $H$ is infinite dimensional we cannot define the elements $Q_h \in \mathcal{H}(H^{\circ})$, because we do not know if the representation morphism $\rho : \mathcal{H}(H^*) \to \mathrm{End}_k(H^{\circ})$ is surjective. This is why we introduced the bigger algebra $\mathcal{HH}(H^{\circ})$ in which the elements $\widetilde{h}$ exist by construction for all $h$ and are a substitute for $Q_h$.

\subsection{The case $H = U_q^{\mathrm{ad}}(\mathfrak{g})$}\label{sectionAlekseevUq}
\indent The left and right coregular actions of $U_q$ on $\mathcal{O}_q$ allow us to define the two-sided Heisenberg double $\mathcal{H}_q = \mathcal{H}\mathcal{H}_q(\mathfrak{g})$ as the smash product
\[ \mathcal{H}\mathcal{H}_q = \mathcal{O}_q \,\#\, \bigl( U_q \otimes U_q^{\mathrm{cop}}\bigr). \]
where the scalars are implicitly extended to $\mathbb{C}(q^{1/D})$. It is the $\mathbb{C}(q^{1/D})$-vector space $\mathcal{O}_q \otimes \bigl( U_q \otimes U_q\bigr)$ endowed with the product \eqref{produitSurHH}.

\smallskip

\indent Let $G$ be the connected, complex, semisimple algebraic group with Lie algebra $\mathfrak{g}$. Recall from \S\ref{sectionPreliminaires} that $\mathcal{O}_q = \mathcal{O}_q(G)$ is the subalgebra of $U_q^{\circ}$ generated over $\mathbb{C}(q)$ by the family of all matrix coefficients of the irreducible $U_q^{\mathrm{ad}}(\mathfrak{g})$-modules of type $1$.

\begin{lem}\label{lemOqgIntegre}
For all $k \geq 1$ the algebra $\mathcal{O}_q(G)^{\otimes k}$ is a domain.
\end{lem}
\begin{proof}
The direct product $G^{\times k}$ is a connected, complex, semisimple algebraic group with Lie algebra $\mathfrak{g}^{\oplus k}$. By definition, $\mathcal{O}_q(G^{\times k})$ is generated over $\mathbb{C}(q)$ by the family of all matrix coefficients of the irreducible $U_q^{\mathrm{ad}}(\mathfrak{g}^{\oplus k})$-modules of type $1$. Let us note that $U_q^{\mathrm{ad}}(\mathfrak{g}^{\oplus k}) \cong U_q^{\mathrm{ad}}(\mathfrak{g})^{\otimes k}$ as Hopf algebras. It is a general fact that given two algebras $A$ and $B$ the irreducible $(A \otimes B)$-modules are all of the form $V \otimes W$ where $V$ is an irreducible $A$-module and $W$ is an irreducible $B$-module, see e.g. \cite[Th. 3.10.2]{introRepTheory}. Hence we have an isomorphism of algebras
\[ \flecheIso{\mathcal{O}_q(G^{\times k})}{\mathcal{O}_q(G)^{\otimes k}}{_{V_1 \otimes \ldots \otimes V_k}\phi^{i_1, \ldots, i_k}_{j_1, \ldots, j_k}}{_{V_1}\phi^{i_1}_{j_1} \otimes \ldots \otimes {_{V_k}\phi^{i_k}_{j_k}}} \]
where each $V_l$ is an irreducible $U_q^{\mathrm{ad}}(\mathfrak{g})$-module of type $1$, $(v_{l,i})$ is a basis of $V_l$ with dual basis $(v^{l,j})$ and the $\phi$'s are the matrix coefficients in these bases. By \cite[Th. I.8.9]{BG} the algebra $\mathcal{O}_q(G^{\times k})$ is a domain, which concludes the proof.\end{proof}

\begin{prop}\label{propHHqUqIntegre}
The algebra $\mathcal{HH}_q^{\otimes g} \otimes U_q^{\otimes n}$ is a domain.
\end{prop}
\begin{proof}
We first introduce a filtration $\mathcal{G}$ on $\mathcal{HH}_q$ which straightforwardly generalizes the filtration of $\mathcal{H}_q$ used in the proof of Proposition \ref{HqSansDivDeZero}. It is defined by 
\[ \mathcal{G}^{\mathbf{m}, \mathbf{n}} = \mathcal{O}_q \,\#\, (\mathcal{F}_{\mathrm{DCK}}^{\mathbf{m}} \otimes \mathcal{F}_{\mathrm{DCK}}^{\mathbf{n}}) \]
for $(\mathbf{m}, \mathbf{n}) \in (\mathbb{N}^{2N+1})^2$. We endow the monoid $(\mathbb{N}^{2N+1})^2$ with the partial order defined by $(\mathbf{m}, \mathbf{n}) \leq (\mathbf{m}', \mathbf{n}')$ if and only if $\mathbf{m} \leq \mathbf{m}'$ and $\mathbf{n} \leq \mathbf{n}'$, where $\leq$ on $\mathbb{N}^{2N+1}$ is the lexicographic order from the right \eqref{RevLexN}. Thanks to formula \eqref{produitSurHH} for the product in $\mathcal{HH}_q$ and to Corollary \ref{coroCoproduitSurFiltrationDCK} we see that $\mathcal{G}$ is a filtration of $\mathcal{HH}_q$. We have
\[ \mathrm{gr}_{\mathcal{G}}(\mathcal{HH}_q) = \mathcal{O}_q \,\#\, \bigl( \mathrm{gr}_{\mathcal{F}_{\mathrm{DCK}}}(U_q) \otimes \mathrm{gr}_{\mathcal{F}_{\mathrm{DCK}}}(U_q) \bigr). \]
Let us explain this equality more precisely. Let $_V\phi^i_j : x \mapsto v^i(x \cdot v_j)$ be the matrix coefficients (see \S\ref{sectionHopfAlgebras}) of a finite-dimensional $U_q^{\mathrm{ad}}$-module $V$ in a basis of weight vectors $(v_i)$ with weights $(\epsilon_i)$. Write
\[ _V\phi^i_j, \quad \overline{E_{\beta_k}}, \quad \overline{F_{\beta_k}}, \quad \overline{K_{\mu}}, \quad \overline{\widetilde{E_{\beta_k}}}, \quad \overline{\widetilde{F_{\beta_k}}}, \quad \overline{\widetilde{K_{\mu}}} \]
instead of
\begin{align*}
_V\phi^i_j + \mathcal{G}^{<(\mathbf{0},\mathbf{0})}, \quad &E_{\beta_k} + \mathcal{G}^{<(d(E_{\beta_k}),\mathbf{0})}, \quad F_{\beta_k} + \mathcal{G}^{<(d(F_{\beta_k}),\mathbf{0})}, \quad K_{\mu} + \mathcal{G}^{<(\mathbf{0}, \mathbf{0})}\\
&\widetilde{E_{\beta_k}} + \mathcal{G}^{<(\mathbf{0},d(E_{\beta_k}))}, \quad \widetilde{F_{\beta_k}} + \mathcal{G}^{<(\mathbf{0},d(F_{\beta_k}))}, \quad \widetilde{K_{\mu}} + \mathcal{G}^{<(\mathbf{0}, \mathbf{0})}
\end{align*}
in $\mathrm{gr}_{\mathcal{G}}(\mathcal{HH}_q)$. Note that $\mathcal{G}^{<(\mathbf{0},\mathbf{0})}=0$. The subalgebra generated by $\overline{E_{\beta_k}}$, $\overline{F_{\beta_k}}$, $\overline{K_{\mu}}$, $\overline{\widetilde{E_{\beta_k}}}$, $\overline{\widetilde{F_{\beta_k}}}$, $\overline{\widetilde{K_{\mu}}}$ for $1 \leq k \leq N$ and $\mu \in P$ is isomorphic to $\mathrm{gr}_{\mathcal{F}_{\mathrm{DCK}}}(U_q)^{\otimes 2}$ and is thus a quasi-polynomial ring (over $\mathbb{C}(q)$). Moreover by Proposition \ref{propHtCoproduit} and \eqref{commutationsHH} we get
\begin{equation}\label{commutationGHHq}
\begin{array}{lll}
\overline{E_{\beta_k}} \, {_V\phi^i_j} = {_V\phi^i_j} \, \overline{E_{\beta_k}}, & \overline{F_{\beta_k}} \, {_V\phi^i_j} = q^{-(\beta_k, \epsilon_j)} {_V\phi^i_j} \, \overline{F_{\beta_k}}, & \overline{K_{\mu}} \, {_V\phi^i_j} = q^{(\mu, \epsilon_j)} {_V\phi^i_j} \, \overline{K_{\nu}},\\[.5em]
\overline{\widetilde{E_{\beta_k}}} \, {_V\phi^i_j} = q^{-(\beta_k, \epsilon_i)} {_V\phi^i_j} \, \overline{\widetilde{E_{\beta_k}}}, & \overline{\widetilde{F_{\beta_k}}} \, {_V\phi^i_j} = {_V\phi^i_j} \, \overline{\widetilde{F_{\beta_k}}}, & \overline{K_{\mu}} \, {_V\phi^i_j} = q^{-(\mu, \epsilon_i)} {_V\phi^i_j} \, \overline{K_{\nu}}
\end{array}
\end{equation}
for all $\mu \in P$. Hence $\mathrm{gr}_{\mathcal{G}}(\mathcal{HH}_q)$ is a quasi-polynomial ring over $\mathcal{O}_q(q^{1/D})$, generated over $\mathcal{O}_q(q^{1/D})$ by $\overline{E_{\beta_k}}$, $\overline{F_{\beta_k}}$, $\overline{K_{\mu}}$, $\overline{\widetilde{E_{\beta_k}}}$, $\overline{\widetilde{F_{\beta_k}}}$, $\overline{\widetilde{K_{\mu}}}$ for $1 \leq k \leq N$ and $\mu \in P$.

\noindent Now for $(\mathbf{m}_1, \mathbf{n}_1, \ldots, \mathbf{m}_g, \mathbf{n}_g, \mathbf{p}_{1}, \ldots, \mathbf{p}_{n}) \in (\mathbb{N}^{2N+1})^{2g+n}$ let
\[ \mathcal{T}^{\mathbf{m}_1, \mathbf{n}_1, \ldots, \mathbf{m}_g, \mathbf{n}_g, \mathbf{p}_{1}, \ldots, \mathbf{p}_{n}} = \mathcal{G}^{\mathbf{m}_1, \mathbf{n}_1} \otimes \ldots \otimes \mathcal{G}^{\mathbf{m}_g, \mathbf{n}_g} \otimes \mathcal{F}_{\mathrm{DCK}}^{\mathbf{p}_1} \otimes \ldots \otimes \mathcal{F}_{\mathrm{DCK}}^{\mathbf{p}_n} \subset \mathcal{HH}_q^{\otimes g} \otimes U_q^{\otimes n}. \]
We endow the monoid $(\mathbb{N}^{2N+1})^{2g+n}$ with the direct product order:
\begin{align*}
&(\mathbf{m}_1, \mathbf{n}_1, \ldots, \mathbf{m}_g, \mathbf{n}_g, \mathbf{p}_{1}, \ldots, \mathbf{p}_{n}) \leq (\mathbf{m}'_1, \mathbf{n}'_1, \ldots, \mathbf{m}'_g, \mathbf{n}'_g, \mathbf{p}'_{1}, \ldots, \mathbf{p}'_{n})\\
&\iff \forall\, i,j,\quad \mathbf{m}_i \leq \mathbf{m}'_i, \quad \mathbf{n}_i \leq \mathbf{n}'_i, \quad \mathbf{p}_j \leq \mathbf{p}'_j.
\end{align*}
It is clear that $\mathcal{T}$ is a filtration of $\mathcal{HH}_q^{\otimes g} \otimes U_q^{\otimes n}$ and that
\[ \mathrm{gr}_{\mathcal{T}}\bigl( \mathcal{HH}_q^{\otimes g} \otimes U_q^{\otimes n} \bigr) = \mathrm{gr}_{\mathcal{G}}(\mathcal{HH}_q)^{\otimes g} \otimes \mathrm{gr}_{\mathcal{F}_{\mathrm{DCK}}}(U_q)^{\otimes n}. \]
We have seen that $\mathrm{gr}_{\mathcal{G}}(\mathcal{HH}_q)$ is a quasipolynomial ring over $\mathcal{O}_q(q^{1/D})$ and we know that $\mathrm{gr}_{\mathcal{F}_{\mathrm{DCK}}}(U_q)$ is a quasipolynomial ring over $\mathbb{C}(q)$ (\S \ref{sectionPrelimUq}). Hence $\mathrm{gr}_{\mathcal{T}}\bigl( \mathcal{HH}_q^{\otimes g} \otimes U_q^{\otimes n} \bigr)$ is a quasipolynomial ring over $\mathcal{O}_q(q^{1/D})^{\otimes g}$. Since $\mathcal{O}_q(q^{1/D})^{\otimes g}$ is a domain by Lemma \ref{lemOqgIntegre}, it follows from a general result (see e.g. \cite[\S 1.2.9]{MC-R}) that $\mathrm{gr}_{\mathcal{T}}(\mathcal{HH}_q^{\otimes g} \otimes U_q^{\otimes n})$ is a domain. The conclusion then follows from Lemma \ref{lemmadomainGr}.
\end{proof}

\begin{teo}\label{TheoremePhignInjectif}
1. The morphism $\Phi_{g,n} : \mathcal{L}_{g,n} \to \mathcal{HH}_q^{\otimes g} \otimes U_q^{\otimes n}$ is injective.
\\2. The algebra $\mathcal{L}_{g,n}$ is a domain.
\end{teo}
\begin{proof}
1. The proof is by induction. First we prove the result for $\Phi_{0,n}$ by induction on $n$. Then we fix $n$ and we prove the result for $\Phi_{g,n}$ by induction on $g$.
\\It was shown in \cite[Th. 3]{Bau1} that the morphism $\Phi_{0,1}$ is injective (see \cite[Th 4.3]{BR1} for this statement in the present framework). Now assume that $\Phi_{0,n}$ is injective for some $n$. Let $\mathrm{ad}^r$ be the right adjoint action of $U_q$ on itself: $\textstyle \mathrm{ad}^r(h)(x) = \sum_{(h)} S(h_{(1)})xh_{(2)}$. For $\lambda \in P$ consider the subspaces
\begin{align*}
(\mathcal{L}_{0,1})_{\lambda} &= \bigl\{ \varphi \in \mathcal{L}_{0,1} \, \big| \, \forall \, \nu \in P, \:\: \mathrm{coad}^r(K_{\nu})(\varphi) = q^{(\lambda,\nu)}\varphi \bigr\}, \\
(U_q)_{\lambda} &= \bigl\{ x \in U_q \, \big| \, \forall \, \nu \in P, \:\: \mathrm{ad}^r(K_{\nu})(x) = q^{(\lambda,\nu)}x \bigr\}.
\end{align*}
We have $\textstyle \mathcal{L}_{0,1} = \bigoplus_{\lambda \in P} (\mathcal{L}_{0,1})_{\lambda}$ and $\textstyle U_q = \bigoplus_{\lambda \in P} (U_q)_{\lambda}$, thus
\[ \mathcal{L}_{0,n+1} = \bigoplus_{\lambda \in P} \, (\mathcal{L}_{0,1})_{\lambda} \,\widetilde{\otimes}\, \mathcal{L}_{0,n}, \qquad U_q^{\otimes (n+1)} =  \bigoplus_{\lambda \in P} \, (U_q)_{\lambda} \otimes (U_q)^{\otimes n} \]
where for the first equality we used that $\mathcal{L}_{0,n+1} = \mathcal{L}_{0,1} \,\widetilde{\otimes}\, \mathcal{L}_{0,n}$ (Prop. \ref{propBraidedTensProduct}). Since
\[ \forall\, h \in U_q, \: \forall \, \psi \in \mathcal{L}_{0,1}, \quad\Phi_{0,1}\bigl( \mathrm{coad}^r(h)(\psi) \bigr) = \mathrm{ad}^r(h)\bigl( \Phi_{0,1}(\psi) \bigr) \]
we have $\Phi_{0,1}\bigl( (\mathcal{L}_{0,1})_{\lambda} \bigr) \subset (U_q)_{\lambda}$. Take $\varphi \in (\mathcal{L}_{0,1})_{\lambda}$ and $x \in \mathcal{L}_{0,n}$. By \eqref{ThetaPoids} we have
\[ \mathrm{coad}^r(\Theta_{(1)})(\varphi) \otimes \Theta_{(2)} = \varphi \otimes K_{\lambda}.  \]
It follows from \eqref{commutationRK} and the expression of $R$ in \eqref{expressionCanoniqueR} that
\begin{align*}
\Phi_{0,n+1}(\varphi \,\widetilde{\otimes}\, x)&= \sum_{(R)}\Phi_{0,1}\!\left( \mathrm{coad}^r\bigl(R_{(1)}\bigr)(\varphi)\right) \otimes \mathsf{D}_{0,n}\bigl(R_{(2)}\bigr)\Phi_{0,n}(x)\\
&\in \Phi_{0,1}\!\left( \mathrm{coad}^r\bigl(\Theta_{(1)}\bigr)(\varphi)\right) \otimes \mathsf{D}_{0,n}\bigl(\Theta_{(2)}\bigr)\Phi_{0,n}(x) + \bigoplus_{\lambda' < \lambda} \, (U_q)_{\lambda'} \otimes (U_q)^{\otimes n}\\
&=\Phi_{0,1}(\varphi) \otimes (K_{\lambda})^{\otimes n}\Phi_{0,n}(x) + \bigoplus_{\lambda' < \lambda} \, (U_q)_{\lambda'} \otimes (U_q)^{\otimes n}
\end{align*}
where $\lambda' < \lambda$ means that $\lambda - \lambda' \in Q_+\!\setminus\!\{0\}$ and we used that $K_ {\lambda}$ is grouplike for the last equality. We are now in position to show our result. Let $y \in \mathcal{L}_{0,n+1}$ be a non-zero element and write it as $\textstyle y = \sum_{\lambda \in P} \sum_{i \in I_{\lambda}} \varphi_{\lambda,i} \,\widetilde{\otimes}\, x_{\lambda,i}$ with $\varphi_{\lambda,i} \in (\mathcal{L}_{0,1})_{\lambda}$ for all $i \in I_{\lambda}$ and such that for each $\lambda$ the elements $\bigl(\varphi_{\lambda,i}\bigr)_{i \in I_{\lambda}}$ are linearly independent over $\mathbb{C}(q^{1/D})$. Take a $\lambda$ maximal for the order $\leq$ on $P$ such that there exists at least one $i \in I_{\lambda}$ with $\varphi_{\lambda,i} \neq 0$ and $x_{\lambda,i} \neq 0$. Then
\[ \Phi_{0,n+1}(y) \in \sum_{i \in I_{\lambda}} \Phi_{0,1}(\varphi_{\lambda,i}) \otimes (K_{\lambda})^{\otimes n}\Phi_{0,n}(x_{\lambda,i}) + \bigoplus_{\lambda' \neq \lambda} \, (U_q)_{\lambda} \otimes (U_q)^{\otimes n} \]
By the induction hypothesis the morphism $\Phi_{0,n}$ is injective so $(K_{\lambda})^{\otimes n}\Phi_{0,n}(x_{\lambda,i}) \neq 0$ ($K_{\lambda}$ being invertible) for at least one $i \in I_{\lambda}$. Moreover the elements $\varphi_{\lambda,i}$ are linearly independent so the elements $\Phi_{0,1}(\varphi_{\lambda,i})$ are linearly independent as well and it follows that $\Phi_{0,n+1}(y) \neq 0$.
\\Let us now examine the case of $\Phi_{g,n}$. Recall first from Theorem \ref{ThmPhi10} that $\Phi_{1,0}$ is injective. We see easily from the formulas in Theorem \ref{thmMorphismeAlekseev} that $\Phi_{g,n}$ actually takes values in $\mathcal{H}_q \otimes \mathcal{HH}_q^{\otimes (g-1)} \otimes U_q^{\otimes n}$. Recall the right action \eqref{actionOnHeisenberg} of $U_q$ on $\mathcal{H}_q$ and consider the subspaces

\begin{align*}
(\mathcal{L}_{1,0})_{\lambda} &= \bigl\{ x \in \mathcal{L}_{1,0} \, \big| \, \forall \, \nu \in P, \:\: \mathrm{coad}^r(K_{\nu})(x) = q^{(\lambda,\nu)}x \bigr\}, \\
(\mathcal{H}_q)_{\lambda} &= \bigl\{ y \in \mathcal{H}_q \,\big|\, \forall \, \nu \in P, \: y \cdot K_{\nu} = q^{(\lambda,\nu)}y \bigr\}.
\end{align*}
One sees easily that $\textstyle \mathcal{H}_q = \bigoplus_{\lambda \in P} (\mathcal{H}_q)_{\lambda}$ and we obtain the decompositions
\[ \mathcal{L}_{g+1,n} = \bigoplus_{\lambda \in P} \, (\mathcal{L}_{1,0})_{\lambda} \,\widetilde{\otimes}\, \mathcal{L}_{g,n}, \quad \mathcal{H}_q \otimes \mathcal{HH}_q^{\otimes g} \otimes U_q^{\otimes n} =  \bigoplus_{\lambda \in P} \, (\mathcal{H}_q)_{\lambda} \otimes \mathcal{HH}_q^{\otimes g} \otimes U_q^{\otimes n}. \]
Moreover by Proposition \ref{propPhi10} we have $\Phi_{1,0}\bigl( (\mathcal{L}_{1,0})_{\lambda} \bigr) \subset (\mathcal{H}_q)_{\lambda}$. After these preliminary remarks the proof (by induction) of the injectivity of $\Phi_{g+1,n}$ is completely similar to the proof for $\Phi_{0,n+1}$ above and is thus left to the reader.
\\2. Follows from Proposition \ref{propHHqUqIntegre} and item 1.
\end{proof}

\section{Topological interpretation of $\mathcal{L}_{g,n}(H)$}\label{sectionTopologicalInterpretation}
Let $H$ be a ribbon Hopf algebra with an invertible antipode. In this last section we show that the algebra $\mathcal{L}_{g,n}(H)$ is isomorphic to the stated skein algebra of a surface as defined below. If moreover $H\text{-}\mathrm{mod}$ is semisimple we show that the subalgebra of invariant elements $\mathcal{L}_{g,n}^H(H)$ is isomorphic to the skein algebra of the surface. Finally we explain why this last property is false when $H$-mod is not semisimple.

\subsection{Stated skein algebras of surfaces}\label{sectionStatedSkein}
\indent In this subsection we recall the definitions of the topological objects that we will use. For details, see \cite{Le, CL}.

\smallskip

\indent Denote by $\Sigma_g$ the compact closed oriented surface of genus $g \geq 0$, by $\Sigma_{g,n}$ the surface obtained from $\Sigma_g$ by removing $n$ points, by $\Sigma_{g,n}^{\circ}$ the punctured bordered surface obtained from $\Sigma_{g,n}$ by removing an open disk $D$, and then by $\Sigma_{g,n}^{\circ,\bullet}$ the surface obtained from $\Sigma_{g,n}^{\circ}$ by removing a point from the boundary component $\partial(\Sigma_{g,n}^{\circ,\bullet})$. Here is a picture of $\Sigma_{g,n}^{\circ,\bullet}$:

\begin{center}
\begingroup%
  \makeatletter%
  \providecommand\color[2][]{%
    \errmessage{(Inkscape) Color is used for the text in Inkscape, but the package 'color.sty' is not loaded}%
    \renewcommand\color[2][]{}%
  }%
  \providecommand\transparent[1]{%
    \errmessage{(Inkscape) Transparency is used (non-zero) for the text in Inkscape, but the package 'transparent.sty' is not loaded}%
    \renewcommand\transparent[1]{}%
  }%
  \providecommand\rotatebox[2]{#2}%
  \newcommand*\fsize{\dimexpr\f@size pt\relax}%
  \newcommand*\lineheight[1]{\fontsize{\fsize}{#1\fsize}\selectfont}%
  \ifx\svgwidth\undefined%
    \setlength{\unitlength}{256.21571674bp}%
    \ifx\svgscale\undefined%
      \relax%
    \else%
      \setlength{\unitlength}{\unitlength * \real{\svgscale}}%
    \fi%
  \else%
    \setlength{\unitlength}{\svgwidth}%
  \fi%
  \global\let\svgwidth\undefined%
  \global\let\svgscale\undefined%
  \makeatother%
  \begin{picture}(1,0.21872731)%
    \lineheight{1}%
    \setlength\tabcolsep{0pt}%
    \put(0,0){\includegraphics[width=\unitlength,page=1]{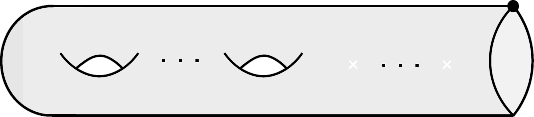}}%
    \put(0.17862914,0.12557988){\color[rgb]{0,0,0}\makebox(0,0)[lt]{\lineheight{1.25}\smash{\begin{tabular}[t]{l}$1$\end{tabular}}}}%
    \put(0.48686139,0.13252033){\color[rgb]{0,0,0}\makebox(0,0)[lt]{\lineheight{1.25}\smash{\begin{tabular}[t]{l}$g$\end{tabular}}}}%
    \put(0.65312417,0.11683069){\color[rgb]{0,0,0}\makebox(0,0)[lt]{\lineheight{1.25}\smash{\begin{tabular}[t]{l}$1$\end{tabular}}}}%
    \put(0.82509597,0.11830895){\color[rgb]{0,0,0}\makebox(0,0)[lt]{\lineheight{1.25}\smash{\begin{tabular}[t]{l}$n$\end{tabular}}}}%
    \put(0,0){\includegraphics[width=\unitlength,page=2]{surface.pdf}}%
  \end{picture}%
\endgroup%

\end{center}

\noindent The puncture on the boundary is represented by a big black dot (\textbullet) while the $n$ inner punctures are represented by the cross marks ($\mathsf{x}$). For our purposes it is much more convenient to use a less familiar view of $\Sigma_{g,n}^{\circ,\bullet}$:
\begin{equation}\label{surfaceRuban}
\begingroup%
  \makeatletter%
  \providecommand\color[2][]{%
    \errmessage{(Inkscape) Color is used for the text in Inkscape, but the package 'color.sty' is not loaded}%
    \renewcommand\color[2][]{}%
  }%
  \providecommand\transparent[1]{%
    \errmessage{(Inkscape) Transparency is used (non-zero) for the text in Inkscape, but the package 'transparent.sty' is not loaded}%
    \renewcommand\transparent[1]{}%
  }%
  \providecommand\rotatebox[2]{#2}%
  \newcommand*\fsize{\dimexpr\f@size pt\relax}%
  \newcommand*\lineheight[1]{\fontsize{\fsize}{#1\fsize}\selectfont}%
  \ifx\svgwidth\undefined%
    \setlength{\unitlength}{380.84571101bp}%
    \ifx\svgscale\undefined%
      \relax%
    \else%
      \setlength{\unitlength}{\unitlength * \real{\svgscale}}%
    \fi%
  \else%
    \setlength{\unitlength}{\svgwidth}%
  \fi%
  \global\let\svgwidth\undefined%
  \global\let\svgscale\undefined%
  \makeatother%
  \begin{picture}(1,0.16735188)%
    \lineheight{1}%
    \setlength\tabcolsep{0pt}%
    \put(0,0){\includegraphics[width=\unitlength,page=1]{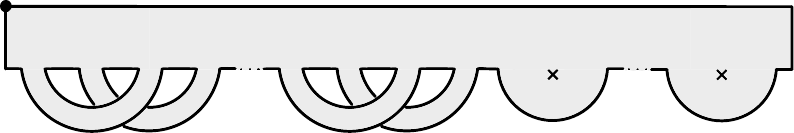}}%
  \end{picture}%
\endgroup%

\end{equation}

\indent Let $\boldsymbol{\Sigma} = \Sigma_{g,n}^{\circ,\bullet} \times [0,1]$ be the thickening of $\Sigma_{g,n}^{\circ,\bullet}$. A point $P \in \boldsymbol{\Sigma}$ is a pair $(p,h)$ and the number $h$ is called the {\em height} of $P$. Recall from \cite[\S 2.2]{Le} that an $H$-colored $\partial\boldsymbol{\Sigma}$-tangle is an oriented, framed, compact, properly embedded
1-dimensional submanifold $\mathbf{T} \subset \boldsymbol{\Sigma}$ such that
\begin{enumerate}
\item the framing is vertical at any point of $\partial\mathbf{T} = \mathbf{T} \cap \partial\boldsymbol{\Sigma}$,
\item the points of $\partial\mathbf{T}$ have distincts heights,
\item each connected component of $\mathbf{T}$ is colored (i.e. labelled) by a finite-dimensional $H$-module.
\end{enumerate}
Here we moreover allow the $\partial\boldsymbol{\Sigma}$-tangles to contain {\em coupons}, which are squares colored by $H$-morphisms in the usual coherent way. The coupons are not allowed to intersect $\partial\boldsymbol{\Sigma}$. This gives {\em $\partial\boldsymbol{\Sigma}$-ribbon graphs}, which are the generalization to surfaces of the homogeneous colored directed ribbon graphs from \cite{RT}.

\smallskip

\indent An isotopy of $\partial\boldsymbol{\Sigma}$-ribbon graphs is an isotopy which preserves all the properties listed above. We represent a $\partial\boldsymbol{\Sigma}$-ribbon graph $\mathbf{T}$ by its diagram, which is the projection of $\mathbf{T}$ on $\Sigma_{g,n}^{\circ,\bullet}$. We use isotopy to ensure that the projection has at most double points, and at double points we record the over/under-passing information as usual. Moreover, the height information of the boundary points $\partial\mathbf{T}$ must be recorded with the diagram; thanks to isotopy it is enough to record the order of the boundary points with respect to the relation $<$ defined by $(p,h)<(p',h')$ if and only $h < h'$.

We will systematically consider $\partial\boldsymbol{\Sigma}$-ribbon graphs up to isotopy.

\smallskip

\indent Let $\mathfrak{o}$ be an orientation of $\partial\Sigma_{g,n}^{\circ,\bullet}$. We say that a diagram of the $\partial\boldsymbol{\Sigma}$-ribbon graph $\mathbf{T}$ is $\mathfrak{o}$-ordered if the boundary points of $\mathbf{T}$ have increasing heights when one goes along $\partial\Sigma_{g,n}^{\circ,\bullet}$ according to $\mathfrak{o}$ and starting from the point $\bullet$. In a $\mathfrak{o}$-ordered diagram the height relation $<$ of the boundary points of $\mathbf{T}$ is given by $\mathfrak{o}$. Using isotopy it is clear that any $\partial\boldsymbol{\Sigma}$-ribbon graph can be represented by a $\mathfrak{o}$-ordered diagram. Here is an example of such a diagram for $(g,n)=(1,1)$:
\begin{equation}\label{exempleTangle}
\begingroup%
  \makeatletter%
  \providecommand\color[2][]{%
    \errmessage{(Inkscape) Color is used for the text in Inkscape, but the package 'color.sty' is not loaded}%
    \renewcommand\color[2][]{}%
  }%
  \providecommand\transparent[1]{%
    \errmessage{(Inkscape) Transparency is used (non-zero) for the text in Inkscape, but the package 'transparent.sty' is not loaded}%
    \renewcommand\transparent[1]{}%
  }%
  \providecommand\rotatebox[2]{#2}%
  \newcommand*\fsize{\dimexpr\f@size pt\relax}%
  \newcommand*\lineheight[1]{\fontsize{\fsize}{#1\fsize}\selectfont}%
  \ifx\svgwidth\undefined%
    \setlength{\unitlength}{232.25295243bp}%
    \ifx\svgscale\undefined%
      \relax%
    \else%
      \setlength{\unitlength}{\unitlength * \real{\svgscale}}%
    \fi%
  \else%
    \setlength{\unitlength}{\svgwidth}%
  \fi%
  \global\let\svgwidth\undefined%
  \global\let\svgscale\undefined%
  \makeatother%
  \begin{picture}(1,0.37687844)%
    \lineheight{1}%
    \setlength\tabcolsep{0pt}%
    \put(-0.39641803,0.03709577){\color[rgb]{0,0,0}\makebox(0,0)[lt]{\begin{minipage}{0.06366424\unitlength}\raggedright \end{minipage}}}%
    \put(0,0){\includegraphics[width=\unitlength,page=1]{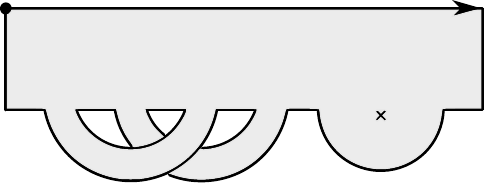}}%
    \put(0.18172988,0.21672238){\color[rgb]{0,0,0}\makebox(0,0)[lt]{\lineheight{1.25}\smash{\begin{tabular}[t]{l}$f$\end{tabular}}}}%
    \put(0,0){\includegraphics[width=\unitlength,page=2]{exempleTangle.pdf}}%
    \put(0.85594006,0.23774041){\color[rgb]{0,0,0}\makebox(0,0)[lt]{\lineheight{1.25}\smash{\begin{tabular}[t]{l}$g$\end{tabular}}}}%
    \put(0,0){\includegraphics[width=\unitlength,page=3]{exempleTangle.pdf}}%
    \put(0.09264031,0.31049015){\color[rgb]{0,0,0}\makebox(0,0)[lt]{\lineheight{1.25}\smash{\begin{tabular}[t]{l}$U$\end{tabular}}}}%
    \put(0.21726033,0.27248345){\color[rgb]{0,0,0}\makebox(0,0)[lt]{\lineheight{1.25}\smash{\begin{tabular}[t]{l}$V$\end{tabular}}}}%
    \put(0.4386446,0.31015976){\color[rgb]{0,0,0}\makebox(0,0)[lt]{\lineheight{1.25}\smash{\begin{tabular}[t]{l}$W$\end{tabular}}}}%
    \put(0.71068414,0.29629705){\color[rgb]{0,0,0}\makebox(0,0)[lt]{\lineheight{1.25}\smash{\begin{tabular}[t]{l}$X$\end{tabular}}}}%
    \put(0.87721113,0.31505116){\color[rgb]{0,0,0}\makebox(0,0)[lt]{\lineheight{1.25}\smash{\begin{tabular}[t]{l}$Y$\end{tabular}}}}%
  \end{picture}%
\endgroup%

\end{equation}
The colors (i.e. labels) $U,V,W,X,Y$ are finite-dimensional $H$-modules and $f : W \otimes V^* \to V^*$, $g : U \to Y$ are $H$-morphisms. With this choice of orientation of the boundary, the diagram implicitly means that the heights of the boundary points labelled by $U$, $W$ and $Y$ respectively form a strictly increasing sequence. In the sequel we always use ordered diagrams with respect to this orientation.

\smallskip

\indent Let $\mathbf{T}$ be a $\partial\boldsymbol{\Sigma}$-ribbon graph, let $p_1, \ldots, p_r$ be the boundary points of $\mathbf{T}$ and let $V_i$ be the color of the strand with endpoint $p_i$, for each $i$. A {\em state} of $\mathbf{T}$ is a map which to each $p_i$ associates an element of $V_i$ or $V_i^*$, depending on if the orientation of the strand points outward or inward the surface at $p_i$, respectively. In practice we label $p_i$ by an element of $V_i$ or $V_i^*$ on the ribbon graph diagram. Here is an example of a state for a ribbon graph like in \eqref{exempleTangle}:

\begin{center}
\begingroup%
  \makeatletter%
  \providecommand\color[2][]{%
    \errmessage{(Inkscape) Color is used for the text in Inkscape, but the package 'color.sty' is not loaded}%
    \renewcommand\color[2][]{}%
  }%
  \providecommand\transparent[1]{%
    \errmessage{(Inkscape) Transparency is used (non-zero) for the text in Inkscape, but the package 'transparent.sty' is not loaded}%
    \renewcommand\transparent[1]{}%
  }%
  \providecommand\rotatebox[2]{#2}%
  \newcommand*\fsize{\dimexpr\f@size pt\relax}%
  \newcommand*\lineheight[1]{\fontsize{\fsize}{#1\fsize}\selectfont}%
  \ifx\svgwidth\undefined%
    \setlength{\unitlength}{126.21546977bp}%
    \ifx\svgscale\undefined%
      \relax%
    \else%
      \setlength{\unitlength}{\unitlength * \real{\svgscale}}%
    \fi%
  \else%
    \setlength{\unitlength}{\svgwidth}%
  \fi%
  \global\let\svgwidth\undefined%
  \global\let\svgscale\undefined%
  \makeatother%
  \begin{picture}(1,0.26469555)%
    \lineheight{1}%
    \setlength\tabcolsep{0pt}%
    \put(0,0){\includegraphics[width=\unitlength,page=1]{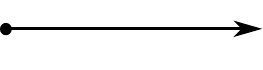}}%
    \put(0.21346234,0.19285028){\color[rgb]{0,0,0}\makebox(0,0)[lt]{\lineheight{1.25}\smash{\begin{tabular}[t]{l}$u$\end{tabular}}}}%
    \put(0.43494575,0.19299738){\color[rgb]{0,0,0}\makebox(0,0)[lt]{\lineheight{1.25}\smash{\begin{tabular}[t]{l}$w$\end{tabular}}}}%
    \put(0.68554456,0.2017813){\color[rgb]{0,0,0}\makebox(0,0)[lt]{\lineheight{1.25}\smash{\begin{tabular}[t]{l}$\gamma$\end{tabular}}}}%
    \put(0,0){\includegraphics[width=\unitlength,page=2]{exempleState.pdf}}%
    \put(0.25773229,0.04118123){\color[rgb]{0,0,0}\makebox(0,0)[lt]{\lineheight{1.25}\smash{\begin{tabular}[t]{l}$U$\end{tabular}}}}%
    \put(0.48701464,0.03467227){\color[rgb]{0,0,0}\makebox(0,0)[lt]{\lineheight{1.25}\smash{\begin{tabular}[t]{l}$W$\end{tabular}}}}%
    \put(0.73185394,0.03821341){\color[rgb]{0,0,0}\makebox(0,0)[lt]{\lineheight{1.25}\smash{\begin{tabular}[t]{l}$Y$\end{tabular}}}}%
    \put(0,0){\includegraphics[width=\unitlength,page=3]{exempleState.pdf}}%
  \end{picture}%
\endgroup%

\end{center}
with $u \in U$, $w \in W$, $\gamma \in Y^*$; in this picture we just show $\mathbf{T}$ in a disk which intersects the boundary. A $\partial\boldsymbol{\Sigma}$-ribbon graph together with a state is called a {\em stated ribbon graph}. If $s = (s_1, \ldots, s_l)$ is a state of the ribbon graph $\mathbf{T}$ (\textit{i.e.} a labelling of the boundary points ordered by increasing height), we denote by $\mathbf{T}^s$ the ribbon graph $\mathbf{T}$ with the state $s$.

\smallskip

\indent Denote by $M^{\mathrm{st}}_H(\Sigma_{g,n}^{\circ,\bullet})$ the vector space freely generated by the (isotopy classes of) stated ribbon graphs, over the base field $k$ of $H$. In $M^{\mathrm{st}}_H(\Sigma_{g,n}^{\circ,\bullet})$ we ``multilinearize'' the states, \textit{i.e} we put the relations $\mathbf{T}^{(s_1, \ldots, \lambda s_i + \mu s'_i, \ldots, s_n)} = \lambda \mathbf{T}^{(s_1, \ldots, s_i, \ldots, s_n)} + \mu \mathbf{T}^{(s_1, \ldots, s'_i, \ldots, s_n)}$ for all $i$, all $\lambda,\mu \in k$ and any $\partial\boldsymbol{\Sigma}$-ribbon graph $\mathbf{T}$ which can be labelled with these states.

\smallskip

\indent Let $\mathbf{T}^s$ be a stated ribbon graph which looks as follows in a disk $D$ which intersects the boundary:
\begin{equation}\label{tangleInDisk}
\begingroup%
  \makeatletter%
  \providecommand\color[2][]{%
    \errmessage{(Inkscape) Color is used for the text in Inkscape, but the package 'color.sty' is not loaded}%
    \renewcommand\color[2][]{}%
  }%
  \providecommand\transparent[1]{%
    \errmessage{(Inkscape) Transparency is used (non-zero) for the text in Inkscape, but the package 'transparent.sty' is not loaded}%
    \renewcommand\transparent[1]{}%
  }%
  \providecommand\rotatebox[2]{#2}%
  \newcommand*\fsize{\dimexpr\f@size pt\relax}%
  \newcommand*\lineheight[1]{\fontsize{\fsize}{#1\fsize}\selectfont}%
  \ifx\svgwidth\undefined%
    \setlength{\unitlength}{126.19012465bp}%
    \ifx\svgscale\undefined%
      \relax%
    \else%
      \setlength{\unitlength}{\unitlength * \real{\svgscale}}%
    \fi%
  \else%
    \setlength{\unitlength}{\svgwidth}%
  \fi%
  \global\let\svgwidth\undefined%
  \global\let\svgscale\undefined%
  \makeatother%
  \begin{picture}(1,0.49958293)%
    \lineheight{1}%
    \setlength\tabcolsep{0pt}%
    \put(0,0){\includegraphics[width=\unitlength,page=1]{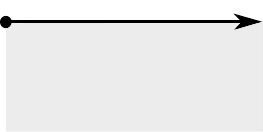}}%
    \put(0.45105949,0.16828492){\color[rgb]{0,0,0}\makebox(0,0)[lt]{\lineheight{1.25}\smash{\begin{tabular}[t]{l}$T$\end{tabular}}}}%
    \put(0.18852961,0.30935651){\color[rgb]{0,0,0}\makebox(0,0)[lt]{\lineheight{1.25}\smash{\begin{tabular}[t]{l}$V_1$\end{tabular}}}}%
    \put(0,0){\includegraphics[width=\unitlength,page=2]{planarTangle.pdf}}%
    \put(0.68954449,0.31080404){\color[rgb]{0,0,0}\makebox(0,0)[lt]{\lineheight{1.25}\smash{\begin{tabular}[t]{l}$V_l$\end{tabular}}}}%
    \put(0.1632712,0.0179534){\color[rgb]{0,0,0}\makebox(0,0)[lt]{\lineheight{1.25}\smash{\begin{tabular}[t]{l}$W_1$\end{tabular}}}}%
    \put(0.68848308,0.01790013){\color[rgb]{0,0,0}\makebox(0,0)[lt]{\lineheight{1.25}\smash{\begin{tabular}[t]{l}$W_m$\end{tabular}}}}%
    \put(0,0){\includegraphics[width=\unitlength,page=3]{planarTangle.pdf}}%
    \put(0.2794529,0.4552337){\color[rgb]{0,0,0}\makebox(0,0)[lt]{\lineheight{1.25}\smash{\begin{tabular}[t]{l}$s_1$\end{tabular}}}}%
    \put(0.46331786,0.46198791){\color[rgb]{0,0,0}\makebox(0,0)[lt]{\lineheight{1.25}\smash{\begin{tabular}[t]{l}$\ldots$\end{tabular}}}}%
    \put(0.66219219,0.4552337){\color[rgb]{0,0,0}\makebox(0,0)[lt]{\lineheight{1.25}\smash{\begin{tabular}[t]{l}$s_l$\end{tabular}}}}%
  \end{picture}%
\endgroup%

\end{equation}
where $m,l \geq 0$ and $T$ is any oriented and colored $(m,l)$-ribbon graph in $\mathbb{R}^3$. We emphasize that $m$ or $l$ may be equal to $0$. Each strand carries an orientation. For $1 \leq i \leq l$ let $\epsilon(i) \in \{\pm\}$ be $+$ (resp. $-$) if the orientation of the strand colored by $V_i$ points inward (resp. outward) the surface at the boundary. Set $V_i^+ = V_i$ and $V_i^- = V_i^*$ so that $s_i \in V_i^{-\epsilon(i)}$. Similarly for $1 \leq j \leq m$ let $\eta(j) \in \{\pm\}$ be $+$ (resp. $-$) if the strand colored by $W_j$ is outgoing (resp. incoming) at the bottom of $T$ and set $W_j^+ = W_j$ and $W_j^- = W_j^*$. The Reshetikhin--Turaev functor $F_{\mathrm{RT}}$ \cite{RT} applied to $T$ gives a linear map
\[ F_{\mathrm{RT}}(T) : W_1^{\eta(1)} \otimes \ldots \otimes W_m^{\eta(m)} \to V_1^{\epsilon(1)} \otimes \ldots \otimes V_l^{\epsilon(l)}. \]
Denote as usual $\langle -, - \rangle$ for the evaluation pairing: $\langle \varphi, v \rangle = \langle v, \varphi \rangle = \varphi(v)$ for any $v \in V_i$, $\varphi \in V_i^*$. We extend this pairing straightforwardly:
\begin{equation}\label{dualityBracket}
\begin{array}{rrcl}
\langle -, - \rangle : &\left( V_1^{-\epsilon(1)} \otimes \ldots \otimes V_l^{-\epsilon(l)} \right) \otimes \left( V_1^{\epsilon(1)} \otimes \ldots \otimes V_l^{\epsilon(l)} \right) & \to & k\\[.5em]
& (s_1 \otimes \ldots \otimes s_l) \otimes (x_1 \otimes \ldots \otimes x_l) & \mapsto & \langle s_1, x_1 \rangle \ldots \langle s_l, x_l \rangle 
\end{array}
\end{equation}
Let $\bigl(w^{j,+}_r\bigr)_r$ be a basis of $W_j$ and $\bigl(w^{j,-}_r\bigr)_r$ be its dual basis. Then we identify $\mathbf{T}^s$ with the following element of $M^{\mathrm{st}}_H(\Sigma_{g,n}^{\circ,\bullet})$:
\begin{equation}\label{skeinRelation}
\begingroup%
  \makeatletter%
  \providecommand\color[2][]{%
    \errmessage{(Inkscape) Color is used for the text in Inkscape, but the package 'color.sty' is not loaded}%
    \renewcommand\color[2][]{}%
  }%
  \providecommand\transparent[1]{%
    \errmessage{(Inkscape) Transparency is used (non-zero) for the text in Inkscape, but the package 'transparent.sty' is not loaded}%
    \renewcommand\transparent[1]{}%
  }%
  \providecommand\rotatebox[2]{#2}%
  \newcommand*\fsize{\dimexpr\f@size pt\relax}%
  \newcommand*\lineheight[1]{\fontsize{\fsize}{#1\fsize}\selectfont}%
  \ifx\svgwidth\undefined%
    \setlength{\unitlength}{381.87954953bp}%
    \ifx\svgscale\undefined%
      \relax%
    \else%
      \setlength{\unitlength}{\unitlength * \real{\svgscale}}%
    \fi%
  \else%
    \setlength{\unitlength}{\svgwidth}%
  \fi%
  \global\let\svgwidth\undefined%
  \global\let\svgscale\undefined%
  \makeatother%
  \begin{picture}(1,0.09897297)%
    \lineheight{1}%
    \setlength\tabcolsep{0pt}%
    \put(0,0){\includegraphics[width=\unitlength,page=1]{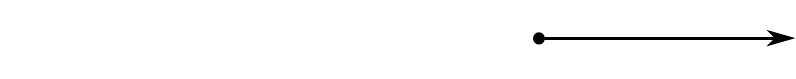}}%
    \put(0.74017001,0.06805863){\color[rgb]{0,0,0}\makebox(0,0)[lt]{\lineheight{1.25}\smash{\begin{tabular}[t]{l}$w^{1,-\eta(1)}_{r_1}$\end{tabular}}}}%
    \put(0.8151311,0.06455063){\color[rgb]{0,0,0}\makebox(0,0)[lt]{\lineheight{1.25}\smash{\begin{tabular}[t]{l}$\ldots$\end{tabular}}}}%
    \put(0.88646875,0.06779516){\color[rgb]{0,0,0}\makebox(0,0)[lt]{\lineheight{1.25}\smash{\begin{tabular}[t]{l}$w^{m,-\eta(m)}_{r_m}$\end{tabular}}}}%
    \put(-0.00058827,0.03803108){\color[rgb]{0,0,0}\makebox(0,0)[lt]{\lineheight{1.25}\smash{\begin{tabular}[t]{l}$\sum_{r_1, \ldots, r_m} \left\langle s_1 \otimes \ldots \otimes s_l, F_{\mathrm{RT}}(T)\!\left( w^{1,\eta(1)}_{r_1} \otimes \ldots w^{m,\eta(m)}_{r_m} \right) \right\rangle$\end{tabular}}}}%
    \put(0,0){\includegraphics[width=\unitlength,page=2]{skeinRelation.pdf}}%
    \put(0.71344173,0.01549328){\color[rgb]{0,0,0}\makebox(0,0)[lt]{\lineheight{1.25}\smash{\begin{tabular}[t]{l}$W_1$\end{tabular}}}}%
    \put(0.90663496,0.0154757){\color[rgb]{0,0,0}\makebox(0,0)[lt]{\lineheight{1.25}\smash{\begin{tabular}[t]{l}$W_m$\end{tabular}}}}%
    \put(0,0){\includegraphics[width=\unitlength,page=3]{skeinRelation.pdf}}%
  \end{picture}%
\endgroup%

\end{equation}
We stress that the stated ribbon graphs appearing in the linear combination \eqref{skeinRelation} are equal to $\mathbf{T}$ outside the disk $D$ represented in \eqref{tangleInDisk}. Such an identification is called a {\em stated skein relation}.

\begin{defi}\label{defStatedSkein}
The stated skein algebra of $\Sigma_{g,n}^{\circ,\bullet}$ associated to $H$, which we denote by $\mathcal{S}^{\mathrm{st}}_H(\Sigma_{g,n}^{\circ,\bullet})$, is the quotient of $M^{\mathrm{st}}_H(\Sigma_{g,n}^{\circ,\bullet})$ by all the stated skein relations \eqref{skeinRelation}.
\end{defi}
\noindent This definition agrees with the one to appear in \cite{CKL} for more general surfaces than $\Sigma_{g,n}^{\circ,\bullet}$.

\smallskip

\indent Given two stated ribbon graphs $\mathbf{T}_1^s$, $\mathbf{T}_2^t$ we can use isotopy to transform $\mathbf{T}_1^s$ into $\textstyle (\mathbf{T}_1^s)^- \subset \Sigma_{g,n}^{\circ,\bullet} \times [0,\frac{1}{2} [$ and $\mathbf{T}_2$ into $\textstyle (\mathbf{T}_2^t)^+ \subset \Sigma_{g,n}^{\circ,\bullet} \times ]\frac{1}{2}, 1]$. The disjoint union $(\mathbf{T}_1^s)^- \cup (\mathbf{T}_2^t)^+$ is denoted by $\mathbf{T}_1^s \ast \mathbf{T}_2^t$. Each boundary point keeps its state in $(\mathbf{T}_1^s)^- \cup (\mathbf{T}_2^t)^+$ so that the result is naturally a stated ribbon graph. This operation $\ast$ defines an associative algebra structure on $M^{\mathrm{st}}_H(\Sigma_{g,n}^{\circ,\bullet})$ which descends to $\mathcal{S}^{\mathrm{st}}_H(\Sigma_{g,n}^{\circ,\bullet})$. The unit is the empty $\partial\boldsymbol{\Sigma}$-ribbon graph.

\smallskip

\indent Here is an example of stated skein relation which we will use later:

\begin{equation}\label{evalBraiding}
\begingroup%
  \makeatletter%
  \providecommand\color[2][]{%
    \errmessage{(Inkscape) Color is used for the text in Inkscape, but the package 'color.sty' is not loaded}%
    \renewcommand\color[2][]{}%
  }%
  \providecommand\transparent[1]{%
    \errmessage{(Inkscape) Transparency is used (non-zero) for the text in Inkscape, but the package 'transparent.sty' is not loaded}%
    \renewcommand\transparent[1]{}%
  }%
  \providecommand\rotatebox[2]{#2}%
  \newcommand*\fsize{\dimexpr\f@size pt\relax}%
  \newcommand*\lineheight[1]{\fontsize{\fsize}{#1\fsize}\selectfont}%
  \ifx\svgwidth\undefined%
    \setlength{\unitlength}{287.44900273bp}%
    \ifx\svgscale\undefined%
      \relax%
    \else%
      \setlength{\unitlength}{\unitlength * \real{\svgscale}}%
    \fi%
  \else%
    \setlength{\unitlength}{\svgwidth}%
  \fi%
  \global\let\svgwidth\undefined%
  \global\let\svgscale\undefined%
  \makeatother%
  \begin{picture}(1,0.29529047)%
    \lineheight{1}%
    \setlength\tabcolsep{0pt}%
    \put(0,0){\includegraphics[width=\unitlength,page=1]{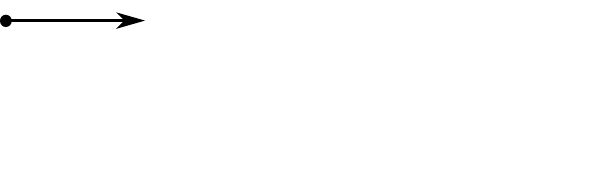}}%
    \put(0.06685406,0.27483276){\color[rgb]{0,0,0}\makebox(0,0)[lt]{\lineheight{1.25}\smash{\begin{tabular}[t]{l}$v$\end{tabular}}}}%
    \put(0.15450349,0.27878623){\color[rgb]{0,0,0}\makebox(0,0)[lt]{\lineheight{1.25}\smash{\begin{tabular}[t]{l}$\varphi$\end{tabular}}}}%
    \put(0,0){\includegraphics[width=\unitlength,page=2]{evalBraiding.pdf}}%
    \put(0.04038675,0.17417135){\color[rgb]{0,0,0}\makebox(0,0)[lt]{\lineheight{1.25}\smash{\begin{tabular}[t]{l}$Y$\end{tabular}}}}%
    \put(0.17251204,0.17513637){\color[rgb]{0,0,0}\makebox(0,0)[lt]{\lineheight{1.25}\smash{\begin{tabular}[t]{l}$X$\end{tabular}}}}%
    \put(0,0){\includegraphics[width=\unitlength,page=3]{evalBraiding.pdf}}%
    \put(0.26393868,0.05025881){\color[rgb]{0,0,0}\makebox(0,0)[lt]{\lineheight{1.25}\smash{\begin{tabular}[t]{l}$= \: \sum_{(R)}$\end{tabular}}}}%
    \put(0,0){\includegraphics[width=\unitlength,page=4]{evalBraiding.pdf}}%
    \put(0.59281727,0.11380332){\color[rgb]{0,0,0}\makebox(0,0)[lt]{\lineheight{1.25}\smash{\begin{tabular}[t]{l}$S(R_{(2)})v$\end{tabular}}}}%
    \put(0.42975207,0.11413277){\color[rgb]{0,0,0}\makebox(0,0)[lt]{\lineheight{1.25}\smash{\begin{tabular}[t]{l}$\varphi(R_{(1)}?)$\end{tabular}}}}%
    \put(0,0){\includegraphics[width=\unitlength,page=5]{evalBraiding.pdf}}%
    \put(0.45785212,0.00457604){\color[rgb]{0,0,0}\makebox(0,0)[lt]{\lineheight{1.25}\smash{\begin{tabular}[t]{l}$Y$\end{tabular}}}}%
    \put(0.66825206,0.00554104){\color[rgb]{0,0,0}\makebox(0,0)[lt]{\lineheight{1.25}\smash{\begin{tabular}[t]{l}$X$\end{tabular}}}}%
    \put(0,0){\includegraphics[width=\unitlength,page=6]{evalBraiding.pdf}}%
    \put(0.26612122,0.21465081){\color[rgb]{0,0,0}\makebox(0,0)[lt]{\lineheight{1.25}\smash{\begin{tabular}[t]{l}$= \: \sum_{i,j} \left\langle v \otimes \varphi, c_{Y,X^*}(y_i \otimes x^j) \right\rangle$\end{tabular}}}}%
    \put(0,0){\includegraphics[width=\unitlength,page=7]{evalBraiding.pdf}}%
    \put(0.82350998,0.27483276){\color[rgb]{0,0,0}\makebox(0,0)[lt]{\lineheight{1.25}\smash{\begin{tabular}[t]{l}$y^i$\end{tabular}}}}%
    \put(0.91115944,0.27878623){\color[rgb]{0,0,0}\makebox(0,0)[lt]{\lineheight{1.25}\smash{\begin{tabular}[t]{l}$x_j$\end{tabular}}}}%
    \put(0,0){\includegraphics[width=\unitlength,page=8]{evalBraiding.pdf}}%
    \put(0.79704269,0.17417134){\color[rgb]{0,0,0}\makebox(0,0)[lt]{\lineheight{1.25}\smash{\begin{tabular}[t]{l}$Y$\end{tabular}}}}%
    \put(0.92916799,0.17513636){\color[rgb]{0,0,0}\makebox(0,0)[lt]{\lineheight{1.25}\smash{\begin{tabular}[t]{l}$X$\end{tabular}}}}%
    \put(0,0){\includegraphics[width=\unitlength,page=9]{evalBraiding.pdf}}%
  \end{picture}%
\endgroup%

\end{equation}

\noindent where $v \in X$, $\varphi \in Y^*$, $c$ is the braiding in $H\text{-}\mathrm{mod}$, $(x_i)$ is a basis of $X$, $(y_j)$ is a basis of $Y$, $(x^i)$, $(y^j)$ are their dual bases and $\varphi(R_{(1)}?)$ denotes the linear form on $Y$ defined by $y \mapsto \varphi(R_{(1)}y)$. The detail of the second equality is
\begin{align*}
&\sum_{i,j} \left\langle v \otimes \varphi, c_{Y,X^*}(y_i \otimes x^j) \right\rangle y^i \otimes x_j = \sum_{(R),i,j} \left\langle v \otimes \varphi, R_{(2)}x^j \otimes R_{(1)}y_i \right\rangle y^i \otimes x_j\\
=\:&\sum_{(R),i,j} \left\langle x^j, S(R_{(2)})v\right\rangle \left\langle \varphi, R_{(1)}y_i \right\rangle y^i \otimes x_j =\sum_{(R)} \varphi(R_{(1)}?) \otimes S(R_{(2)})v.
\end{align*}
One proves similarly that
\begin{equation}\label{evalBraiding2}
\begingroup%
  \makeatletter%
  \providecommand\color[2][]{%
    \errmessage{(Inkscape) Color is used for the text in Inkscape, but the package 'color.sty' is not loaded}%
    \renewcommand\color[2][]{}%
  }%
  \providecommand\transparent[1]{%
    \errmessage{(Inkscape) Transparency is used (non-zero) for the text in Inkscape, but the package 'transparent.sty' is not loaded}%
    \renewcommand\transparent[1]{}%
  }%
  \providecommand\rotatebox[2]{#2}%
  \newcommand*\fsize{\dimexpr\f@size pt\relax}%
  \newcommand*\lineheight[1]{\fontsize{\fsize}{#1\fsize}\selectfont}%
  \ifx\svgwidth\undefined%
    \setlength{\unitlength}{212.41435739bp}%
    \ifx\svgscale\undefined%
      \relax%
    \else%
      \setlength{\unitlength}{\unitlength * \real{\svgscale}}%
    \fi%
  \else%
    \setlength{\unitlength}{\svgwidth}%
  \fi%
  \global\let\svgwidth\undefined%
  \global\let\svgscale\undefined%
  \makeatother%
  \begin{picture}(1,0.18069311)%
    \lineheight{1}%
    \setlength\tabcolsep{0pt}%
    \put(0,0){\includegraphics[width=\unitlength,page=1]{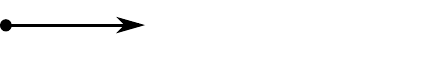}}%
    \put(0.09047003,0.14285349){\color[rgb]{0,0,0}\makebox(0,0)[lt]{\lineheight{1.25}\smash{\begin{tabular}[t]{l}$v$\end{tabular}}}}%
    \put(0.20863547,0.14329933){\color[rgb]{0,0,0}\makebox(0,0)[lt]{\lineheight{1.25}\smash{\begin{tabular}[t]{l}$w$\end{tabular}}}}%
    \put(0,0){\includegraphics[width=\unitlength,page=2]{evalBraiding2.pdf}}%
    \put(0.05465323,0.00663377){\color[rgb]{0,0,0}\makebox(0,0)[lt]{\lineheight{1.25}\smash{\begin{tabular}[t]{l}$Y$\end{tabular}}}}%
    \put(0.23345131,0.00793968){\color[rgb]{0,0,0}\makebox(0,0)[lt]{\lineheight{1.25}\smash{\begin{tabular}[t]{l}$X$\end{tabular}}}}%
    \put(0,0){\includegraphics[width=\unitlength,page=3]{evalBraiding2.pdf}}%
    \put(0.35717416,0.06845382){\color[rgb]{0,0,0}\makebox(0,0)[lt]{\lineheight{1.25}\smash{\begin{tabular}[t]{l}$= \: \sum_{(R)}$\end{tabular}}}}%
    \put(0,0){\includegraphics[width=\unitlength,page=4]{evalBraiding2.pdf}}%
    \put(0.77410156,0.15835881){\color[rgb]{0,0,0}\makebox(0,0)[lt]{\lineheight{1.25}\smash{\begin{tabular}[t]{l}$R_{(2)}v$\end{tabular}}}}%
    \put(0.58263152,0.15533685){\color[rgb]{0,0,0}\makebox(0,0)[lt]{\lineheight{1.25}\smash{\begin{tabular}[t]{l}$R_{(1)}w$\end{tabular}}}}%
    \put(0,0){\includegraphics[width=\unitlength,page=5]{evalBraiding2.pdf}}%
    \put(0.60193264,0.00663379){\color[rgb]{0,0,0}\makebox(0,0)[lt]{\lineheight{1.25}\smash{\begin{tabular}[t]{l}$Y$\end{tabular}}}}%
    \put(0.85134731,0.00793965){\color[rgb]{0,0,0}\makebox(0,0)[lt]{\lineheight{1.25}\smash{\begin{tabular}[t]{l}$X$\end{tabular}}}}%
    \put(0,0){\includegraphics[width=\unitlength,page=6]{evalBraiding2.pdf}}%
  \end{picture}%
\endgroup%

\end{equation}

\subsection{Isomorphism $\mathcal{S}^{\mathrm{st}}_H(\Sigma_{g,n}^{\circ,\bullet}) \cong \mathcal{L}_{g,n}(H)$}\label{sectionIsoSstLgn}
This subsection is a generalization of \cite[\S 5]{FaitgHol}, which dealt with the case $H = U_q^{\mathrm{ad}}(\mathfrak{sl}_2)$. The main ingredient in the definition of the isomorphism will be the {\em holonomy map}
\[ \mathrm{hol} : \bigl\{\partial\boldsymbol{\Sigma}\text{-ribbon graphs} \bigr\} \to \bigl\{\text{tensors with coefficients in } \mathcal{L}_{g,n}(H) \bigr\} \]
defined in \cite[\S 4.1]{FaitgHol}. This map depends on a presentation of $\Sigma_{g,n}^{\circ,\bullet}$ as in \eqref{surfaceRuban}.

\smallskip

\indent Let $\mathbf{T}^s \subset \Sigma_{g,n}^{\circ, \bullet} \times [0,1]$ be a stated ribbon graph with state $s = (s_1, \ldots, s_l)$ which looks as follows near the boundary:
\begin{center}
\begingroup%
  \makeatletter%
  \providecommand\color[2][]{%
    \errmessage{(Inkscape) Color is used for the text in Inkscape, but the package 'color.sty' is not loaded}%
    \renewcommand\color[2][]{}%
  }%
  \providecommand\transparent[1]{%
    \errmessage{(Inkscape) Transparency is used (non-zero) for the text in Inkscape, but the package 'transparent.sty' is not loaded}%
    \renewcommand\transparent[1]{}%
  }%
  \providecommand\rotatebox[2]{#2}%
  \newcommand*\fsize{\dimexpr\f@size pt\relax}%
  \newcommand*\lineheight[1]{\fontsize{\fsize}{#1\fsize}\selectfont}%
  \ifx\svgwidth\undefined%
    \setlength{\unitlength}{113.66472394bp}%
    \ifx\svgscale\undefined%
      \relax%
    \else%
      \setlength{\unitlength}{\unitlength * \real{\svgscale}}%
    \fi%
  \else%
    \setlength{\unitlength}{\svgwidth}%
  \fi%
  \global\let\svgwidth\undefined%
  \global\let\svgscale\undefined%
  \makeatother%
  \begin{picture}(1,0.24204434)%
    \lineheight{1}%
    \setlength\tabcolsep{0pt}%
    \put(0,0){\includegraphics[width=\unitlength,page=1]{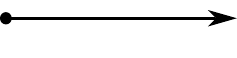}}%
    \put(0.19410107,0.19808386){\color[rgb]{0,0,0}\makebox(0,0)[lt]{\lineheight{1.25}\smash{\begin{tabular}[t]{l}$s_1$\end{tabular}}}}%
    \put(0.65684454,0.1948548){\color[rgb]{0,0,0}\makebox(0,0)[lt]{\lineheight{1.25}\smash{\begin{tabular}[t]{l}$s_l$\end{tabular}}}}%
    \put(0,0){\includegraphics[width=\unitlength,page=2]{boundaryPoints.pdf}}%
    \put(0.09595768,0.0358081){\color[rgb]{0,0,0}\makebox(0,0)[lt]{\lineheight{1.25}\smash{\begin{tabular}[t]{l}$V_1$\end{tabular}}}}%
    \put(0,0){\includegraphics[width=\unitlength,page=3]{boundaryPoints.pdf}}%
    \put(0.69721404,0.03497471){\color[rgb]{0,0,0}\makebox(0,0)[lt]{\lineheight{1.25}\smash{\begin{tabular}[t]{l}$V_l$\end{tabular}}}}%
    \put(0.37897518,0.06976563){\color[rgb]{0,0,0}\makebox(0,0)[lt]{\lineheight{1.25}\smash{\begin{tabular}[t]{l}$\ldots$\end{tabular}}}}%
  \end{picture}%
\endgroup%

\end{center}
\noindent As before, let $\epsilon(i) \in \{\pm\}$ be $+$ (resp. $-$) if the orientation of the strand colored by $V_i$ points inward (resp. outward) the surface at the boundary, and set $V_i^+ = V_i$, $V_i^- = V_i^*$ so that $s_i \in V_i^{-\epsilon(i)}$. By definition $\mathbf{T}$ is the $\partial\boldsymbol{\Sigma}$-ribbon graph obtained from $\mathbf{T}^s$ by forgetting the states. Then the holonomy map gives an element
\[ \mathrm{hol}(\mathbf{T}) \in \mathcal{L}_{g,n}(H) \otimes V_1^{\epsilon(1)} \otimes \ldots \otimes V_l^{\epsilon(l)}. \]
Let $\langle -, - \rangle$ be the evaluation pairing as in \eqref{dualityBracket}. Then $\langle s_1 \otimes \ldots \otimes s_l, - \rangle$ is a linear map $V_1^{\epsilon(1)} \otimes \ldots \otimes V_l^{\epsilon(l)} \to k$ and we define
\[ \mathrm{hol}^{\mathrm{st}}(\mathbf{T}^s) = \left(\mathrm{id}_ {\mathcal{L}_{g,n}(H)} \otimes \langle s_1 \otimes \ldots \otimes s_l, - \rangle\right)\!(\mathrm{hol}(\mathbf{T})) \in \mathcal{L}_{g,n}(H). \]
Extending linearly, this gives a map $\mathrm{hol}^{\mathrm{st}} : M^{\mathrm{st}}_H(\Sigma_{g,n}^{\circ,\bullet}) \to \mathcal{L}_{g,n}(H)$. It is immediate from the definition of $\mathrm{hol}$ in \cite{FaitgHol} that $\mathrm{hol}^{\mathrm{st}}$ is compatible with the stated skein relations, so it descends to $\mathcal{S}^{\mathrm{st}}_H(\Sigma_{g,n}^{\circ,\bullet})$.

\begin{defi}
We call $\mathrm{hol}^{\mathrm{st}} : \mathcal{S}^{\mathrm{st}}_H(\Sigma_{g,n}^{\circ,\bullet}) \to \mathcal{L}_{g,n}(H)$ the stated holonomy map.
\end{defi}

\begin{prop}\label{statedHolMorphism}
The stated holonomy map is a morphism of algebras.
\end{prop}
\begin{proof}
This follows from one of the main properties of $\mathrm{hol}$ proven in \cite{FaitgHol}. Let $\mathbf{T}^s_1$, $\mathbf{T}^t_2$ be stated ribbon graphs and let
\begin{center}
\begingroup%
  \makeatletter%
  \providecommand\color[2][]{%
    \errmessage{(Inkscape) Color is used for the text in Inkscape, but the package 'color.sty' is not loaded}%
    \renewcommand\color[2][]{}%
  }%
  \providecommand\transparent[1]{%
    \errmessage{(Inkscape) Transparency is used (non-zero) for the text in Inkscape, but the package 'transparent.sty' is not loaded}%
    \renewcommand\transparent[1]{}%
  }%
  \providecommand\rotatebox[2]{#2}%
  \newcommand*\fsize{\dimexpr\f@size pt\relax}%
  \newcommand*\lineheight[1]{\fontsize{\fsize}{#1\fsize}\selectfont}%
  \ifx\svgwidth\undefined%
    \setlength{\unitlength}{315.48854944bp}%
    \ifx\svgscale\undefined%
      \relax%
    \else%
      \setlength{\unitlength}{\unitlength * \real{\svgscale}}%
    \fi%
  \else%
    \setlength{\unitlength}{\svgwidth}%
  \fi%
  \global\let\svgwidth\undefined%
  \global\let\svgscale\undefined%
  \makeatother%
  \begin{picture}(1,0.0878418)%
    \lineheight{1}%
    \setlength\tabcolsep{0pt}%
    \put(0,0){\includegraphics[width=\unitlength,page=1]{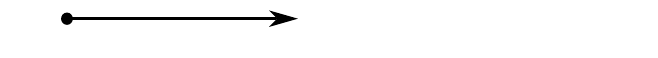}}%
    \put(0.16287845,0.07136599){\color[rgb]{0,0,0}\makebox(0,0)[lt]{\lineheight{1.25}\smash{\begin{tabular}[t]{l}$s_1$\end{tabular}}}}%
    \put(0.32959642,0.07020261){\color[rgb]{0,0,0}\makebox(0,0)[lt]{\lineheight{1.25}\smash{\begin{tabular}[t]{l}$s_l$\end{tabular}}}}%
    \put(0,0){\includegraphics[width=\unitlength,page=2]{boundaryPoints2.pdf}}%
    \put(0.12751919,0.01290101){\color[rgb]{0,0,0}\makebox(0,0)[lt]{\lineheight{1.25}\smash{\begin{tabular}[t]{l}$V_1$\end{tabular}}}}%
    \put(0,0){\includegraphics[width=\unitlength,page=3]{boundaryPoints2.pdf}}%
    \put(0.34414083,0.01260075){\color[rgb]{0,0,0}\makebox(0,0)[lt]{\lineheight{1.25}\smash{\begin{tabular}[t]{l}$V_l$\end{tabular}}}}%
    \put(0.2294852,0.02513528){\color[rgb]{0,0,0}\makebox(0,0)[lt]{\lineheight{1.25}\smash{\begin{tabular}[t]{l}$\ldots$\end{tabular}}}}%
    \put(-0.0011425,0.02281649){\color[rgb]{0,0,0}\makebox(0,0)[lt]{\lineheight{1.25}\smash{\begin{tabular}[t]{l}$\mathbf{T}^s_1 =$\end{tabular}}}}%
    \put(0,0){\includegraphics[width=\unitlength,page=4]{boundaryPoints2.pdf}}%
    \put(0.70964941,0.07136598){\color[rgb]{0,0,0}\makebox(0,0)[lt]{\lineheight{1.25}\smash{\begin{tabular}[t]{l}$t_1$\end{tabular}}}}%
    \put(0.87636736,0.07200365){\color[rgb]{0,0,0}\makebox(0,0)[lt]{\lineheight{1.25}\smash{\begin{tabular}[t]{l}$t_m$\end{tabular}}}}%
    \put(0,0){\includegraphics[width=\unitlength,page=5]{boundaryPoints2.pdf}}%
    \put(0.66828665,0.012901){\color[rgb]{0,0,0}\makebox(0,0)[lt]{\lineheight{1.25}\smash{\begin{tabular}[t]{l}$W_1$\end{tabular}}}}%
    \put(0,0){\includegraphics[width=\unitlength,page=6]{boundaryPoints2.pdf}}%
    \put(0.89091181,0.01260074){\color[rgb]{0,0,0}\makebox(0,0)[lt]{\lineheight{1.25}\smash{\begin{tabular}[t]{l}$W_m$\end{tabular}}}}%
    \put(0.77625618,0.02513527){\color[rgb]{0,0,0}\makebox(0,0)[lt]{\lineheight{1.25}\smash{\begin{tabular}[t]{l}$\ldots$\end{tabular}}}}%
    \put(0.54578838,0.02281649){\color[rgb]{0,0,0}\makebox(0,0)[lt]{\lineheight{1.25}\smash{\begin{tabular}[t]{l}$\mathbf{T}^t_2 =$\end{tabular}}}}%
  \end{picture}%
\endgroup%

\end{center}
be their boundary points and states. By definition of the product of stated ribbon graphs and of an ordered $\partial\boldsymbol{\Sigma}$-ribbon graph diagram, the boundary points of $\mathbf{T}^s_1 \ast \mathbf{T}_2^t$ are arranged as follows:
\begin{center}
\begingroup%
  \makeatletter%
  \providecommand\color[2][]{%
    \errmessage{(Inkscape) Color is used for the text in Inkscape, but the package 'color.sty' is not loaded}%
    \renewcommand\color[2][]{}%
  }%
  \providecommand\transparent[1]{%
    \errmessage{(Inkscape) Transparency is used (non-zero) for the text in Inkscape, but the package 'transparent.sty' is not loaded}%
    \renewcommand\transparent[1]{}%
  }%
  \providecommand\rotatebox[2]{#2}%
  \newcommand*\fsize{\dimexpr\f@size pt\relax}%
  \newcommand*\lineheight[1]{\fontsize{\fsize}{#1\fsize}\selectfont}%
  \ifx\svgwidth\undefined%
    \setlength{\unitlength}{269.48703354bp}%
    \ifx\svgscale\undefined%
      \relax%
    \else%
      \setlength{\unitlength}{\unitlength * \real{\svgscale}}%
    \fi%
  \else%
    \setlength{\unitlength}{\svgwidth}%
  \fi%
  \global\let\svgwidth\undefined%
  \global\let\svgscale\undefined%
  \makeatother%
  \begin{picture}(1,0.10283644)%
    \lineheight{1}%
    \setlength\tabcolsep{0pt}%
    \put(0,0){\includegraphics[width=\unitlength,page=1]{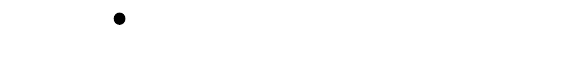}}%
    \put(0.28437267,0.08354818){\color[rgb]{0,0,0}\makebox(0,0)[lt]{\lineheight{1.25}\smash{\begin{tabular}[t]{l}$s_1$\end{tabular}}}}%
    \put(0.47954943,0.08218621){\color[rgb]{0,0,0}\makebox(0,0)[lt]{\lineheight{1.25}\smash{\begin{tabular}[t]{l}$s_l$\end{tabular}}}}%
    \put(-0.00078325,0.02720825){\color[rgb]{0,0,0}\makebox(0,0)[lt]{\lineheight{1.25}\smash{\begin{tabular}[t]{l}$\mathbf{T}^s_1 \ast \mathbf{T}^t_2 =$\end{tabular}}}}%
    \put(0,0){\includegraphics[width=\unitlength,page=2]{boundaryPoints3.pdf}}%
    \put(0.6600864,0.08354819){\color[rgb]{0,0,0}\makebox(0,0)[lt]{\lineheight{1.25}\smash{\begin{tabular}[t]{l}$t_1$\end{tabular}}}}%
    \put(0.8552632,0.08429471){\color[rgb]{0,0,0}\makebox(0,0)[lt]{\lineheight{1.25}\smash{\begin{tabular}[t]{l}$t_m$\end{tabular}}}}%
    \put(0,0){\includegraphics[width=\unitlength,page=3]{boundaryPoints3.pdf}}%
    \put(0.24297756,0.01510322){\color[rgb]{0,0,0}\makebox(0,0)[lt]{\lineheight{1.25}\smash{\begin{tabular}[t]{l}$V_1$\end{tabular}}}}%
    \put(0,0){\includegraphics[width=\unitlength,page=4]{boundaryPoints3.pdf}}%
    \put(0.49657657,0.01475171){\color[rgb]{0,0,0}\makebox(0,0)[lt]{\lineheight{1.25}\smash{\begin{tabular}[t]{l}$V_l$\end{tabular}}}}%
    \put(0.36234922,0.02942589){\color[rgb]{0,0,0}\makebox(0,0)[lt]{\lineheight{1.25}\smash{\begin{tabular}[t]{l}$\ldots$\end{tabular}}}}%
    \put(0,0){\includegraphics[width=\unitlength,page=5]{boundaryPoints3.pdf}}%
    \put(0.6042833,0.01510324){\color[rgb]{0,0,0}\makebox(0,0)[lt]{\lineheight{1.25}\smash{\begin{tabular}[t]{l}$W_1$\end{tabular}}}}%
    \put(0,0){\includegraphics[width=\unitlength,page=6]{boundaryPoints3.pdf}}%
    \put(0.8722903,0.01475171){\color[rgb]{0,0,0}\makebox(0,0)[lt]{\lineheight{1.25}\smash{\begin{tabular}[t]{l}$W_m$\end{tabular}}}}%
    \put(0.73806304,0.02942589){\color[rgb]{0,0,0}\makebox(0,0)[lt]{\lineheight{1.25}\smash{\begin{tabular}[t]{l}$\ldots$\end{tabular}}}}%
  \end{picture}%
\endgroup%

\end{center}
Let $\mathbf{T}_1$, $\mathbf{T}_2$ be the $\partial\boldsymbol{\Sigma}$-ribbon graphs without states and write $\textstyle \mathrm{hol}(\mathbf{T}_1) = \sum_i x_i \otimes \mathbf{v}_i$, $\textstyle \mathrm{hol}(\mathbf{T}_2) = \sum_j y_j \otimes \mathbf{w}_j$ with $x_i, y_j \in \mathcal{L}_{g,n}(H)$, $\mathbf{v}_i \in V_1^{\epsilon(1)} \otimes \ldots \otimes V_l^{\epsilon(l)}$ and $\mathbf{w}_j \in W_1^{\eta(1)} \otimes \ldots \otimes W_m^{\eta(m)}$, where the $\epsilon$'s and $\eta$'s are signs as above.
We know from \cite[Th. 4.4]{FaitgHol} that $\textstyle \mathrm{hol}(\mathbf{T}_1 \ast \mathbf{T}_2) = \sum_{i,j} x_iy_j \otimes \mathbf{v}_i \otimes \mathbf{w}_j$. Hence:
\begin{align*}
\mathrm{hol}^{\mathrm{st}}(\mathbf{T}^s_1 \ast \mathbf{T}^t_2) &= \left(\mathrm{id}_ {\mathcal{L}_{g,n}(H)} \otimes \langle s_1 \otimes \ldots \otimes s_l \otimes t_1 \otimes \ldots \otimes t_m, - \rangle\right)\!(\mathrm{hol}(\mathbf{T}_1 \ast \mathbf{T}_2))\\
&= \sum_{i,j} x_iy_j \, \langle s_1 \otimes \ldots \otimes s_l, \mathbf{v}_i \rangle \, \langle t_1 \otimes \ldots \otimes t_m, \mathbf{w}_j\rangle = \mathrm{hol}^{\mathrm{st}}(\mathbf{T}^s_1) \, \mathrm{hol}^{\mathrm{st}}(\mathbf{T}^t _2). \qedhere
\end{align*}
\end{proof}

\indent We now look for a right inverse to $\mathrm{hol}^{\mathrm{st}}$. Recall that $H^{\circ}$ is spanned by the matrix coefficients $_V\phi^f_v$ of finite-dimensional $H$-modules and that $\mathcal{L}_{g,n}(H)$ is $(H^{\circ})^{\otimes (2g+n)}$ as a vector space. We define a linear map $\xi_{g,n} : \mathcal{L}_{g,n}(H) \to \mathcal{S}^{\mathrm{st}}_H(\Sigma_{g,n}^{\circ,\bullet})$ by

\smallskip

\begin{equation}\label{defXign}
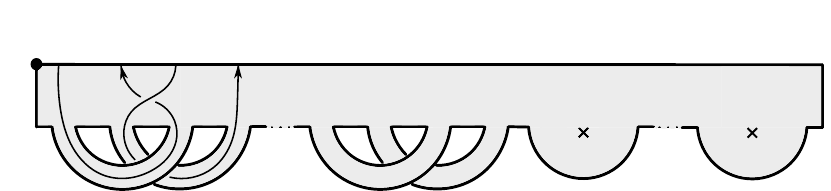
\end{equation}

\noindent Note that this definition implicitly extracts the tensor $f \otimes x \in X^* \otimes X$ from the matrix coefficient $_X\phi^f_x$, which might be ill-defined. Such a situation never occurs in the following cases, where therefore $\xi_{g,n}$ is well-defined:
\begin{align}
\begin{split}\label{assumptionsForSkein}
\bullet &\text{ the base field of } H \text{ is algebraically closed and } H\text{-}\mathrm{mod} \text{ is semisimple,}\\
\bullet & \text{ or } H = U_q^{\mathrm{ad}}(\mathfrak{g}) \text{ over } \mathbb{C}(q),\\
\bullet & \text{ or } H \text{ is finite-dimensional.}
\end{split}
\end{align}
Here $H\text{-}\mathrm{mod}$ is the abelian category of finite-dimensional $H$-modules. Indeed, in the first case we have $\textstyle H^{\circ} = \bigoplus_{X \in \mathrm{Irr}(H)} C(X)$ and $C(X) \cong X^* \otimes X$ \cite[Th.\,27.8]{CR}, where $\mathrm{Irr}(H)$ is the set of isomorphism classes of irreducible $H$-modules. The second case similarly follows from $\textstyle \mathcal{O}_q(G) = \bigoplus_{\mu \in P_+} V_{\mu}^* \otimes V_{\mu}$, called Peter--Weyl decomposition \cite[\S 3.10]{VY}. The last case is proven in \cite[\S 5.1.2]{FaitgDer} and is based on the use of the regular representation $_HH$. In the sequel we assume that $H$ belongs to one of these three cases.

\begin{lem}\label{lemmeXiSurjective}
The linear map $\xi_{g,n} : \mathcal{L}_{g,n}(H) \to \mathcal{S}^{\mathrm{st}}_H(\Sigma_{g,n}^{\circ,\bullet})$ is surjective.
\end{lem}
\begin{proof}
Due to the stated skein relations \eqref{skeinRelation} in $\mathcal{S}^{\mathrm{st}}_H(\Sigma_{g,n}^{\circ,\bullet})$ we see that every element is equal to a linear combination of ``standard'' stated tangles of the form
\begin{center}
\begingroup%
  \makeatletter%
  \providecommand\color[2][]{%
    \errmessage{(Inkscape) Color is used for the text in Inkscape, but the package 'color.sty' is not loaded}%
    \renewcommand\color[2][]{}%
  }%
  \providecommand\transparent[1]{%
    \errmessage{(Inkscape) Transparency is used (non-zero) for the text in Inkscape, but the package 'transparent.sty' is not loaded}%
    \renewcommand\transparent[1]{}%
  }%
  \providecommand\rotatebox[2]{#2}%
  \newcommand*\fsize{\dimexpr\f@size pt\relax}%
  \newcommand*\lineheight[1]{\fontsize{\fsize}{#1\fsize}\selectfont}%
  \ifx\svgwidth\undefined%
    \setlength{\unitlength}{446.63583731bp}%
    \ifx\svgscale\undefined%
      \relax%
    \else%
      \setlength{\unitlength}{\unitlength * \real{\svgscale}}%
    \fi%
  \else%
    \setlength{\unitlength}{\svgwidth}%
  \fi%
  \global\let\svgwidth\undefined%
  \global\let\svgscale\undefined%
  \makeatother%
  \begin{picture}(1,0.17637223)%
    \lineheight{1}%
    \setlength\tabcolsep{0pt}%
    \put(0,0){\includegraphics[width=\unitlength,page=1]{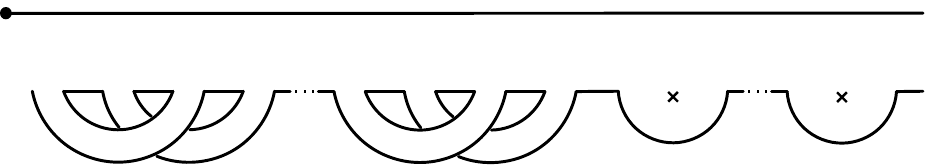}}%
    \put(0.50127145,0.17063306){\color[rgb]{0,0,0}\makebox(0,0)[lt]{\lineheight{1.25}\smash{\begin{tabular}[t]{l}$s$\end{tabular}}}}%
    \put(0,0){\includegraphics[width=\unitlength,page=2]{standardForm.pdf}}%
    \put(0.18702328,0.15014741){\color[rgb]{0,0,0}\makebox(0,0)[lt]{\lineheight{1.25}\smash{\begin{tabular}[t]{l}$\ldots$\end{tabular}}}}%
    \put(0,0){\includegraphics[width=\unitlength,page=3]{standardForm.pdf}}%
    \put(0.27429474,0.15037311){\color[rgb]{0,0,0}\makebox(0,0)[lt]{\lineheight{1.25}\smash{\begin{tabular}[t]{l}$\ldots$\end{tabular}}}}%
    \put(0,0){\includegraphics[width=\unitlength,page=4]{standardForm.pdf}}%
    \put(0.04031936,0.15013887){\color[rgb]{0,0,0}\makebox(0,0)[lt]{\lineheight{1.25}\smash{\begin{tabular}[t]{l}$\ldots$\end{tabular}}}}%
    \put(0.10740354,0.15141107){\color[rgb]{0,0,0}\makebox(0,0)[lt]{\lineheight{1.25}\smash{\begin{tabular}[t]{l}$\ldots$\end{tabular}}}}%
    \put(0,0){\includegraphics[width=\unitlength,page=5]{standardForm.pdf}}%
    \put(0.51111277,0.15014741){\color[rgb]{0,0,0}\makebox(0,0)[lt]{\lineheight{1.25}\smash{\begin{tabular}[t]{l}$\ldots$\end{tabular}}}}%
    \put(0,0){\includegraphics[width=\unitlength,page=6]{standardForm.pdf}}%
    \put(0.59838427,0.15037311){\color[rgb]{0,0,0}\makebox(0,0)[lt]{\lineheight{1.25}\smash{\begin{tabular}[t]{l}$\ldots$\end{tabular}}}}%
    \put(0,0){\includegraphics[width=\unitlength,page=7]{standardForm.pdf}}%
    \put(0.36440888,0.15013887){\color[rgb]{0,0,0}\makebox(0,0)[lt]{\lineheight{1.25}\smash{\begin{tabular}[t]{l}$\ldots$\end{tabular}}}}%
    \put(0.43149307,0.15141107){\color[rgb]{0,0,0}\makebox(0,0)[lt]{\lineheight{1.25}\smash{\begin{tabular}[t]{l}$\ldots$\end{tabular}}}}%
    \put(0,0){\includegraphics[width=\unitlength,page=8]{standardForm.pdf}}%
    \put(0.66412211,0.15056291){\color[rgb]{0,0,0}\makebox(0,0)[lt]{\lineheight{1.25}\smash{\begin{tabular}[t]{l}$\ldots$\end{tabular}}}}%
    \put(0.75392156,0.15031452){\color[rgb]{0,0,0}\makebox(0,0)[lt]{\lineheight{1.25}\smash{\begin{tabular}[t]{l}$\ldots$\end{tabular}}}}%
    \put(0,0){\includegraphics[width=\unitlength,page=9]{standardForm.pdf}}%
    \put(0.84547794,0.15056291){\color[rgb]{0,0,0}\makebox(0,0)[lt]{\lineheight{1.25}\smash{\begin{tabular}[t]{l}$\ldots$\end{tabular}}}}%
    \put(0.9352773,0.15031452){\color[rgb]{0,0,0}\makebox(0,0)[lt]{\lineheight{1.25}\smash{\begin{tabular}[t]{l}$\ldots$\end{tabular}}}}%
  \end{picture}%
\endgroup%

\end{center}
 In this picture each strand carries some orientation and color and $s = (s_1, \ldots, s_N)$ is some state. We put the crossings in view of the definition \eqref{defXign}; it is clear that we can always add such crossings since they correspond to the braiding isomorphism through the Reshetikhin--Turaev functor. Now we observe that the stated skein relations imply an orientation reversal relation:
\begin{center}
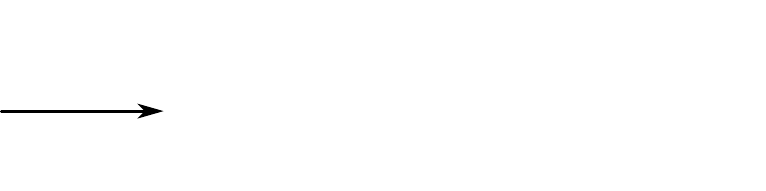
\end{center}
where $\langle -, v \rangle : \varphi \mapsto \varphi(v)$ and $f(g^{-1} ?) : x \mapsto f(g^{-1}x)$ with $g \in H$ the pivotal element, and $v,x\in V$. We similarly have an all-in-one relation:
\begin{center}
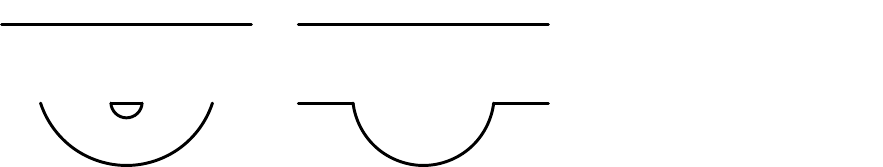
\end{center}
In these pictures the handle represents either one of the handles or one of the punctured half-disks in \eqref{surfaceRuban}. As a result the ``standard'' stated tangles above can be rewritten as the ones in \eqref{defXign}, which by definition span the image of $\xi_{g,n}$.
\end{proof}

\begin{teo}\label{ThStatedHolIso}
Let $H$ be a ribbon Hopf algebra satisfying any one of the conditions in \eqref{assumptionsForSkein}. Then the stated holonomy map $\mathrm{hol}^{\mathrm{st}} : \mathcal{S}^{\mathrm{st}}_H(\Sigma_{g,n}^{\circ,\bullet}) \to \mathcal{L}_{g,n}(H)$ is an isomorphism of algebras.
\end{teo}
\begin{proof}
We already know from Proposition \ref{statedHolMorphism} that $\mathrm{hol}^{\mathrm{st}}$ is a morphism of algebras. It is immediate from the definition of $\mathrm{hol}$ in \cite{FaitgHol} that $\mathrm{hol}^{\mathrm{st}} \circ \xi_{g,n} = \mathrm{id}_{\mathcal{L}_{g,n}(H)}$ (the role of the crossings in the definition \eqref{defXign} of $\xi_{g,n}$ is precisely to obtain this equality). This implies that the linear map $\xi_{g,n}$ is injective. It follows from Lemma \ref{lemmeXiSurjective} that $\xi_{g,n}$ is an isomorphism of vector spaces. Hence $\mathrm{hol}^{\mathrm{st}} = \xi^{-1}_{g,n}$ is an isomorphism as well.
\end{proof}

\indent A corollary of the previous proof is that $\xi_{g,n}$ is a morphism of algebras. It is nevertheless worthwile to mention that the direct proof (\textit{i.e.} without resorting on $\mathrm{hol}^{\mathrm{st}}$) of this fact allows one to recover the product in $\mathcal{L}_{g,n}(H)$ by means of stated skein relations, thus giving a topological flavour to the formulas in Proposition \ref{productLgn}. We explain this for $\mathcal{S}^{\mathrm{st}}_H(\Sigma_{0,1}^{\circ,\bullet})$ and $\mathcal{S}^{\mathrm{st}}_H(\Sigma_{1,0}^{\circ,\bullet})$.

\smallskip

\indent So let us first discuss the case of $\xi_{0,1} : \mathcal{L}_{0,1}(H) \to \mathcal{S}^{\mathrm{st}}_H(\Sigma_{0,1}^{\circ,\bullet})$. It is useful to note that
\begin{equation}\label{coregularOnMatrixCoeffs}
h \rhd {_V\phi^f_v} = {_V\phi^f_{hv}}, \qquad {_V\phi^f_v} \lhd h = {_V\phi^{f(h?)}_v}
\end{equation}
where we recall that $\rhd$ and $\lhd$ are the left and right coregular actions of $H$ on $H^{\circ}$, see  \eqref{coregularActions}. We then have the following stated skein computation, which recovers the formula \eqref{produitL01} for the product in $\mathcal{L}_{0,1}(H)$:
\begin{center}
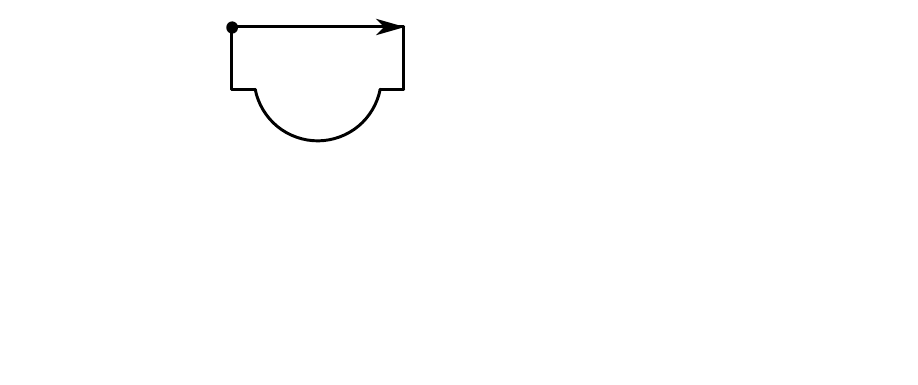
\end{center}
The first equality is by definition of the product in $\mathcal{S}^{\mathrm{st}}_H(\Sigma_{0,1}^{\circ,\bullet})$, the second equality uses the stated skein relations \eqref{evalBraiding} and \eqref{evalBraiding2}, the third equality uses the all-in-one relation (see the proof of Lemma \ref{lemmeXiSurjective}), the fourth equality is the definition of $\xi_{0,1}$ and the last equality uses \eqref{coregularOnMatrixCoeffs} and \eqref{hopfMatrixCoeffs}.

\smallskip

\indent Let us now consider $\xi_{1,0} : \mathcal{L}_{1,0}(H) \to \mathcal{S}^{\mathrm{st}}_H(\Sigma_{1,0}^{\circ,\bullet})$. Recall that the product in $\mathcal{L}_{1,0}(H)$ is fully described by the formulas \eqref{productL10WithEmbeddings}. Let us write
\[ \xi_b = \xi_{1,0} \circ \mathfrak{i}_B, \qquad \xi_a = \xi_{1,0} \circ \mathfrak{i}_A \]
which are morphisms of algebras $\mathcal{L}_{0,1}(H) \to \mathcal{S}^{\mathrm{st}}_H(\Sigma_{1,0}^{\circ,\bullet})$, and we recall that $\mathfrak{i}_B$ and $\mathfrak{i}_A$ are defined in \eqref{embeddingL01inL10}. Then the definition of $\xi_{1,0}$ in \eqref{defXign} is such that 
\begin{equation}\label{Xi10XiaXib}
\xi_{1,0}(\beta \otimes \alpha) = \xi_b(\beta) \, \xi_a(\alpha).
\end{equation}
We have two natural morphisms of algebras $\mathfrak{e}_b, \mathfrak{e}_a : \mathcal{S}^{\mathrm{st}}_H(\Sigma_{0,1}^{\circ,\bullet}) \to \mathcal{S}^{\mathrm{st}}_H(\Sigma_{1,0}^{\circ,\bullet})$ defined by
\begin{center}
\begingroup%
  \makeatletter%
  \providecommand\color[2][]{%
    \errmessage{(Inkscape) Color is used for the text in Inkscape, but the package 'color.sty' is not loaded}%
    \renewcommand\color[2][]{}%
  }%
  \providecommand\transparent[1]{%
    \errmessage{(Inkscape) Transparency is used (non-zero) for the text in Inkscape, but the package 'transparent.sty' is not loaded}%
    \renewcommand\transparent[1]{}%
  }%
  \providecommand\rotatebox[2]{#2}%
  \newcommand*\fsize{\dimexpr\f@size pt\relax}%
  \newcommand*\lineheight[1]{\fontsize{\fsize}{#1\fsize}\selectfont}%
  \ifx\svgwidth\undefined%
    \setlength{\unitlength}{396.51535151bp}%
    \ifx\svgscale\undefined%
      \relax%
    \else%
      \setlength{\unitlength}{\unitlength * \real{\svgscale}}%
    \fi%
  \else%
    \setlength{\unitlength}{\svgwidth}%
  \fi%
  \global\let\svgwidth\undefined%
  \global\let\svgscale\undefined%
  \makeatother%
  \begin{picture}(1,0.13720742)%
    \lineheight{1}%
    \setlength\tabcolsep{0pt}%
    \put(0,0){\includegraphics[width=\unitlength,page=1]{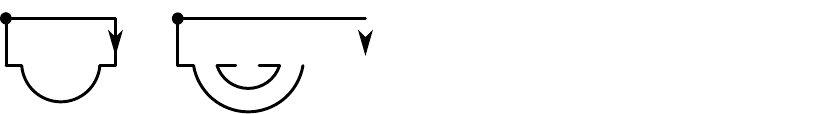}}%
    \put(0.23093838,0.12424651){\color[rgb]{0,0,0}\makebox(0,0)[lt]{\lineheight{1.25}\smash{\begin{tabular}[t]{l}$f$\end{tabular}}}}%
    \put(0.36199029,0.12609925){\color[rgb]{0,0,0}\makebox(0,0)[lt]{\lineheight{1.25}\smash{\begin{tabular}[t]{l}$v$\end{tabular}}}}%
    \put(0.15835319,0.05828366){\color[rgb]{0,0,0}\makebox(0,0)[lt]{\lineheight{1.25}\smash{\begin{tabular}[t]{l}$\overset{\text{\normalsize $\mathfrak{e}_b$}}{\longmapsto}$\end{tabular}}}}%
    \put(0.03166589,0.1244853){\color[rgb]{0,0,0}\makebox(0,0)[lt]{\lineheight{1.25}\smash{\begin{tabular}[t]{l}$f$\end{tabular}}}}%
    \put(0.1056362,0.12347203){\color[rgb]{0,0,0}\makebox(0,0)[lt]{\lineheight{1.25}\smash{\begin{tabular}[t]{l}$v$\end{tabular}}}}%
    \put(0,0){\includegraphics[width=\unitlength,page=2]{exempleInjection.pdf}}%
    \put(0.83129417,0.12400766){\color[rgb]{0,0,0}\makebox(0,0)[lt]{\lineheight{1.25}\smash{\begin{tabular}[t]{l}$f$\end{tabular}}}}%
    \put(0.96139068,0.12418857){\color[rgb]{0,0,0}\makebox(0,0)[lt]{\lineheight{1.25}\smash{\begin{tabular}[t]{l}$v$\end{tabular}}}}%
    \put(0.70692047,0.05828361){\color[rgb]{0,0,0}\makebox(0,0)[lt]{\lineheight{1.25}\smash{\begin{tabular}[t]{l}$\overset{\text{\normalsize $\mathfrak{e}_a$}}{\longmapsto}$\end{tabular}}}}%
    \put(0,0){\includegraphics[width=\unitlength,page=3]{exempleInjection.pdf}}%
    \put(0.04546268,0.08554862){\color[rgb]{0,0,0}\makebox(0,0)[lt]{\lineheight{1.25}\smash{\begin{tabular}[t]{l}$_V$\end{tabular}}}}%
    \put(0,0){\includegraphics[width=\unitlength,page=4]{exempleInjection.pdf}}%
    \put(0.58019448,0.12448527){\color[rgb]{0,0,0}\makebox(0,0)[lt]{\lineheight{1.25}\smash{\begin{tabular}[t]{l}$f$\end{tabular}}}}%
    \put(0.6541648,0.12347201){\color[rgb]{0,0,0}\makebox(0,0)[lt]{\lineheight{1.25}\smash{\begin{tabular}[t]{l}$v$\end{tabular}}}}%
    \put(0,0){\includegraphics[width=\unitlength,page=5]{exempleInjection.pdf}}%
    \put(0.59399128,0.08554858){\color[rgb]{0,0,0}\makebox(0,0)[lt]{\lineheight{1.25}\smash{\begin{tabular}[t]{l}$_V$\end{tabular}}}}%
    \put(0,0){\includegraphics[width=\unitlength,page=6]{exempleInjection.pdf}}%
    \put(0.242859,0.08619951){\color[rgb]{0,0,0}\makebox(0,0)[lt]{\lineheight{1.25}\smash{\begin{tabular}[t]{l}$_V$\end{tabular}}}}%
    \put(0,0){\includegraphics[width=\unitlength,page=7]{exempleInjection.pdf}}%
    \put(0.84699226,0.08561413){\color[rgb]{0,0,0}\makebox(0,0)[lt]{\lineheight{1.25}\smash{\begin{tabular}[t]{l}$_V$\end{tabular}}}}%
    \put(0,0){\includegraphics[width=\unitlength,page=8]{exempleInjection.pdf}}%
  \end{picture}%
\endgroup%

\end{center}
(we use that every element of $\mathcal{S}^{\mathrm{st}}_H(\Sigma_{0,1}^{\circ,\bullet})$ can be written as a linear combination of such stated tangles, see the proof of Lemma \ref{lemmeXiSurjective}). They satisfy $\xi_b = \mathfrak{e}_b \circ \xi_{0,1}$, $\xi_a = \mathfrak{e}_a \circ \xi_{0,1}$. Hence by the previous computation for $\xi_{0,1}$ we find
\[ \xi_b(\varphi)\,\xi_b(\psi) = \mathfrak{e}_b\bigl( \xi_{0,1}(\varphi)\,\xi_{0,1}(\psi) \bigr) = \mathfrak{e}_b\bigl( \xi_{0,1}(\varphi\psi) \bigr) = \xi_b(\varphi\psi) \]
and similarly for $\xi_a$; these two equalities correspond to the first and second formulas of \eqref{productL10WithEmbeddings}. Finally:
\begin{center}
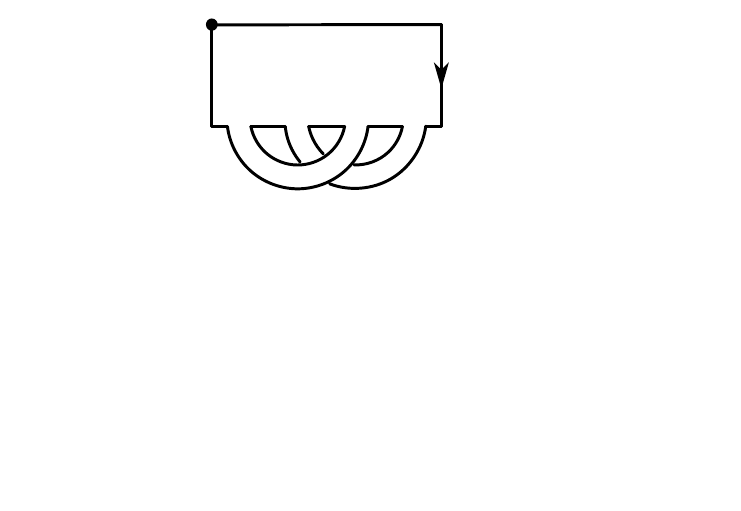
\end{center}

\bigskip

\noindent The first equality is by definition of the product in $\mathcal{S}^{\mathrm{st}}_H(\Sigma_{1,0}^{\circ,\bullet})$, the second equality is a trick, the third equality uses stated skein relations like \eqref{evalBraiding} and \eqref{evalBraiding2}, the fourth equality uses the definition of $\xi_{1,0}$, the last equality uses \eqref{coregularOnMatrixCoeffs} and \eqref{Xi10XiaXib}. This equality corresponds to the third formula of \eqref{productL10WithEmbeddings}.

\smallskip

\indent For general $g$ and $n$ write
\[ \xi_{b_i} = \xi_{g,n} \circ \mathfrak{i}_{B(i)}, \qquad \xi_{a_i} = \xi_{g,n} \circ \mathfrak{i}_{A(i)}, \qquad \xi_{m_j} = \xi_{g,n} \circ \mathfrak{i}_{M(j)}. \]
for all $1 \leq i \leq g$, $g+1 \leq j \leq g+n$ and observe that by definition of $\xi_{g,n}$ (see \eqref{defXign}) we have
\begin{align*}
\xi_{g,n}\bigl( \varphi_1 \otimes \ldots \otimes \varphi_{2g+n} \bigr) = \xi_{b_1}(\varphi_1) \, \xi_{a_1}(\varphi_2) \ldots \xi_{b_g}(\varphi_{2g-1}) \, \xi_{a_g}(\varphi_g) \, \xi_{m_{g+1}}(\varphi_{2g+1}) \ldots \xi_{m_{g+n}}(\varphi_{2g+n}).
\end{align*}
As above, one can show using the stated skein relations that the elements $\xi_{b_i}(\varphi)$, $\xi_{a_i}(\varphi)$, $\xi_{m_j}(\varphi) \in \mathcal{S}^{\mathrm{st}}_H(\Sigma_{g,n}^{\circ,\bullet})$ with $\varphi \in H^{\circ}$ obey the same product than the elements $\mathfrak{i}_{B(i)}(\varphi)$, $\mathfrak{i}_{A(i)}(\varphi)$, $\mathfrak{i}_{M(j)}(\varphi) \in \mathcal{L}_{g,n}(H)$. The reader can treat the case $(g,n)=(0,2)$ as an exercise.

\smallskip

\indent The computations above reveal the topological relevance of the elements $\mathfrak{i}_{B(i)}(\varphi)$, $\mathfrak{i}_{A(i)}(\varphi)$, $\mathfrak{i}_{M(j)}(\varphi)$ and of the formulas in Proposition \ref{productLgn} for their products.

\subsection{Isomorphism $\mathcal{S}_H(\Sigma_{g,n}^{\circ}) \cong \mathcal{L}_{g,n}^H(H)$ for semisimple $H\text{-}\mathrm{mod}$}\label{sectionIsoSLgninv}
This part is a generalization of \cite[\S 8.2]{BR1}, which dealt with the case $H = U_q^{\mathrm{ad}}(\mathfrak{sl}_2)$.

\smallskip

\indent Recall that $\Sigma_{g,n}^{\circ}$ is the compact oriented surface of genus $g$ with one boundary component and $n$ punctures not belonging to the boundary. Recall that $F_{\mathrm{RT}}$ denotes the Reshetikhin--Turaev functor \cite{RT}, and $k$ is the base field of $H$.
\begin{defi}
The skein algebra $\mathcal{S}_H(\Sigma_{g,n}^{\circ})$ is the $k$-vector space generated by the isotopy classes of ribbon links with coupons (\textit{i.e.} ribbon graphs without boundary points) in $\Sigma_{g,n}^{\circ} \times [0,1]$, modulo the skein relations:
\begin{equation}\label{usualSkeinRelation}
\begingroup%
  \makeatletter%
  \providecommand\color[2][]{%
    \errmessage{(Inkscape) Color is used for the text in Inkscape, but the package 'color.sty' is not loaded}%
    \renewcommand\color[2][]{}%
  }%
  \providecommand\transparent[1]{%
    \errmessage{(Inkscape) Transparency is used (non-zero) for the text in Inkscape, but the package 'transparent.sty' is not loaded}%
    \renewcommand\transparent[1]{}%
  }%
  \providecommand\rotatebox[2]{#2}%
  \newcommand*\fsize{\dimexpr\f@size pt\relax}%
  \newcommand*\lineheight[1]{\fontsize{\fsize}{#1\fsize}\selectfont}%
  \ifx\svgwidth\undefined%
    \setlength{\unitlength}{209.63965807bp}%
    \ifx\svgscale\undefined%
      \relax%
    \else%
      \setlength{\unitlength}{\unitlength * \real{\svgscale}}%
    \fi%
  \else%
    \setlength{\unitlength}{\svgwidth}%
  \fi%
  \global\let\svgwidth\undefined%
  \global\let\svgscale\undefined%
  \makeatother%
  \begin{picture}(1,0.23412325)%
    \lineheight{1}%
    \setlength\tabcolsep{0pt}%
    \put(-0.00078435,0.10909138){\color[rgb]{0,0,0}\makebox(0,0)[lt]{\lineheight{1.25}\smash{\begin{tabular}[t]{l}$\sum_i \lambda_i$\end{tabular}}}}%
    \put(0.40349837,0.10824069){\color[rgb]{0,0,0}\makebox(0,0)[lt]{\lineheight{1.25}\smash{\begin{tabular}[t]{l}$= 0$ \hspace{1em}if\; \; $\displaystyle\sum_i \lambda_iF_{\mathrm{RT}}(T_i) =  0$.\end{tabular}}}}%
    \put(0,0){\includegraphics[width=\unitlength,page=1]{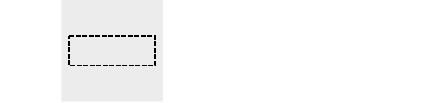}}%
    \put(0.23081831,0.1020627){\color[rgb]{0,0,0}\makebox(0,0)[lt]{\lineheight{1.25}\smash{\begin{tabular}[t]{l}$T_i$\end{tabular}}}}%
    \put(0.2223914,0.19244067){\color[rgb]{0,0,0}\makebox(0,0)[lt]{\lineheight{1.25}\smash{\begin{tabular}[t]{l}$\ldots$\end{tabular}}}}%
    \put(0,0){\includegraphics[width=\unitlength,page=2]{usualSkeinRelations.pdf}}%
    \put(0.14312695,0.20663588){\color[rgb]{0,0,0}\makebox(0,0)[lt]{\lineheight{1.25}\smash{\begin{tabular}[t]{l}$_{Y_1}$\end{tabular}}}}%
    \put(0,0){\includegraphics[width=\unitlength,page=3]{usualSkeinRelations.pdf}}%
    \put(0.32765603,0.2069265){\color[rgb]{0,0,0}\makebox(0,0)[lt]{\lineheight{1.25}\smash{\begin{tabular}[t]{l}$_{Y_l}$\end{tabular}}}}%
    \put(0.2223914,0.04218284){\color[rgb]{0,0,0}\makebox(0,0)[lt]{\lineheight{1.25}\smash{\begin{tabular}[t]{l}$\ldots$\end{tabular}}}}%
    \put(0,0){\includegraphics[width=\unitlength,page=4]{usualSkeinRelations.pdf}}%
    \put(0.13736881,0.05026268){\color[rgb]{0,0,0}\makebox(0,0)[lt]{\lineheight{1.25}\smash{\begin{tabular}[t]{l}$_{X_1}$\end{tabular}}}}%
    \put(0,0){\includegraphics[width=\unitlength,page=5]{usualSkeinRelations.pdf}}%
    \put(0.32795138,0.0502102){\color[rgb]{0,0,0}\makebox(0,0)[lt]{\lineheight{1.25}\smash{\begin{tabular}[t]{l}$_{X_k}$\end{tabular}}}}%
  \end{picture}%
\endgroup%

\end{equation}
where $\lambda_i \in k$ and the $T_i$'s are any ribbon graphs in $[0,1]^3$ having common colors on the incoming and outgoing strands. The sum on the left represents a linear combination of ribbon links which are equal outside of some cube in $\Sigma_{g,n}^{\circ} \times [0,1]$ which is depicted in grey.
\end{defi}
\noindent Note that up to introducing identity coupons, we can always assume that the incoming and outgoing strands of $T_i$ have this orientation. Then $F_{\mathrm{RT}}(T_i) \in \mathrm{Hom}_H\bigl( X_1 \otimes \ldots \otimes X_k, Y_1 \otimes \ldots \otimes Y_l \bigr)$. The product in $\mathcal{S}_H(\Sigma_{g,n}^{\circ})$ is given by stacking (see after Definition \ref{defStatedSkein}).

\smallskip

\indent Recall that $\Sigma_{g,n}^{\circ, \bullet}$ is $\Sigma_{g,n}^{\circ}$ with a puncture on the boundary. This puncture is irrelevant when we restrict to ribbon links, so if $L$ is the isotopy class of a ribbon link in $\Sigma_{g,n}^{\circ} \times [0,1]$ we have the corresponding isotopy class $I(L)$ in $\Sigma_{g,n}^{\circ, \bullet} \times [0,1]$. Note that $I(L)$ can be seen as a stated ribbon graph without boundary points (and thus without states): $I(L) \in \mathcal{S}_H^{\mathrm{st}}(\Sigma_{g,n}^{\circ, \bullet})$.
\begin{lem}\label{LemmaMapIWellDef}
The map $L \mapsto I(L)$ induces a well-defined morphism of algebras $I : \mathcal{S}_H(\Sigma_{g,n}^{\circ}) \to \mathcal{S}_H^{\mathrm{st}}(\Sigma_{g,n}^{\circ, \bullet})$.
\end{lem}
\begin{proof}
Let $L_1, \ldots, L_k$ be links in $\Sigma_{g,n}^{\circ} \times [0,1]$ and $\textstyle L = \sum_i \lambda_i L_i \in \mathcal{S}_H(\Sigma_{g,n}^{\circ})$. We have to show that if $L = 0$ is a skein relation \eqref{usualSkeinRelation} in $\mathcal{S}_H(\Sigma_{g,n}^{\circ})$ then $\textstyle I(L) = \sum_i \lambda_i I(L_i) = 0$ follows from the stated skein relations in $\mathcal{S}_H^{\mathrm{st}}(\Sigma_{g,n}^{\circ, \bullet})$. For notational simplicity we take ribbon graphs $T_i$ which have $2$ incoming strands and $1$ outgoing strand in \eqref{usualSkeinRelation}. Using isotopy we can assume that the cube in \eqref{usualSkeinRelation} is very close to $\partial(\Sigma_{g,n}^{\circ})$. Therefore in $\mathcal{S}_H^{\mathrm{st}}(\Sigma_{g,n}^{\circ, \bullet})$ we have
\begin{center}
\begingroup%
  \makeatletter%
  \providecommand\color[2][]{%
    \errmessage{(Inkscape) Color is used for the text in Inkscape, but the package 'color.sty' is not loaded}%
    \renewcommand\color[2][]{}%
  }%
  \providecommand\transparent[1]{%
    \errmessage{(Inkscape) Transparency is used (non-zero) for the text in Inkscape, but the package 'transparent.sty' is not loaded}%
    \renewcommand\transparent[1]{}%
  }%
  \providecommand\rotatebox[2]{#2}%
  \newcommand*\fsize{\dimexpr\f@size pt\relax}%
  \newcommand*\lineheight[1]{\fontsize{\fsize}{#1\fsize}\selectfont}%
  \ifx\svgwidth\undefined%
    \setlength{\unitlength}{367.93349484bp}%
    \ifx\svgscale\undefined%
      \relax%
    \else%
      \setlength{\unitlength}{\unitlength * \real{\svgscale}}%
    \fi%
  \else%
    \setlength{\unitlength}{\svgwidth}%
  \fi%
  \global\let\svgwidth\undefined%
  \global\let\svgscale\undefined%
  \makeatother%
  \begin{picture}(1,0.22374486)%
    \lineheight{1}%
    \setlength\tabcolsep{0pt}%
    \put(0.06826433,0.02073411){\color[rgb]{0,0,0}\makebox(0,0)[lt]{\lineheight{1.25}\smash{\begin{tabular}[t]{l}$\displaystyle = \sum_{i,j,k,l} \lambda_i\,\mathrm{ev}_Y \circ \bigl(\mathrm{id}_{Y^*} \otimes F_{\mathrm{RT}}(T_i) \bigr)(y^j \otimes x^{(1)}_k \otimes x^{(2)}_l)$\end{tabular}}}}%
    \put(0,0){\includegraphics[width=\unitlength,page=1]{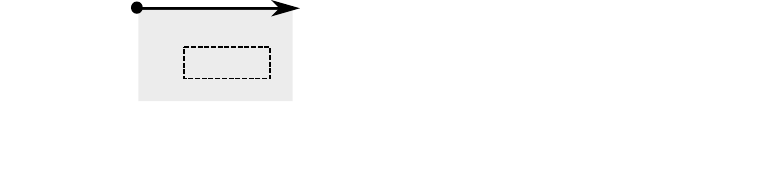}}%
    \put(0.28155071,0.13366402){\color[rgb]{0,0,0}\makebox(0,0)[lt]{\lineheight{1.25}\smash{\begin{tabular}[t]{l}$T_i$\end{tabular}}}}%
    \put(0,0){\includegraphics[width=\unitlength,page=2]{fromSkeinToStatedSkein.pdf}}%
    \put(0.29814315,0.1745336){\color[rgb]{0,0,0}\makebox(0,0)[lt]{\lineheight{1.25}\smash{\begin{tabular}[t]{l}$_Y$\end{tabular}}}}%
    \put(0.2689412,0.09792553){\color[rgb]{0,0,0}\makebox(0,0)[lt]{\lineheight{1.25}\smash{\begin{tabular}[t]{l}$_{X_1}$\end{tabular}}}}%
    \put(0.3438521,0.09727831){\color[rgb]{0,0,0}\makebox(0,0)[lt]{\lineheight{1.25}\smash{\begin{tabular}[t]{l}$_{X_2}$\end{tabular}}}}%
    \put(-0.0009011,0.14350498){\color[rgb]{0,0,0}\makebox(0,0)[lt]{\lineheight{1.25}\smash{\begin{tabular}[t]{l}$I(L) = \sum_i \lambda_i$\end{tabular}}}}%
    \put(0,0){\includegraphics[width=\unitlength,page=3]{fromSkeinToStatedSkein.pdf}}%
    \put(0.82056071,0.02762859){\color[rgb]{0,0,0}\makebox(0,0)[lt]{\lineheight{1.25}\smash{\begin{tabular}[t]{l}$_{X_1}$\end{tabular}}}}%
    \put(0.90566368,0.02698137){\color[rgb]{0,0,0}\makebox(0,0)[lt]{\lineheight{1.25}\smash{\begin{tabular}[t]{l}$_{X_2}$\end{tabular}}}}%
    \put(0,0){\includegraphics[width=\unitlength,page=4]{fromSkeinToStatedSkein.pdf}}%
    \put(0.76001653,0.02702497){\color[rgb]{0,0,0}\makebox(0,0)[lt]{\lineheight{1.25}\smash{\begin{tabular}[t]{l}$_Y$\end{tabular}}}}%
    \put(0.74829309,0.07625522){\color[rgb]{0,0,0}\makebox(0,0)[lt]{\lineheight{1.25}\smash{\begin{tabular}[t]{l}$y_j$\end{tabular}}}}%
    \put(0.80712497,0.07302488){\color[rgb]{0,0,0}\makebox(0,0)[lt]{\lineheight{1.25}\smash{\begin{tabular}[t]{l}$x^{(1),k}$\end{tabular}}}}%
    \put(0.88571896,0.07279151){\color[rgb]{0,0,0}\makebox(0,0)[lt]{\lineheight{1.25}\smash{\begin{tabular}[t]{l}$x^{(2),l}$\end{tabular}}}}%
    \put(0.98455734,0.02790645){\color[rgb]{0,0,0}\makebox(0,0)[lt]{\lineheight{1.25}\smash{\begin{tabular}[t]{l}$=0$\end{tabular}}}}%
  \end{picture}%
\endgroup%

\end{center}
where all ribbon graphs are equal outside of the grey cube and $(y_j)$, $(x^{(1)}_k)$, $(x^{(2)}_l)$ are bases of $Y$, $X_1$, $X_2$ with dual bases $(y^j)$, $(x^{(1), k})$, $(x^{(2), l})$. For the second equality we used \eqref{skeinRelation} and for the third equality we used that $\textstyle \sum_i \lambda_i \, F_{\mathrm{RT}}(T_i) = 0$ by assumption. It is clear that $I$ is a morphism of algebras since both products are by stacking.
\end{proof}

\indent Note that since $I(L)$ is a stated ribbon graph without boundary points (and thus without states), $\mathrm{hol}^{\mathrm{st}}\bigl( I(L) \bigr)$ is just an element of $\mathcal{L}_{g,n}(H)$. Thus we can make the following definition:
\begin{defi}\label{defWilsonLoopMap}
The morphism $W = \mathrm{hol}^{\mathrm{st}} \circ I :\mathcal{S}_H(\Sigma_{g,n}^{\circ}) \to \mathcal{L}_{g,n}(H)$ is called the Wilson loop map.
\end{defi}
Recall that the subalgebra of $H$-invariant elements is
\[ \mathcal{L}_{g,n}^H(H) = \bigl\{ x \in \mathcal{L}_{g,n}(H) \, \big| \, \forall \, h \in H, \: \mathrm{coad}^r(h)(x) = \varepsilon(h)x \bigr\}. \]

\begin{teo}\label{thWilsonIso}
Assume that the ribbon Hopf algebra $H$ satisfies any one of the conditions in \eqref{assumptionsForSkein} and has semisimple category $H\text{-}\mathrm{mod}$. Then the Wilson loop map takes values in $\mathcal{L}^H_{g,n}(H)$ and provides an isomorphism of algebras $\mathcal{S}_H(\Sigma_{g,n}^{\circ}) \overset{\sim}{\to} \mathcal{L}^H_{g,n}(H)$.
\end{teo}
It is well-known that the ribbon category $\mathcal{C}$ of finite dimensional $U_q^{\mathrm{ad}}(\mathfrak{g})$-modules of type $1$ is semisimple. Therefore we can apply this theorem to $H=U_q^{\mathrm{ad}}(\mathfrak{g})$, meaning that the ribbon links with coupons defining $\mathcal{S}_H(\Sigma_{g,n}^{\circ})$ are colored by objects and morphisms in $\mathcal{C}$. Since $\mathcal{C}$ is equivalent to the category of finite dimensional modules of type $1$ over the simply connected quantum group $U_q(\mathfrak{g})$, we can indifferently take $H=U_q(\mathfrak{g})$.

\indent It is true that $W$ takes values in $\mathcal{L}^H_{g,n}(H)$ even if $H$-mod is not semisimple, but in general it is not an isomorphism. See \S \ref{sectionRemarksSemisimplicity} for more details on the non-semisimple case. 

\medskip

{\em For the rest of this section we assume that $H$-mod is semisimple.} In order to prove the theorem we first define two auxiliary maps $\mathcal{E}$ and $\mathcal{F}$. 

\smallskip

\indent Let $\mathrm{Irr}(H)$ be the set of irreducible finite-dimensional $H$-modules up to isomorphisms. Since $H\text{-}\mathrm{mod}$ is semisimple the matrix coefficients of the modules in $\mathrm{Irr}(H)$ form a basis of $H^{\circ}$, \textit{i.e.} $\textstyle H^{\circ} = \bigoplus_{X \in \mathrm{Irr}(H)} C(X)$. Hence
\begin{equation}\label{decompositionLgnMatCoeffs}
\mathcal{L}_{g,n}(H) = \bigoplus_{\mathbf{X} \in \mathrm{Irr}(H)^{2g+n}} C(X_1) \otimes C(X_2) \otimes \ldots \otimes C(X_{2g+n})
\end{equation}
where $\mathbf{X} = (X_1, \ldots, X_{2g+n})$. For any such $\mathbf{X} \in \mathrm{Irr}(H)^{2g+n}$, let
\[ \textstyle \mathcal{E}_{\mathbf{X}} : \mathrm{Hom}_k\!\left( \bigotimes_{i=1}^{2g+n} X_i \otimes X_i^*, k \right) \overset{\sim}{\longrightarrow} C(X_1) \otimes \ldots \otimes C(X_{2g+n}) \]
($\mathrm{Hom}_k$ denotes the space of all linear maps) be the isomorphism of vector spaces defined by
\[ \mathcal{E}_{\mathbf{X}}(f) = \sum_{\substack{i_1, \ldots, i_N\\j_1, \ldots, j_N}} f\!\left( x_{i_1}^{(1)} \otimes x^{(1),j_1} \otimes \ldots \otimes x_{i_N}^{(N)} \otimes x^{(N),j_N} \right) \, {_{X_1}\phi^{i_1}_{j_1}} \otimes \ldots \otimes {_{X_N}\phi^{i_N}_{j_N}} \]
where $N = 2g+n$, $\bigl( x_{i_l}^{(l)} \bigr)$ is a basis of $X_l$ with dual basis $\bigl( x^{(l), j_l} \bigr)$ and $_{X_l}\phi^{i_l}_{j_l}$ denotes the matrix coefficients of $X_l$ in this basis. When $f$ is $H$-linear we can represent the element $\mathcal{E}_{\mathbf{X}}(f)$ by using the diagrammatic calculus of \cite[\S 3]{FaitgHol}, as follows:
\smallskip
\begin{equation}\label{diagramInvariant}
\begingroup%
  \makeatletter%
  \providecommand\color[2][]{%
    \errmessage{(Inkscape) Color is used for the text in Inkscape, but the package 'color.sty' is not loaded}%
    \renewcommand\color[2][]{}%
  }%
  \providecommand\transparent[1]{%
    \errmessage{(Inkscape) Transparency is used (non-zero) for the text in Inkscape, but the package 'transparent.sty' is not loaded}%
    \renewcommand\transparent[1]{}%
  }%
  \providecommand\rotatebox[2]{#2}%
  \newcommand*\fsize{\dimexpr\f@size pt\relax}%
  \newcommand*\lineheight[1]{\fontsize{\fsize}{#1\fsize}\selectfont}%
  \ifx\svgwidth\undefined%
    \setlength{\unitlength}{416.38977499bp}%
    \ifx\svgscale\undefined%
      \relax%
    \else%
      \setlength{\unitlength}{\unitlength * \real{\svgscale}}%
    \fi%
  \else%
    \setlength{\unitlength}{\svgwidth}%
  \fi%
  \global\let\svgwidth\undefined%
  \global\let\svgscale\undefined%
  \makeatother%
  \begin{picture}(1,0.14846606)%
    \lineheight{1}%
    \setlength\tabcolsep{0pt}%
    \put(0,0){\includegraphics[width=\unitlength,page=1]{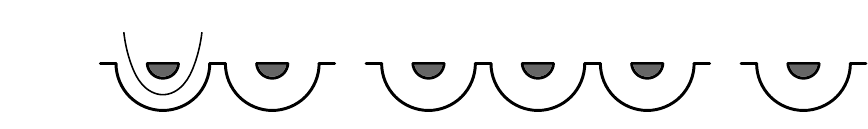}}%
    \put(0.15043853,0.09278534){\color[rgb]{0,0,0}\makebox(0,0)[lt]{\lineheight{1.25}\smash{\begin{tabular}[t]{l}$_{X_1}$\end{tabular}}}}%
    \put(0,0){\includegraphics[width=\unitlength,page=2]{diagramInvariant.pdf}}%
    \put(0.2765223,0.09278534){\color[rgb]{0,0,0}\makebox(0,0)[lt]{\lineheight{1.25}\smash{\begin{tabular}[t]{l}$_{X_2}$\end{tabular}}}}%
    \put(0,0){\includegraphics[width=\unitlength,page=3]{diagramInvariant.pdf}}%
    \put(0.45664203,0.09278534){\color[rgb]{0,0,0}\makebox(0,0)[lt]{\lineheight{1.25}\smash{\begin{tabular}[t]{l}$_{X_{2g-1}}$\end{tabular}}}}%
    \put(0,0){\includegraphics[width=\unitlength,page=4]{diagramInvariant.pdf}}%
    \put(0.58272578,0.09278534){\color[rgb]{0,0,0}\makebox(0,0)[lt]{\lineheight{1.25}\smash{\begin{tabular}[t]{l}$_{X_{2g}}$\end{tabular}}}}%
    \put(0,0){\includegraphics[width=\unitlength,page=5]{diagramInvariant.pdf}}%
    \put(0.70880963,0.09278534){\color[rgb]{0,0,0}\makebox(0,0)[lt]{\lineheight{1.25}\smash{\begin{tabular}[t]{l}$_{X_{2g+1}}$\end{tabular}}}}%
    \put(0,0){\includegraphics[width=\unitlength,page=6]{diagramInvariant.pdf}}%
    \put(0.88892928,0.09278534){\color[rgb]{0,0,0}\makebox(0,0)[lt]{\lineheight{1.25}\smash{\begin{tabular}[t]{l}$_{X_{2g+n}}$\end{tabular}}}}%
    \put(0,0){\includegraphics[width=\unitlength,page=7]{diagramInvariant.pdf}}%
    \put(0.5571836,0.12304783){\color[rgb]{0,0,0}\makebox(0,0)[lt]{\lineheight{1.25}\smash{\begin{tabular}[t]{l}$f$\end{tabular}}}}%
    \put(0.20643905,0.00245282){\color[rgb]{0,0,0}\makebox(0,0)[lt]{\lineheight{1.25}\smash{\begin{tabular}[t]{l}$B(1)$\end{tabular}}}}%
    \put(0.32710797,0.00240923){\color[rgb]{0,0,0}\makebox(0,0)[lt]{\lineheight{1.25}\smash{\begin{tabular}[t]{l}$A(1)$\end{tabular}}}}%
    \put(0.51540638,0.00223782){\color[rgb]{0,0,0}\makebox(0,0)[lt]{\lineheight{1.25}\smash{\begin{tabular}[t]{l}$B(g)$\end{tabular}}}}%
    \put(0.63634363,0.00232925){\color[rgb]{0,0,0}\makebox(0,0)[lt]{\lineheight{1.25}\smash{\begin{tabular}[t]{l}$A(g)$\end{tabular}}}}%
    \put(0.76481387,0.00222913){\color[rgb]{0,0,0}\makebox(0,0)[lt]{\lineheight{1.25}\smash{\begin{tabular}[t]{l}$M(g+1)$\end{tabular}}}}%
    \put(0.94557596,0.0021845){\color[rgb]{0,0,0}\makebox(0,0)[lt]{\lineheight{1.25}\smash{\begin{tabular}[t]{l}$M(g+n)$\end{tabular}}}}%
    \put(-0.0003669,0.0718649){\color[rgb]{0,0,0}\makebox(0,0)[lt]{\lineheight{1.25}\smash{\begin{tabular}[t]{l}$\mathcal{E}_{\mathbf{X}}(f) = $\end{tabular}}}}%
  \end{picture}%
\endgroup%

\end{equation}

\begin{lem}\label{lemmaIsoInvariantsHomH}
By restriction to $H$-linear maps, the $k$-linear isomorphism $\mathcal{E}_{\mathbf{X}}$ gives an isomorphism of vector spaces $\textstyle \mathrm{Hom}_H\!\left( \bigotimes_{i=1}^{2g+n} X_i \otimes X_i^*, k \right) \overset{\sim}{\longrightarrow} \bigl( C(X_1) \otimes \ldots \otimes C(X_{2g+n}) \bigr)^H$. It follows that we have an isomorphism of vector spaces
\[ \mathcal{E} : \bigoplus_{\mathbf{X} \in \mathrm{Irr}(H)^{2g+n}} \mathrm{Hom}_H\!\left( {\textstyle \bigotimes_{i=1}^{2g+n}} X_i \otimes X_i^*, k \right) \overset{\sim}{\longrightarrow} \mathcal{L}_{g,n}^H(H). \]
\end{lem}
\begin{proof}
Note that the action of $H$ on $\textstyle \bigotimes_{i=1}^{2g+n} X_i \otimes X_i^*$ is by coproduct
\[ h \cdot \bigl( x_1 \otimes \varphi_1 \otimes \ldots \otimes x_N \otimes \varphi_N \bigr) = \sum_{(h)} h_{(1)} \cdot x_1 \otimes h_{(2)} \cdot \varphi_1 \otimes \ldots \otimes h_{(2N-1)} \cdot x_N \otimes h_{(2N)} \cdot \varphi_N \]
while the action of $H$ on $k$ is by counit. We also recall that if $X$ is a $H$-module, then the action of $h \in H$ on $\varphi \in X^*$ is defined by $\langle h \cdot \varphi, x \rangle = \langle \varphi, S(h) \cdot x \rangle$ for all $x \in X$. Consider the right action $\mathrm{act}^r$ of $H$ on $\textstyle \mathrm{Hom}_k\!\left( \bigotimes_{i=1}^{2g+n} X_i \otimes X_i^*, k \right)$ defined by $\mathrm{act}^r(h)(f)\bigl( \boldsymbol{v} \bigr) = f(h \cdot \boldsymbol{v} )$ for all $h \in H$ and $\boldsymbol{v} \in \textstyle \bigotimes_{i=1}^{2g+n} X_i \otimes X_i^*$. Then $\mathcal{E}_{\mathbf{X}}$ intertwines the right actions $\mathrm{act}^r$ and $\mathrm{coad}^r$. This is due to the definition of $\mathrm{coad}^r$ (see \eqref{coadLgn} and \eqref{coadL01}) and to the following formulas, which come from \eqref{coregularOnMatrixCoeffs}:
\begin{equation}\label{balancingActions}
\begin{array}{l}
\sum_{i_l,j_l} \bigl( h \rhd {_{X_l}\phi^{i_l}_{j_l}} \bigr) \otimes x_{i_l} \otimes x^{j_l} = \sum_{i_l,j_l} {_{X_l}\phi^{i_l}_{j_l}} \otimes x_{i_l} \otimes S^{-1}(h) \cdot x^{j_l},\\[1.5em]
\sum_{i_l,j_l} \bigl( {_{X_l}\phi^{i_l}_{j_l}} \lhd h \bigr) \otimes x_{i_l} \otimes x^{j_l} = \sum_{i_l,j_l} {_{X_l}\phi^{i_l}_{j_l}} \otimes h \cdot x_{i_l} \otimes x^{j_l}.
\end{array}
\end{equation}
Now note that a linear form $\textstyle f : \bigotimes_{i=1}^{2g+n} X_i \otimes X_i^* \to k$ is $H$-linear if and only if $f$ is an invariant element under the action $\mathrm{act}^r$, \textit{i.e.} $\mathrm{act}^r(h)(f) = \varepsilon(h)f$ for all $h \in H$. Being an intertwiner between $\mathrm{act}^r$ and $\mathrm{coad}^r$, the linear isomorphism $\mathcal{E}_{\mathbf{X}}$ restricts to an isomorphism between the invariant elements for $\mathrm{act}^r$ in $\textstyle  \mathrm{Hom}_k\!\left( \bigotimes_{i=1}^{2g+n} X_i \otimes X_i^*, k \right)$ and the invariant elements for $\mathrm{coad}^r$ in $C(X_1) \otimes \ldots \otimes C(X_{2g+n})$.
\end{proof}

\indent Now, for $\mathbf{X} = (X_1, \ldots, X_{2g+n}) \in \mathrm{Irr}(H)^{2g+n}$, let
\[ \textstyle \mathcal{F}_{\mathbf{X}} : \mathrm{Hom}_H\!\left( \bigotimes_{i=1}^{2g+n} X_i \otimes X_i^*, k \right) \to \mathcal{S}_H(\Sigma_{g,n}^{\circ}) \]
defined by
\begin{center}
\begingroup%
  \makeatletter%
  \providecommand\color[2][]{%
    \errmessage{(Inkscape) Color is used for the text in Inkscape, but the package 'color.sty' is not loaded}%
    \renewcommand\color[2][]{}%
  }%
  \providecommand\transparent[1]{%
    \errmessage{(Inkscape) Transparency is used (non-zero) for the text in Inkscape, but the package 'transparent.sty' is not loaded}%
    \renewcommand\transparent[1]{}%
  }%
  \providecommand\rotatebox[2]{#2}%
  \newcommand*\fsize{\dimexpr\f@size pt\relax}%
  \newcommand*\lineheight[1]{\fontsize{\fsize}{#1\fsize}\selectfont}%
  \ifx\svgwidth\undefined%
    \setlength{\unitlength}{431.6425824bp}%
    \ifx\svgscale\undefined%
      \relax%
    \else%
      \setlength{\unitlength}{\unitlength * \real{\svgscale}}%
    \fi%
  \else%
    \setlength{\unitlength}{\svgwidth}%
  \fi%
  \global\let\svgwidth\undefined%
  \global\let\svgscale\undefined%
  \makeatother%
  \begin{picture}(1,0.17531391)%
    \lineheight{1}%
    \setlength\tabcolsep{0pt}%
    \put(0,0){\includegraphics[width=\unitlength,page=1]{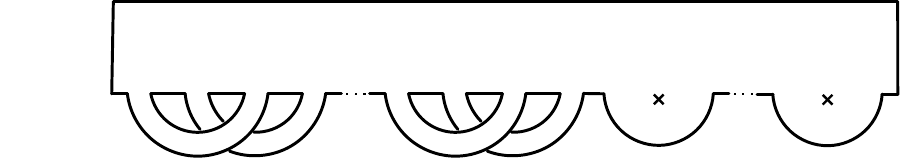}}%
    \put(-0.00049774,0.08874784){\color[rgb]{0,0,0}\makebox(0,0)[lt]{\lineheight{1.25}\smash{\begin{tabular}[t]{l}$\mathcal{F}_{\mathbf{X}}(f)=$\end{tabular}}}}%
    \put(0,0){\includegraphics[width=\unitlength,page=2]{defSectionDeW.pdf}}%
    \put(0.15587771,0.09917036){\color[rgb]{0,0,0}\makebox(0,0)[lt]{\lineheight{1.25}\smash{\begin{tabular}[t]{l}$_{X_1}$\end{tabular}}}}%
    \put(0.28014577,0.10115663){\color[rgb]{0,0,0}\makebox(0,0)[lt]{\lineheight{1.25}\smash{\begin{tabular}[t]{l}$_{X_2}$\end{tabular}}}}%
    \put(0,0){\includegraphics[width=\unitlength,page=3]{defSectionDeW.pdf}}%
    \put(0.44312977,0.09836345){\color[rgb]{0,0,0}\makebox(0,0)[lt]{\lineheight{1.25}\smash{\begin{tabular}[t]{l}$_{X_{2g-1}}$\end{tabular}}}}%
    \put(0.56739773,0.10034972){\color[rgb]{0,0,0}\makebox(0,0)[lt]{\lineheight{1.25}\smash{\begin{tabular}[t]{l}$_{X_{2g}}$\end{tabular}}}}%
    \put(0,0){\includegraphics[width=\unitlength,page=4]{defSectionDeW.pdf}}%
    \put(0.69645035,0.09993191){\color[rgb]{0,0,0}\makebox(0,0)[lt]{\lineheight{1.25}\smash{\begin{tabular}[t]{l}$_{X_{2g+1}}$\end{tabular}}}}%
    \put(0,0){\includegraphics[width=\unitlength,page=5]{defSectionDeW.pdf}}%
    \put(0.88428969,0.09971253){\color[rgb]{0,0,0}\makebox(0,0)[lt]{\lineheight{1.25}\smash{\begin{tabular}[t]{l}$_{X_{2g+n}}$\end{tabular}}}}%
    \put(0,0){\includegraphics[width=\unitlength,page=6]{defSectionDeW.pdf}}%
    \put(0.55873271,0.12590622){\color[rgb]{0,0,0}\makebox(0,0)[lt]{\lineheight{1.25}\smash{\begin{tabular}[t]{l}$f$\end{tabular}}}}%
  \end{picture}%
\endgroup%

\end{center}
The map $\mathcal{F}_{\mathbf{X}}$ is $k$-linear thanks to the skein relations \eqref{usualSkeinRelation}. Let
\[ \mathcal{F} : \bigoplus_{\mathbf{X} \in \mathrm{Irr}(H)^{2g+n}} \mathrm{Hom}_H\!\left( {\textstyle \bigotimes_{i=1}^{2g+n}} X_i \otimes X_i^*, k \right) \to \mathcal{S}_H(\Sigma_{g,n}^{\circ}) \]
be the sum of the linear maps $\mathcal{F}_{\mathbf{X}}$.

\begin{lem}\label{lemmaSurjectionOnSkeinAlg}
The linear map $\mathcal{F}$ is surjective.
\end{lem}
\begin{proof}
As in the proof of Lemma \ref{lemmeXiSurjective} we observe that every element of $\mathcal{S}_H(\Sigma_{g,n}^{\circ})$ is a linear combination of links of the form
\begin{center}
\begingroup%
  \makeatletter%
  \providecommand\color[2][]{%
    \errmessage{(Inkscape) Color is used for the text in Inkscape, but the package 'color.sty' is not loaded}%
    \renewcommand\color[2][]{}%
  }%
  \providecommand\transparent[1]{%
    \errmessage{(Inkscape) Transparency is used (non-zero) for the text in Inkscape, but the package 'transparent.sty' is not loaded}%
    \renewcommand\transparent[1]{}%
  }%
  \providecommand\rotatebox[2]{#2}%
  \newcommand*\fsize{\dimexpr\f@size pt\relax}%
  \newcommand*\lineheight[1]{\fontsize{\fsize}{#1\fsize}\selectfont}%
  \ifx\svgwidth\undefined%
    \setlength{\unitlength}{411.66567055bp}%
    \ifx\svgscale\undefined%
      \relax%
    \else%
      \setlength{\unitlength}{\unitlength * \real{\svgscale}}%
    \fi%
  \else%
    \setlength{\unitlength}{\svgwidth}%
  \fi%
  \global\let\svgwidth\undefined%
  \global\let\svgscale\undefined%
  \makeatother%
  \begin{picture}(1,0.18382137)%
    \lineheight{1}%
    \setlength\tabcolsep{0pt}%
    \put(-0.0005219,0.09547778){\color[rgb]{0,0,0}\makebox(0,0)[lt]{\lineheight{1.25}\smash{\begin{tabular}[t]{l}$L_{\mathbf{X}} =$\end{tabular}}}}%
    \put(0,0){\includegraphics[width=\unitlength,page=1]{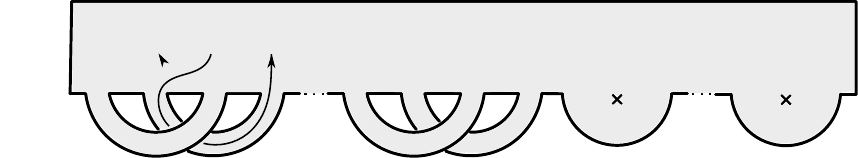}}%
    \put(0.11491492,0.1039828){\color[rgb]{0,0,0}\makebox(0,0)[lt]{\lineheight{1.25}\smash{\begin{tabular}[t]{l}$_{X_1}$\end{tabular}}}}%
    \put(0.24521335,0.10606546){\color[rgb]{0,0,0}\makebox(0,0)[lt]{\lineheight{1.25}\smash{\begin{tabular}[t]{l}$_{X_2}$\end{tabular}}}}%
    \put(0,0){\includegraphics[width=\unitlength,page=2]{standardLink.pdf}}%
    \put(0.41610648,0.10313673){\color[rgb]{0,0,0}\makebox(0,0)[lt]{\lineheight{1.25}\smash{\begin{tabular}[t]{l}$_{X_{2g-1}}$\end{tabular}}}}%
    \put(0.5464048,0.10521939){\color[rgb]{0,0,0}\makebox(0,0)[lt]{\lineheight{1.25}\smash{\begin{tabular}[t]{l}$_{X_{2g}}$\end{tabular}}}}%
    \put(0,0){\includegraphics[width=\unitlength,page=3]{standardLink.pdf}}%
    \put(0.68171996,0.1047813){\color[rgb]{0,0,0}\makebox(0,0)[lt]{\lineheight{1.25}\smash{\begin{tabular}[t]{l}$_{X_{2g+1}}$\end{tabular}}}}%
    \put(0,0){\includegraphics[width=\unitlength,page=4]{standardLink.pdf}}%
    \put(0.87867458,0.10455128){\color[rgb]{0,0,0}\makebox(0,0)[lt]{\lineheight{1.25}\smash{\begin{tabular}[t]{l}$_{X_{2g+n}}$\end{tabular}}}}%
    \put(0,0){\includegraphics[width=\unitlength,page=5]{standardLink.pdf}}%
    \put(0.53718447,0.13433898){\color[rgb]{0,0,0}\makebox(0,0)[lt]{\lineheight{1.25}\smash{\begin{tabular}[t]{l}$T_{\mathbf{X}}$\end{tabular}}}}%
  \end{picture}%
\endgroup%

\end{center}
where $T_{\mathbf{X}}$ is some oriented and colored $(4g+2n, 0)$-ribbon graph in $[0,1]^3$. We can assume that the colors $\mathbf{X} = (X_1, \ldots, X_{2g+n})$ are irreducible. Indeed, if a strand is colored by a non-irreducible module $V$, then write $\textstyle V = \bigoplus_l S_l$ where the $S_l$ are irreducible. We can introduce coupons containing the injections $I_l : S_l \to V$ and projections $\pi_l : V \to S_l$ thanks to the relation $\textstyle \mathrm{id}_V = \sum_l I_l \circ \pi_l$. By sliding these coupons along the strand we get a sum of links where the strands under consideration are colored by the modules $S_l$.
\\Thanks to the skein relations \eqref{usualSkeinRelation} we can replace the ribbon graph $T_{\mathbf{X}}$ by a coupon colored by $\textstyle F_{\mathrm{RT}}(T_{\mathbf{X}}) \in \mathrm{Hom}_H\!\left( \bigotimes_{i=1}^{2g+n} X_i \otimes X_i^*, k \right)$ and it follows that $L_{\mathbf{X}} = \mathcal{F}_{\mathbf{X}}\bigl( F_{\mathrm{RT}}(T_{\mathbf{X}}) \bigr)$.
\end{proof}

\begin{proof}[Proof of Theorem \ref{thWilsonIso}]
By definition of $\mathrm{hol}$ \cite[Def. 4.2]{FaitgHol} and \eqref{diagramInvariant}, we have a commutative diagram
\[ \xymatrix{
\bigoplus_{\mathbf{X} \in \mathrm{Irr}(H)^{2g+n}} \mathrm{Hom}_H\!\left( \bigotimes_{i=1}^{2g+n} X_i \otimes X_i^*, k \right) \ar[r]^-{\mathcal{E}} \ar[d]_{\mathcal{F}} & \mathcal{L}_{g,n}^H(H) \ar@{^{(}->}[d]\\
\mathcal{S}_H(\Sigma_{g,n}^{\circ}) \ar[r]_W & \mathcal{L}_{g,n}(H)
} \]
The surjectivity of $\mathcal{F}$ (Lemma \ref{lemmaSurjectionOnSkeinAlg}) implies that $W$ takes values in $\mathcal{L}_{g,n}^H(H)$. Since $\mathcal{E}$ is an isomorphism (Lemma \ref{lemmaIsoInvariantsHomH}), $\mathcal{F}$ is necessarily injective, so it is an isomorphism of vector spaces. It follows that $W$ is an isomorphism between $\mathcal{S}_H(\Sigma_{g,n}^{\circ})$ and $\mathcal{L}^H_{g,n}(H)$.
\end{proof}

Recall the map $I$ introduced before Lemma \ref{LemmaMapIWellDef}. Our results imply the following important fact, which is not obvious despite seeming completely natural:
\begin{cor}\label{coroISkeinInjective}
Under the assumptions of Theorem \ref{thWilsonIso}, the morphism of algebras $I : \mathcal{S}_H(\Sigma_{g,n}^{\circ}) \to \mathcal{S}_H^{\mathrm{st}}(\Sigma_{g,n}^{\circ, \bullet})$ is injective.
\end{cor}
\begin{proof}
By the definition of $W$ (Definition \ref{defWilsonLoopMap}), Theorem \ref{ThStatedHolIso} and Theorem \ref{thWilsonIso} we have a commutative diagram where the rows are isomorphisms:
\[ \xymatrix@R=.6em{
\mathcal{S}_H^{\mathrm{st}}(\Sigma_{g,n}^{\circ,\bullet}) \ar[rr]_{\mathrm{hol}^{\mathrm{st}}}^{\sim} & & \mathcal{L}_{g,n}(H)\\
 & \circlearrowleft & \\
\mathcal{S}_H(\Sigma_{g,n}^{\circ}) \ar[rr]_W^{\sim} \ar[uu]^I & & \mathcal{L}_{g,n}^H(H) \ar@{^{(}->}[uu]_{\text{(subalgebra)}}
} \]
It immediately follows that $I$ is injective.
\end{proof}

\noindent This corollary means that the skein algebra $\mathcal{S}_H(\Sigma_{g,n}^{\circ})$ is isomorphic to the subalgebra of ribbon links (with coupons) in the stated skein algebra $\mathcal{S}_H^{\mathrm{st}}(\Sigma_{g,n}^{\circ, \bullet})$.

\subsection{Remarks on semisimplicity}\label{sectionRemarksSemisimplicity}

\indent If the category $H\text{-}\mathrm{mod}$ is not semisimple:
\begin{itemize}
\item the matrix coefficients of irreducible modules do not form a generating family of $H^{\circ}$,
\item the matrix coefficients of indecomposable modules form a generating family of $H^{\circ}$ but not a free family (even if one restricts to projective modules). \\Hence $\textstyle H^{\circ} = \sum_{X \in \mathrm{Ind}(H)} C(X)$, where $\mathrm{Ind}(H)$ denotes the set of indecomposable $H$-modules up to isomorphisms.
\end{itemize}
Consequently the decomposition \eqref{decompositionLgnMatCoeffs} is no longer true, and instead one has
\[ \mathcal{L}_{g,n}(H) = \sum_{\mathbf{X} \in \mathrm{Ind}(H)^{2g+n}} C(X_1) \otimes C(X_2) \otimes \ldots \otimes C(X_{2g+n}). \]
So there is a surjection $\textstyle \bigoplus_{\mathbf{X} \in \mathrm{Ind}(H)^{2g+n}} \mathrm{Hom}_k\!\left( \bigotimes_{i=1}^{2g+n} X_i \otimes X_i^*, k \right) \twoheadrightarrow \mathcal{L}_{g,n}(H)$ and a commutative diagram
\[ \xymatrix{
\bigoplus_{\mathbf{X} \in \mathrm{Ind}(H)^{2g+n}} \mathrm{Hom}_H\!\left( \bigotimes_{i=1}^{2g+n} X_i \otimes X_i^*, k \right) \ar[r]^-{\mathcal{E}} \ar[d]_{\mathcal{F}} & \mathcal{L}_{g,n}^H(H) \ar@{^{(}->}[d]\\
\mathcal{S}_H(\Sigma_{g,n}^{\circ}) \ar[r]_W & \mathcal{L}_{g,n}(H)
} \]
Moreover $\mathcal{F}$ is still surjective, thus $W$ takes values in $\mathcal{L}_{g,n}^H(H)$. But in general nothing can be said about the surjectivity of $\mathcal{E}$. So when $H$ is not semisimple, Theorem \ref{thWilsonIso} will in general not be true:
\[ \mathrm{im}(W) \subsetneq \mathcal{L}_{g,n}^H(H). \]
Non-semisimplicity arises notably at roots of unity. For instance take $(g,n)=(0,1)$ and $H = \overline{U}_{\epsilon} = \overline{U}_{\epsilon}(\mathfrak{sl}_2)$, the restricted quantum group of $\mathfrak{sl}(2)$ at $\epsilon = e^{i\pi/p}$ with $p \geq 2$ (see e.g. \cite[\S 3]{FGST}). We have $\mathcal{L}_{0,1}^{\overline{U}_{\epsilon}} \cong \mathcal{Z}(\overline{U}_{\epsilon})$ (see e.g. \cite[Th. 3.7]{Faitg3}) and it is known that $\dim\bigl( \mathcal{Z}(\overline{U}_{\epsilon}) \bigr) = 3p-1$ \cite[Prop. 4.4.4]{FGST}. On the other hand one can check that $\dim\bigl( \mathrm{im}(W) \bigr) = 2p$, showing that the inclusion is indeed strict. 

In \cite[\S 4.4]{FaitgHol} a generalization of Wilson loops was defined in order to produce more $H$-invariant elements. The idea is roughly to use other symmetric linear forms than the traces; for $\Sigma^{\circ}_{0,1}$ these generalized Wilson loops recover the whole of $\mathcal{L}^H_{0,1}(H)$.
\smallskip

Another interesting situation where non-semisimplicity arises is when working over the ring $A=\mc[q^{1/D},q^{-1/D}]$ and specializing $q$ to a root of unity. To define an integral version $\mathcal{L}_{g,n}^A$ of $\mathcal{L}_{g,n}$ defined over $A$, such that $\mathcal{L}_{g,n}^A\otimes_A \mc(q^{1/D}) = \mathcal{L}_{g,n}$, one considers Lusztig's restricted quantum group $U_{\mathcal{A}}^{\mathrm{res}}(\mathfrak{g})\subset U_q^{\mathrm{ad}}(\mathfrak{g})$, which is defined over $\mathcal{A}:=\mc[q,q^{-1}]$, and its restricted dual $\mathcal{O}_{\mathcal{A}}(G)$, associated to the category $\mathcal{C}_{\mathcal{A}}$ of finite rank $U_{\mathcal{A}}^{\mathrm{res}}(\mathfrak{g})$-modules of type $1$. The algebra $\mathcal{L}_{g,n}^A$ in the case $g=0$ is described in \cite{BR1}.

\section{The case of surfaces without boundary}\label{qredsection} By Theorem \ref{thWilsonIso}, we know that if the ribbon Hopf algebra $H$ has semisimple category $H\text{-}\mathrm{mod}$, the Wilson loop map provides an isomorphism of algebras $\mathcal{S}_H(\Sigma_{g,n}^{\circ}) \overset{\sim}{\to} \mathcal{L}^H_{g,n}(H)$. Here we address the question of defining algebras which are related to the surface $\Sigma_{g,n}$ obtained by closing $\Sigma_{g,n}^{\circ}$, and derived from $\mathcal{S}_H(\Sigma_{g,n}^{\circ}) $ and $\mathcal{L}^H_{g,n}(H)$ respectively. We stress that the construction and results  presented below work for any $n\geq0$, including in particular the closed surface of genus $g$.

\subsection{Quantum moment maps and quantum reduction}\label{prelqrsection}
Let $H$ be a Hopf algebra over a field $k$, and $H'\subset H$ a subalgebra and right $H$-coideal (so $\Delta(H')\subset H'\otimes H$). On $H$ we have the right adjoint action, 
$$\mathrm{ad}^r(h)(g) = \sum_{(h)} S(h_{(1)})gh_{(2)}, \quad (h,g\in H),$$
and we assume that $H'$ is a stable subspace under this action. Let $A$ be a left $H$-module algebra. We denote by $\smallblacktriangle$ the action of $H$ on $A$. 
\begin{defi}\label{defQMM} A morphism of algebras $\mu\colon H'\ra A$ is a quantum moment map \emph{(QMM)} if for every $a\in A$ and $h'\in H'$ it satisfies
$$\sum_{(h')}(h_{(2)}'\smallblacktriangle a) \mu(h_{(1)}')= \mu(h') a.$$
\end{defi}
This definition is implicit in \cite{Lu}, and explicit in \cite[Section 1.5]{VV} (we use an opposite coproduct, which is suited to our conventions). The following facts follow readily from the definitions and are properties we can expect of a quantum analog of moment map (see \cite{Lu}, Theorem 3.10).
\begin{lem}\label{lemQMM} Assume we are given a QMM $\mu\colon H'\ra A$. Then:
\smallskip

\begin{itemize}
\item[$(i)$] \emph{(Equivariance)} For every $h\in H$ and $h'\in H'$ we have $$\mu\bigl(\mathrm{ad}^r(h)(h')\bigr) = S(h)\smallblacktriangle \mu(h').$$

\item[$(ii)$] \emph{(Recovering the action)}  For every $a\in A$ and $h\in H$ such that $(S\otimes id)\Delta(h)\in H'\otimes H'$ (e.g. when $H'=H$), we have $$\textstyle \sum_{(h)} \mu(S(h_{(1)})) a \mu(h_{(2)}) = S(h)\smallblacktriangle a.$$
\end{itemize}
\end{lem} 
Given a morphism of algebras $\chi\colon H'\ra k$, put
\begin{equation}\label{defiIchi}
I_\chi := A\mu({\rm Ker}(\chi)).
\end{equation}
This is a left ideal of $A$ (but in general not a two-sided ideal). Since ${\rm Ker}(\chi) = \{h'-\chi(h')1, h'\in H'\}$, and $A$ is a $H$-module algebra and $H'$ is stable under $\mathrm{ad}^r$, $(i)$ in the lemma implies that $I_\chi$ is $H'$-stable. Therefore we can consider the left $(A/I_\chi)$-submodule $(A/I_\chi)^{H'}\subset A/I_\chi$ formed by the $H'$-invariant elements. Define the map $\varphi_\chi\colon H'\ra H$, $\textstyle h'\mapsto \sum_{(h')} \chi(h_{(1)}')h_{(2)}'$, and denote by $\varepsilon_{\vert H'}\colon H'\ra k$ the restriction of the counit of $H$.
\begin{prop}\label{QRalg} The product of $A$ descends to the space $(A/I_\chi)^{H'}$ and gives it a structure of algebra. Moreover, if ${\rm Im}(\varphi_\chi)\subset H'$ and $\varphi_\chi({\rm Ker}(\chi))= {\rm Ker}(\varepsilon_{\vert H'})$, then the algebra $(A/I_\chi)^{H'}$ is isomorphic to $End_A(A/I_\chi)^{op}$.
\end{prop}
Both claims follow from arguments in Proposition 1.5.2 of \cite{VV}, to which we refer. The first claim is also detailed in Proposition 3.12 of \cite{GJS}. For completeness let us explain it. The point is to observe that a coset $a+I_\chi \in A/I_\chi$ is $H'$-invariant if and only if ${\rm Ker}(\varepsilon_{\vert H'}) \smallblacktriangle a \subset I_\chi$. But ${\rm Ker}(\varepsilon_{\vert H'})$ obviously contains the elements $\textstyle \sum_{(h')} \chi(h_{(1)}')h_{(2)}'\in H$, where $h'\in {\rm Ker}(\chi)$, and
\begin{align*}\left(\sum_{(h')} \chi(h_{(1)}')h_{(2)}'\right)\smallblacktriangle a 
& = \sum_{(h')} (h_{(2)}'\smallblacktriangle a)(\mu(h'_{(1)})) + \sum_{(h')} (h_{(2)}'\smallblacktriangle a)(\chi(h'_{(1)}) - \mu(h_{(1)}'))
\\ & = \sum_{(h')} (h_{(2)}'\smallblacktriangle a)\mu(h_{(1)}') \quad {\rm modulo}\ I_\chi \\
& = \mu(h') a \quad {\rm modulo}\ I_\chi.
\end{align*}
The third equality follows from the QMM equation in Definition \ref{defQMM}. Therefore, for every $h'\in {\rm Ker}(\chi)$, $\mu(h') a \in {\rm Ker}(\varepsilon_{\vert H'}) \smallblacktriangle a + I_\chi \subset I_\chi$, whence $I_\chi a \subset I_\chi$, which shows $I_\chi$ is stable under right multiplication by elements which are $H'$-invariant modulo $I_\chi$. It follows that the product of $A$ descends to $(A/I_\chi)^{H'}$.  

\medskip

From \cite{VV} we take:


\begin{defi}\label{defQuantumReduction} The algebra $A/\!\!/_{\!\chi} H' := (A/I_\chi)^{H'}$ is the quantum reduction of $A$ along the character $\chi$.
\end{defi}
By the defining property of $\mu$, it is immediate that given an $A$-module $V$ the subspace
$$V^{H',\chi} :=\{v\in V, \mu(h') \cdot v = \chi(h')v, \forall h'\in H'\}$$
is a $(A/\!\!/_{\!\chi} H')$-module. Therefore we have a functor $A$-Mod$\ra (A/\!\!/_{\!\chi} H')$-Mod.
\medskip

Let us observe the following fact in the case of the quantum reduction by the counit $\varepsilon$, which will be important in Section \ref{Lgnqrsection}. Consider the canonical quotient map of $H$-modules, $p \colon A \ra  A/I_\varepsilon$. The restriction $\pi$ of $p$ to $A^{H'}$ has image contained in $A/\!\!/_{\!\varepsilon} H'$, and as explained after Proposition \ref{QRalg} it is a morphism of algebras
\begin{equation}\label{defPi}
\pi\colon A^{H'} \ra  A/\!\!/_{\!\varepsilon} H'.
\end{equation}
Denote by $Z(H')$ the center of $H'$.
\begin{defi}\label{defiseparate} We will say that $Z(H')$ separates the simple types if the following property holds: for any finite collection of non isomorphic and non-trivial simple $H$-modules $X_1,\ldots,X_k$, there exists an element $z\in Z(H')$ such that $\varepsilon(z)=0$ and $z\smallblacktriangle v  = z_i v$ for every $v\in X_i$, where $z_i \in k$ satisfy $z_i\ne z_j$ for $i\ne j$. 
\end{defi}
\begin{lem} \label{lemsepare} Assume that the $H$-module $A$ is completely reducible and $Z(H')$ separates the simple types. Then the map  $\pi$ is surjective.
\end{lem}
\begin{proof}
Let $a+I_\varepsilon \in A/\!\!/_{\!\varepsilon} H'$. By definition we have $h'\smallblacktriangle a - \varepsilon(h')a \in I_\varepsilon$, for every $h'\in H'$. Since $A$ is completely reducible, we can decompose $a$ in the form $a=\textstyle \sum_{i=0}^k a_i$ with $a_i$ the isotypical component of $a$ of type $X_i$, $i\in \{1,\ldots,k\}$, and $a_0$ the isotypical component of $a$ of trivial type, ie. the $H$-invariant component of $a$. Then
$$h'\smallblacktriangle a - \varepsilon(h')a = \sum_{i=1}^k  (h'\smallblacktriangle a_i - \varepsilon(h')a_i).$$ 
Let us apply this to an element $z\in Z(H')$, separating the components of type $X_i$ as in Definition \ref{defiseparate}. For every integer $s\geq 1$ we have
$$z^s\smallblacktriangle a - \varepsilon(z^s)a = \sum_{i=1}^k  z_i^s a_i.$$
Therefore $\textstyle \sum_{i=1}^k  z_i^s a_i \in I_\varepsilon$ for every $s\geq 1$, with $z_i\ne z_j$ for $i\ne j$. Because the Vandermonde matrix $(z_i^s)$ is invertible for $1\leq s\leq k$, we find $a_i \in I_\varepsilon$ for every $i\in \{1,\ldots,k\}$. Therefore $a+I_\varepsilon = a_0 +I_\varepsilon$, showing that $\pi$ is surjective. \end{proof}

\subsection{A quantum moment map for $\mathcal{L}_{g,n}(H)$} \label{QMMLgn} Let $H$ be a ribbon Hopf algebra with an invertible antipode; we denote the ribbon element of $H$ by $v$.

\medskip

\noindent{\em Assumption: the morphism $\Phi_{0,1} : \mathcal{L}_{0,1}(H) \to H$ from \eqref{RSDmap} is injective.}

\medskip

\noindent When $\Phi_{0,1}$ is an isomorphism, $H$ is called {\em factorizable}. In general, let $H'$ be the image of the morphism $\Phi_{0,1}$. Then $H'$ is a subalgebra of $H$ such that $H' \subset H^{\mathrm{lf}}$. The quantum group $U_q^{\mathrm{ad}}(\mathfrak{g})$ satisfies the above assumption, up to a slight adaptation: in that case $H' = U_q(\mathfrak{g})^{\mathrm{lf}}$ by \cite[Th. 3]{Bau1}.
\smallskip

Let us quote the following well-known fact:

\begin{lem}
$H'$ is a right coideal (\textit{i.e.} $\Delta(H') \subset H' \otimes H$) and is stable by the right adjoint action $\mathrm{ad}^r$ of $H$.
\end{lem}
\begin{proof}
By definition an element of $H'$ has the form $\textstyle \sum_{(R^1),(R^2)} \varphi\bigl(R^1_{(1)}R^2_{(2)}\bigr)R^1_{(2)}R^2_{(1)}$ for some $\varphi \in H^{\circ}$. Let us compute the coproduct of such an element:
\begin{align*}
&\sum_{(R^1),(R^2), (R^1_{(2)}), (R^2_{(1)})} \varphi\bigl(R^1_{(1)}R^2_{(2)}\bigr) \, (R^1_{(2)})_{(1)} \, (R^2_{(1)})_{(1)} \otimes (R^1_{(2)})_{(2)} \, (R^2_{(1)})_{(2)}\\
=\:& \sum_{(R^1),\ldots, (R^4)} \varphi\bigl(R^1_{(1)}R^2_{(1)} R^3_{(2)} R^4_{(2)} \bigr) \, R^2_{(2)}R^3_{(1)} \otimes R^1_{(2)}R^4_{(1)}\\
=\:& \sum_{(R^1),(R^4),(\varphi)} \varphi_{(1)}\bigl( R^1_{(1)} \bigr) \, \varphi_{(3)}\bigl( R^4_{(2)} \bigr) \, \Phi_{0,1}(\varphi_{(2)}) \otimes R^1_{(2)}R^4_{(1)} \in H' \otimes H.
\end{align*}
The first equality is by quasi-triangularity of $H$ while the second is by definition of $\Phi_{0,1}$ and of the coproduct in $H^{\circ}$.  The last claim of the lemma is obvious, since $\Phi_{0,1} : (\mathcal{L}_{0,1}(H), \mathrm{coad}^r) \to (H, \mathrm{ad}^r)$ is $H$-linear.
\end{proof}
\begin{remark} For any Hopf algebra $H$, the subalgebra $H^{\mathrm{lf}}$ is always a right coideal subalgebra of $H$ (see \cite{KLNY}, Theorem 1).
\end{remark}
Recall the matrices $\overset{V}{B}(i), \overset{V}{A}(i), \overset{V}{M}(j) \in \mathcal{L}_{g,n}(H) \otimes \mathrm{End}(V)$ from \eqref{matricesABMgn}, where $1 \leq i \leq g$ and $g+1 \leq j \leq g+n$ and $V$ is any finite-dimensional $H$-module. Let
\begin{equation}\label{defCi}\overset{V}{C}(i) = v^2_V \, \overset{V}{B}(i) \overset{V}{A}(i){^{-1}} \overset{V}{B}(i){^{-1}} \overset{V}{A}(i)
\end{equation}
where $v^2_V \in \mathrm{End}(V)$ is the representation of $v^2$ on $V$, where $v$ is the ribbon element of $H$. It is well-known that the matrices $\overset{V}{B}(i), \overset{V}{A}(i), \overset{V}{M}(j)$ are invertible (see e.g. \cite[Lem. 6.1.2]{FaitgThesis}), so this formula makes sense.

Now we can construct the building blocks which will be used to define the quantum moment map in \eqref{defmutotal}. Denote by $_V\phi^k_l \in H^{\circ}$ the matrix coefficients of $V$ and by $\overset{V}{C}(i){^k_l}$ the coefficients of $\overset{V}{C}(i)$, both taken in some basis of $V$. Consider the linear map given by
\[ \fonc{\mu^{(i)}}{H'}{\mathcal{L}_{g,n}(H)}{\Phi_{0,1}\bigl(_V\phi^k_l\bigr)}{\overset{V}{C}(i){^k_l}} \]
for all $V$, $k$ and $l$. It is well-defined since by assumption $\Phi_{0,1}$ is injective.
\begin{lem}\label{muilem} $\mu^{(i)}$ is a morphism of algebras. Equivalently, we have the set of fusion relations
\begin{equation}\label{fusionmui}
\overset{V \otimes W}{C(i)} = \overset{V}{C(i)}_1 \, (R')_{V,W} \, \overset{W}{C(i)}_2 \, (R')_{V,W}^{-1}.
\end{equation} 
\end{lem}
\begin{proof} This follows from the definition of $\Ll_{0,1}(H)$ (see \eqref{fusionrelation}), the fact that $\Phi_{0,1}$ is a morphism of algebras, and finally Lemma \ref{fusionforXY}. \end{proof}

\begin{remark}\label{normliftholonomy} Relations similar to \eqref{fusionmui} are satisfied for any collection of matrices which like $C(i)$ ``lifts'' in a suitable sense a simple closed curve in $\Sigma_{g,n} \!\setminus\!D$. For complete details on this we refer to \cite[Prop. 5.3.14]{FaitgThesis} (see also the seminal work \cite[\S 9.1]{ASduke}, but with different conventions). As shows the proof of Lemma \ref{fusionforXY}, the definition of such a lift needs a normalization factor which is a suitable power of $v_V$.  
\end{remark}

\begin{defi}\label{defMatriceC}
For any finite-dimensional $H$-module $V$, we define
\begin{align*}\overset{V}{C} & = \overset{V}{C}(1) \ldots \overset{V}{C}(g) \overset{V}{M}(g+1) \ldots \overset{V}{M}(g+n) \\ &  = v^{2g}_V \ \prod_{i=1}^g\ [\overset{V}{B}(i),\overset{V}{A}(i){^{-1}}]\ \prod_{j=1}^n\ \overset{V}{M}(g+j)\in \mathcal{L}_{g,n}(H) \otimes \mathrm{End}(V).
\end{align*}
\end{defi}
\noindent Note that if $b_i,a_i,m_j$ ($1 \leq i \leq g$, $g+1 \leq j \leq g+n$) are the standard generators of $\pi_1(\Sigma_{g,n} \!\setminus\! D)$, then $\partial(\Sigma_{g,n} \!\setminus\! D) = b_1 a_1^{-1} b_1^{-1} a_1 \ldots b_g a_g^{-1} b_g^{-1} a_g \, m_{g+1} \ldots m_{g+n}$. So the matrices $\overset{V}{C}$ are in some sense a ``lift'' of $\partial(\Sigma_{g,n} \!\setminus\! D)$.
\medskip

Consider the linear map given by
\begin{equation}\label{defmutotal}
\fonc{\mu}{H'}{\mathcal{L}_{g,n}(H)}{\Phi_{0,1}\bigl(_V\phi^k_l\bigr)}{\overset{V}{C}{^k_l}}
\end{equation}
for all $V$, $k$ and $l$. In a component-free form, $\mu$ is given by
\begin{equation}\label{componentFreeDefOfMu}
\sum_{(R^1), (R^2)} \mu\bigl( R^1_{(2)} R^2_{(1)} \bigr) \otimes \bigl( R^1_{(1)} R^2_{(2)} \bigr)_V = \overset{V}{C}.
\end{equation}

\begin{lem}\label{mulem} $\mu$ is a morphism of algebras.
\end{lem}
\begin{proof} Since $\Phi_{0,1}$ is a morphism of algebras it suffices to show that the collection of matrices $C$ satisfies the fusion relation \eqref{fusionrelation}. For ease of notation let us write $\overset{V}{C}(j) = \overset{V}{M}(j)$ for $g+1 \leq j \leq g+n$, so that $\overset{V}{C} = \overset{V}{C}(1) \ldots \overset{V}{C}(g+n)$. We saw in \eqref{fusionmui} that the collection of matrices $\overset{V}{C}(i)$ satisfies the fusion relation for all $1 \leq i \leq g+n$. Straightforward computations using Proposition \ref{presentationLgn} reveal that we also have
\[ R_{V,W} \, \overset{V}{C}(j)_1 \, R{_{V,W}^{-1}} \, \overset{W}{C}(i)_2 = \overset{W}{C}(i)_2 \, R_{V,W} \, \overset{V}{C}(j)_1 \, R{_{V,W}^{-1}} \quad \quad \text{for } 1 \leq j < i \leq g+n. \] The fusion relation for $\overset{V}{C}$ follows from all these relations, arguing like in Lemma \ref{fusionforXY}. \end{proof}

\indent The following commutation relations (written with the notations introduced before Proposition \ref{presentationL10}) for the collection of matrices $\overset{V}{C}$ will play a key role in the proof of the next theorem:
\begin{lem}\label{lemmaCommutationC}
Let $V,W$ be finite-dimensional $H$-modules and let $\overset{W}{X}(i)$ be $\overset{W}{B}(i)$, $\overset{W}{A}(i)$ or $\overset{W}{M}(i)$. We have
\[ R_{V,W} \, \overset{V}{C}_1 \, (R')_{V,W} \, \overset{W}{X}(i)_2 = \overset{W}{X}(i)_2 \, R_{V,W} \, \overset{V}{C}_1 \, (R')_{V,W} \]
in $\mathcal{L}_{g,n}(H) \otimes \mathrm{End}(V) \otimes \mathrm{End}(W)$.
\end{lem}
\begin{proof}
Recall the notation $\overset{V}{C}(j) = \overset{V}{M}(j)$ for $g+1 \leq j \leq g+n$ from the proof of Lemma \ref{mulem}. Straightforward computations using Proposition \ref{presentationLgn} reveal that
\begin{align}
(R'){_{V,W}^{-1}} \, \overset{V}{C}(j)_1 \, (R')_{V,W} \, \overset{W}{X}(i)_2 &= \overset{W}{X}(i)_2 \, (R'){_{V,W}^{-1}} \, \overset{V}{C}(j)_1 \, (R')_{V,W} \quad \text{for } 1 \leq i < j \leq g+n,\nonumber\\
R_{V,W} \, \overset{V}{C}(i)_1 \, (R')_{V,W} \, \overset{W}{X}(i)_2 &= \overset{W}{X}(i)_2 \, R_{V,W} \, \overset{V}{C}(i)_1 \, (R')_{V,W} \quad \text{for } 1 \leq i \leq g+n\label{echangeCiXi}\\
R_{V,W} \, \overset{V}{C}(j)_1 \, R{_{V,W}^{-1}} \, \overset{W}{X}(i)_2 &= \overset{W}{X}(i)_2 \, R_{V,W} \, \overset{V}{C}(j)_1 \, R{_{V,W}^{-1}} \quad \quad \text{for } 1 \leq j < i \leq g+n.\nonumber
\end{align}
The result easily follows.
\end{proof}

\indent Let us define a left action $\smallblacktriangle$ of $H$ on $\mathcal{L}_{g,n}(H)$ by
\[ h \smallblacktriangle x = \mathrm{coad}^r\bigl( S^{-1}(h) \bigr)(x) \]
for $h\in H$, $x \in \mathcal{L}_{g,n}(H)$ and $\mathrm{coad}^r$ is defined in \eqref{coadLgn}. Since $S$ is an anti-morphism of bialgebras, $\smallblacktriangle$ endows $\mathcal{L}_{g,n}(H)$ with the structure of $H^{\mathrm{cop}}$-module-algebra: $\textstyle h \smallblacktriangle (xy) = \sum_{(h)} (h_{(2)} \smallblacktriangle x) (h_{(1)} \smallblacktriangle y)$. If $\overset{V}{X}(i)$ is one of the matrices $\overset{V}{B}(i), \overset{V}{A}(i), \overset{V}{M}(i)$ we define the matrix $h \smallblacktriangle \overset{V}{X}(i)$ in the obvious way on coefficients: $\bigl( h \smallblacktriangle \overset{V}{X}(i) \bigr)^k_l = h \smallblacktriangle \bigl(\overset{V}{X}(i){^k_l}\bigr)$ for all $1 \leq k,l \leq \dim(V)$. One checks easily that
\begin{equation}\label{actionTriangleMatrixForm}
h \smallblacktriangle \overset{V}{X}(i) = \sum_{(h)} S^{-1}(h_{(2)})_V \, \overset{V}{X}(i) \, (h_{(1)})_V.
\end{equation}
Clearly, this applies also to the matrix $\overset{V}{C}$, because $h \smallblacktriangle \overset{V}{X}(i)^{-1}$ is given by the same formula.
\begin{teo}\label{muQMMteo}
The morphism $\mu : H' \to \mathcal{L}_{g,n}(H)$ is a quantum moment map:
\[ \mu(h')\,x = \sum_{(h')} \, (h'_{(2)} \smallblacktriangle x) \, \mu(h'_{(1)}) \]
for all $h' \in H'$ and $x \in \mathcal{L}_{g,n}(H)$.
\end{teo}
\begin{proof}
Note first that, since $\textstyle (R')_{V,W} = \sum_{(R)} (R_{(2)})_{V,1}(R_{(1)})_{W,2}$ and $\textstyle \sum_{(R^1), (R^2)} R^1_{(1)} R^2_{(1)} \otimes R^2_{(2)} S\bigl(R^1_{(2)}\bigr) = 1 \otimes 1$, the commutation relations of Lemma \ref{lemmaCommutationC} can be rewritten as:
\begin{multline}\label{echangeCX}
\overset{V}{C}_1 \, \overset{W}{X}(i)_2 \,=\, \sum_{(R)} (R_{(1)})_{W,2} \, R{^{-1}_{V,W}} \, \overset{W}{X}(i)_2 \, R_{V,W} \, \overset{V}{C}_1 \, (R')_{V,W} \, S\bigl( R_{(2)} \bigr)_{V,1}\\
=\sum_{(R^1), \ldots, (R^4)} \Bigl( \bigl(R^1_{(1)}R^2_{(2)}\bigr)_W \, \overset{W}{X}(i) \, \bigl(R^3_{(2)} R^4_{(1)}\bigr)_W \Bigr)_2 \Bigl( \bigl(S\bigl( R^2_{(1)} \bigr) R^3_{(1)}\bigr)_V \, \overset{V}{C}_1 \, \bigl(R^4_{(2)} S\bigl( R^1_{(2)} \bigr)\bigr)_V \Bigr)_1.
\end{multline}
For the second equality we split all the $R$-matrices and we used obvious commutation relations to rearrange the terms; also recall that $R^{-1} = (S \otimes \mathrm{id})(R)$. To prove the theorem we can assume that $h'$ is a coefficient of some matrix $\textstyle \sum_{(R^1), (R^2)} R^1_{(2)} R^2_{(1)} \otimes \bigl( R^1_{(1)} R^2_{(2)} \bigr)_V \in H' \otimes \mathrm{End}(V)$ and that $x$ is a coefficient of some matrix $\overset{W}{X}(i) \in \mathcal{L}_{g,n}(H) \otimes \mathrm{End}(W)$ which is $\overset{W}{B}(i), \overset{W}{A}(i)$ or $\overset{W}{M}(i)$. This reduces the proof to a matrix computation in $\mathcal{L}_{g,n}(H) \otimes \mathrm{End}(V) \otimes \mathrm{End}(W)$ based on \eqref{componentFreeDefOfMu}. Starting from the right-hand side of the desired formula we get:
\begin{align*}
&\sum_{(R^1), (R^2), (R^1_{(2)}R^2_{(1)})} \Bigl( \bigl( R^1_{(2)}R^2_{(1)} \bigr)_{(2)} \smallblacktriangle \overset{W}{X}(i) \Bigr)_2 \, \Bigl( \mu\bigl( \bigl( R^1_{(2)}R^2_{(1)} \bigr)_{(1)} \bigr) \otimes \bigl( R^1_{(1)}R^2_{(2)} \bigr)_V \Bigr)_1\\
=\:&\sum_{(R^1), \ldots, (R^4)} \Bigl( R^1_{(2)} R^4_{(1)} \smallblacktriangle \overset{W}{X}(i) \Bigr)_2 \, \Bigl( \mu\bigl( R^2_{(2)} R^3_{(1)} \bigr) \otimes \bigl( R^1_{(1)} \, R^2_{(1)} \, R^3_{(2)} \, R^4_{(2)} \bigr)_V \Bigr)_1\\
=\:&\sum_{(R^1), (R^2)} \Bigl( R^1_{(2)} R^2_{(1)} \smallblacktriangle \overset{W}{X}(i) \Bigr)_2 \, \Bigl( \bigl( R^1_{(1)} \bigr)_V \, \overset{V}{C} \, \bigl(R^2_{(2)} \bigr)_V \Bigr)_1\\
=\:&\sum_{(R^1), \ldots, (R^4)} \Bigl( S^{-1}\bigl( R^1_{(2)} R^4_{(1)} \bigr)_W \, \overset{W}{X}(i) \, \bigl( R^2_{(2)} R^3_{(1)} \bigr)_W \Bigr)_2 \, \Bigl( \bigl( R^1_{(1)} R^2_{(1)} \bigr)_V \, \overset{V}{C} \, \bigl( R^3_{(2)} R^4_{(2)} \bigr)_V \Bigr)_1\\
=\:&\sum_{(R^1), \ldots, (R^4)} \Bigl( \bigl( R^4_{(1)} R^1_{(2)} \bigr)_W \, \overset{W}{X}(i) \, \bigl( R^2_{(2)} R^3_{(1)} \bigr)_W \Bigr)_2 \, \Bigl( \bigl( S\bigl(R^1_{(1)}\bigr) R^2_{(1)} \bigr)_V \, \overset{V}{C} \, \bigl( R^3_{(2)} S\bigl(R^4_{(2)}\bigr) \bigr)_V \Bigr)_1\\
=\:& \overset{V}{C}_1 \, \overset{W}{X}(i)_2 = \sum_{(R^1),(R^2)} \Bigl( \mu\bigl( R^1_{(2)}R^2_{(1)} \bigr) \otimes  \bigl( R^1_{(1)}R^2_{(2)} \bigr)_V \Bigr)_1 \, \overset{W}{X}(i)_2.
\end{align*}
For the first equality we used the quasi-triangularity of $R$, for the second equality we used the definition of $\mu$, for the third equality we used \eqref{actionTriangleMatrixForm} and the quasi-triangularity of $R$, for the fourth equality we used that $S$ is an anti-morphism and that $(S \otimes S)(R) = R$, for the fifth equality we used \eqref{echangeCX} and for the last equality we used the definition of $\mu$ in \eqref{componentFreeDefOfMu}.
\end{proof}
\begin{remark}\label{rubanmu} In the case $\mathfrak{g}=\mathfrak{gl}_n(\mathbb{C})$ and $(g,n)=(1,0)$, quantum moment maps have been obtained in \cite{VV}, Section 1.8, and \cite{JordanQuiver}. In that situation, Theorem \ref{muQMMteo} is Proposition 7.21 of \cite{JordanQuiver}, and Definition 7.25 of \cite{JordanQuiver} provides the extension to any $(g,n)$ by external tensor product. The proof in that paper uses a matrix $\tilde{R}$, which in our notations is $(\mathrm{id} \otimes S)(R)$ evaluated in the fundamental representation of $\mathfrak{gl}_n(\mathbb{C})$. 
\end{remark}

\subsection{The quantum reduction of $\Ll_{g,n}(H)$ along the character $\varepsilon$}\label{Lgnqrsection} In this section we assume the following hypotheses hold for $\Ll_{g,n}(H)$:
\begin{itemize}
\item[(i)] $H$ is a ribbon Hopf algebra over a field $k$, with an invertible antipode.
\item[(ii)] The algebra morphism $\Phi_{0,1} : \mathcal{L}_{0,1}(H) \to H$ is injective.\\ 
Denoting by $H'$ the image of $\Phi_{0,1}$, we thus have the quantum moment map defined in \eqref{defmutotal},
$$\mu\colon H'\ra \Ll_{g,n}(H).$$ 
\item[(iii)] The $H$-module $\Ll_{g,n}(H)$ is completely reducible.
\item[(iv)] $Z(H')$ separates the simple types.
\item[(v)] $\Ll_{g,n}^{H'}(H) = \Ll_{g,n}^{H}(H)$.
\end{itemize}
The quantum groups as defined in Section \ref{sectionPrelimUq} satisfy these hypotheses (see the comments below before Corollary \ref{LgnqrNothfinavril25}).
\smallskip

In particular, under the above hypotheses Lemma \ref{lemsepare} and the results of Section \ref{QMMLgn} apply. 
\begin{defi} $\mathcal{L}_{g,n}^{\mathrm{qr}}(H)$ is the quantum reduction of $\mathcal{L}_{g,n}(H)$ along the counit $\varepsilon$, following Definition \ref{defQuantumReduction}, that is
$$\mathcal{L}_{g,n}^{\mathrm{qr}}(H) := \mathcal{L}_{g,n}(H) /\!\!/_{\varepsilon} \,H' = (\mathcal{L}_{g,n}(H)/I_{\varepsilon})^{H'}.$$
\end{defi}

\begin{prop}\label{qrNoethfg} Under the above hypotheses, if $\Ll_{g,n}^{H}(H)$ is a Noetherian algebra, then $\mathcal{L}_{g,n}^{\mathrm{qr}}(H)$ is a Noetherian algebra, and if $\Ll_{g,n}^{H}(H)$ is a finitely generated algebra over $k$, then $\mathcal{L}_{g,n}^{\mathrm{qr}}(H)$ is a finitely generated algebra over $k$.
\end{prop} 
\proof By Lemma \ref{lemsepare} and the hypothesis $\Ll_{g,n}^{H'}(H) = \Ll_{g,n}^{H}(H)$, we have a surjective morphism of algebras $\pi\colon \mathcal{L}_{g,n}^H(H) \twoheadrightarrow  \mathcal{L}_{g,n}^{\mathrm{qr}}(H)$. The conclusion follows.\Endproof
\medskip

By abusing notations, we will denote ${\rm id}_W\otimes 1 \in {\rm End}(W) \otimes \Ll_{g,n}(H)$ by ${\rm id}_W$. Define
$$\mathfrak{C}_{g,n}(H) := {\rm Vect}_{k} \left\lbrace \begin{array}{ll} \text{the matrix coefficients of the matrices } \overset{W}{C}-{\rm id}_W \\ \quad\:\: \text{for all finite dimensional } H\text{-modules\ } W \end{array} \right\rbrace.$$
From \eqref{defiIchi} and \eqref{defmutotal} we see that $I_\varepsilon$ is the left ideal of $\Ll_{g,n}(H)$ of the form
\begin{equation}\label{expressIe}
I_\varepsilon = \Ll_{g,n}(H)\mathfrak{C}_{g,n}(H).
\end{equation}
Recall the Reynolds operator $\mathfrak{R}\colon \Ll_{g,n}(H)\ra \Ll_{g,n}^H(H)$, defined by \eqref{defReynolds}. Note that $\mathfrak{R}(I_\varepsilon)$ is a two-sided ideal of $\Ll_{g,n}^H(H)$, thanks to Lemma \ref{lemmaReynolds}.
\begin{prop} We have $ {\rm Ker}(\pi)= I_\varepsilon \cap \mathcal{L}_{g,n}^H(H) = \mathfrak{R}(I_\varepsilon)$. Therefore $\pi$ factors to an isomorphism $\Ll_{g,n}^H(H)/\mathfrak{R}(I_\varepsilon) \cong (\mathcal{L}_{g,n}(H)/I_{\varepsilon})^{H'}$.
\end{prop}
\begin{proof} The first equality and the inclusion $I_\varepsilon \cap \mathcal{L}_{g,n}^H(H) \subset \mathfrak{R}(I_\varepsilon)$ are immediate, by the definitions of $\pi$ and $\mathfrak{R}$ respectively. For the converse inclusion, it is enough to show that $\mathfrak{R}(I_\varepsilon) \subset I_\varepsilon$. It follows from the remark below \eqref{actionTriangleMatrixForm} that $\mathfrak{C}_{g,n}(H)$ is an $H$-submodule. Since $\mathcal{L}_{g,n}(H)$ is an $H$-module algebra, $I_\varepsilon = \Ll_{g,n}(H)\mathfrak{C}_{g,n}(H)$ is an $H$-submodule too. Because $\mathcal{L}_{g,n}(H)$ is completely reducible, $I_\varepsilon$ is also completely reducible. Because $\mathfrak{R}$ is the projection onto the isotypical components of trivial type, we obtain $\mathfrak{R}(I_\varepsilon) \subset I_\varepsilon$.
\end{proof}

\smallskip

We are going to describe generators of ${\rm Ker}(\pi)$, and for this, a few observations are necessary. Let $V$ be a finite dimensional $H$-module. Recall that the action $\mathrm{coad}^r$ is defined in \eqref{coadLgn}, and $h_V$ denotes the representation of $h\in H$ on $V$.

\begin{defi} Given a finite family $(v_i)_{i\in I}$ of vectors of $V$, we say that a tensor 
$$t = \sum_{i\in I} a_i \otimes v_i \in \mathcal{L}_{g,n}(H) \otimes V$$
is {\it equivariant} if for every $h\in H$ we have
$$\sum_{i\in I} \mathrm{coad}^r(h)(a_i) \otimes v_i = \sum_{i\in I} a_i \otimes h_V(v_i).$$
\end{defi}
A main example of equivariant tensor is provided by the matrix $\overset{W}{X}\in \mathcal{L}_{g,n}(H) \otimes \mathrm{End}(W)$ which is $\overset{W}{B}(i), \overset{W}{A}(i)$, $\overset{W}{M}(i)$ or any product of them and their inverses, like $\overset{W}{C}$, taking $V=\mathrm{End}(W) \cong W\otimes W^*$. Indeed, expanding $\overset{W}{X} = \textstyle \sum_{i,j} \overset{W}{X}{^i_{j}} \otimes w_i \otimes w^j$ where $(w_i)$ is a basis of $V$ with dual basis $(w^j)$, we have by \eqref{balancingActions},
$$\sum_{i,j} \mathrm{coad}^r\bigl( \overset{W}{X}{^i_j} \bigr) \otimes \bigl( w_i \otimes w^j \bigr) = \sum_{(h),i,j} \overset{W}{X}{^i_{j}} \otimes \bigl( h_{(1)} \cdot w_i \otimes h_{(2)} \cdot w^j \bigr).$$
Another source of examples is provided by products of equivariant tensors. Indeed, given another finite dimensional $H$-module $V'$ with a basis $(v'_j)$, and an equivariant tensor $\textstyle t' = \sum_{i=1}^{\dim V'} a_i' \otimes v_i' \in A\otimes V'$, the tensor
$$t_1t'_2 = \sum_{i,j} a_ia_j' \otimes v_i \otimes v'_j\in A\otimes V \otimes V'$$ is equivariant.
Note also the following fact. Assume the $H$-module $A$ is completely reducible, and denote by $\mathfrak{R}\colon A\ra A^H$ the Reynolds operator (\textit{i.e.} the projection on the submodule of $H$-invariant elements). 
\begin{lem} With $V$ and an equivariant tensor $t$ as above, we have
\begin{equation}\label{obsRP}\big(\mathfrak{R} \otimes {\rm id}_V \big)(t) = \big({\rm id} \otimes P \big)(t)\end{equation}
where $P\colon V \ra V$ is the projector onto the isotypical component of trivial type, so ${\rm Im}(P) = V^H$.
\end{lem} The proof is immediate by taking a basis $(v_i)$ adapted to the decomposition $V = {\rm Im}(P) \oplus {\rm Ker}(P)$.  
\smallskip

In the next statement we use the graphical calculus of \cite[\S 3]{FaitgHol} to represent invariant elements of $\Ll_{g,n}(H)$.

\begin{prop}\label{propDescriptionKerPi} The vector space ${\rm Ker}(\pi)$ is generated by the elements of the form
\smallskip

\begin{center}
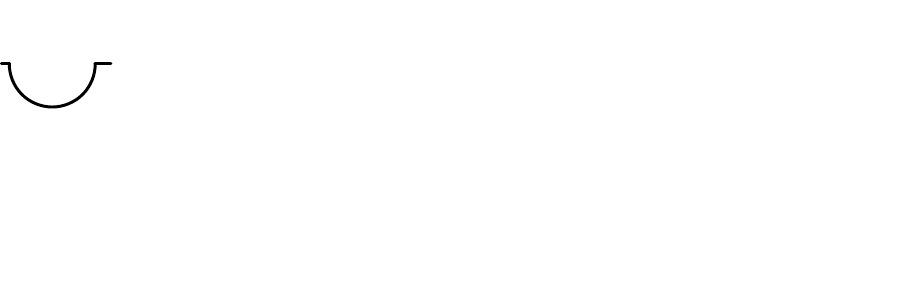
\end{center}

for all $X_i, Y \in \mathrm{Irr}(H)$ and any $f\in  \mathrm{Hom}_H\textstyle \!\left( (\bigotimes_{i=1}^{2g+n} X_i \otimes X_i^*) \otimes Y \otimes Y^*, k \right)$.
\end{prop}
\noindent For instance, if $(g,n)=(1,1)$, the element represented by this diagram is
\begin{align*}
&\sum_{a,b,k,l,m,n,o,p} \overset{X_1}{B(1)}{^a_b} \, \overset{X_2}{A(1)}{^k_l} \, \overset{X_3}{M(2)}{^m_n} \, \overset{Y}{C}{^o_p} \:\, f\bigl( {_1x_a} \otimes {_1x^b} \otimes {_2x_k} \otimes {_2x^l} \otimes {_3x_m} \otimes {_3x^n} \otimes y_o \otimes y^p \bigr)\\
&- \sum_{a,b,k,l,m,n,o} \overset{X_1}{B(1)}{^a_b} \, \overset{X_2}{A(1)}{^k_l} \, \overset{X_3}{M(2)}{^m_n} \:\, f\bigl( {_1x_a} \otimes {_1x^b} \otimes {_2x_k} \otimes {_2x^l} \otimes {_3x_m} \otimes {_3x^n} \otimes y_o \otimes y^o \bigr)
\end{align*}
where the matrices $\overset{X_1}{B(1)}$, $\overset{X_2}{A(1)}$, $\overset{X_3}{M(2)}$, $\overset{Y}{C}$ with coefficients in $\mathcal{L}_{g,n}(H)$ have been introduced in \eqref{matricesABMgn} and Def. \ref{defMatriceC}, $(_kx_r)$ a basis of $X_k$ with dual basis $(_kx^r)$ for $i=1,2,3$, and $(y_s)$ a basis of $Y$ with dual basis $(y^s)$.

\begin{proof} To simplify notations and diagrams we give the details in the case $(g,n)=(1,1)$. Because of \eqref{expressIe} and ${\rm Ker}(\pi) = \mathfrak{R}(I_\varepsilon)$, it is enough to show that if $x(\overset{Y}{C}-{\rm id}_Y)^i_j$ with $x\in \Ll_{1,1}(H)$ satisfies $\mathfrak{R}(x(\overset{Y}{C}-{\rm id}_Y)^i_j) = x(\overset{Y}{C}-{\rm id}_Y)^i_j$, then it is of the form required. Without loss of generality, by linearity we can assume $x\in C(X_1) \otimes C(X_2)\otimes C(X_3)$ for some simple modules $X_1$, $X_2$, $X_3$. Let $S\subset C(X_1) \otimes C(X_2)\otimes C(X_3)$ be the submodule generated by $x$ under $\mathrm{coad}^r$. We thus have an isomorphism $f\colon \mathrm{coad}^r(H)(x) \ra S$. We can view $S$ as a submodule (a direct summand) of $\textstyle \bigotimes_{i=1}^{3} X_i \otimes X_i^*$. Denote by $p_S\colon  \textstyle \bigotimes_{i=1}^{3} X_i \otimes X_i^* \ra S$ the associated projection. Define the tensor $t_x\in \Ll_{1,1}(H) \otimes S$ by
\begin{center}
\begingroup%
  \makeatletter%
  \providecommand\color[2][]{%
    \errmessage{(Inkscape) Color is used for the text in Inkscape, but the package 'color.sty' is not loaded}%
    \renewcommand\color[2][]{}%
  }%
  \providecommand\transparent[1]{%
    \errmessage{(Inkscape) Transparency is used (non-zero) for the text in Inkscape, but the package 'transparent.sty' is not loaded}%
    \renewcommand\transparent[1]{}%
  }%
  \providecommand\rotatebox[2]{#2}%
  \newcommand*\fsize{\dimexpr\f@size pt\relax}%
  \newcommand*\lineheight[1]{\fontsize{\fsize}{#1\fsize}\selectfont}%
  \ifx\svgwidth\undefined%
    \setlength{\unitlength}{172.34063278bp}%
    \ifx\svgscale\undefined%
      \relax%
    \else%
      \setlength{\unitlength}{\unitlength * \real{\svgscale}}%
    \fi%
  \else%
    \setlength{\unitlength}{\svgwidth}%
  \fi%
  \global\let\svgwidth\undefined%
  \global\let\svgscale\undefined%
  \makeatother%
  \begin{picture}(1,0.44697027)%
    \lineheight{1}%
    \setlength\tabcolsep{0pt}%
    \put(0,0){\includegraphics[width=\unitlength,page=1]{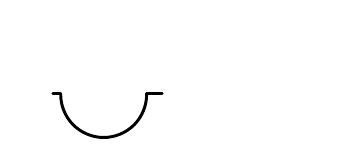}}%
    \put(0.23774912,0.00405489){\color[rgb]{0,0,0}\makebox(0,0)[lt]{\lineheight{1.25}\smash{\begin{tabular}[t]{l}$B(1)$\end{tabular}}}}%
    \put(0,0){\includegraphics[width=\unitlength,page=2]{tenseur_tx.pdf}}%
    \put(0.52048747,0.00372925){\color[rgb]{0,0,0}\makebox(0,0)[lt]{\lineheight{1.25}\smash{\begin{tabular}[t]{l}$A(1)$\end{tabular}}}}%
    \put(0,0){\includegraphics[width=\unitlength,page=3]{tenseur_tx.pdf}}%
    \put(0.80522533,0.00436766){\color[rgb]{0,0,0}\makebox(0,0)[lt]{\lineheight{1.25}\smash{\begin{tabular}[t]{l}$M(2)$\end{tabular}}}}%
    \put(0,0){\includegraphics[width=\unitlength,page=4]{tenseur_tx.pdf}}%
    \put(0.53773538,0.308145){\color[rgb]{0,0,0}\makebox(0,0)[lt]{\lineheight{1.25}\smash{\begin{tabular}[t]{l}$p_S$\end{tabular}}}}%
    \put(0,0){\includegraphics[width=\unitlength,page=5]{tenseur_tx.pdf}}%
    \put(0.20303195,0.23071388){\color[rgb]{0,0,0}\makebox(0,0)[lt]{\lineheight{1.25}\smash{\begin{tabular}[t]{l}$_{X_1}$\end{tabular}}}}%
    \put(0,0){\includegraphics[width=\unitlength,page=6]{tenseur_tx.pdf}}%
    \put(0.48590202,0.23071388){\color[rgb]{0,0,0}\makebox(0,0)[lt]{\lineheight{1.25}\smash{\begin{tabular}[t]{l}$_{X_2}$\end{tabular}}}}%
    \put(0,0){\includegraphics[width=\unitlength,page=7]{tenseur_tx.pdf}}%
    \put(0.76877192,0.23071388){\color[rgb]{0,0,0}\makebox(0,0)[lt]{\lineheight{1.25}\smash{\begin{tabular}[t]{l}$_{X_3}$\end{tabular}}}}%
    \put(0,0){\includegraphics[width=\unitlength,page=8]{tenseur_tx.pdf}}%
    \put(-0.00088646,0.17471334){\color[rgb]{0,0,0}\makebox(0,0)[lt]{\lineheight{1.25}\smash{\begin{tabular}[t]{l}$t_x =$\end{tabular}}}}%
    \put(0,0){\includegraphics[width=\unitlength,page=9]{tenseur_tx.pdf}}%
    \put(0.57133886,0.42424704){\color[rgb]{0,0,0}\makebox(0,0)[lt]{\lineheight{1.25}\smash{\begin{tabular}[t]{l}$_S$\end{tabular}}}}%
  \end{picture}%
\endgroup%

\end{center}
Letting $({}_kx_r)$ and $({}_kx^r)$ be dual basis of $X_k$ and $X_k^*$, this reads
$$t_x =\sum_{a,b,k,l,m,n} \overset{X_1}{B(1)}{}^a_b \,  \overset{X_2}{A(1)}{}^k_l \, \overset{X_3}{M(2)}{}^m_n \otimes p_S\big( {}_1x_a\otimes {}_1x^b\otimes {}_2x_k\otimes {}_2x^l\otimes {}_3x_m \otimes {}_3x^n\big).$$
By construction $t_x$ is an equivariant tensor. Let $(s_r)$ be a basis of $S$, chosen so that $f(x)=s_1$. Then it is immediate that $({\rm id} \otimes s^1)(t_x) = x$. It follows that $(t_x)_1(\overset{Y}{C}-{\rm id}_Y)_2 \in \Ll_{1,1}(H) \otimes S\otimes Y\otimes Y^*$ is equivariant, and therefore from \eqref{obsRP} we get
$$(\mathfrak{R}\otimes {\rm id})\big((t_x)_1(\overset{Y}{C}-{\rm id}_Y)_2\big)  =  ({\rm id} \otimes P) \big((t_x)_1(\overset{Y}{C}-{\rm id}_Y)_2\big)$$
where $P\colon S\otimes Y\otimes Y^* \ra S\otimes Y\otimes Y^*$ is the projection onto the subspace of invariant elements. Letting $(y_s)$ and $(y^t)$ be dual basis of $Y$ and $Y^*$, denote by $\eta\colon S\otimes Y\otimes Y^* \ra k$ the linear form such that $\eta(s_r\otimes y_s\otimes y^t) = \delta_{r,1}\delta_{s,i}\delta_{t,j}$. Then we have
$$x(\overset{Y}{C}-{\rm id}_Y)^i_j = ({\rm id}\otimes \eta)\big((t_x)_1(\overset{Y}{C}-{\rm id}_Y)_2\big),$$
and
\begin{align*}
({\rm id}\otimes (\eta\circ P))\big((t_x)_1(\overset{Y}{C}-{\rm id}_Y)_2\big) & = (\mathfrak{R}\otimes \eta)\big((t_x)_1(\overset{Y}{C}-{\rm id}_Y)_2\big) \\
& = \mathfrak{R}\big(x(\overset{Y}{C}-{\rm id}_Y)^i_j\big) = x(\overset{Y}{C}-{\rm id}_Y)^i_j.
\end{align*}
The map $\eta\circ P\colon S\otimes Y\otimes Y^* \ra k$ is $H$-linear,  and the left-hand side of the identity is the element of the form
\smallskip

\begin{center}
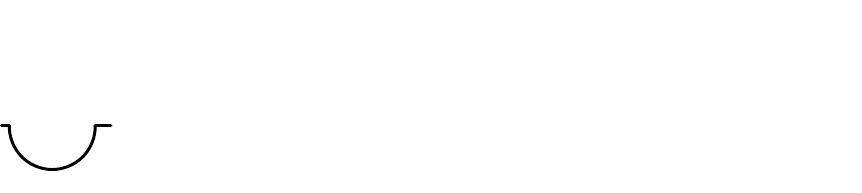
\end{center}
 This concludes the proof.
\end{proof}
\begin{remark}  The proposition above proves the definition of the ideal $\mathcal{I}_C$ in \cite[p.\,10]{BNR}.
\end{remark}
Finally, note that the hypotheses (i)--(v) stated at the beginning of the section are satisfied by $H=U^{\mathrm{ad}}_q=U^{\mathrm{ad}}_q(\mathfrak{g})$. Indeed, (ii) is satisfied as a well-known particular case of Theorem \ref{TheoremePhignInjectif} (\cite{Bau1}), and (iii) follows from the fact that $\Ll_{g,n}(\mathfrak{g}) = \Oo_q(q^{1/D})^{\otimes (2g+n)}$ as a right $U_q$-module, and the complete reducibility of the $U_q$-module $\Oo_q$. As for (i), $U^{\mathrm{ad}}_q$ is ribbon in its categorical completion, explained in Section \ref{sectionCategoricalCompletion}.

 It remains to check (iv) and (v). As already mentioned, we have $H'=U_q^{{\rm lf}}$ by \cite[Th. 3]{Bau1}. Also we have the following decomposition as a free $T_{2-}^{-1}U_q^{\rm lf}$-module \cite[Theorem 6.4]{JL} :
$$U_q = T_{2-}^{-1}U_q^{\rm lf}[T/T_{2}]$$
where $T\subset U_q$ is the multiplicative Abelian group formed by the elements $K_{\lambda}$, $\lambda \in P$, and we denote by $T_2\subset T$ the subgroup formed by the elements $K_{\lambda}$, $\lambda \in 2P$, by $T_{2-} \subset T_2$ the set formed by the elements $K_{-\lambda}$, $\lambda \in 2P_+$, and by $T/T_2$ the quotient group. As a result
$$Z(U_q^{{\rm lf}})=Z(U_q)\quad {\rm and}\quad \Ll_{g,n}^{U_q^{{\rm lf}}} = \Ll_{g,n}^{U_q}.$$
The last equality proves (v). To prove (iv), let us show that $Z(U_q)$ separates the simple types. There is an algebra isomorphism $\Xi\colon Z(U_q)\ra \mc(q)[K_{2\varpi_1}^{\pm 1},\ldots,K_{2\varpi_N}^{\pm 1}]^W$, where $W$ is acting by the shifted Weyl group action (see e.g. \cite{VY}, Section 3.13). Then let $X_1,\ldots,X_k$ be non-trivial and non isomorphic simple type $1$ finite dimensional $U_q$-modules. They are highest weight modules, of non-zero dominant weights $\lambda_1,\ldots,\lambda_k$ lying in disjoint $W$-orbits. Therefore, there exists an element $P\in \mc(q)[K_{2\varpi_1}^{\pm 1},\ldots,K_{2\varpi_N}^{\pm 1}]^W$ such that $P(0)=0$ and $P(\lambda_i)\ne P(\lambda_j)$ for $i\ne j$. Taking $z=\Xi^{-1}(P)$ provides an element satisfying the separation property of Definition \ref{defiseparate}.
\medskip

By Proposition \ref{qrNoethfg} and Theorem \ref{thmLgnUqFinGen}, we get:
\begin{cor}\label{LgnqrNothfinavril25} The algebra $\Ll_{g,n}^{{\rm qr}}(\mathfrak{g})$ is Noetherian and finitely generated.
\end{cor}

\subsection{Topological interpretation of the quantum reduction $\mathcal{L}_{g,n}^{\mathrm{qr}}(H)$}\label{topqrsection} Let us call {\em $\partial$-move} the transformation shown in the picture below; this picture represents a neighborhood of the boundary of $\Sigma_{g,n}^{\circ}$, and the strand represents the projection on $\Sigma_{g,n}^{\circ}$ of a portion of some $H$-colored ribbon link with coupons in $\Sigma_{g,n}^{\circ} \times [0,1]$. 
\medskip

\begin{center}
\begingroup%
  \makeatletter%
  \providecommand\color[2][]{%
    \errmessage{(Inkscape) Color is used for the text in Inkscape, but the package 'color.sty' is not loaded}%
    \renewcommand\color[2][]{}%
  }%
  \providecommand\transparent[1]{%
    \errmessage{(Inkscape) Transparency is used (non-zero) for the text in Inkscape, but the package 'transparent.sty' is not loaded}%
    \renewcommand\transparent[1]{}%
  }%
  \providecommand\rotatebox[2]{#2}%
  \newcommand*\fsize{\dimexpr\f@size pt\relax}%
  \newcommand*\lineheight[1]{\fontsize{\fsize}{#1\fsize}\selectfont}%
  \ifx\svgwidth\undefined%
    \setlength{\unitlength}{240.25800989bp}%
    \ifx\svgscale\undefined%
      \relax%
    \else%
      \setlength{\unitlength}{\unitlength * \real{\svgscale}}%
    \fi%
  \else%
    \setlength{\unitlength}{\svgwidth}%
  \fi%
  \global\let\svgwidth\undefined%
  \global\let\svgscale\undefined%
  \makeatother%
  \begin{picture}(1,0.24830176)%
    \lineheight{1}%
    \setlength\tabcolsep{0pt}%
    \put(0,0){\includegraphics[width=\unitlength,page=1]{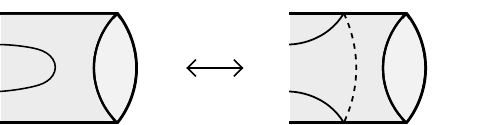}}%
    \put(0.26120084,0.20944257){\color[rgb]{0,0,0}\makebox(0,0)[lt]{\lineheight{1.25}\smash{\begin{tabular}[t]{l}$\partial(\Sigma_{g,n}^{\circ})$\end{tabular}}}}%
    \put(0.84058447,0.2141168){\color[rgb]{0,0,0}\makebox(0,0)[lt]{\lineheight{1.25}\smash{\begin{tabular}[t]{l}$\partial(\Sigma_{g,n}^{\circ})$\end{tabular}}}}%
  \end{picture}%
\endgroup%

\end{center}
\medskip

As usual we consider such links up to isotopy. We say that two linear combinations of links $L_1$ and $L_2$ are $\partial$-equivalent, denoted by $L_1 \sim_{\partial} L_2$, if their diagrams on $\Sigma_{g,n}^{\circ}$ can be related by a finite sequence of $\partial$-moves. Because the skein relations are local, the relation $\sim_{\partial}$ is compatible with skein equivalence, in the sense that if $L_1 \sim_{\partial} L_2$ and if $L_1$ (resp. $L_2$) is skein equivalent to $L'_1$ (resp. to $L'_2$), then $L'_1 \sim_{\partial} L'_2$. Also, $\sim_{\partial}$ is compatible with the product in $\mathcal{S}_H(\Sigma_{g,n}^{\circ})$, which we recall is given by stacking. Hence $\mathcal{S}_H(\Sigma_{g,n}^{\circ})/\sim_{\partial}$ is an algebra. 

The inclusion $\Sigma_{g,n}^{\circ} \subset \Sigma_{g,n}$ yields a well-defined, surjective morphism of algebras $\mathcal{S}_H(\Sigma_{g,n}^{\circ})\ra \mathcal{S}_H(\Sigma_{g,n})$, and skein classes of $\partial$-equivalent $H$-colored ribbon links have the same image. Since (diagrams of) isotopic links in $\Sigma_{g,n}$ can be related by finite sequences of $\partial$-moves and isotopies in the subsurface $\Sigma_{g,n}^{\circ}$, it factors into an isomorphism
$$\mathcal{S}_H(\Sigma_{g,n}^{\circ})/\sim_{\partial}\ \stackrel{\cong }{\lra} \mathcal{S}_H(\Sigma_{g,n}).$$
Below we will identify $\mathcal{S}_H(\Sigma_{g,n}^{\circ})/\sim_{\partial}$ and $\mathcal{S}_H(\Sigma_{g,n})$ by using this isomorphism.
\smallskip

Recall the Wilson loop morphism $W : \mathcal{S}_H(\Sigma_{g,n}^{\circ}) \to \mathcal{L}_{g,n}^H(H)$ (Definition \ref{defWilsonLoopMap}; as explained thereafter, $W$ takes values in $\mathcal{L}_{g,n}^H(H)$), and recall that $\pi\colon \mathcal{L}_{g,n}^H(H) \to \mathcal{L}_{g,n}^{\rm qr}(H)$ denotes the restriction to $\mathcal{L}_{g,n}^H(H)$ of the canonical projection $p : \mathcal{L}_{g,n}(H) \to \mathcal{L}_{g,n}(H)/I_{\epsilon}$ (see \eqref{defPi}).
\begin{prop}\label{prop_Wqr_passe_au_quotient}
If $L_1 \sim_{\partial} L_2$ then $\pi \circ W(L_1) = \pi \circ W(L_2)$. It follows that there is a morphism of algebras $W^{\mathrm{qr}} : \mathcal{S}_H(\Sigma_{g,n}) \to \mathcal{L}_{g,n}^{\mathrm{qr}}(H)$ such that the diagram
\begin{equation}\label{defWqr}
\xymatrix@C=4em{
\mathcal{S}_H(\Sigma_{g,n}^{\circ}) \ar@{->>}[d]_{\sim_{\partial}} \ar[r]^{W}& \mathcal{L}_{g,n}^H(H) \ar[d]^{\pi}\\
\mathcal{S}_H(\Sigma_{g,n}) \ar[r]_{W^{\mathrm{qr}}}& \mathcal{L}_{g,n}^{\mathrm{qr}}(H) 
}
\end{equation}
commutes.
\end{prop}
We do the proof for $(g,n)=(1,1)$. The general case is similar except that it requires to draw more cumbersome diagrams. The key-point is the following lemma, which uses the graphical calculus from \cite[\S 3]{FaitgHol}.
\begin{lem}\label{lemmeSimplificationWqr}We have:
\begin{center}
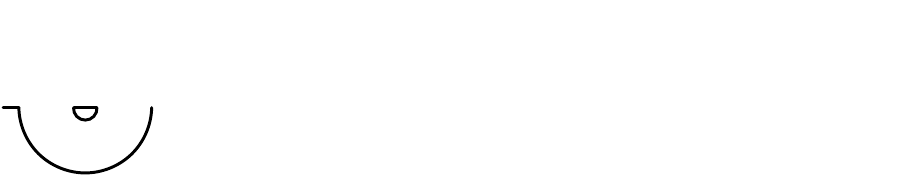
\end{center}

\noindent with the matrix $\overset{X}{C}$ from Definition \ref{defMatriceC}.
\end{lem}
\begin{proof}
For ease of notation we did not put the colors on the matrices in the diagrams; namely the first label in the left-hand side should actually be $\overset{X \otimes U \otimes X}{B(1)}$ {\it etc}.
\\Note first that due to the diagram for the inverse of $B(1)$ and $A(1)$ \cite[Prop. 3.3]{FaitgHol}, the matrix $\overset{X}{C}(1)$ defined in \eqref{defCi} can be written as
\begin{center}
\begingroup%
  \makeatletter%
  \providecommand\color[2][]{%
    \errmessage{(Inkscape) Color is used for the text in Inkscape, but the package 'color.sty' is not loaded}%
    \renewcommand\color[2][]{}%
  }%
  \providecommand\transparent[1]{%
    \errmessage{(Inkscape) Transparency is used (non-zero) for the text in Inkscape, but the package 'transparent.sty' is not loaded}%
    \renewcommand\transparent[1]{}%
  }%
  \providecommand\rotatebox[2]{#2}%
  \newcommand*\fsize{\dimexpr\f@size pt\relax}%
  \newcommand*\lineheight[1]{\fontsize{\fsize}{#1\fsize}\selectfont}%
  \ifx\svgwidth\undefined%
    \setlength{\unitlength}{275.16732491bp}%
    \ifx\svgscale\undefined%
      \relax%
    \else%
      \setlength{\unitlength}{\unitlength * \real{\svgscale}}%
    \fi%
  \else%
    \setlength{\unitlength}{\svgwidth}%
  \fi%
  \global\let\svgwidth\undefined%
  \global\let\svgscale\undefined%
  \makeatother%
  \begin{picture}(1,0.17729334)%
    \lineheight{1}%
    \setlength\tabcolsep{0pt}%
    \put(0,0){\includegraphics[width=\unitlength,page=1]{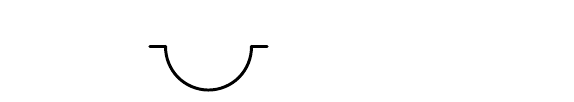}}%
    \put(0.40230358,0.00233567){\color[rgb]{0,0,0}\makebox(0,0)[lt]{\lineheight{1.25}\smash{\begin{tabular}[t]{l}$B(1)$\end{tabular}}}}%
    \put(0,0){\includegraphics[width=\unitlength,page=2]{diagrammeC.pdf}}%
    \put(0.57843095,0.00289628){\color[rgb]{0,0,0}\makebox(0,0)[lt]{\lineheight{1.25}\smash{\begin{tabular}[t]{l}$A(1)$\end{tabular}}}}%
    \put(0,0){\includegraphics[width=\unitlength,page=3]{diagrammeC.pdf}}%
    \put(0.75778977,0.00309618){\color[rgb]{0,0,0}\makebox(0,0)[lt]{\lineheight{1.25}\smash{\begin{tabular}[t]{l}$B(1)$\end{tabular}}}}%
    \put(0,0){\includegraphics[width=\unitlength,page=4]{diagrammeC.pdf}}%
    \put(0.93714304,0.00292197){\color[rgb]{0,0,0}\makebox(0,0)[lt]{\lineheight{1.25}\smash{\begin{tabular}[t]{l}$A(1)$\end{tabular}}}}%
    \put(0,0){\includegraphics[width=\unitlength,page=5]{diagrammeC.pdf}}%
    \put(0.2200115,0.09087547){\color[rgb]{0,0,0}\makebox(0,0)[lt]{\lineheight{1.25}\smash{\begin{tabular}[t]{l}$=$\end{tabular}}}}%
    \put(0,0){\includegraphics[width=\unitlength,page=6]{diagrammeC.pdf}}%
    \put(0.05016561,0.15231032){\color[rgb]{0,0,0}\makebox(0,0)[lt]{\lineheight{1.25}\smash{\begin{tabular}[t]{l}$_X$\end{tabular}}}}%
    \put(0,0){\includegraphics[width=\unitlength,page=7]{diagrammeC.pdf}}%
    \put(0.14671488,0.00292197){\color[rgb]{0,0,0}\makebox(0,0)[lt]{\lineheight{1.25}\smash{\begin{tabular}[t]{l}$C(1)$\end{tabular}}}}%
    \put(0,0){\includegraphics[width=\unitlength,page=8]{diagrammeC.pdf}}%
    \put(0.30199982,0.15317414){\color[rgb]{0,0,0}\makebox(0,0)[lt]{\lineheight{1.25}\smash{\begin{tabular}[t]{l}$_X$\end{tabular}}}}%
  \end{picture}%
\endgroup%

\end{center}
Then, using several times the diagrammatic commutation relations from \cite[Prop. 3.2 and 3.3]{FaitgHol}, one shows that
\begin{center}
\begingroup%
  \makeatletter%
  \providecommand\color[2][]{%
    \errmessage{(Inkscape) Color is used for the text in Inkscape, but the package 'color.sty' is not loaded}%
    \renewcommand\color[2][]{}%
  }%
  \providecommand\transparent[1]{%
    \errmessage{(Inkscape) Transparency is used (non-zero) for the text in Inkscape, but the package 'transparent.sty' is not loaded}%
    \renewcommand\transparent[1]{}%
  }%
  \providecommand\rotatebox[2]{#2}%
  \newcommand*\fsize{\dimexpr\f@size pt\relax}%
  \newcommand*\lineheight[1]{\fontsize{\fsize}{#1\fsize}\selectfont}%
  \ifx\svgwidth\undefined%
    \setlength{\unitlength}{324.75703527bp}%
    \ifx\svgscale\undefined%
      \relax%
    \else%
      \setlength{\unitlength}{\unitlength * \real{\svgscale}}%
    \fi%
  \else%
    \setlength{\unitlength}{\svgwidth}%
  \fi%
  \global\let\svgwidth\undefined%
  \global\let\svgscale\undefined%
  \makeatother%
  \begin{picture}(1,0.27204962)%
    \lineheight{1}%
    \setlength\tabcolsep{0pt}%
    \put(0,0){\includegraphics[width=\unitlength,page=1]{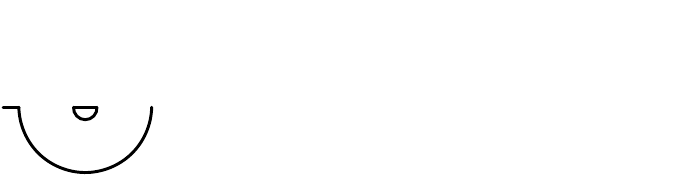}}%
    \put(-0.00047042,0.25316757){\color[rgb]{0,0,0}\makebox(0,0)[lt]{\lineheight{1.25}\smash{\begin{tabular}[t]{l}$_X$\end{tabular}}}}%
    \put(0,0){\includegraphics[width=\unitlength,page=2]{diagrammeBCA.pdf}}%
    \put(0.17748882,0.00206027){\color[rgb]{0,0,0}\makebox(0,0)[lt]{\lineheight{1.25}\smash{\begin{tabular}[t]{l}$B(1)$\end{tabular}}}}%
    \put(0,0){\includegraphics[width=\unitlength,page=3]{diagrammeBCA.pdf}}%
    \put(0.39659519,0.00197902){\color[rgb]{0,0,0}\makebox(0,0)[lt]{\lineheight{1.25}\smash{\begin{tabular}[t]{l}$A(1)$\end{tabular}}}}%
    \put(0,0){\includegraphics[width=\unitlength,page=4]{diagrammeBCA.pdf}}%
    \put(0.48095913,0.10853182){\color[rgb]{0,0,0}\makebox(0,0)[lt]{\lineheight{1.25}\smash{\begin{tabular}[t]{l}$=$\end{tabular}}}}%
    \put(0,0){\includegraphics[width=\unitlength,page=5]{diagrammeBCA.pdf}}%
    \put(0.06864653,0.25272952){\color[rgb]{0,0,0}\makebox(0,0)[lt]{\lineheight{1.25}\smash{\begin{tabular}[t]{l}$_U$\end{tabular}}}}%
    \put(0.16433776,0.25318402){\color[rgb]{0,0,0}\makebox(0,0)[lt]{\lineheight{1.25}\smash{\begin{tabular}[t]{l}$_V$\end{tabular}}}}%
    \put(0,0){\includegraphics[width=\unitlength,page=6]{diagrammeBCA.pdf}}%
    \put(0.64368264,0.03375765){\color[rgb]{0,0,0}\makebox(0,0)[lt]{\lineheight{1.25}\smash{\begin{tabular}[t]{l}$B(1)$\end{tabular}}}}%
    \put(0,0){\includegraphics[width=\unitlength,page=7]{diagrammeBCA.pdf}}%
    \put(0.79291576,0.03423263){\color[rgb]{0,0,0}\makebox(0,0)[lt]{\lineheight{1.25}\smash{\begin{tabular}[t]{l}$C(1)$\end{tabular}}}}%
    \put(0,0){\includegraphics[width=\unitlength,page=8]{diagrammeBCA.pdf}}%
    \put(0.94488687,0.03440203){\color[rgb]{0,0,0}\makebox(0,0)[lt]{\lineheight{1.25}\smash{\begin{tabular}[t]{l}$A(1)$\end{tabular}}}}%
    \put(0,0){\includegraphics[width=\unitlength,page=9]{diagrammeBCA.pdf}}%
    \put(0.62668576,0.25314739){\color[rgb]{0,0,0}\makebox(0,0)[lt]{\lineheight{1.25}\smash{\begin{tabular}[t]{l}$_U$\end{tabular}}}}%
    \put(0.7206066,0.25398293){\color[rgb]{0,0,0}\makebox(0,0)[lt]{\lineheight{1.25}\smash{\begin{tabular}[t]{l}$_V$\end{tabular}}}}%
    \put(0.53377897,0.25329164){\color[rgb]{0,0,0}\makebox(0,0)[lt]{\lineheight{1.25}\smash{\begin{tabular}[t]{l}$_X$\end{tabular}}}}%
    \put(0,0){\includegraphics[width=\unitlength,page=10]{diagrammeBCA.pdf}}%
  \end{picture}%
\endgroup%

\end{center}
Finally, it is an exercise to check that the matrix relation \eqref{echangeCiXi} is equivalent to
\begin{center}
\begingroup%
  \makeatletter%
  \providecommand\color[2][]{%
    \errmessage{(Inkscape) Color is used for the text in Inkscape, but the package 'color.sty' is not loaded}%
    \renewcommand\color[2][]{}%
  }%
  \providecommand\transparent[1]{%
    \errmessage{(Inkscape) Transparency is used (non-zero) for the text in Inkscape, but the package 'transparent.sty' is not loaded}%
    \renewcommand\transparent[1]{}%
  }%
  \providecommand\rotatebox[2]{#2}%
  \newcommand*\fsize{\dimexpr\f@size pt\relax}%
  \newcommand*\lineheight[1]{\fontsize{\fsize}{#1\fsize}\selectfont}%
  \ifx\svgwidth\undefined%
    \setlength{\unitlength}{226.41732085bp}%
    \ifx\svgscale\undefined%
      \relax%
    \else%
      \setlength{\unitlength}{\unitlength * \real{\svgscale}}%
    \fi%
  \else%
    \setlength{\unitlength}{\svgwidth}%
  \fi%
  \global\let\svgwidth\undefined%
  \global\let\svgscale\undefined%
  \makeatother%
  \begin{picture}(1,0.31182482)%
    \lineheight{1}%
    \setlength\tabcolsep{0pt}%
    \put(0,0){\includegraphics[width=\unitlength,page=1]{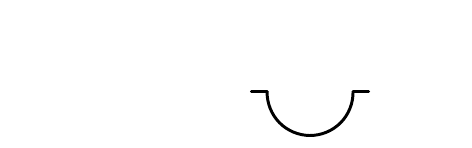}}%
    \put(0.7079673,0.00283857){\color[rgb]{0,0,0}\makebox(0,0)[lt]{\lineheight{1.25}\smash{\begin{tabular}[t]{l}$X(i)$\end{tabular}}}}%
    \put(0,0){\includegraphics[width=\unitlength,page=2]{diagrammeEchangeCX.pdf}}%
    \put(0.92201675,0.00351988){\color[rgb]{0,0,0}\makebox(0,0)[lt]{\lineheight{1.25}\smash{\begin{tabular}[t]{l}$C(i)$\end{tabular}}}}%
    \put(0,0){\includegraphics[width=\unitlength,page=3]{diagrammeEchangeCX.pdf}}%
    \put(0.61756225,0.28568644){\color[rgb]{0,0,0}\makebox(0,0)[lt]{\lineheight{1.25}\smash{\begin{tabular}[t]{l}$_V$\end{tabular}}}}%
    \put(0.82161451,0.28470308){\color[rgb]{0,0,0}\makebox(0,0)[lt]{\lineheight{1.25}\smash{\begin{tabular}[t]{l}$_W$\end{tabular}}}}%
    \put(0.48222482,0.10696849){\color[rgb]{0,0,0}\makebox(0,0)[lt]{\lineheight{1.25}\smash{\begin{tabular}[t]{l}$=$\end{tabular}}}}%
    \put(0,0){\includegraphics[width=\unitlength,page=4]{diagrammeEchangeCX.pdf}}%
    \put(0.27627718,0.18510428){\color[rgb]{0,0,0}\makebox(0,0)[lt]{\lineheight{1.25}\smash{\begin{tabular}[t]{l}$_W$\end{tabular}}}}%
    \put(0,0){\includegraphics[width=\unitlength,page=5]{diagrammeEchangeCX.pdf}}%
    \put(0.39734777,0.00355106){\color[rgb]{0,0,0}\makebox(0,0)[lt]{\lineheight{1.25}\smash{\begin{tabular}[t]{l}$X(i)$\end{tabular}}}}%
    \put(0,0){\includegraphics[width=\unitlength,page=6]{diagrammeEchangeCX.pdf}}%
    \put(0.06096677,0.18510428){\color[rgb]{0,0,0}\makebox(0,0)[lt]{\lineheight{1.25}\smash{\begin{tabular}[t]{l}$_V$\end{tabular}}}}%
    \put(0,0){\includegraphics[width=\unitlength,page=7]{diagrammeEchangeCX.pdf}}%
    \put(0.17830411,0.00355109){\color[rgb]{0,0,0}\makebox(0,0)[lt]{\lineheight{1.25}\smash{\begin{tabular}[t]{l}$C(i)$\end{tabular}}}}%
    \put(0,0){\includegraphics[width=\unitlength,page=8]{diagrammeEchangeCX.pdf}}%
  \end{picture}%
\endgroup%

\end{center}
and the matrix relation in Lemma \ref{lemmaCommutationC} is equivalent to the same diagrammatic identity but with $C$ instead of $C(i)$. From these facts one easily finishes the computation.
\end{proof}

\begin{proof}[Proof of Proposition \ref{prop_Wqr_passe_au_quotient}]
Consider the following $H$-colored ribbon links in $\Sigma_{1,1}^{\circ} \times [0,1]$ which are related by a $\partial$-move:
\begin{center}
\begingroup%
  \makeatletter%
  \providecommand\color[2][]{%
    \errmessage{(Inkscape) Color is used for the text in Inkscape, but the package 'color.sty' is not loaded}%
    \renewcommand\color[2][]{}%
  }%
  \providecommand\transparent[1]{%
    \errmessage{(Inkscape) Transparency is used (non-zero) for the text in Inkscape, but the package 'transparent.sty' is not loaded}%
    \renewcommand\transparent[1]{}%
  }%
  \providecommand\rotatebox[2]{#2}%
  \newcommand*\fsize{\dimexpr\f@size pt\relax}%
  \newcommand*\lineheight[1]{\fontsize{\fsize}{#1\fsize}\selectfont}%
  \ifx\svgwidth\undefined%
    \setlength{\unitlength}{440.16734268bp}%
    \ifx\svgscale\undefined%
      \relax%
    \else%
      \setlength{\unitlength}{\unitlength * \real{\svgscale}}%
    \fi%
  \else%
    \setlength{\unitlength}{\svgwidth}%
  \fi%
  \global\let\svgwidth\undefined%
  \global\let\svgscale\undefined%
  \makeatother%
  \begin{picture}(1,0.20006297)%
    \lineheight{1}%
    \setlength\tabcolsep{0pt}%
    \put(0,0){\includegraphics[width=\unitlength,page=1]{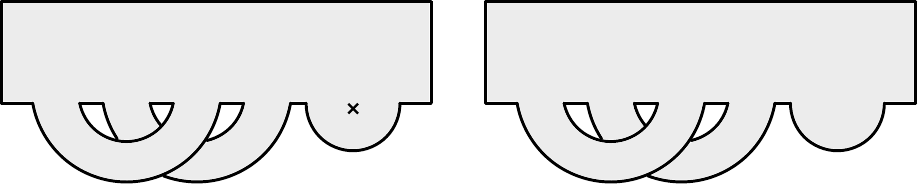}}%
    \put(0.21775541,0.13905404){\color[rgb]{0,0,0}\makebox(0,0)[lt]{\lineheight{1.25}\smash{\begin{tabular}[t]{l}$T$\end{tabular}}}}%
    \put(0,0){\includegraphics[width=\unitlength,page=2]{L_et_Lbord.pdf}}%
    \put(0.05816269,0.11512652){\color[rgb]{0,0,0}\makebox(0,0)[lt]{\lineheight{1.25}\smash{\begin{tabular}[t]{l}$_U$\end{tabular}}}}%
    \put(0,0){\includegraphics[width=\unitlength,page=3]{L_et_Lbord.pdf}}%
    \put(0.13845288,0.11445711){\color[rgb]{0,0,0}\makebox(0,0)[lt]{\lineheight{1.25}\smash{\begin{tabular}[t]{l}$_V$\end{tabular}}}}%
    \put(0.36722648,0.11282728){\color[rgb]{0,0,0}\makebox(0,0)[lt]{\lineheight{1.25}\smash{\begin{tabular}[t]{l}$_W$\end{tabular}}}}%
    \put(0.05873924,0.1801495){\color[rgb]{0,0,0}\makebox(0,0)[lt]{\lineheight{1.25}\smash{\begin{tabular}[t]{l}$_X$\end{tabular}}}}%
    \put(0,0){\includegraphics[width=\unitlength,page=4]{L_et_Lbord.pdf}}%
    \put(0.7436189,0.14076229){\color[rgb]{0,0,0}\makebox(0,0)[lt]{\lineheight{1.25}\smash{\begin{tabular}[t]{l}$T$\end{tabular}}}}%
    \put(0,0){\includegraphics[width=\unitlength,page=5]{L_et_Lbord.pdf}}%
    \put(0.58402626,0.11683479){\color[rgb]{0,0,0}\makebox(0,0)[lt]{\lineheight{1.25}\smash{\begin{tabular}[t]{l}$_U$\end{tabular}}}}%
    \put(0,0){\includegraphics[width=\unitlength,page=6]{L_et_Lbord.pdf}}%
    \put(0.66431644,0.11616539){\color[rgb]{0,0,0}\makebox(0,0)[lt]{\lineheight{1.25}\smash{\begin{tabular}[t]{l}$_V$\end{tabular}}}}%
    \put(0.89309001,0.11453556){\color[rgb]{0,0,0}\makebox(0,0)[lt]{\lineheight{1.25}\smash{\begin{tabular}[t]{l}$_W$\end{tabular}}}}%
    \put(0.58055577,0.17714226){\color[rgb]{0,0,0}\makebox(0,0)[lt]{\lineheight{1.25}\smash{\begin{tabular}[t]{l}$_X$\end{tabular}}}}%
    \put(0,0){\includegraphics[width=\unitlength,page=7]{L_et_Lbord.pdf}}%
  \end{picture}%
\endgroup%

\end{center}
where $T$ is some ribbon graph in $[0,1]^{3}$. We denote them by $L_T$ and $L^{\partial}_T$ respectively and we want to show that $\pi \circ W(L^{\partial}_T) = \pi \circ W(L_T)$. Recall that $p : \mathcal{L}_{g,n}(H) \to \mathcal{L}_{g,n}(H)/I_{\varepsilon}$ denotes the canonical projection. If $\textstyle \mathbf{x} = \sum_i x_i \otimes v_{1,i} \otimes \ldots \otimes v_{k,i} \in \mathcal{L}_{g,n}(H) \otimes V_1 \otimes \ldots \otimes V_k$ is some tensor, define $\textstyle p(\mathbf{x}) = \sum_i p(x_i) \otimes v_{1,i} \otimes \ldots \otimes v_{k,i}$. Note that $W(L^{\partial}_T)$ is equal to the diagram obtained by plugging the tangle $T$ atop the diagram at the left-hand side of Lemma \ref{lemmeSimplificationWqr}. By \eqref{expressIe} we have $p\bigl(a\overset{X}{C}\bigr) = p(a) \otimes \mathrm{id}_X$ for all $a \in \mathcal{L}_{g,n}(H)$, and hence
\begin{center}
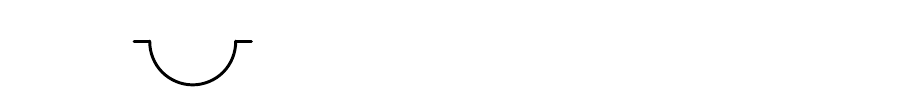
\end{center}
Since the map $\pi$ is just the restriction of $p$ to $\mathcal{L}_{g,n}^H(H)$, applying this equality to the right-hand side of Lemma \ref{lemmeSimplificationWqr} we recover $\pi \circ W(L_T)$, as desired.
\end{proof}

\begin{teo} Assume that $H$ satisfies \eqref{assumptionsForSkein}, $H$-mod is semisimple and the hypotheses of \S \ref{Lgnqrsection} hold true. Then $W^{\rm qr}\colon \mathcal{S}_H(\Sigma_{g,n}) \to \mathcal{L}_{g,n}^{\mathrm{qr}}(H)$ is an isomorphism of algebras.
\end{teo}
\noindent Note that in particular this result applies to $H = U_q^{\mathrm{ad}}(\mathfrak{g})$, where as usual the ribbon links are colored by the objects and morphisms from the category of finite-dimensional $U_q^{\mathrm{ad}}(\mathfrak{g})$-modules of type $1$.
\begin{proof} Under the hypothesis, $W\colon \mathcal{S}_H(\Sigma_{g,n}^{\circ}) \overset{\sim}{\to} \mathcal{L}^H_{g,n}(H)$ is an isomorphism of algebras (Theorem \ref{thWilsonIso}) and $\pi : \mathcal{L}_{g,n}^H(H) \to \mathcal{L}_{g,n}^{\rm qr}(H)$ is surjective (Lemma \ref{lemsepare}). It follows from its definition in \eqref{defWqr} that $W^{\mathrm{qr}}$ is surjective. To have diagrams with a reasonable size, we prove injectivity for $(g,n) = (1,1)$; this is completely representative of the general situation. So let $L \in \mathcal{S}_H(\Sigma_{1,1}^{\circ})$ and assume that $\pi \circ W(L) = 0$. Then by Proposition \ref{propDescriptionKerPi} we can write
\begin{center}
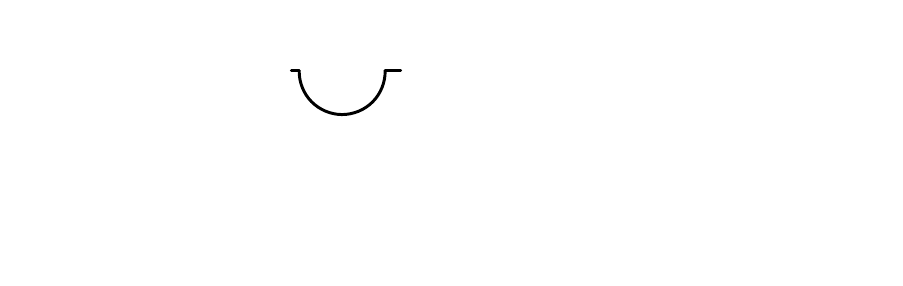
\end{center}
with $X_1,X_2,X_3, Y \in \mathrm{Irr}(H)$, only a finite number of coefficients $\lambda_{X_1,X_2,X_3,Y} \in \mathbb{C}$ are non-zero and $f_{X_1,X_2,X_3,Y} \in \mathrm{Hom}_H\bigl( X_1 \otimes X_1^* \otimes \ldots \otimes X_3 \otimes X_3^* \otimes Y \otimes Y^*, k \bigr)$ where $k$ is the trivial $H$-module. Consider the following ribbon graphs in $[0,1]^3$:
\begin{center}
\begingroup%
  \makeatletter%
  \providecommand\color[2][]{%
    \errmessage{(Inkscape) Color is used for the text in Inkscape, but the package 'color.sty' is not loaded}%
    \renewcommand\color[2][]{}%
  }%
  \providecommand\transparent[1]{%
    \errmessage{(Inkscape) Transparency is used (non-zero) for the text in Inkscape, but the package 'transparent.sty' is not loaded}%
    \renewcommand\transparent[1]{}%
  }%
  \providecommand\rotatebox[2]{#2}%
  \newcommand*\fsize{\dimexpr\f@size pt\relax}%
  \newcommand*\lineheight[1]{\fontsize{\fsize}{#1\fsize}\selectfont}%
  \ifx\svgwidth\undefined%
    \setlength{\unitlength}{322.7697498bp}%
    \ifx\svgscale\undefined%
      \relax%
    \else%
      \setlength{\unitlength}{\unitlength * \real{\svgscale}}%
    \fi%
  \else%
    \setlength{\unitlength}{\svgwidth}%
  \fi%
  \global\let\svgwidth\undefined%
  \global\let\svgscale\undefined%
  \makeatother%
  \begin{picture}(1,0.19504931)%
    \lineheight{1}%
    \setlength\tabcolsep{0pt}%
    \put(0,0){\includegraphics[width=\unitlength,page=1]{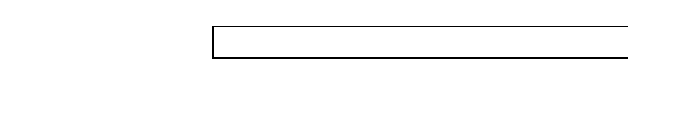}}%
    \put(0.54638063,0.12588952){\color[rgb]{0,0,0}\makebox(0,0)[lt]{\lineheight{1.25}\smash{\begin{tabular}[t]{l}$f_{X_1,X_2,X_3,Y}$\end{tabular}}}}%
    \put(0,0){\includegraphics[width=\unitlength,page=2]{def_f_prime.pdf}}%
    \put(0.27954854,0.16179435){\color[rgb]{0,0,0}\makebox(0,0)[lt]{\lineheight{1.25}\smash{\begin{tabular}[t]{l}$_Y$\end{tabular}}}}%
    \put(0.349665,0.0074577){\color[rgb]{0,0,0}\makebox(0,0)[lt]{\lineheight{1.25}\smash{\begin{tabular}[t]{l}$_{X_1}$\end{tabular}}}}%
    \put(0.43155228,0.00702286){\color[rgb]{0,0,0}\makebox(0,0)[lt]{\lineheight{1.25}\smash{\begin{tabular}[t]{l}$_{X_2}$\end{tabular}}}}%
    \put(0.50907239,0.00635185){\color[rgb]{0,0,0}\makebox(0,0)[lt]{\lineheight{1.25}\smash{\begin{tabular}[t]{l}$_{X_1}$\end{tabular}}}}%
    \put(0.59080743,0.00674274){\color[rgb]{0,0,0}\makebox(0,0)[lt]{\lineheight{1.25}\smash{\begin{tabular}[t]{l}$_{X_2}$\end{tabular}}}}%
    \put(0.67568092,0.00652613){\color[rgb]{0,0,0}\makebox(0,0)[lt]{\lineheight{1.25}\smash{\begin{tabular}[t]{l}$_{X_3}$\end{tabular}}}}%
    \put(0.75297783,0.00671831){\color[rgb]{0,0,0}\makebox(0,0)[lt]{\lineheight{1.25}\smash{\begin{tabular}[t]{l}$_{X_3}$\end{tabular}}}}%
    \put(0.98688347,0.16874702){\color[rgb]{0,0,0}\makebox(0,0)[lt]{\lineheight{1.25}\smash{\begin{tabular}[t]{l}$_Y$\end{tabular}}}}%
    \put(-0.00047332,0.09307643){\color[rgb]{0,0,0}\makebox(0,0)[lt]{\lineheight{1.25}\smash{\begin{tabular}[t]{l}$T(f_{X_1,X_2,X_3,Y}) =$\end{tabular}}}}%
  \end{picture}%
\endgroup%

\end{center}
and using the notations $L_T$, $L^{\partial}_T$ introduced in the proof of Proposition \ref{prop_Wqr_passe_au_quotient} define
\[ L' = \sum_{X_1,X_2,X_3,Y} \lambda_{X_1,X_2,X_3,Y} \bigl( L_{T(f_{X_1,X_2,X_3,Y})}^{\partial} - L_{T(f_{X_1,X_2,X_3,Y})} \bigr) \in \mathcal{S}_H(\Sigma_{1,1}^{\circ}). \]
Note that $L' \sim_{\partial} 0$, simply because $L^{\partial}_T \sim_{\partial} L_T$ for any ribbon graph $T$. Moreover, by definition of $W$ and by Lemma \ref{lemmeSimplificationWqr}, we have $W(L') = W(L)$. It follows that $L' = L$ since $W$ is an isomorphism (Theorem \ref{thWilsonIso}). Thus $L \sim_{\partial} 0$, which means that $L = 0$ in $\mathcal{S}_H(\Sigma_{1,1})$.
\end{proof}

\end{document}